\documentclass[11pt, a4paper,leqno]{amsart}
\usepackage{amsmath,amsthm,amscd,amssymb,amsfonts, amsbsy}
\usepackage{latexsym}
\usepackage{txfonts}
\usepackage{exscale}

\usepackage[colorlinks,citecolor=red,pagebackref,hypertexnames=false]{hyperref}

\usepackage{pgf}

\newcommand{\kaj}[1]{\marginpar{\scriptsize \textbf{Kaj:} #1}}

\calclayout
\allowdisplaybreaks

\theoremstyle{plain}
\newtheorem{theorem}[equation]{Theorem}
\newtheorem{lemma}[equation]{Lemma}
\newtheorem{corollary}[equation]{Corollary}

\theoremstyle{definition}
\newtheorem{definition}[equation]{Definition}

\theoremstyle{remark}
\newtheorem{remark}[equation]{Remark}

\numberwithin{equation}{section}

\renewcommand{\P}{\mathcal{P}}

\def\mean#1{\mathchoice%
          {\mathop{\kern 0.2em\vrule width 0.6em height 0.69678ex depth -0.58065ex
                  \kern -0.8em \intop}\nolimits_{\kern -0.4em#1}}%
          {\mathop{\kern 0.1em\vrule width 0.5em height 0.69678ex depth -0.60387ex
                  \kern -0.6em \intop}\nolimits_{#1}}%
          {\mathop{\kern 0.1em\vrule width 0.5em height 0.69678ex
              depth -0.60387ex
                  \kern -0.6em \intop}\nolimits_{#1}}%
          {\mathop{\kern 0.1em\vrule width 0.5em height 0.69678ex depth -0.60387ex
                  \kern -0.6em \intop}\nolimits_{#1}}}

\def\div{\mathop{\operatorname{div}}\nolimits}
\def\adj{\mathop{\operatorname{adj}}\nolimits}

\begin{document}
\title[$L^2$ solvability of BVPs for divergence form parabolic equations]{$L^2$ Solvability of boundary value problems\\ for divergence form parabolic equations \\with complex coefficients}

\address{Kaj Nystr\"{o}m\\Department of Mathematics, Uppsala University\\
S-751 06 Uppsala, Sweden}
\email{kaj.nystrom@math.uu.se}

\author{K. Nystr{\"o}m}

\maketitle
\begin{abstract}
\noindent\medskip
 We consider parabolic operators of the form $$\partial_t+\mathcal{L},\ \mathcal{L}=-\mbox{div}\, A(X,t)\nabla,$$ in
$\mathbb R_+^{n+2}:=\{(X,t)=(x,x_{n+1},t)\in \mathbb R^{n}\times \mathbb R\times \mathbb R:\ x_{n+1}>0\}$, $n\geq 1$. We assume that $A$ is a $(n+1)\times (n+1)$-dimensional matrix which is bounded, measurable, uniformly elliptic and complex, and we assume, in addition, that the entries of A are independent of the spatial coordinate $x_{n+1}$ as well as of the time coordinate $t$. For such operators we prove that the boundedness and invertibility of the corresponding layer potential operators are stable on $L^2(\mathbb R^{n+1},\mathbb C)=L^2(\partial\mathbb R^{n+2}_+,\mathbb C)$ under complex, $L^\infty$ perturbations of the coefficient matrix. Subsequently, using this general result, we establish solvability of the Dirichlet, Neumann and Regularity
problems for $\partial_t+\mathcal{L}$, by way of  layer potentials and with data in $L^2$,  assuming that the coefficient matrix is a small complex perturbation of either
a constant matrix or of a real and symmetric matrix.

\noindent

\medskip

\end{abstract}


    \section{Introduction and statement of main results}
In this paper  we study the solvability of the Dirichlet, Neumann and Regularity
problems with data in $L^2$, in the following these problems are referred to as $(D2)$, $(N2)$ and $(R2)$, see \eqref{eq10} below for the exact definition of these problems, by way of layer potentials and for second order parabolic equations of the form
    \begin{eqnarray}\label{eq1}
    \mathcal{H}u:=(\partial_t+\mathcal{L})u = 0,
    \end{eqnarray}
     where
    \begin{eqnarray}\label{eq2}
    \mathcal{L}=-\mbox{div }A(X,t)\nabla =-\sum_{i,j=1}^{n+1}\partial_{x_i}(A_{i,j}(X,t)\partial_{x_j})
    \end{eqnarray}
    is defined in $\mathbb R^{n+2}=\{(X,t)=(x_1,..,x_{n+1},t)\in \mathbb R^{n+1}\times\mathbb R\}$, $n\geq 1$. $A=A(X,t)=\{A_{i,j}(X,t)\}_{i,j=1}^{n+1}$ is assumed to be a $(n+1)\times (n+1)$-dimensional matrix with complex coefficients
    satisfying the uniform ellipticity condition
     \begin{eqnarray}\label{eq3}
     (i)&&\Lambda^{-1}|\xi|^2\leq \mbox{Re }A(X,t)\xi\cdot\bar\xi =\mbox{Re }\bigl (\sum_{i,j=1}^{n+1} A_{i,j}(X,t)\xi_i\bar\xi_j\bigr ),\notag\\
     (ii)&& |A(X,t)\xi\cdot\zeta|\leq \Lambda|\xi||\zeta|,
    \end{eqnarray}
    for some $\Lambda$, $1\leq\Lambda<\infty$, and for all $\xi,\zeta\in \mathbb C^{n+1}$, $(X,t)\in\mathbb R^{n+2}$. Here $u\cdot v=u_1v_1+...+u_{n+1}v_{n+1}$,
    $\bar u$ denotes the complex conjugate of $u$ and $u\cdot \bar v$ is the standard inner product on $\mathbb C^{n+1}$. In addition, we consistently assume that
     \begin{eqnarray}\label{eq4}
A(x_1,..,x_{n+1},t)=A(x_1,..,x_{n}),\mbox{i.e., $A$ is independent of $x_{n+1}$ and $t$}.
    \end{eqnarray}
    We study $(D2)$, $(N2)$ and $(R2)$ for the operator $\mathcal{H}$ in $\mathbb R^{n+2}_+=\{(x,x_{n+1},t)\in \mathbb R^{n}\times\mathbb R\times\mathbb R:\ x_{n+1}>0\}$,  with data prescribed
on  $\mathbb R^{n+1}=\{(x,x_{n+1},t)\in \mathbb R^{n}\times\mathbb R\times\mathbb R:\ x_{n+1}=0\}$. Assuming \eqref{eq3}-\eqref{eq4}, as well as the De Giorgi-Moser-Nash estimates stated in \eqref{eq14+}-\eqref{eq14++} below, we first prove (Theorem \ref{thper2}, Corollary \ref{thper3}) that the solvability of $(D2)$, $(N2)$ and $(R2)$, by way of layer potentials,  is stable under small complex $L^\infty$ perturbations of the coefficient matrix. Subsequently, using Theorem \ref{thper2}, Corollary \ref{thper3}, we establish the solvability for $(D2)$, $(N2)$ and $(R2)$, by way of layer potentials,  when the coefficient matrix is  either
     \begin{eqnarray}\label{res}
       (i)&&\mbox{a small complex perturbation of a constant}\notag\\
       &&\mbox{(complex) matrix (Theorem \ref{th2-}), or},\notag\\
       (ii)&&\mbox{a real and symmetric matrix (Theorem \ref{th2}), or},\notag\\
       (iii)&&\mbox{a small complex perturbation of a real and}\notag\\
       &&\mbox{symmetric matrix (Theorem \ref{th2+})}.
       \end{eqnarray}
We emphasize that in \eqref{res}  $(i)-(iii)$ the unique solutions can be represented
 in terms of layer potentials and we remark that for the class of operators we consider, solvability of these boundary value problems in the upper half space can readily be generalized, by a change of coordinates, to the geometrical setting of a  domain given as the region above a time-independent Lipschitz graph. We emphasize that already in the case when $A$ is real and symmetric our contribution is twofold. First,  we prove solvability of $(D2)$, $(N2)$ and $(R2)$. Second,  we prove solvability of $(D2)$, $(N2)$ and $(R2)$ by way of layer potentials. To our knowledge Theorem \ref{thper2},
 Corollary \ref{thper3}, Theorem \ref{th2-}, Theorem \ref{th2}, and Theorem \ref{th2+} are all new, see subsection \ref{hist} below for an outline of the state of the art of second order parabolic boundary value problems with non-smooth coefficients, but we note that in \cite{CNS} we, together with A. Castro and O. Sande, develop some of the estimates used in this paper. \cite{CNS} should be seen as a companion to this paper. We claim that our results, and the tools developed, pave the way for important developments in the area of parabolic PDEs, see Remark \ref{fut1} and Remark \ref{fut2}  below.

 The main results of this paper  can be seen as parabolic analogues of the elliptic results established in \cite{AAAHK} and we recall that in \cite{AAAHK} the authors establish results concerning the solvability of $(D2)$, $(N2)$ and $(R2)$, by way of  layer potentials and for elliptic operators of the form $-\mbox{div}\, A(X)\nabla,$ in
$\mathbb R_+^{n+1}:=\{X=(x,x_{n+1})\in \mathbb R^{n}\times \mathbb R:\ x_{n+1}>0\}$, $n\geq 1$, assuming  that $A$ is a $(n+1)\times (n+1)$-dimensional matrix which is bounded, measurable, uniformly elliptic and complex, and assuming, in addition, that the entries of A are independent of the spatial coordinate $x_{n+1}$. If  $A$ is also real and symmetric,  $(D2)$, $(N2)$ and $(R2)$ were solved in \cite{JK}, \cite{KP}, \cite{KP1}, and the major achievement in \cite{AAAHK} is that the authors prove that solutions can be represented by way of  layer potentials. We refer to \cite{AAAHK} for a thorough account of the history of these problems in the context of elliptic equations. In \cite{HMM} a version of
 \cite{AAAHK}, but in the context of $L^p$ and relevant endpoint spaces, was developed and in \cite{HMaMi} the structural assumption that A is independent of the spatial coordinate $x_{n+1}$ is challenged. The core of the impressive arguments and estimates in \cite{AAAHK} is based on the fine and elaborated techniques developed in  the proof of the Kato conjecture, see \cite{AHLMcT} and \cite{AHLeMcT}, \cite{HLMc}. In this context it is also relevant to mention the novel approach to the Dirichlet, Neumann and Regularity
problems developed in \cite{AAM}, \cite{AA}, and \cite{AR}. This approach is based on a reduction of the PDE to a first order system which is then solved using functional calculus.

While our set up and our results coincide, in the stationary case, with the set up  and results established in   \cite{AAAHK} for  elliptic equations, we claim that our results are not straightforward generalizations of the corresponding results in \cite{AAAHK}. First, our results rely on \cite{N} where certain square function estimates are established for second order parabolic operators of the form $\mathcal{H}$, and where, in particular,  a parabolic version of the technology in \cite{AHLMcT} is developed. Second, in general the presence of the (first order) time-derivative forces us to consider fractional time-derivatives leading, as in \cite{LM}, \cite{HL}, \cite{H}, see also \cite{HL1}, to rather elaborate additional estimates. Having said this we acknowledge, once and for all, the influence
 that the work in \cite{AAAHK} has had on our understanding of the topic, and on this paper, and we believe that \cite{AAAHK} as well as this paper represent important contributions to the theory of partial differential equations.

\subsection{Statement of main results} Consider $\mathcal{H}=\partial_t+\mathcal{L}=\partial_t-\div A\nabla$. We let
 $\mathcal{H}^\ast$ be the hermitian adjoint of $\mathcal{H}$, i.e.,
 $$\int_{\mathbb R^{n+2}} (\mathcal{H}\psi)\bar\phi\, dXdt=\int_{\mathbb R^{n+2}} \psi(\overline{\mathcal{H}^\ast\phi})\, dXdt,$$
whenever $\phi$, $\psi\in C_0^\infty(\mathbb R^{n+2},\mathbb C)$.
Then $\mathcal{H}^\ast=-\partial_t+\mathcal{L}^\ast=-\partial_t-\div A^\ast\nabla$, where $\mathcal{L}^\ast$ and  $A^\ast=\bar A$ are the hermitian adjoints of $\mathcal{L}$ and $A$, respectively. The following are our main results.

     \begin{theorem}\label{thper2} Consider two operators $\mathcal{H}_0=\partial_t-\div A^0\nabla$, $\mathcal{H}_1=\partial_t-\div A^1\nabla$. Assume that $\mathcal{H}_0$,  $\mathcal{H}_0^\ast$,  $\mathcal{H}_1$,  $\mathcal{H}_1^\ast$ satisfy \eqref{eq3}-\eqref{eq4} as well as the De Giorgi-Moser-Nash estimates stated in \eqref{eq14+}-\eqref{eq14++} below.  Assume that
     \begin{eqnarray*}\label{afa}
    &&\mbox{$\mathcal{H}_0$,  $\mathcal{H}_0^\ast$, have bounded, invertible and good layer potentials in the sense}\notag\\
      &&\mbox{of Definition \ref{blayer+}, for some constant $\Gamma_0$}.
     \end{eqnarray*}
     Then there exists a constant $\varepsilon_0$, depending at most
     on $n$, $\Lambda$,  the De Giorgi-Moser-Nash constants and $\Gamma_0$,  such that if
    $$||A^1-A^0||_\infty\leq\varepsilon_0,$$
    then there exists a constant $\Gamma_1$,  depending at most
     on $n$, $\Lambda$,  the De Giorgi-Moser-Nash constants and  $\Gamma_0$, such that
     \begin{eqnarray*}\label{afa+}
  &&\mbox{$\mathcal{H}_1$,  $\mathcal{H}_1^\ast$, have bounded, invertible and good layer potentials in the sense}\notag\\
      &&\mbox{of Definition \ref{blayer+}, with constant $\Gamma_1$}.
     \end{eqnarray*} \end{theorem}

      \begin{corollary}\label{thper3} Consider two operators $\mathcal{H}_0=\partial_t-\div A^0\nabla$, $\mathcal{H}_1=\partial_t-\div A^1\nabla$. Assume that $\mathcal{H}_0$,  $\mathcal{H}_0^\ast$,  $\mathcal{H}_1$,  $\mathcal{H}_1^\ast$ satisfy \eqref{eq3}-\eqref{eq4} as well as the De Giorgi-Moser-Nash estimates stated in \eqref{eq14+}-\eqref{eq14++} below.  Assume that
     \begin{eqnarray*}\label{afaa}
     &&\mbox{$(D2)$, $(N2)$ and $(R2)$ are uniquely solvable, for the operators}\notag\\
      &&\mbox{$\mathcal{H}_0$, $\mathcal{H}_0^\ast$, by way of layer potentials and for a constant $\Gamma_0$, }\notag\\
      &&\mbox{in the sense of Definition \ref{sollayer}.}
     \end{eqnarray*}
Then there exists a constant $\varepsilon_0$, depending at most
     on $n$, $\Lambda$,  the De Giorgi-Moser-Nash constants and  $\Gamma_0$,  such that if
    $$||A^1-A^0||_\infty\leq\varepsilon_0,$$
    then then there exists a constant $\Gamma_1$,  depending at most
     on $n$, $\Lambda$,  the De Giorgi-Moser-Nash constants and  $\Gamma_0$, such that
         \begin{eqnarray*}\label{afaa+}
&&\mbox{$(D2)$, $(N2)$ and $(R2)$ are uniquely solvable, for the operators}\notag\\
      &&\mbox{$\mathcal{H}_1$, $\mathcal{H}_1^\ast$, by way of layer potentials and with constant $\Gamma_1$, }\notag\\
      &&\mbox{in the sense of Definition \ref{sollayer}.}
     \end{eqnarray*}
    \end{corollary}

 \begin{theorem}\label{th2-}  Consider two operators $\mathcal{H}_0=\partial_t-\div A^0\nabla$, $\mathcal{H}_1=\partial_t-\div A^1\nabla$. Assume that $A^0$ is constant and that $\mathcal{H}_0$, $\mathcal{H}_1$ satisfy \eqref{eq3}-\eqref{eq4}. Then there exists a constant $\varepsilon_0$, depending at most
     on $n$, $\Lambda$ such that if
    $$||A^1-A^0||_\infty\leq\varepsilon_0,$$
      then $(D2)$ for the operator $\mathcal{H}_1$ has a unique solution and $(N2)$ and $(R2)$ for the operator $\mathcal{H}_1$ have unique solutions
    modulo constants. The solutions can be represented
 in terms of layer potentials. \end{theorem}

 \begin{theorem}\label{th2} Assume that $\mathcal{H}=\partial_t-\div A\nabla$ satisfies \eqref{eq3}-\eqref{eq4}. Assume in addition that
 $A$ is real and symmetric. Then $(D2)$ for the operator $\mathcal{H}$ has a unique solution and $(N2)$ and $(R2)$ for the operator $\mathcal{H}$ have unique solutions
    modulo constants. The solutions can be represented
 in terms of layer potentials.
               \end{theorem}

                   \begin{theorem}\label{th2+} Assume that $\mathcal{H}_0=\partial_t-\div A^0\nabla$, $\mathcal{H}_1=\partial_t-\div A^1\nabla$ satisfy \eqref{eq3}-\eqref{eq4}. Assume that $A^0$ is real and symmetric. Then there exists a constant $\varepsilon_0$, depending at most
     on $n$, $\Lambda$ such that if
    $$||A^1-A^0||_\infty\leq\varepsilon_0,$$
    then $(D2)$ for the operator $\mathcal{H}_1$ has a unique solution and $(N2)$ and $(R2)$ for the operator $\mathcal{H}_1$ have unique solutions
    modulo constants. The solutions can be represented
 in terms of layer potentials. \end{theorem}

               \begin{remark} Assuming \eqref{eq3}-\eqref{eq4}, as well as the De Giorgi-Moser-Nash estimates stated in \eqref{eq14+}-\eqref{eq14++} below, Theorem \ref{thper2} states that the property of having bounded, invertible and good layer potentials in the sense of Definition \ref{blayer+}  is stable under small complex $L^\infty$ perturbations of the coefficient matrix. Corollary \ref{thper3} emphasizes that the solvability of $(D2)$, $(N2)$ and $(R2)$, is stable under small complex $L^\infty$ perturbations of the coefficient matrix.
\end{remark}

\begin{remark} Note that Theorem \ref{th2-} gives the existence and uniqueness for $(D2)$, $(N2)$ and $(R2)$ whenever the matrix $A^1$ is a small perturbation of a constant (complex) matrix
 $A^0$. Theorem \ref{th2} states that we  have existence and uniqueness for $(D2)$, $(N2)$ and $(R2)$  when $A$  is real and symmetric and
 satisfies \eqref{eq3}-\eqref{eq4}. Theorem \ref{th2+} states that the latter result is true whenever $A^1$ is a (small) complex perturbation of
 a  real and symmetric matrix $A^0$. In all cases the unique solutions can be represented
 in terms of layer potentials.
 \end{remark}

 \begin{remark}\label{fut1} In forthcoming papers we intend  to generalize the present paper  to the context of $L^p$ and relevant endpoint spaces, and to  challenge the assumption in \eqref{eq4}. The ambition is to develop parabolic versions of  \cite{HMM}, \cite{HMaMi}, and \cite{HKMP}.
 \end{remark}

  \begin{remark}\label{fut2} The underlying theme of  this paper, as well as in \cite{AAAHK}, is to basically reduce all estimates to two core estimates involving single layer potentials. To briefly discuss this, and to be consistent with the notation used in the bulk of the paper,  we let, based on \eqref{eq4},  $\lambda=x_{n+1}$ and when using the symbol
    $\lambda$  we will write the point $(X,t)=(x_1,..,x_{n},x_{n+1},t)$ as $ (x,t, \lambda)=(x_1,..,x_{n},t,\lambda)$.  We let $L^2(\mathbb R^{n+1},\mathbb C)$ denote the standard Hilbert space of functions $f:\mathbb R^{n+1}\to \mathbb C$ which are square integrable and we let $||f||_2$ denote the norm of $f$.   We let
\begin{eqnarray}\label{keyestint-ex+-}
|||\cdot|||_\pm=\biggl (\int_{\mathbb R^{n+2}_\pm}|\cdot|^2\, \frac{dxdtd\lambda}{|\lambda|}\biggr )^{1/2},
\end{eqnarray}
where $\mathbb R^{n+2}_\pm=\{(x,t,\lambda)\in \mathbb R^{n}\times\mathbb R\times\mathbb R:\ \pm\lambda>0\}$. In our case the core estimates referred to above are embedded in the statement that $\mathcal{H}$,  $\mathcal{H}^\ast$ have bounded, invertible and good layer  potentials with constant $\Gamma \geq 1$, see display \eqref{keyestint-} in Definition \ref{blayer+}. The estimates read
   \begin{eqnarray}\label{keyestint-intro}
     (i)&&\sup_{\lambda\neq 0}||\partial_\lambda \mathcal{S}_{\lambda}^{\mathcal{H}}f||_2+\sup_{\lambda\neq 0}||\partial_\lambda \mathcal{S}_{\lambda}^{\mathcal{H}^\ast}f||_2\leq \Gamma||f||_2,\notag\\
     (ii)&&|||\lambda\partial_\lambda^2 \mathcal{S}_{\lambda}^{\mathcal{H}}f|||_\pm+|||\lambda\partial_\lambda^2 \mathcal{S}_{\lambda}^{\mathcal{H}^\ast}f|||_\pm\leq  \Gamma||f||_2,
     \end{eqnarray}
     whenever $f\in L^2(\mathbb R^{n+1},\mathbb C)$ and where $\mathcal{S}_{\lambda}^{\mathcal{H}}f$ and $\mathcal{S}_{\lambda}^{\mathcal{H}^\ast}f $ are the single layer potentials associated to $\mathcal{H}$ and $\mathcal{H}^\ast$, respectively. See \eqref{eq11} for the definition of $\mathcal{S}_{\lambda}^{\mathcal{H}}f$ and $\mathcal{S}_{\lambda}^{\mathcal{H}^\ast}f$.  Note that \eqref{keyestint-intro} $(i)$ is a uniform (in $\lambda$) $L^2$-estimate involving the first order partial derivative, in the $\lambda$-coordinate, of the single layer potentials, while  \eqref{keyestint-intro} $(ii)$  is a square function estimate involving the second order partial derivatives, in the $\lambda$-coordinate, of the single layer potentials. A key technical challenge in the proof of Theorem \ref{thper2}, Corollary \ref{thper3} is to prove that these estimates are stable under small complex perturbations of the coefficient matrix. However, in the elliptic case and after \cite{AAAHK} appeared, it was proved in \cite{R}, see \cite{GH} for an alternative proof, that if $-\mbox{div}\, A(X)\nabla$ satisfies the basic assumptions imposed in
     \cite{AAAHK}, then the elliptic version of \eqref{keyestint-intro} $(ii)$ always holds. In fact, the approach in \cite{R}, which is based on functional calculus, even dispenses of the De Giorgi-Moser-Nash estimates underlying \cite{AAAHK}. Furthermore, in the elliptic case \eqref{keyestint-intro} $(ii)$ can be seen to imply \eqref{keyestint-intro} $(i)$ by the results of \cite{AAAHK} and \cite{AA}. Hence, in the elliptic case, and under the assumptions of \cite{AAAHK},  the elliptic version of \eqref{keyestint-intro} always holds. Based on this it is fair to pose the question whether or not a similar line of development can be anticipated in the parabolic case. Based on \cite{N}, this paper and \cite{CNS}, we anticipated that a parabolic version of \cite{GH} may be possible to develop and this is currently work in progress. To develop a parabolic version of \cite{AA} is a very interesting project.
 \end{remark}

 \subsection{Relation to the literature}\label{hist} To put our work in context, and to briefly outline previous work devoted
to parabolic singular integral operators, parabolic  layer potentials, as well as the Dirichlet, Neumann and Regularity
problems with data in $L^2$ and $L^p$,  for second order parabolic operators in divergence form, it is fair to first mention \cite{FR}, \cite{FR1}, \cite{FR2} where a theory of singular integral operators with mixed homogeneity was developed in the context of time-independent $C^1$-cylinders. In the setting of time-independent Lipschitz cylinders and the heat equation,  $(D2)$ was solved in \cite{FS}, while
$(D2)$, $(N2)$ and $(R2)$ were solved in \cite{B}, \cite{B1} by way of  layer potentials. In this context the natural pull-back of the heat operator to a half-space is a second order parabolic operator of the form $\mathcal{H}$ with defining matrix $A$ being real, symmetric, and satisfying \eqref{eq3}-\eqref{eq4}. A major breakthrough in the field, in the setting of time-dependent Lipschitz type cylinders and the heat equation, was achieved in \cite{LS}, \cite{LM}, \cite{HL}, \cite{H}, see also \cite{HL1}. In particular, in these papers the correct notion of time-dependent Lipschitz type cylinders, correct from the perspective of parabolic singular integral operators, parabolic  layer potentials, parabolic measure, as well as the Dirichlet, Neumann and Regularity
problems with data in $L^p$ for the heat operator, was found. In \cite{HL} the authors solved $(D2)$, $(N2)$ and $(R2)$ for the heat operator. The Neumann and Regularity
problems with data in $L^p$ were considered  in \cite{HL2} and \cite{HL3}. Due to the modest regularity assumption in the time-direction imposed in \cite{LM}, \cite{HL}, \cite{H}, in this setting a more elaborate pull-back to a half-space has to be employed and the resulting operator, in the case of the heat operator, turns out to be an operator of the form $$\mathcal{H}-B\cdot\nabla=\partial_t-\mbox{div}\, A(X,t)\nabla-B\cdot\nabla,$$ where the term $B\cdot\nabla$ is a singular drift term. In this case, $A$ and $B$ will in general depend on $x_{n+1}$ as well as $t$, i.e., $A$ will not satisfy \eqref{eq4}. Instead the geometry underlying \cite{LM}, \cite{HL}, \cite{H}, will reveal itself through the fact that certain measures, defined based on $A$ and $B$, turn out to be Carleson measures. The fine properties of associated parabolic measures were analyzed in the impressive and influential work \cite{HL4}, this work also being strongly influential in the solution of the Kato conjecture, see \cite{AHLeMcT}. A fine contribution to the field, simplifying parts of  \cite{HL4}, was given in \cite{NR}.

\subsection{Proofs and organization of the paper} Concerning the proofs of our main results it is fair to say that the skeleton of our proofs is similar to
the skeleton of \cite{AAAHK}. However, due to the presence of the time derivative many of the important details are different. To briefly discuss proofs, and the organization of the paper, we need to introduce some notation. Based on \eqref{eq4} we let $\lambda=x_{n+1}$ and when using the symbol
    $\lambda$ we will write the point $(X,t)=(x_1,..,x_{n},x_{n+1},t)$ as $ (x,t, \lambda)=(x_1,..,x_{n},t,\lambda)$. Using this notation, and assuming \eqref{eq3}-\eqref{eq4}, we study $(D2)$, $(N2)$ and $(R2)$ for the operator $\mathcal{H}$ in $$\mathbb R^{n+2}_+=\{(x,t,\lambda)\in \mathbb R^{n}\times\mathbb R\times\mathbb R:\ \lambda>0\},$$ with data prescribed
on $$\mathbb R^{n+1}=\{(x,t,\lambda)\in \mathbb R^{n}\times\mathbb R\times\mathbb R:\ \lambda=0\}.$$
We write $\nabla =(\nabla_{||},\partial_\lambda)$ where $\nabla_{||}=(\partial_{x_1},...,\partial_{x_n})$. We let $L^2(\mathbb R^{n+1},\mathbb C)$ denote the standard Hilbert space of functions $f:\mathbb R^{n+1}\to \mathbb C$ which are square integrable, we let $||f||_2$ denote the norm of $f$ and we will use the notation $|||\cdot|||_\pm$ introduced in \eqref{keyestint-ex+-}.  In the following we refer the reader to Section \ref{sec2} for notation and the precise definitions of
     the operators $\mathbb D$, $H_t$, $D^t_{1/2}$, the non-tangential maximal operators $N_\ast^\pm$, $\tilde N_\ast^\pm$, and the parabolic Sobolev space $\mathbb H=\mathbb H(\mathbb R^{n+1},\mathbb C)$.

In Section \ref{sec2}, which is of preliminary nature, we introduce notation, function spaces, weak solutions, state energy estimates, and we introduce non-tangential maximal functions and the problems
$(D2)$, $(N2)$ and $(R2)$. We here also  state the De Giorgi-Moser-Nash estimates referred to in the statements of Theorem \ref{thper2} and
Corollary \ref{thper3}, we introduce layer potentials and we state Definition \ref{blayer+} and Definition \ref{sollayer}.

In Section \ref{sec7} we establish a number of harmonic analysis results and collect some  of the results from \cite{N} to be used in the sequel. In particular, we introduce the resolvent and establish the existence of  a parabolic Hodge decomposition. We collect estimates from \cite{N} concerning
uniform (in $\lambda$) $L^2$-estimates and off-diagonal estimates, square function estimates for resolvents, and the Littlewood-Paley theory. In this section  we also prove some consequences of uniform (in $\lambda$) $L^2$-estimates and off-diagonal estimates.

In Section \ref{sec3} we collect and prove a number of estimates related to the boundedness of single layer potentials: off-diagonal estimates, uniform  (in $\lambda$) $L^2$-estimates, estimates of non-tangential maximal functions and square functions. Much of the material in this section is a summary of the key results established in \cite{CNS}. The essence of the results stated in Section \ref{sec3} is that if  $\mathcal{H}=\partial_t-\div A\nabla$ satisfies \eqref{eq3}-\eqref{eq4}, and if we let
 $\mathcal{S}_\lambda=\mathcal{S}_\lambda^{\mathcal{H}}$, $\mathcal{S}_\lambda^\ast=\mathcal{S}_\lambda^{\mathcal{H}^\ast}$ denote the single layer potentials associated to $\mathcal{H}$ and $\mathcal{H}^\ast$, respectively, then the $L^2$-norms of non-tangential maximal functions in the upper half-space, $||N_\ast^+(\partial_\lambda \mathcal{S}_\lambda f)||_2$, $||\tilde N_\ast^+(\nabla_{||}\mathcal{S}_\lambda f)||_2$, $||\tilde N_\ast^+(H_tD^t_{1/2}\mathcal{S}_\lambda f)||_2$, appropriate square functions involving partial derivatives, and fractional in time derivatives, of $\mathcal{S}_\lambda f$, as well as the Sobolev semi-norms $||\mathbb D \mathcal{S}_\lambda f||_2$, can be bounded by a constant times
\begin{eqnarray}\label{keyestint-ex}
\Phi_+(f)+||f||_2^2,
     \end{eqnarray}
     where
     \begin{eqnarray}\label{keyestint-ex+}
\Phi_+(f):=\sup_{\lambda> 0}||\partial_\lambda \mathcal{S}_\lambda f||_2+|||\lambda\partial_\lambda^2 \mathcal{S}_{\lambda}f|||_+.
     \end{eqnarray}
The same results hold with $\mathbb R_+^{n+2}$, $N_\ast^+$, $\tilde N_\ast^+$,  replaced by  $\mathbb R_-^{n+2}$, $N_\ast^-$, $\tilde N_\ast^-$,  and with $\mathcal{S}_\lambda$ replaced by $\mathcal{S}_\lambda^\ast$. In Section \ref{sec3} we also, in analogy with  \cite{AAAHK}, introduced smoothed single layer potentials $\mathcal{S}_\lambda^\eta=\mathcal{S}_\lambda^{\mathcal{H},\eta}$ in order to make certain otherwise formal manipulations rigorous. In particular, in contrast to $\partial_\lambda\mathcal{S}_\lambda$,  $\partial_\lambda\mathcal{S}_\lambda^\eta$ does not, for $\eta>0$, jump across the boundary defined by $\mathbb R^{n+1}$.

     In Section \ref{sec4} we are concerned with boundary traces theorems for weak solutions, weak solutions for which the appropriate non-tangential maximal functions are controlled,  and the existence of boundary layer potentials. In particular, assuming that
        \begin{eqnarray}\label{keyestint-ex+++}
\sup_{\lambda\neq 0}\biggl (||\partial_\lambda \mathcal{S}_\lambda||_{2\to 2}+||\partial_\lambda \mathcal{S}_\lambda^\ast||_{2\to 2}+||\mathbb D \mathcal{S}_\lambda||_{2\to 2}+||\mathbb D \mathcal{S}_\lambda^\ast||_{2\to 2}\biggr )<\infty
     \end{eqnarray}
     we prove, see Lemma \ref{trace4}, the existence of  boundary layer potential operators
     $$\mp\frac 12+\mathcal{K},\ \pm\frac 12+\tilde{\mathcal{K}},\ \mathbb D \mathcal{S}_\lambda|_{\lambda=0},$$
     relevant to the solution of $(D2)$, $(N2)$ and $(R2)$, respectively.        By the results of Section \ref{sec3}, \eqref{keyestint-ex+++}  holds whenever the key estimates in \eqref{keyestint-} of  Definition \ref{blayer+}, see \eqref{keyestint-intro} above, hold. At this stage we prove that the boundary layer potential operators exist in the sense of weak limits in $L^2(\mathbb R^{n+1},\mathbb C)$ as $\pm\lambda\to 0$.

     In  Section \ref{sec4+} we establish the uniqueness of solutions to $(D2)$, $(N2)$ and $(R2)$.

     In Section \ref{sec4++} we are concerned with the existence of non-tangential limits of layer potentials. In particular, we prove, under assumptions, that the weak limits established in Section \ref{sec4} can be strengthened to strong limits in the non-tangential sense.

     Starting from Section \ref{sec8}, the rest of the paper is  devoted to the proof of Theorem \ref{thper2}, Corollary \ref{thper3} and  Theorem \ref{th2-}-Theorem \ref{th2+}.  The smoothed single layer potentials operators $\mathcal{S}_\lambda^{\mathcal{H}_0,\eta}$ and $\mathcal{S}_\lambda^{\mathcal{H}_1,\eta}$ are introduced in \eqref{sop}. The proof of
     Theorem \ref{thper2} is based on a representation formula for the difference $\partial_\lambda \mathcal{S}_\lambda^{\mathcal{H}_1,\eta}f(x,t)- \partial_\lambda\mathcal{S}_\lambda^{\mathcal{H}_0,\eta}f(x,t)$. Indeed,
    \begin{eqnarray}\label{iest1mod++preint}
\partial_\lambda \mathcal{S}_\lambda^{\mathcal{H}_1,\eta}f(x,t)- \partial_\lambda\mathcal{S}_\lambda^{\mathcal{H}_0,\eta}f(x,t)&=&\mathcal{H}_0^{-1}\div {\bf \varepsilon}\nabla D_{n+1}
\mathcal{S}_\cdot^{\mathcal{H}_1,\eta}f(x,t),\notag\\
\quad\quad\lambda\partial_\lambda^2 \mathcal{S}_\lambda^{\mathcal{H}_1,\eta}f(x,t)- \lambda^2\partial_\lambda\mathcal{S}_\lambda^{\mathcal{H}_0,\eta}f(x,t)&=&\lambda\mathcal{H}_0^{-1}\partial_\lambda\div {\bf \varepsilon}\nabla D_{n+1}
\mathcal{S}_\cdot^{\mathcal{H}_1,\eta}f(x,t),
     \end{eqnarray}
     where $D_{n+1}=\partial_{x_{n+1}}=\partial_\lambda$ and
     \begin{eqnarray}\label{keyestaahmoa}{\bf\varepsilon}(x):=A^1(x)-A^0(x).
      \end{eqnarray}
     ${\bf \varepsilon}$ is a (complex) matrix valued function and throughout the paper we assume that $||{\bf \varepsilon}||_\infty\leq\varepsilon_0$. To complete the proof of Theorem \ref{thper2} the idea is to control, in the appropriate $L^2$-sense, and based on the assumptions stated in Theorem \ref{thper2}, the differences or errors defined in \eqref{iest1mod++preint}. To do this several quite involved and technical estimates have to proved. To highlight one such estimate, it becomes important to control
        \begin{eqnarray}\label{keyestint-ex+han}
\int_{0}^\infty\int_{\mathbb R^{n+1}}|\theta_\lambda{\bf \varepsilon}\nabla\mathcal{S}_{\lambda}^{\mathcal{H}_1,\eta}f|^2\,  \frac{dxdtd\lambda}\lambda,
     \end{eqnarray}
      whenever ${f}\in L^2(\mathbb R^{n+1},\mathbb C)$, and where
     \begin{eqnarray*}
          \theta_\lambda{\bf f}:=\lambda^2\partial_\lambda^2 (\mathcal{S}_{\lambda}^{\mathcal{H}_0}\nabla)\cdot{\bf f},
     \end{eqnarray*}
     whenever ${\bf f}\in L^2(\mathbb R^{n+1},\mathbb C^{n+1})$.  We write ${\bf \varepsilon}=({\bf \varepsilon}_1,...,{\bf \varepsilon}_{n+1})$ where ${\bf \varepsilon}_i$, for $i\in \{1,...,n+1\}$, is a $(n+1)$-dimensional column vector, and we let  $\tilde{\bf \varepsilon}$ be the $(n+1)\times n$ matrix defined to equal the first $n$ columns of ${\bf \varepsilon}$, i.e., $\tilde{\bf \varepsilon}=({\bf \varepsilon}_1,...,{\bf \varepsilon}_{n})$. Then
          \begin{eqnarray}\label{decc}
\quad\theta_\lambda{\bf \varepsilon} \nabla \mathcal{S}_\lambda^{\mathcal{H}_1,\eta}f =\theta_\lambda\tilde{\bf \varepsilon}\nabla_{||}\mathcal{S}_\lambda^{\mathcal{H}_1,\eta}f+\mathcal{R}_\lambda\partial_\lambda \mathcal{S}_\lambda^{\mathcal{H}_1,\eta}f+(\theta_\lambda {\bf \varepsilon}_{n+1})\mathcal{P}_\lambda\partial_\lambda \mathcal{S}_\lambda^{\mathcal{H}_1,\eta}f,
\end{eqnarray}
where
\begin{eqnarray}
\mathcal{R}_\lambda=\theta_\lambda{\bf \varepsilon}_{n+1}-(\theta_\lambda {\bf \varepsilon}_{n+1})\P_\lambda,
\end{eqnarray}
and where $\P_\lambda$ is a standard parabolic approximation of the identity. One important step is then to prove that $|\theta_\lambda {\bf \varepsilon}_{n+1}|^2\, \lambda^{-1}dxdtd\lambda$ defines a Carleson measure on $\mathbb R^{n+2}_+$ and that the approximation to the zero operator $\mathcal{R}_\lambda$ can be controlled. This can then be used to control the contribution to \eqref{keyestint-ex+han} from the last two pieces on the right hand side in \eqref{decc}. An other important step is to handle the contribution from $\theta_\lambda\tilde{\bf \varepsilon}\nabla_{||}\mathcal{S}_\lambda^{\mathcal{H}_1,\eta}f$, and to do this we introduce the resolvent $$\mathcal{E}_\lambda^1:=(I+\lambda^2(\partial_t+(\mathcal{L}_1)_{||}))^{-1},$$
defined and analyzed in \cite{N}. Here
 \begin{eqnarray*}\label{block}(\mathcal{L}_1)_{||}=-\div_{||}(A_{||}^1\nabla_{||}),
 \end{eqnarray*}
 and $\div_{||}$ is the divergence operator in the variables $(\partial_{x_1},...,\partial_{x_n})$ only. $A_{||}^1$ is the $n\times n$-dimensional sub matrix of $A^1$ defined by $\{A^1_{i,j}\}_{i,j=1}^n$. Then
    \begin{eqnarray*}\label{block+}\mathcal{L}_1=(\mathcal{L}_1)_{||}-\sum_{j=1}^{n+1}A_{n+1,j}^1D_{n+1}D_j-\sum_{i=1}^{n}D_iA_{i,n+1}D_{n+1},
    \end{eqnarray*}
    where $D_i=\partial_{x_i}$ for $i\in\{1,...,n+1\}$. Using $\mathcal{E}_\lambda^1$ we write
 \begin{eqnarray}
 \theta_\lambda\tilde{\bf \varepsilon}\nabla_{||}\mathcal{S}_\lambda^{\mathcal{H}_1,\eta}f&=&\theta_\lambda\tilde{\bf \varepsilon}\nabla_{||}(I-\mathcal{E}_\lambda^1)\mathcal{S}_\lambda^{\mathcal{H}_1,\eta}f+ \theta_\lambda\tilde{\bf \varepsilon}\nabla_{||}\mathcal{E}_\lambda^1\mathcal{S}_\lambda^{\mathcal{H}_1,\eta}f\notag\\
 &=&\theta_\lambda \tilde{\bf \varepsilon}\nabla_{||}\mathcal{E}_\lambda^1\lambda^2(\partial_t+(\mathcal{L}_1)_{||})\mathcal{S}_\lambda ^{\mathcal{H}_1,\eta}f+ \theta_\lambda\tilde{\bf \varepsilon}\nabla_{||}\mathcal{E}_\lambda^1\mathcal{S}_\lambda^{\mathcal{H}_1,\eta}f.
\end{eqnarray}
To handle the contribution to \eqref{keyestint-ex+han} from the first term on the second line on the right hand side in the last display we have to make use of the recent  square function estimates involving the resolvent $\mathcal{E}_\lambda^1$ established in \cite{N}. As previously mentioned, the estimates in \cite{N} are the parabolic counterparts of the main and hard estimates in \cite{AHLMcT} established in the context of the solution of the Kato conjecture. Using this brief technical digression as a motivation or guide, the rest of the paper is  organized as follows.

In Section \ref{sec8} we prove, using the results of Section \ref{sec7} and techniques and arguments from \cite{N}, certain square function estimates for composed operators involving $\theta_\lambda$ and the resolvents mentioned above. This section is a technical core of the paper.

In Section \ref{sec9}  we establish a number of preliminary technical estimates needed in the proof of  Theorem \ref{thper2}. These estimates rely on the results of Section \ref{sec7} and Section \ref{sec8}.

In Section \ref{sec10} we give the final proof of Theorem \ref{thper2} and Corollary \ref{thper3} and it is fair to say that, at this stage, the proof become notational in line with the corresponding arguments in \cite{AAAHK}. Indeed,  by expanding the errors in \eqref{iest1mod++preint} in a manner similar to \cite{FJK} and
\cite{AAAHK}, we are then in the proof of Theorem \ref{thper2} confronted with a number of pieces. The most involved piece can  be estimated using the technical estimates established in Section \ref{sec9}. To conclude the proof of Theorem \ref{thper2} we then use analytic an perturbation result for our operators, see Lemma \ref{analpert} below, stating that  there exists a constant $c$, depending at most
     on $n$, $\Lambda$,  such that if $||{\bf \varepsilon}||_\infty\leq\varepsilon_0$, then
     \begin{eqnarray}\label{keyest+aaaedhan}
     ||\mathcal{K}^{\mathcal{H}_0}-\mathcal{K}^{\mathcal{H}_1}||_{2\to 2}+||\tilde {\mathcal{K}}^{\mathcal{H}_0}-\tilde {\mathcal{K}}^{\mathcal{H}_1}||_{2\to 2}&\leq& c\varepsilon_0,\notag\\
     ||\mathbb D\mathcal{S}_\lambda^{\mathcal{H}_0}|_{\lambda=0}-\mathbb D\mathcal{S}_\lambda^{\mathcal{H}_1}|_{\lambda=0}||_{2\to 2}
     &\leq& c\varepsilon_0.
    \end{eqnarray}
As a consequence of all these estimates we are able to extrapolate all the estimates related to the
boundedness, invertibility and goodness of the layer potentials associated to $\mathcal{H}_0$, $\mathcal{H}_0^\ast$,  to the corresponding estimates, assuming
$||{\bf \varepsilon}||_\infty\leq\varepsilon_0$, related to the
boundedness, invertibility and goodness  of the layer potentials associated to $\mathcal{H}_1$, $\mathcal{H}_1^\ast$. We can then complete the proof of Theorem \ref{thper2} using  the method of continuity. Corollary \ref{thper3} basically follows directly from Theorem \ref{thper2}, a few additional estimates/remarks, see Remark \ref{resea},  and from the uniqueness results proved in Section \ref{sec4}.

     In Section \ref{sec5} we prove Theorem \ref{th2-}-Theorem \ref{th2+}, using Theorem \ref{thper2} and the method of continuity. To do this in the case of Theorem \ref{th2}, we first establish Rellich type estimates, assuming that $A$ is real and symmetric, related to invertibility. In addition we here also use the main results established in \cite{CNS}, see Theorem 1.5 and Theorem 1.8 in \cite{CNS} and Theorem \ref{th2ag} stated below. The proof of Theorem 1.8 in \cite{CNS} is based on a local parabolic Tb-theorem for square functions, see Theorem 8.4 in \cite{CNS}, and on a version of the main result in \cite{FS} for equation of the form \eqref{eq1}, assuming in addition that $A$ is real and symmetric, see Theorem 8.7 in \cite{CNS}. Both Theorem 8.4 and Theorem 8.7 in \cite{CNS} are of independent interest.

\section{Preliminaries}\label{sec2}

    Let
$x=(x_1,..,x_{n})$, $X=(x,x_{n+1})$, $(x,t)=(x_1,..,x_{n},t)$, $(X,t)=(x_1,..,x_{n}, x_{n+1},t)$. Given $(X,t)=(x,x_{n+1}, t)$, $r>0$, we let $Q_r(x,t)$ and
 $\tilde Q_r(X,t)$ denote, respectively, the standard parabolic cubes in $\mathbb R^{n+1}$ and $\mathbb R^{n+2}$,  centered at $(x,t)$ and $(X,t)$, and of size $r$. By $Q$, $\tilde Q$ we denote any such parabolic cubes and we let $l(Q)$, $l(\tilde Q)$, $(x_Q,t_Q)$, $(X_{\tilde Q},t_{\tilde Q})$ denote their sizes and centers, respectively. Given $\gamma>0$, we let $\gamma Q$, $\gamma \tilde Q$ be the cubes which have the same centers as $Q$ and $\tilde Q$, respectively, but with sizes defined by $\gamma l(Q)$ and $\gamma l(\tilde Q)$.  We let $L^2(\mathbb R^{n+1},\mathbb C)$ denote the standard Hilbert space of functions $f:\mathbb R^{n+1}\to \mathbb C$ equipped with the inner product $(f,g):=\int f\bar g\, dxdt$ and we let $||f||_2:=(f,f)^{1/2}$ denote the norm of $f$. Given $p$, $1\leq p\leq \infty$, we let $L^p(\mathbb R^{n+1},\mathbb C)$ denote the standard Banach space of functions $f:\mathbb R^{n+1}\to \mathbb C$ which are $p$-integrable and we let $||f||_p$ denote the norm of $f$.  Given a set $E\subset \mathbb R^{n+1}$ we let $|E|$ denote its Lebesgue measure and by
$1_E$ we denote the indicator function for $E$. By $||\cdot||_{L^p(E)}$ we mean $||\cdot 1_E||_p$.   A function $f$ belongs to  $L^{p,\infty}(\mathbb R^{n+1},\mathbb C)$ if there exists a constant $c$ such that
$$l_f(\tau):=|\{(x,t)\in\mathbb R^{n+1}:\ |f(x,t)|\geq \tau\}|\leq\frac {c^p}{\tau^p}$$
whenever $\tau>0$. The best constant $c$ for which this inequality is valid is the $L^{p,\infty}(\mathbb R^{n+1},\mathbb C)$-norm of $f$ and
$$||f||_{{p,\infty}}:=||f||_{L^{p,\infty}}=||f||_{L^{p,\infty}(\mathbb R^{n+1},\mathbb C)}=\sup_{\tau>0}\tau(l_f(\tau))^{1/p}.$$
Given functions $f$, $\tilde f$, defined on $\mathbb R^{n+1}$, $\mathbb R^{n+2}$, respectively, we let
$$\mean{E}f\, dxdt,\ \mean{\tilde E}\tilde f\, dXdt$$ denote the averages of $f$, $\tilde f$ on the sets $E\subset\mathbb R^{n+1}$, $\tilde E\subset\mathbb R^{n+2}$, respectively. Furthermore, as mentioned and based on \eqref{eq4}, we will frequently also use a different convention concerning the labeling of the coordinates: we let $\lambda=x_{n+1}$ and when using the symbol
    $\lambda$, the point $(X,t)=(x,x_{n+1},t)$ will be written as $ (x,t, \lambda)=(x_1,..,x_{n},t,\lambda)$.  We write $\nabla =(\nabla_{||},\partial_\lambda)$ where $\nabla_{||}=(\partial_{x_1},...,\partial_{x_n})$. we let $$\mathbb R^{n+2}_\pm=\{(x,t,\lambda)\in \mathbb R^{n}\times\mathbb R\times\mathbb R:\ \pm\lambda>0\},$$
and
$$|||\cdot|||_\pm=\biggl (\int_{\mathbb R^{n+2}_\pm}|\cdot|^2\, \frac{dxdtd\lambda}{|\lambda|}\biggr )^{1/2},\ |||\cdot|||=\biggl (\int_{\mathbb R^{n+2}}|\cdot|^2\, \frac{dxdtd\lambda}{|\lambda|}\biggr )^{1/2}.$$

\subsection{Differential operators} Given $(x,t)\in\mathbb R^{n}\times\mathbb R$ we let $\|(x,t)\|$ be the unique positive
solution $\rho$ to the equation
\begin{eqnarray}\frac{t^2}{\rho^4}+\sum\limits^{n}_{i=1}\frac{x^2_i}{\rho^2}=1.\end{eqnarray}
Then $\|(\gamma x,\gamma^2t)\|=\gamma\|(x,t)\|$, $\gamma>0$, and we call $\|(x,t)\|$
the parabolic norm of $(x,t)$. Given $\beta\geq 0$, we define the operator $\mathbb D_\beta$
through the relation
\begin{eqnarray}\widehat{\mathbb D_\beta f}(\xi,\tau):=\|(\xi,\tau)\|^\beta\hat{f}(\xi,\tau),
\end{eqnarray}
where $\widehat{\mathbb D_\beta f}$ and $\hat f$ denote the Fourier transform of $\mathbb D_\beta f$ and $f$, respectively. We define the parabolic first order differential operator $\mathbb D$ through $\mathbb D=\mathbb D_1$. Similarly, given $\beta\geq 0$ we let  $\mathbb I_\beta$ denote the  operator defined on the Fourier transform side
through the relation $$\widehat{\mathbb I_\beta f}(\xi,\tau)=||(\xi,\tau)||^{-\beta}\hat f(\xi,\tau).$$
Note that $\mathbb I_\beta \mathbb D=\mathbb D\mathbb I_\beta=\mathbb D_{1-\beta}$ whenever $\beta\in [0,1]$. Given $\beta\in (0,1)$ we also define the
fractional (in time) differentiation operators $D_{\beta}^t$ through the relation
\begin{eqnarray}\widehat {D_{\beta}^t f}(\xi,\tau):=|\tau|^{\beta}\hat{f}(\xi,\tau).\end{eqnarray}
We let $H_t$ denote a Hilbert transform in the $t$-variable defined through the multiplier $i\mbox{sgn}({\tau})$. We make the construction so that $$\partial_t=D_{1/2}^tH_tD_{1/2}^t.$$
In the following we will also use  the parabolic half-order time derivative
\begin{eqnarray}\widehat{\mathbb D_{n+1}f}(\xi,\tau):=\frac{\tau}{\|(\xi,\tau)\|}\hat
f(\xi,\tau).\end{eqnarray}
By applying
Plancherel's theorem  we have
\begin{eqnarray}\label{uau}
\|\mathbb D_{n+1}f\|_2\leq c\|D^t_{1/2}f\|_2,\end{eqnarray}
with a constant depending only on $n$.

\subsection{Function spaces}\label{fspace}
Given $\beta\in [-1,1]$ we let
$\mathbb H^\beta:=\mathbb H^\beta(\mathbb R^{n+1},\mathbb C)$ be the closure of $C_0^\infty(\mathbb R^{n+1},\mathbb C)$ with respect to
\begin{eqnarray}\label{fspace}\|f\|_{\mathbb H^\beta}:=\|\|(\xi,\tau)\|^\beta \hat f\|_2.\end{eqnarray}
We let $\mathbb H=\mathbb H^1$. By applying
Plancherel's theorem  we have
\begin{eqnarray}\label{uau+}
\|f\|_{\mathbb H}\approx\|\nabla_{||}  f\|_2+\|H_tD^t_{1/2}f\|_2,\end{eqnarray}
with constants depending only on $n$. Furthermore, we let $\tilde{\mathbb H}:=\tilde{\mathbb H}(\mathbb R^{n+2},\mathbb C)$ be the closure of $C_0^\infty(\mathbb R^{n+2},\mathbb C)$
with respect to
\begin{eqnarray}\label{ez}
\|F\|_{\tilde{\mathbb H}}:=\biggl (\int_{-\infty}^\infty\int_{\mathbb R^{n+1}}\biggl(|\partial_\lambda F|^2+|\mathbb DF|^2\biggr )\, dxdtd\lambda\biggr )^{1/2}.\end{eqnarray}
Similarly, we let $\tilde{\mathbb H}_+:=\tilde{\mathbb H}_+(\mathbb R^{n+2}_+,\mathbb C)$ be the closure of $C^\infty(\mathbb R^{n+2}_+,\mathbb C)$ with respect to the expression in the last display but with integration over the interval $(-\infty,\infty)$ replaced by integration over the interval $(0,\infty)$ only. Given $F\in \tilde{\mathbb H}_+$ we let
\begin{eqnarray}\label{exop}
\tilde E(F)(x,t,\lambda)&=&F(x,t,\lambda),\mbox{ if }\lambda>0,\notag\\
\tilde E(F)(x,t,\lambda)&=&-3F(x,t,-\lambda)+4F(x,t,-\lambda/2),\mbox{ if } \lambda<0.
\end{eqnarray}
It is easily seen that $\tilde E(F)\in \tilde{\mathbb H}$ and we can conclude that there is a bijection between the spaces
$\tilde{\mathbb H}$ and $\tilde{\mathbb H}_+$. Furthermore, given $F\in C_0^\infty(\mathbb R^{n+2},\mathbb C)$ we see, by a straightforward calculation, that
\begin{eqnarray*}\label{trcc1}
||\mathbb D_{1/2}F||_2^2&=&-\int_{0}^\infty\int_{\mathbb R^{n+1}}\partial_\lambda|\mathbb D_{1/2}F|^2\, dxdtd\lambda\notag\\
&\leq& c\biggl (\int_{0}^\infty\int_{\mathbb R^{n+1}}|\mathbb D F|^2\, dxdtd\lambda\biggr )^{1/2}\biggl (\int_{0}^\infty\int_{\mathbb R^{n+1}}|\partial_\lambda F|^2\, dxdtd\lambda\biggr )^{1/2}.\end{eqnarray*}
Hence
\begin{eqnarray}\label{trcc1}
||\mathbb D_{1/2}F||_2\leq c\|F\|_{\tilde{\mathbb H}_+},\end{eqnarray}
whenever $F\in C_0^\infty(\mathbb R^{n+2},\mathbb C)$. Similarly, it is easy to see that there exists a linear extension operator $E:\mathbb H^{1/2}\to
\tilde{\mathbb H}$ such that
\begin{eqnarray}\label{trcc1+}\|E(f)\|_{\tilde{\mathbb H}}\leq c||f||_{\mathbb H^{1/2}},
\end{eqnarray}
whenever $f\in \mathbb H^{1/2}$. In particular, we can conclude that
\begin{eqnarray}\label{trcc2}
\mbox{space of traces of $\tilde{\mathbb H}_+$ onto $\mathbb R^{n+1}$ equals $\mathbb H^{1/2}$}.\end{eqnarray}
The dual of $\mathbb H^{1/2}$ is $\mathbb H^{-1/2}$.

\subsection{Definition of weak solutions} Let $\Omega\subset\{X=(x,x_{n+1})\in\mathbb R^n\times\mathbb R\}$ be a domain and let, given $-\infty<t_1< t_2<\infty$,
$\Omega_{t_1,t_2}=\Omega\times (t_1,t_2)$. We let $W^{1,2}(\Omega,\mathbb C)$ denote the standard Sobolev space of complex valued functions $v$, defined on $\Omega$, such that $v$ and $\nabla v$ are in $L^{2}(\Omega,\mathbb C)$. $L^2(t_1,t_2,W^{1,2}(\Omega,\mathbb C))$ is the space of  functions $u:\Omega_{t_1,t_2}\to \mathbb C$ such that
$$||u||_{L^2(t_1,t_2,W^{1,2}(\Omega,\mathbb C))}:=\biggl (\int_{t_1}^{t_2}||u(\cdot,t)||_{W^{1,2}(\Omega,\mathbb C)}^2\, dt\biggr )^{1/2}<\infty.$$
We say that $u\in L^2(t_1,t_2,W^{1,2}(\Omega,\mathbb C))$ is a weak solution to the equation
\begin{eqnarray}\label{ggag}
\mathcal{H}u=(\partial_t+\mathcal{L})u=0,\end{eqnarray}
in $\Omega_{t_1,t_2}$, if \begin{equation}\label{weak}
\int_{\mathbb R^{n+2}} \bigl( A\nabla u\cdot\nabla\bar \phi-u \partial_t\bar\phi\bigr )\, dXdt=0,
\end{equation}
whenever $\phi \in C_0^{\infty} (\Omega_{t_1,t_2},\mathbb C)$. Similarly, we say that $u$ is a solution to \eqref{ggag} in $\mathbb R^{n+2}$,  $\mathbb R^{n+2}_+$, if $u\phi\in L^2(-\infty,\infty,W^{1,2}(\mathbb R^n\times\mathbb R,\mathbb C))$,
$u\phi\in L^2(-\infty,\infty,W^{1,2}(\mathbb R^n\times\mathbb R_+,\mathbb C))$ whenever $\phi \in C_0^{\infty} (\mathbb R^{n+2},\mathbb C)$, $\phi \in C_0^{\infty} (\mathbb R^{n+2}_+,\mathbb C)$, and if
\eqref{weak} holds whenever $\phi \in C_0^{\infty} (\mathbb R^{n+2},\mathbb C)$, $\phi \in C_0^{\infty} (\mathbb R^{n+2}_+,\mathbb C)$, respectively. Assuming  that $\mathcal{H}$  satisfies \eqref{eq3}-\eqref{eq4} as well as the De Giorgi-Moser-Nash estimates stated in \eqref{eq14+}-\eqref{eq14++} below, it follows that any weak solution is smooth as a function of $t$ and that in this case
     \begin{equation}\label{weak+}
\int_{\mathbb R^{n+2}} \bigl( A\nabla u\cdot\nabla\bar \phi+\partial_tu \bar\phi\bigr )\, dXdt=0,
\end{equation}
whenever $\phi \in C_0^{\infty} (\Omega_{t_1,t_2},\mathbb C)$. Furthermore, if $u$ is globally defined in $\mathbb R^{n+2}$, and if
$D_{1/2}^tu\overline{H_tD_{1/2}^t\phi}$ is integrable in $\mathbb R^{n+2}$, whenever $\phi\in C_0^\infty(\mathbb R^{n+2},\mathbb C)$, then
\begin{eqnarray}\label{eq4-ed}
\tilde B(u,\phi)=0\mbox{ whenever } \phi\in C_0^\infty(\mathbb R^{n+2},\mathbb C),
    \end{eqnarray}
where the bilinear form $\tilde B(\cdot,\cdot)$ is defined on $ \tilde{\mathbb H}\times  \tilde{\mathbb H}$ as
       \begin{eqnarray*}
  \tilde B(u,\phi)= \int_{-\infty}^\infty\int_{\mathbb R^{n+1}}
      \bigl(A\nabla u\cdot\nabla\bar \phi-D_{1/2}^tu\overline{H_tD_{1/2}^t\phi}\bigr)\, dxdtd\lambda.
      \end{eqnarray*}
      Similar statements hold with $\tilde{\mathbb H}$, $\mathbb R^{n+2}$, $\tilde B$, replaced by  $\tilde{\mathbb H}_+$, $\mathbb R^{n+2}_+$, $\tilde B_+$, where $\tilde B_+$ is defined as in the last display but with integration in $\lambda$ over $\mathbb R_+$ only.  In particular, whenever $u$ is a weak solution to \eqref{ggag} in $\mathbb R^{n+2}$ or $\mathbb R^{n+2}_+$,  such that $u\in \tilde{\mathbb H}$ or $u\in \tilde{\mathbb H}_+$, then
      \eqref{eq4-ed} holds or \eqref{eq4-ed} holds with $\mathbb R^{n+2}$ replaced by $\mathbb R^{n+2}_+$. From now on, whenever we write  $\mathcal{H}u=0$ in a bounded domain $\Omega_{t_1,t_2}$, then we mean that
      \eqref{weak} holds whenever $\phi \in C_0^{\infty} (\Omega_{t_1,t_2},\mathbb C)$, and when we write that $\mathcal{H}u=0$ in $\mathbb R^{n+2}$, $\mathbb R^{n+2}_+$, then we mean that \eqref{weak} holds whenever $\phi \in C_0^{\infty} (\mathbb R^{n+2},\mathbb C)$, $\phi \in C_0^{\infty} (\mathbb R^{n+2}_+,\mathbb C)$.

\subsection{Existence of weak solutions (in $\mathbb R^{n+2}$)}
 Consider the space $\tilde{\mathbb H}:=\tilde{\mathbb H}(\mathbb R^{n+2},\mathbb C)$ and let
 $\tilde{\mathbb H}^\ast:=\tilde{\mathbb H}^\ast(\mathbb R^{n+2},\mathbb C)$ denotes its dual space. Given $F\in \tilde{\mathbb H}^\ast$, one can arguing as a in the proof of Lemma \ref{parahodge} below and conclude that there exists a weak solution $u\in \tilde{\mathbb H}$ to the equation $\mathcal{H}u=F$, in $\mathbb R^{n+2}$, in the sense that
    \begin{eqnarray}\label{parris}
  \tilde B(u,\phi)= \langle F,\phi\rangle
      \end{eqnarray}
      whenever $\phi\in \tilde{\mathbb H}$ and where $ \langle \cdot,\cdot\rangle$ is the duality pairing on $\tilde{\mathbb H}$. Furthermore,
$$||u||_{\mathbb H}\leq c||F||_{\tilde{\mathbb H}^\ast},$$
for some constant $c$ depending only on $n$ and $\Lambda$. The solution is unique up to a constant. Throughout the paper we let $\mathcal{H}^{-1}:
\tilde{\mathbb H}^\ast\to \tilde{\mathbb H} $ denote the operator which maps $F$ to $u$. Furthermore, arguing as in the proof of Lemma \ref{le8-} stated below, one can prove the following lemma.
\begin{lemma}\label{gara}  Consider the operator $\mathcal{H}=\partial_t-\div (A\nabla \cdot)$ and assume that $A$ satisfies \eqref{eq3}, \eqref{eq4}. Let $\Theta$ denote any of  the operators
       \begin{eqnarray}\label{aa1}
       &&\mbox{$\nabla\mathcal{H}^{-1}$, $D_{1/2}^t\mathcal{H}^{-1}$},
       \end{eqnarray}
       or
   \begin{eqnarray}\label{aa3}
       &&\mbox{ $\nabla\mathcal{H}^{-1} D_{1/2}^t $, $D_{1/2}^t\mathcal{H}^{-1} D_{1/2}^t$},
       \end{eqnarray}
       and let $\tilde \Theta$ denote any of  the operators
         \begin{eqnarray}\label{aa2}
       &&\mbox{$\nabla\mathcal{H}^{-1}\div $, $D_{1/2}^t\mathcal{H}^{-1}\div$}.
       \end{eqnarray}
       Then there exist $c$, depending only on $n, \Lambda$, such that
           \begin{eqnarray}
       (i)&&\int_{\mathbb R^{n+2}}\ |\Theta_\lambda f(X,t)|^2\, dXdt\leq c\int_{\mathbb R^{n+2}}\ |f(X,t)|^2\, dXdt,\notag\\
       (ii)&&\int_{\mathbb R^{n+2}}\ |\tilde \Theta_\lambda  {\bf f}(X,t)|^2\, dXdt\leq
       c\int_{\mathbb R^{n+2}}\ |{\bf f}(X,t)|^2\, dxdt,
       \end{eqnarray}
      whenever $f\in L^2(\mathbb R^{n+2},\mathbb C)$, ${\bf f}\in L^2(\mathbb R^{n+2},\mathbb C^{n+1})$. Furthermore, the corresponding statements hold with
      $\mathcal{H}^{-1}$ replaced by $(\mathcal{H}^\ast)^{-1}$.\end{lemma}

\begin{remark} Naturally, weak solutions to the problem $\mathcal{H}u=0$ in $\mathbb R^{n+2}_+$ can, as above, be constructed by first extending the boundary data on $\mathbb R^{n+1}$ to $\mathbb R^{n+2}_+$ by using the heat operator and then by subsequently solving an inhomogeneous problem similar to \eqref{parris} but with $\tilde B$, $\mathbb R^{n+2}$, replaced by $\tilde B_+$, $\mathbb R^{n+2}_+$.
\end{remark}

      \subsection{De Giorgi-Moser-Nash estimates}  We say  that solutions to
     $\mathcal{H}u=0$ satisfy De Giorgi--Moser-Nash estimates if there exist, for $p$, $1\leq p<\infty$, fixed, constants $c$ and $\alpha\in (0,1)$ such that the following is true. Let $\tilde Q\subset\mathbb R^{n+2}$ be a parabolic cube and assume that
$\mathcal{H}u=0$ in $2\tilde Q$. Then
                       \begin{eqnarray}\label{eq14+}
                \sup_{ \tilde Q}|u|\leq c\biggl (\mean{2\tilde Q}|u|^p\biggr )^{1/p},
    \end{eqnarray}
         and
        \begin{eqnarray}\label{eq14++}
                &&|u(X,t)-u(\tilde X,\tilde t)|\leq c\biggl (\frac {||(X-\tilde X,t-\tilde t)||}r\biggr )^{\alpha}
                \biggl (\mean{2\tilde Q}|u|^p\biggr )^{1/p},
    \end{eqnarray}
    whenever $(X,t)$, $(\tilde X,\tilde t)\in \tilde Q_{r}$. Given $p$, the constants $c$ and $\alpha$ will be referred to as the
    De Giorgi-Moser-Nash constants.  If $A$ is a (complex) constant matrix, or if $A$ real then solutions to
     $\mathcal{H}u=0$ satisfy De Giorgi--Moser-Nash estimates. The following result is due to Auscher \cite{A}, see also \cite{AT}.
     \begin{lemma}\label{auscher} Assume that $\mathcal{H}_0=\partial_t-\div A^0\nabla$, $\mathcal{H}_1=\partial_t-\div A^1\nabla$ satisfy \eqref{eq3}-\eqref{eq4}. Assume that solutions to $\mathcal{H}_0u=0$ satisfy De Giorgi--Moser-Nash estimates for all $p\in[1,\infty)$. Then there exists a constant $\varepsilon_0$, depending at most
     on $n$, $\Lambda$, and the De Giorgi-Moser-Nash constants for $\mathcal{H}_0$, such that if
    $$||A^1-A^0||_\infty\leq\varepsilon_0,$$
    then solutions to
     $\mathcal{H}_1u=0$ satisfy De Giorgi--Moser-Nash estimates for all $p\in[1,\infty)$.  Furthermore, the same statements hold with $\mathcal{H}_1$ replaced by $\mathcal{H}_1^\ast$.
     \end{lemma}

         \begin{remark}\label{auscher+} Based on Lemma \ref{auscher} we can conclude that if $A^0$ is either a (complex) constant matrix or a real and symmetric matrix, and if $A^1$ is as in Lemma \ref{auscher},   then solutions to
     $\mathcal{H}_1u=0$ satisfy De Giorgi--Moser-Nash estimates for all $p\in[1,\infty)$.
     \end{remark}

  \subsection{Energy estimates}

     \begin{lemma}\label{le1--} Assume that $\mathcal{H}$ satisfies \eqref{eq3}-\eqref{eq4}. Let $\tilde Q\subset\mathbb R^{n+2}$ be a parabolic cube and assume that
$\mathcal{H}u=0$ in $2\tilde Q$.  Then there exists a constant
     $c=c(n,\Lambda)$, $1\leq c<\infty$, such that
     \begin{eqnarray*}
     \mean{\tilde Q}|\nabla u(X,t)|^2\, dXdt\leq \frac c{l(\tilde Q)^2}\mean{2\tilde Q}|u(X,t)|^2\, dXdt.
\end{eqnarray*}
\end{lemma}
\begin{proof} The lemma follows by standard arguments.
\end{proof}

\begin{lemma}\label{le1} Assume that $\mathcal{H}$ satisfies \eqref{eq3}-\eqref{eq4}.  Let $Q\subset\mathbb R^{n+1}$ be a parabolic cube, $\lambda_0\in \mathbb R$, and let $\beta_1>1$, $\beta_2\in (0,1]$ be fixed constants. Let $I=(\lambda_0-\beta_2l(Q),\lambda_0+\beta_2l(Q))$, $\gamma I=
(\lambda_0-\gamma \beta_2l(Q),\lambda_0+\gamma \beta_2l(Q))$ for $\gamma\in (0,1)$. Assume that $\mathcal{H}u=0$ in $\beta_1^2Q\times I$. Then there exists a constant
     $c=c(n,\Lambda,\beta_1,\beta_2)$, $1\leq c<\infty$, such that
\begin{eqnarray*}
(i)&&\mean{Q}|\nabla u(x,t,\lambda_0)|^2\, dxdt\leq c\mean{\beta_1Q\times\frac 1 4I}|\nabla
u(X,t)|^2\, dXdt,\notag\\
(ii)&&\mean{Q}|\nabla u(x,t,\lambda_0)|^2\, dxdt\leq \frac c{l(Q)^2}\mean{\beta_1^2Q\times\frac 1 2I}|u(X,t)|^2\,
dXdt.
\end{eqnarray*}
\end{lemma}
\begin{proof} For the proof we refer to the proof of  Lemma 2.12  in \cite{CNS}.
\end{proof}

\begin{lemma}\label{le1a} Assume that $\mathcal{H}$ satisfies \eqref{eq3}-\eqref{eq4}. Let $\tilde Q\subset\mathbb R^{n+2}$ be a parabolic cube and assume that
$\mathcal{H}u=0$ in $2\tilde Q$.  Then there exists a constant
     $c=c(n,\Lambda)$, $1\leq c<\infty$, such that
\begin{eqnarray*}
\mean{\tilde Q}|\partial_t u(X,t)|^2\, dXdt\leq \frac {c}{l(\tilde Q)^2}\mean{2\tilde Q}|\nabla u(X,t)|^2\, dXdt.
\end{eqnarray*}
\end{lemma}
\begin{proof} In the following we can, without loss of generality, assume that $A$ is smooth. Let $\phi\in C_0^\infty(2\tilde Q)$ be a standard  cut-off function for $\tilde Q$. Let
\begin{eqnarray*}
I:=\int|\partial_t u|^2\phi^4\, dXdt,
\end{eqnarray*}
and let
\begin{eqnarray*}
II=\int|\nabla u|^2\phi^2\, dXdt,\ III:=\int|\nabla \partial_t u|^2\phi^6\, dXdt.
\end{eqnarray*}
Using that $\partial_t u=\nabla \cdot A\nabla u$, and partial integration, we see that
\begin{eqnarray*}
-I&=&-\int (\nabla \cdot (A\nabla u)\partial_t \bar u )\phi^4\, dXdt\notag\\
&=&\int ( A\nabla u\cdot \nabla(\partial_t \bar u))\phi^4\, dXdt+4\int \partial_t \bar u( A\nabla u\cdot \nabla\phi)\phi^3\, dXdt.
\end{eqnarray*}
Hence,
\begin{eqnarray*}
I&\leq&r^2\epsilon III+\frac {c(\epsilon)}{r^2}II
\end{eqnarray*}
where $\epsilon$ is a degree of freedom. As $\partial_t u$ is a solution to the underlying equation we can conclude, using  Lemma \ref{le1--},
that the lemma holds.
\end{proof}

\subsection{Non-tangential maximal functions} Given $(x_0,t_0)\in \mathbb R^{n+1}$, and $\beta>0$, we define the cone
    $$\Gamma^\beta(x_0,t_0)=\{(x,t,\lambda)\in\mathbb R^{n+2}_+:\ ||(x-x_0,t-t_0)||<\beta\lambda\}.$$
    Consider a function $U$ defined on $\mathbb R^{n+2}_+$. The non-tangential maximal operator $N_\ast^\beta$ is defined
               \begin{eqnarray}\label{eq6}
           N_{\ast}^\beta(U)(x_0,t_0):=\sup_{(x,t,\lambda)\in \Gamma^\beta(x_0,t_0)}\ |U(x,t,\lambda)|.
    \end{eqnarray}
    Given $(x,t)\in\mathbb R^{n+1}$, $\lambda>0$, we let
    \begin{eqnarray}\label{eq6+a} Q_\lambda(x,t)=\{(y,s):\ |x_i-y_i|<\lambda,\ |t-s|<\lambda^2\}
    \end{eqnarray} denote the standard
    parabolic cube on $\mathbb R^{n+1}$, with center $(x,t)$ and side length $\lambda$. We let
    $$W_\lambda(x,t)=\{(y,s,\sigma):\ (y,s)\in Q_\lambda(x,t),\lambda/2<\sigma<3\lambda/2\}$$
    be an associated Whitney type set. Using this notation we also introduce
           \begin{eqnarray}\label{eq6+}
           \tilde N_\ast^\beta(U)(x_0,t_0):=\sup_{(x,t,\lambda)\in \Gamma^\beta(x_0,t_0)}\biggl (
           \mean{W_\lambda(x,t)}|U(y,s,\sigma)|^2\, dydsd\sigma\biggr )^{1/2}.
    \end{eqnarray}
    We let
             \begin{eqnarray}\label{eq6+se}
           \Gamma(x_0,t_0):=\Gamma^1(x_0,t_0),\ N_{\ast}(U):=N_{\ast}^1(U),\ \tilde N_{\ast}(U):=\tilde N_{\ast}^1(U).
    \end{eqnarray}
    Furthermore, in many estimates it is necessary to increase the $\beta$ in $\Gamma^\beta$  as the estimates progress. We will use the convention, when the exact $\beta$ is not important, that $N_{\ast\ast}(U)$, $\tilde N_{\ast\ast}(U)$, equal $N_{\ast}^\beta(U)$, $\tilde N_{\ast}^\beta(U)$, for some appropriate $\beta>1$. Given a function $u$ defined on $\mathbb R^{n+2}_+$, and a function $f$ defined on  $\mathbb R^{n+1}$, we
    in the following say that $u$ converges to $f$ non-tangentially almost everywhere as we approach $\mathbb R^{n+1}$,  if
     $$\lim_{(x,t,\lambda)\in \Gamma(x_0,t_0)\to (x_0,t_0,0)}u(x,t,\lambda)=f(x_0,t_0)$$
     holds for almost every  $(x_0,t_0)\in \mathbb R^{n+1}$. As a short notation we will write
     $$\lim_{\lambda\to 0}u(\cdot,\cdot,\lambda)=f(\cdot,\cdot)\ \mbox{n.t}$$
     or simply that $u\to f$ n.t. At instances we will also use the notation
     $$\Gamma^\pm(x_0,t_0)=\{(x,t,\lambda)\in\mathbb R^{n+2}_\pm:\ ||(x-x_0,t-t_0)||<\pm\lambda\},$$
     and the associated non-tangential maximal operators $N_\ast^\pm$ defined through
               \begin{eqnarray}\label{eq6han}
           N_\ast^\pm(U)(x_0,t_0):=\sup_{(x,t,\lambda)\in \Gamma^\pm(x_0,t_0)}\ |U(x,t,\lambda)|,
    \end{eqnarray}
    for any function $U$ defined on $\mathbb R^{n+2}_\pm$. Similarly we introduce  the non-tangential maximal operators  $\tilde N_\ast^\pm$ in the natural way. If we need to emphasize a particular construction of the cone, with a particular opening defined by $\beta> 1$, we will use the notation $N_{\ast}^{\beta,\pm}$,  $\tilde N_{\ast}^{\beta,\pm}$. We let $N_{\ast\ast}^\pm$, $\tilde N_{\ast\ast}^\pm$, equal $N_{\ast}^{\beta,\pm}$, $\tilde N_{\ast}^{\beta,\pm}$, for some $\beta>1$.

     \subsection{Boundary value problems}  We say that $u$ solves
the Dirichlet problem in $\mathbb R^{n+2}_+$ with data $f\in L^2(\mathbb R^{n+1},\mathbb C)$, if
               \begin{eqnarray}\label{eq7}
               \mathcal{H}u &=& 0\mbox{ in $\mathbb R^{n+2}_+$},\notag\\
               \lim_{\lambda\to 0}u(\cdot,\cdot,\lambda)&=&f(\cdot,\cdot)\mbox{ n.t},
               \end{eqnarray}
               and
                  \begin{eqnarray}\label{eq7+}
               &&\sup_{\lambda>0}||u(\cdot,\cdot,\lambda)||_2+|||\lambda\nabla u|||_+<\infty.
               \end{eqnarray}
               We say that $u$ solves
the Neumann problem in $\mathbb R^{n+2}_+$ with data $g\in L^2(\mathbb R^{n+1},\mathbb C)$  if
             \begin{eqnarray}\label{eq8}
                      \mathcal{H}u &=& 0\mbox{ in $\mathbb R^{n+2}_+$},\notag\\
               \lim_{\lambda\to 0}-\sum_{j=1}^{n+1}A_{n+1,j}(\cdot)\partial_{x_j}u(\cdot,\cdot,\lambda)&=&g(\cdot,\cdot)\mbox{ n.t},
    \end{eqnarray}
      and
         \begin{eqnarray}\label{eq8+}
                   \tilde N_\ast(\nabla u)\in L^2(\mathbb R^{n+1}).
    \end{eqnarray}
We say that $u$ solves
the Regularity problem in $\mathbb R^{n+2}_+$ with data $f\in {\mathbb H}(\mathbb R^{n+1},\mathbb C)$  if
    \begin{eqnarray}\label{eq9}
               \mathcal{H}u &=& 0\mbox{ in $\mathbb R^{n+2}_+$},\notag\\
               \lim_{\lambda\to 0}u(\cdot,\cdot,\lambda)&=&f(\cdot,\cdot)\mbox{ n.t},
               \end{eqnarray}
               and
                    \begin{eqnarray}\label{eq9+}
                   \tilde N_\ast(\nabla u)\in L^2(\mathbb R^{n+1}),\ \tilde N_\ast(H_tD^t_{1/2}u)\in L^2(\mathbb R^{n+1}).
    \end{eqnarray}
    We denote the problems in \eqref{eq7}-\eqref{eq7+}, \eqref{eq8}-\eqref{eq8+}, \eqref{eq9}-\eqref{eq9+}, by
               \begin{eqnarray}\label{eq10}
        \mbox{$(D2)$, $(N2)$ and $(R2)$, respectively.}
               \end{eqnarray}

\subsection{Layer potentials}

Assume that $\mathcal{H}=\partial_t+\mathcal{L}$   satisfies \eqref{eq3}-\eqref{eq4}. By functional calculus, see \cite{AT}, \cite{K},  $\mathcal{L}$ defines an
$L^2$-contraction semigroup $e^{-t\mathcal{L}}$, for $t>0$.  Let $K_t(X,Y)$ denote the distribution kernel of $e^{-t\mathcal{L}}$. We introduce
\begin{eqnarray}
\Gamma(x,t,\lambda,y,s,\sigma)=\Gamma(X,t,Y,s):=K_{t-s}(X,Y)=K_{t-s}(x,\lambda,y,\sigma)
\end{eqnarray}
whenever $(x,t,\lambda),\ (y,s,\sigma)\in\mathbb R^{n+2}$, $t-s>0$ and we put $\Gamma(x,t,\lambda,y,s,\sigma)=0$ whenever $t-s<0$. Then $\Gamma(x,t,\lambda,y,s,\sigma)$, for $(x,t,\lambda)$, $(y,s,\sigma)\in \mathbb R^{n+2}$ is a  fundamental solution, heat kernel, associated to the operator $\mathcal{H}$.  In particular, the fundamental solution $\Gamma$ associated to $\mathcal{H}$ coincides with the kernel
$K$. We let
$$\Gamma^\ast(y,s,\sigma,x,t,\lambda)=\overline{\Gamma(x,t,\lambda,y,s,\sigma)}$$ and we note that this is then a fundamental solution associated to  $\mathcal{H}^\ast$.  Based on \eqref{eq4} we   let
                    \begin{eqnarray}\label{eq11-}
                    \Gamma_\lambda(x,t,y,s)&=&\Gamma(x,t,\lambda,y,s,0),\notag\\
                    \Gamma_\lambda^\ast(y,s,x,t)&=&\Gamma^\ast(y,s,0,x,t,\lambda),
    \end{eqnarray}
    whenever $(x,t),\ (y,s)\in\mathbb R^{n+1}$, $\lambda\in\mathbb R$. We define associated single  layer potentials
    \begin{eqnarray}\label{eq11}
    \mathcal{S}_\lambda^{\mathcal{H}} f(x,t)&:=&\int_{\mathbb R^{n+1}}\Gamma_\lambda(x,t,y,s)f(y,s)\, dyds,\notag\\
     \mathcal{S}_\lambda^{\mathcal{H}^\ast} f(x,t)&:=&\int_{\mathbb R^{n+1}}\Gamma_\lambda^\ast(y,s,x,t)f(y,s)\, dyds,
    \end{eqnarray}
    whenever $f\in C_0^\infty(\mathbb R^{n+1},\mathbb C)$. We also introduce double layer potentials
    \begin{eqnarray}\label{eq11ed}
    \mathcal{D}_\lambda^{\mathcal{H}} f(x,t)&:=&\int_{\mathbb R^{n+1}}\overline{\partial_{\nu^\ast}\Gamma_\lambda^\ast(y,s,x,t)}f(y,s)\, dyds,\notag\\
    \mathcal{D}_\lambda^{\mathcal{H}^\ast} f(x,t)&:=&\int_{\mathbb R^{n+1}}\overline{\partial_{\nu}\Gamma_\lambda(x,t,y,s)}f(y,s)\, dyds,
    \end{eqnarray}
    whenever $\lambda\neq 0$ and where
        \begin{eqnarray}\label{eq12}
        \partial_{\nu^\ast}=-\sum_{j=1}^{n+1}A^\ast_{n+1,j}(y)\partial_{y_j},\ \partial_{\nu}=-\sum_{j=1}^{n+1}A_{n+1,j}(y)\partial_{y_j}.
    \end{eqnarray}
    We also note that
        \begin{eqnarray}\label{eq11edsea}
    \mathcal{D}_\lambda^{\mathcal{H}}&=&\mathcal{S}_\lambda^{\mathcal{H}}\overline{\partial_{\nu}}=-\sum_{j=1}^{n+1}\mathcal{S}_\lambda^{\mathcal{H}}
    \overline{A_{n+1,j}(y)}\partial_{y_j},\notag\\
     \mathcal{D}_\lambda^{\mathcal{H}^\ast}&=&\mathcal{S}_\lambda^{\mathcal{H}^\ast}\overline{\partial_{\nu^\ast}}=-\sum_{j=1}^{n+1}
     \mathcal{S}_\lambda^{\mathcal{H}^\ast}
    \overline{A^\ast_{n+1,j}(y)}\partial_{y_j}.
    \end{eqnarray}
    An other way to write these relations is
          \begin{eqnarray}\label{eq11edsea+}
    \mathcal{D}_\lambda^{\mathcal{H}}&=&\adj\bigl (-e_{n+1}\cdot A^\ast\nabla \mathcal{S}_\tau^{\mathcal{H}^\ast}|_{\tau=-\lambda}\bigr ),\notag\\
    \mathcal{D}_\lambda^{\mathcal{H}^\ast}&=&\adj\bigl (-e_{n+1}\cdot A\nabla \mathcal{S}_\tau^{\mathcal{H}}|_{\tau=-\lambda}\bigr ),
    \end{eqnarray}
    where we, here and throughout the paper, by $\mathcal{O}^\ast$ or
    $\adj(\mathcal{O})$ denote the hermitian adjoint of a given operator $\mathcal{O}$. In Lemma \ref{trace4} below we prove, under assumptions, the existence of boundary layer potential operators
     $$\mp\frac 12+\mathcal{K}^{\mathcal{H}},\ \pm\frac 12+\tilde{\mathcal{K}}^{\mathcal{H}},\ \mathbb D \mathcal{S}_\lambda^{\mathcal{H}}|_{\lambda=0},$$
     such that
     \begin{eqnarray}\label{eq14in}
                \mathcal{D}_{\pm\lambda}^{\mathcal{H}} f&\to &\bigl (\pm\frac 1 2+\mathcal{K}^{\mathcal{H}}\bigr) f,\notag\\
                -\sum_{j=1}^{n+1}A_{n+1,j}(\cdot)\partial_{x_j}\mathcal{S}_{\pm\lambda}^{\mathcal{H}} f&\to &\bigl (\pm\frac 1 2+\tilde{\mathcal{K}}^{\mathcal{H}}\bigr) f,\notag\\
                (\mathbb D\mathcal{S}_\sigma^{\mathcal{H}})|_{\sigma=\pm\lambda}  f&\to& \mathbb D \mathcal{S}_\lambda^{\mathcal{H}}|_{\lambda=0}f,
    \end{eqnarray}
    as $\lambda\to 0$, whenever $f\in L^2(\mathbb R^{n+1},\mathbb C)$. We prove similar results with $\mathcal{S}_{\lambda}^{\mathcal{H}}$, $\mathcal{D}_{\lambda}^{\mathcal{H}}$, $\mathcal{K}^{\mathcal{H}}$, $\tilde{\mathcal{K}}^{\mathcal{H}}$, $\mathbb D \mathcal{S}_\lambda^{\mathcal{H}}|_{\lambda=0}$,  replaced by $\mathcal{S}_{\lambda}^{\mathcal{H}^\ast}$, $\mathcal{D}_{\lambda}^{\mathcal{H}^\ast}$,  $\mathcal{K}^{\mathcal{H}^\ast}$, $\tilde{\mathcal{K}}^{\mathcal{H}^\ast}$, $\mathbb D \mathcal{S}_\lambda^{\mathcal{H}^\ast}|_{\lambda=0}$. The limits in \eqref{eq14in} are interpreted in the sense of Lemma \ref{trace4}, Lemma \ref{trace5}, and Lemma \ref{trace7-}, and we refer to the bulk of the paper for details. In the formulation of Theorem \ref{thper2} and  Corollary \ref{thper3} we used the following definitions, Definition \ref{blayer+} and Definition \ref{sollayer}.

\begin{definition}\label{blayer+} Consider $\mathcal{H}=\partial_t-\div A\nabla$. Assume that $\mathcal{H}$,  $\mathcal{H}^\ast$   satisfy \eqref{eq3}-\eqref{eq4}. We say that $\mathcal{H}$,  $\mathcal{H}^\ast$ have bounded, invertible and good layer  potentials with constant $\Gamma \geq 1$, if statements $(i)-(xiii)$ below hold whenever  $f\in L^2(\mathbb R^{n+1},\mathbb C)$.

\noindent
First,
 \begin{eqnarray}\label{keyestint-}
     (i)&&\sup_{\lambda\neq 0}||\partial_\lambda \mathcal{S}_{\lambda}^{\mathcal{H}}f||_2+\sup_{\lambda\neq 0}||\partial_\lambda \mathcal{S}_{\lambda}^{\mathcal{H}^\ast}f||_2\leq \Gamma||f||_2,\notag\\
     (ii)&&|||\lambda\partial_\lambda^2 \mathcal{S}_{\lambda}^{\mathcal{H}}f|||_\pm+|||\lambda\partial_\lambda^2 \mathcal{S}_{\lambda}^{\mathcal{H}^\ast}f|||_\pm\leq  \Gamma||f||_2.
     \end{eqnarray}
Second,
\begin{eqnarray}\label{keyestint+a}
(iii)&& ||N_\ast^\pm(\partial_\lambda \mathcal{S}_\lambda^{\mathcal{H}} f)||_2+||N_\ast^\pm(\partial_\lambda \mathcal{S}_\lambda^{\mathcal{H}^\ast} f)||_2\leq\Gamma  ||f||_2,\notag\\
  (iv)&&\sup_{\lambda\neq 0}||\mathbb D\mathcal{S}_{\lambda}^{\mathcal{H}}f||_{2}+\sup_{\lambda\neq 0}||\mathbb D\mathcal{S}_{\lambda}^{\mathcal{H}^\ast}f||_{2}\leq \Gamma ||f||_2,\notag\\
(v)&&||\tilde N_\ast^\pm(\nabla_{||}\mathcal{S}_\lambda^{\mathcal{H}} f)||_2+||\tilde N_\ast^\pm(\nabla_{||}\mathcal{S}_\lambda^{\mathcal{H}^\ast} f)||_2\leq\Gamma  ||f||_2,\notag\\
(vi)&&||\tilde N_\ast^\pm(H_tD^t_{1/2}\mathcal{S}_\lambda^{\mathcal{H}} f)||_2+||\tilde N_\ast^\pm(H_tD^t_{1/2}\mathcal{S}_\lambda^{\mathcal{H}^\ast} f)||_2\leq\Gamma  ||f||_2.\end{eqnarray}
Third,
\begin{eqnarray*}\label{keyestint+aga}
   (vii)&&\mbox{$\mathcal{K}^{\mathcal{H}}$, $\tilde{\mathcal{K}}^{\mathcal{H}}$, $\mathbb D \mathcal{S}_\lambda^{\mathcal{H}}|_{\lambda=0}$, $\mathcal{K}^{\mathcal{H}^\ast}$, $\tilde{\mathcal{K}}^{\mathcal{H}^\ast}$, $\mathbb D \mathcal{S}_\lambda^{\mathcal{H}^\ast}|_{\lambda=0}$, exist in the}\notag\\
   &&\mbox{sense of Lemma \ref{trace4}, Lemma \ref{trace5}, and Lemma \ref{trace7-}}.
   \end{eqnarray*}
Fourth, with constants of comparison defined by  $\Gamma $,
   \begin{eqnarray*}\label{keyestint+aah}
   (viii)&&\mbox{$||(\pm\frac 1 2 I+\mathcal{K}^{\mathcal{H}})f||_2\approx||f||_2\approx ||(\pm\frac 1 2 I+\tilde {\mathcal{K}}^{\mathcal{H}})f||_2$,}\notag\\
      (ix)&&\mbox{$||(\pm\frac 1 2 I+\mathcal{K}^{\mathcal{H}^\ast})f||_2\approx||f||_2\approx||(\pm\frac 1 2 I+\tilde {\mathcal{K}}^{\mathcal{H}^\ast})f||_2$},\notag\\
   (x)&&\mbox{$||\mathbb D \mathcal{S}_\lambda^{\mathcal{H}}|_{\lambda=0}f||_{2}\approx||f||_2\approx||\mathbb D \mathcal{S}_\lambda^{\mathcal{H}^\ast}|_{\lambda=0}f||_{2}$}.
\end{eqnarray*}
Fifth,
  \begin{eqnarray*}
   (xi)&&\mbox{$(\pm\frac 1 2 I+\mathcal{K}^{\mathcal{H}})$,  $(\pm\frac 1 2 I+\tilde {\mathcal{K}}^{\mathcal{H}})$, are bijections on $L^2(\mathbb R^{n+1},\mathbb C)$,}\notag\\
   (xii)&&\mbox{$(\pm\frac 1 2 I+\mathcal{K}^{\mathcal{H}^\ast})$,  $(\pm\frac 1 2 I+\tilde {\mathcal{K}}^{\mathcal{H}^\ast})$, are bijections on $L^2(\mathbb R^{n+1},\mathbb C)$,}\notag\\
   (xiii)&&\mbox{$\mathcal{S}_\lambda^{\mathcal{H}}|_{\lambda=0}$, $\mathcal{S}_\lambda^{\mathcal{H}^\ast}|_{\lambda=0}$, are bijections from $L^2(\mathbb R^{n+1},\mathbb C)$ to $\mathbb H(\mathbb R^{n+1},\mathbb C)$.}
\end{eqnarray*}
\end{definition}

\begin{remark}\label{resea} Assume that  $\mathcal{H}$,  $\mathcal{H}^\ast$ have  bounded, invertible and good layer  potentials with constant $ \Gamma$ in the sense of Definition \ref{blayer+}. Then \begin{eqnarray}\label{keyestint+adase}
(i')&&\sup_{\lambda\neq 0}|| \mathcal{D}_\lambda^{\mathcal{H}} f||_2+\sup_{\lambda\neq 0}||\mathcal{D}_\lambda^{\mathcal{H}^\ast}f||_2\leq c ||f||_2,\notag\\
(ii')&&|||\lambda \nabla \mathcal{D}_\lambda^{\mathcal{H}} f|||_\pm+|||\lambda \nabla \mathcal{D}_\lambda^{\mathcal{H}^\ast} f|||_\pm\leq c  ||f||_2,
\end{eqnarray}
for some constant $c$ depending only on $n$, $\Lambda$,  the De Giorgi-Moser-Nash constants and $\Gamma$. Indeed, $(i')$ is a simple consequence of
\eqref{eq11edsea+} and Definition \ref{blayer+} $(i)$, $(iv)$. That $(ii')$ holds is proved in Lemma \ref{th0uu} below. In particular, the statements of Definition \ref{blayer+} are strong enough to ensure the validity of the quantitative estimates for the double layer potential operators
$\mathcal{D}_\lambda^{\mathcal{H}}$, $\mathcal{D}_\lambda^{\mathcal{H}^\ast}$, underlying the solvability of $(D2)$ for $\mathcal{H}$, $\mathcal{H}^\ast$. \end{remark}

\begin{definition}\label{sollayer} Consider $\mathcal{H}=\partial_t-\div A\nabla$. Assume that $\mathcal{H}$,  $\mathcal{H}^\ast$   satisfy \eqref{eq3}-\eqref{eq4}. Assume that  $\mathcal{H}$,  $\mathcal{H}^\ast$ have  bounded, invertible and good layer  potentials with constant $ \Gamma$ in the sense of Definition \ref{blayer+}.  We then say that $(D2)$, $(N2)$ and $(R2)$ are uniquely solvable, for the operators $\mathcal{H}$,  $\mathcal{H}^\ast$, by way of layer potentials and with constant $\Gamma$,  if $(D2)$ for the operators $\mathcal{H}$,  $\mathcal{H}^\ast$ have unique solutions, and if $(N2)$ and $(R2)$ for the operators $\mathcal{H}$,  $\mathcal{H}^\ast$ have unique solutions, modulo a constant.
\end{definition}

\section{Harmonic analysis}\label{sec7} In the following we establish a number of harmonic analysis results, and collect some  results from \cite{N}, to be used in the forthcoming sections. Throughout the section we assume that  $\mathcal{H}$,  $\mathcal{H}^\ast$ satisfy \eqref{eq3}-\eqref{eq4}. Recall that $\nabla =(\nabla_{||},\partial_\lambda)$ where $\nabla_{||}=(\partial_{x_1},...,\partial_{x_n})$.  We will also use the notation $D_i=\partial_{x_i}$ for
$i\in\{1,...,n+1\}$.  We
let
 \begin{eqnarray}\label{block}\mathcal{L}_{||}=-\div_{||}(A_{||}\nabla_{||})
 \end{eqnarray}
where $\div_{||}$ is the divergence operator in the variables $(\partial_{x_1},...,\partial_{x_n})$ only and where $A_{||}$ is the $n\times n$-dimensional sub matrix of $A$ defined by $\{A_{i,j}\}_{i,j=1}^n$. Then
    \begin{eqnarray}\label{block+}\mathcal{L}=\mathcal{L}_{||}-\sum_{j=1}^{n+1}A_{n+1,j}D_{n+1}D_j-\sum_{i=1}^{n}D_iA_{i,n+1}D_{n+1}.
    \end{eqnarray}
    We also let
     \begin{eqnarray}\label{blocked}\mathcal{H}_{||}=\partial_t+\mathcal{L}_{||},\ \mathcal{H}_{||}^\ast=-\partial_t+\mathcal{L}_{||}^\ast.
 \end{eqnarray}
    Using this notation,  the equation $\mathcal{H}u=0$ can formally be written
     \begin{eqnarray}\label{block++}\mathcal{H}_{||}u-\sum_{j=1}^{n+1}A_{n+1,j}D_{n+1}D_ju-\sum_{i=1}^{n}D_i(A_{i,n+1}D_{n+1}u)=0.
     \end{eqnarray}

\subsection{Resolvents and a parabolic Hodge decomposition associated to $\mathcal{H}_{||}$} Recall the function space $\mathbb H={\mathbb H}(\mathbb R^{n+1},\mathbb C)$. We let
$\bar{\mathbb H}=\bar{\mathbb H}(\mathbb R^{n+1},\mathbb C)$ be the closure of $C_0^\infty(\mathbb R^{n+1},\mathbb C)$ with respect to the norm
\begin{eqnarray}\|f\|_{\bar{\mathbb H}}:=\|f\|_{\mathbb H}+\|f\|_2.\end{eqnarray}
Let  $B:\mathbb H\times  \mathbb H\to\mathbb C$ be defined as
       \begin{eqnarray*}
  B(u,\phi)= \int_{\mathbb R^{n+1}}
      (A_{||}\nabla_{||} u\cdot\nabla_{||}\bar \phi-D_{1/2}^tu\overline{H_tD_{1/2}^t\phi})\, dxdt,
      \end{eqnarray*}
      and let, for $\delta\in (0,1)$, $ B_\delta:\mathbb H\times  \mathbb H\to\mathbb C$ be defined as
       \begin{eqnarray}\label{form2}
  B_\delta(u,\phi)&=&\int_{\mathbb R^{n+1}}A_{||}\nabla_{||} u\cdot\overline{\nabla_{||}(I+\delta H_t)\phi}\, dxdt\notag\\
  &&-\int_{\mathbb R^{n+1}}D_{1/2}^tu\overline{H_tD_{1/2}^t(I+\delta H_t)\phi}\, dxdt.
      \end{eqnarray}

     \begin{definition}\label{de1} Let $g\in L^2(\mathbb R^{n+1},\mathbb C^n)$. We say that a function
      $u\in {\mathbb H}(\mathbb R^{n+1},\mathbb C)$ is a (weak) solution to the equation $\mathcal{H}_{||}u=-\div_{||} g$,
in $\mathbb R^{n+1}$, if
     \begin{eqnarray*}\label{eq4-}
B(u,\phi)=\int_{\mathbb R^{n+1}}g\cdot\nabla_{||}\bar\phi\, dxdt,
    \end{eqnarray*}
    whenever $\phi\in {\mathbb H}(\mathbb R^{n+1},\mathbb C)$.
      \end{definition}

          \begin{definition}\label{de2} Let $\lambda>0$ be given. Let $f\in L^2(\mathbb R^{n+1},\mathbb C)$. We say that a function
      $u\in\bar{\mathbb{H}}(\mathbb R^{n+1},\mathbb C)$ is a (weak) solution to the equation $u+\lambda^2\mathcal{H}_{||}u=f$,
in $\mathbb R^{n+1}$, if
     \begin{eqnarray*}\label{eq4-a}
\int_{\mathbb R^{n+1}}u\bar \phi\, dxdt+\lambda^2 B(u,\phi)=\int_{\mathbb R^{n+1}}f\bar \phi\, dxdt
    \end{eqnarray*}
whenever $\phi\in \bar{\mathbb H}(\mathbb R^{n+1},\mathbb C)$.
      \end{definition}

\begin{lemma}\label{parahodge} Consider the operator $\mathcal{H}_{||}=\partial_t-\div_{||} (A_{||}\nabla_{||} \cdot)$ and assume that $A$ satisfies \eqref{eq3}, \eqref{eq4}. Let  $g\in L^2(\mathbb R^{n+1},\mathbb C^n)$. Then there exists a weak solution $u\in {\mathbb H}(\mathbb R^{n+1},\mathbb C)$ to the equation $\mathcal{H}_{||}u=-\div_{||} g$, in $\mathbb R^{n+1}$, in the sense of Definition \ref{de1}. Furthermore,
$$||u||_{\mathbb H}\leq c||g||_2,$$
for some constant $c$ depending only on $n$ and $\Lambda$. The solution is unique up to a constant.
\end{lemma}
\begin{proof} This is Lemma 2.6 in \cite{N}. We here include the proof for completion. Consider the functionals
$$\Lambda_{g}(\phi)=\int_{\mathbb R^{n+1}}g\cdot{\nabla_{||}\bar \phi}\, dxdt,\  \Lambda_{g}^\delta(\phi)=\int_{\mathbb R^{n+1}}g\cdot{\nabla_{||}\overline{\phi_\delta}}\, dxdt, $$
$\phi_\delta=(I+\delta H_t)\phi$, $\phi\in {\mathbb H}(\mathbb R^{n+1},\mathbb C)$.  Then $\Lambda_{g}$ and $\Lambda_{g}^\delta$ are bounded linear functional on $\mathbb H={\mathbb H}(\mathbb R^{n+1},\mathbb C)$ and
$$|\Lambda_{g}(\phi)|+|\Lambda_{g}^\delta(\phi)|\leq c||g||_2||\phi||_{\mathbb H}.$$
Consider the bilinear form $B_\delta(\cdot,\cdot)$ introduced in \eqref{form2}. If $\delta=\delta(n,\Lambda)$ is small enough, then $B_\delta(\cdot,\cdot)$ is a bilinear, bounded, coercive form on $\mathbb H\times \mathbb H$. Hence, using the Lax-Milgram theorem we see that there exists a unique
$u\in {\mathbb H}$ such that
$$B(u,\phi_\delta)\equiv B_\delta(u,\phi)=\Lambda_{g}^\delta(\phi)\equiv \Lambda_{g}(\phi_\delta)$$
for all $\phi\in \mathbb H$. Using that $(I+\delta H_t)$ is invertible on $\mathbb H$, if $0<\delta\ll 1$ is small enough, we can conclude that
$$B(u,\psi)= \Lambda_{g}(\psi),$$
whenever $\psi\in {\mathbb H}$. The bound $||u||_{\mathbb H}\leq c||g||_2$ follows readily. This completes the existence and quantitative part of the lemma. The statement concerning uniqueness follows immediately.\end{proof}

\begin{lemma}\label{parahodge+} Let $\lambda>0$ be given. Consider the operator $\mathcal{H}_{||}=\partial_t-\div_{||} (A_{||}\nabla_{||} \cdot)$ and assume that $A$ satisfies \eqref{eq3}, \eqref{eq4}. Let  $f\in L^2(\mathbb R^{n+1},\mathbb C)$. Then there exists a weak solution $u\in\bar{\mathbb{H}}(\mathbb R^{n+1},\mathbb C)$ to the equation $u+\lambda^2\mathcal{H}_{||}u=f$, in $\mathbb R^{n+1}$, in the sense of Definition \ref{de2}. Furthermore,
$$||u||_{2}+||\lambda\nabla_{||} u||_{2}+||\lambda D_{1/2}^tu||_{2}\leq c||f||_2,$$
for some constant $c$ depending only on $n$ and $\Lambda$. The solution is unique.
\end{lemma}
\begin{proof} See the proof of Lemma 2.7 in \cite{N}.\end{proof}

    \begin{remark}\label{rem1} Definition \ref{de1}, Definition \ref{de2}, Lemma \ref{parahodge}, and Lemma \ref{parahodge+}, all have analogous formulations for the
    operator  $\mathcal{H}_{||}^\ast$.
    \end{remark}

\subsection{Estimates of resolvents} Let $\lambda>0$ be given. Consider the operator $\mathcal{H}_{||}=\partial_t-\div_{||} (A_{||}\nabla_{||} \cdot)$. Let  $f\in L^2(\mathbb R^{n+1},\mathbb C)$. Then by Lemma \ref{parahodge+} the equation $u+\lambda^2\mathcal{H}_{||}u=f$ has a unique weak solution $u\in\bar{\mathbb H}$. From now on we will denote this solution by $ \mathcal{E}_\lambda f$. In the case of the operator $\mathcal{H}_{||}^\ast$ we denote the corresponding solution by $ \mathcal{E}_\lambda^\ast f$. In this sense $\mathcal{E}_\lambda=(I+\lambda^2\mathcal{H}_{||})^{-1}$ and $\mathcal{E}_\lambda^\ast=(I+\lambda^2\mathcal{H}_{||}^\ast)^{-1}$. We here collect some estimates of quantities build on
$\mathcal{E}_\lambda f$ and $\mathcal{E}_\lambda^\ast f$ to be used in the forthcoming sections.

 \begin{lemma}\label{le8-}  Let $\lambda>0$ be given. Consider the operator $\mathcal{H}_{||}=\partial_t-\div_{||} (A_{||}\nabla_{||} \cdot)$ and assume that $A$ satisfies \eqref{eq3}, \eqref{eq4}. Let $\Theta_\lambda$ denote any of  the operators
       \begin{eqnarray}\label{aa1}
       &&\mbox{$\mathcal{E}_\lambda$, $\lambda \nabla_{||}\mathcal{E}_\lambda$, $\lambda D_{1/2}^t\mathcal{E}_\lambda$},
       \end{eqnarray}
       or
   \begin{eqnarray}\label{aa3}
       &&\mbox{$\lambda\mathcal{E}_\lambda D_{1/2}^t$, $\lambda^2 \nabla_{||}\mathcal{E}_\lambda D_{1/2}^t $, $\lambda^2 D_{1/2}^t\mathcal{E}_\lambda D_{1/2}^t$},
       \end{eqnarray}
       and let $\tilde \Theta_\lambda$ denote any of  the operators
         \begin{eqnarray}\label{aa2}
       &&\mbox{$\lambda\mathcal{E}_\lambda\div_{||} $, $\lambda^2 \nabla_{||}\mathcal{E}_\lambda\div_{||} $, $\lambda^2 D_{1/2}^t\mathcal{E}_\lambda\div_{||}$}.
       \end{eqnarray}
       Then there exist $c$, depending only on $n, \Lambda$, such that
           \begin{eqnarray}
       (i)&&\int_{\mathbb R^{n+1}}\ |\Theta_\lambda f(x,t)|^2\, dxdt\leq c\int_{\mathbb R^{n+1}}\ |f(x,t)|^2\, dxdt,\notag\\
       (ii)&&\int_{\mathbb R^{n+1}}\ |\tilde \Theta_\lambda  {\bf f}(x,t)|^2\, dxdt\leq
       c\int_{\mathbb R^{n+1}}\ |{\bf f}(x,t)|^2\, dxdt,
       \end{eqnarray}
      whenever $f\in L^2(\mathbb R^{n+1},\mathbb C)$, ${\bf f}\in L^2(\mathbb R^{n+1},\mathbb C^{n})$.\end{lemma}
\begin{proof}  This is Lemma 2.11 in \cite{N}.\end{proof}

\begin{lemma}\label{le8-+} Let $\lambda>0$ be given. Consider the operator $\mathcal{H}_{||}=\partial_t-\div_{||} (A_{||}\nabla_{||} \cdot)$ and assume that $A$ satisfies \eqref{eq3}, \eqref{eq4}. Let $\Theta_\lambda$ denote any of  the operators
       \begin{eqnarray}\label{aa1a}
       &&\mbox{$\mathcal{E}_\lambda$, $\lambda \nabla_{||}\mathcal{E}_\lambda$},
       \end{eqnarray}
       and let $\tilde \Theta_\lambda$ denote any of the operators
         \begin{eqnarray}\label{aa2a}
       &&\mbox{$\lambda\mathcal{E}_\lambda\div_{||} $, $\lambda^2\nabla_{||}\mathcal{E}_\lambda\div_{||} $.}
       \end{eqnarray}
       Let $E$ and $F$ be two closed sets in $\mathbb R^{n+1}$ and let $d_p(E,F)$ denote the parabolic distance between $E$ and $F$, i.e., $$d_p(E,F)=\min\{||(x-y,y-s)||\ (x,t)\in E,\ (y,s)\in F\}.$$ Then there exist $c$, $1\leq c<\infty$, depending only on $n$, $\Lambda$, such that
       \begin{eqnarray}
       (i)&&\int_{F}\ |\Theta_\lambda f(x,t)|^2\, dxdt\leq ce^{-c^{-1}(d_p(E,F)/{\lambda})}\int_{E}\ |f(x,t)|^2\, dxdt,\notag\\
       (ii)&&\int_{F}\ |\tilde \Theta_\lambda  {\bf f}(x,t)|^2\, dxdt\leq
       ce^{-c^{-1}(d_p(E,F)/{\lambda})}\int_{E}\ |{\bf f}(x,t)|^2\, dxdt,
       \end{eqnarray}
       whenever $f\in L^2(\mathbb R^{n+1},\mathbb C)$, ${\bf f}\in L^2(\mathbb R^{n+1},\mathbb C^{n+1})$, and $\mbox{supp }f\subset E$, $\mbox{supp }{\bf f}\subset E$.
\end{lemma}
\begin{proof}  This is Lemma 2.13 in \cite{N}.\end{proof}

\begin{theorem}\label{thm1}  Consider the operators $\mathcal{H}_{||}=\partial_t+\mathcal{L}_{||}=\partial_t-\div_{||} (A_{||}\nabla_{||} \cdot)$, $\mathcal{H}_{||}^\ast=-\partial_t+\mathcal{L}_{||}^\ast=-\partial_t-\div_{||} (A_{||}^\ast\nabla_{||} \cdot)$, and assume that $A$ satisfies \eqref{eq3}, \eqref{eq4}.
Then there exists a constant $c$, $1\leq c<\infty$, depending only on $n$, $\Lambda$, such that
        \begin{eqnarray}\label{ea1}
      |||\lambda \mathcal{E}_\lambda \mathcal{H}_{||}f|||_++|||\lambda \mathcal{E}_\lambda^\ast\mathcal{H}_{||}^\ast f|||_+\leq c||\mathbb D f||_2,
      \end{eqnarray}
      and
      \begin{eqnarray}\label{ea2}
               (i)&&|||\partial_\lambda\mathcal{E}_\lambda f|||_++|||\partial_\lambda\mathcal{E}_\lambda^\ast f|||_+\leq c||\mathbb Df||_2,\notag\\
               (ii)&&|||\lambda\partial_t\mathcal{E}_\lambda f|||_++|||\lambda\partial_t\mathcal{E}_\lambda^\ast f|||_+\leq c||\mathbb Df||_2,\notag\\
       (iii)&&|||\lambda\mathcal{E}_\lambda \mathcal{L}_{||} f|||_++|||\lambda\mathcal{E}_\lambda^\ast \mathcal{L}_{||}^\ast f|||_+\leq c||\mathbb Df||_2,\notag\\
       (iv)&&|||\lambda \mathcal{L}_{||}\mathcal{E}_\lambda  f|||_++|||\lambda \mathcal{L}_{||}^\ast\mathcal{E}_\lambda^\ast f|||_+\leq c||\mathbb Df||_2,
      \end{eqnarray}
      whenever $f\in\mathbb H(\mathbb R^{n+1},\mathbb C)$. \end{theorem}
      \begin{proof}  \eqref{ea1} is Theorem 1.17 in \cite{N}, \eqref{ea2} $(i)-(iv)$ is Corollary 1.18 in \cite{N}.
       \end{proof}
\begin{remark}\label{commuter} Note that $\mathcal{E}_\lambda$ and $\mathcal{H}_{||}$ commute. To see this we let, arguing formally, $u=\mathcal{E}_\lambda f$ and $\tilde u=\mathcal{H}_{||}u$. Then, by definition
      $u$ satisfies $u+\lambda^2\mathcal{H}_{||}u=f$ and hence $\tilde u+\lambda^2\mathcal{H}_{||}\tilde u=\mathcal{H}_{||}f$. In particular,
      $\tilde u=\mathcal{E}_\lambda \mathcal{H}_{||}f$ and we can conclude, by uniqueness of $\tilde u$, that
      \begin{eqnarray}\label{ea3-}\mathcal{H}_{||}\mathcal{E}_\lambda=\mathcal{E}_\lambda\mathcal{H}_{||},
      \end{eqnarray}
      i.e., $\mathcal{E}_\lambda$ and $\mathcal{H}_{||}$ commute. Furthermore,
       \begin{eqnarray}\label{ea3}
      \mathcal{L}_{||}\mathcal{E}_\lambda-\mathcal{E}_\lambda \mathcal{L}_{||}&=&\mathcal{H}_{||}\mathcal{E}_\lambda-\mathcal{E}_\lambda\mathcal{H}_{||}\notag\\
      &&-(\partial_t\mathcal{E}_\lambda-\mathcal{E}_\lambda \partial_t)=0+0,
      \end{eqnarray}
      by \eqref{ea3-} and as  $\partial_t$ and $\mathcal{E}_\lambda$ commute.
\end{remark}

For reference we here also state the following lemma which is important in the proof of Theorem \ref{thm1} and which will be used in Section \ref{sec8}.
\begin{lemma} \label{ilem2--+} Let $\lambda>0$ be given. Assume that  $\mathcal{H}_{||}=\partial_t+\mathcal{L}_{||}=\partial_t-\div A_{||}\nabla_{||}$ satisfies \eqref{eq3}-\eqref{eq4}. Consider a map
     \begin{eqnarray*}
\gamma_\lambda:\mathbb R^{n+1}\to \mathbb C^n.
\end{eqnarray*}
Then there exist an $\epsilon\in (0,1)$, depending only on $n$, $\Lambda$, a finite set $W$ of unit vectors in
     $\mathbb C^n$, whose
     cardinality depends on $\epsilon$ and $n$, and, for each cube $Q\subset\mathbb R^{n+1}$, a mapping
     $f^{\epsilon}_{Q,w}:\mathbb R^{n+1}\to\mathbb C$  such that the following hold.
     \begin{eqnarray*}
     (i)&& \int_{\mathbb R^{n+1}}|\mathbb D f^{\epsilon}_{Q,w}|^2\, dxdt+\int_{\mathbb R^{n+1}}|\mathbb D_{n+1}f^{\epsilon}_{Q,w}|^2\, dxdt\leq c_1 |Q|,\notag\\
     (ii)&& \int_{\mathbb R^{n+1}}|\partial_tf^{\epsilon}_{Q,w}|^2\, dxdt+\int_{\mathbb R^{n+1}}|\mathcal{L}_{||}f^{\epsilon}_{Q,w}|^2\, dxdt\leq c_2|Q|/l(Q)^2,\notag\\
     (iii)&&\frac 1{|Q|}\int_0^{l(Q)}\int_Q|\gamma_\lambda(x,t)|^2\frac {dxdtd\lambda}\lambda\notag\\
     &&\leq c_3\sum_{w\in W}\frac 1{|Q|}\int_0^{l(Q)}\int_Q|\gamma_\lambda\cdot\mathcal{A}_\lambda^Q \nabla_{||} f^{\epsilon}_{Q,w}|\frac {dxdtd\lambda}\lambda,
     \end{eqnarray*}
     for some constants  $c_1$, $c_2$, $c_3$. $c_1$ depends only on $n$, $\Lambda$, but $c_2$ and $c_3$ are also allowed to depend on
     $\epsilon$. Here $\mathcal{A}_\lambda^Q$ is the dyadic averaging operator induced by $Q$ and defined in \eqref{dy}.
     \end{lemma}
     \begin{proof} This is a consequence of Lemma 3.3 in \cite{N}.
     \end{proof}
\subsection{Littlewood-Paley theory}
We here introduce parabolic approximations of the identity, chosen based on a finite stock of functions and fixed throughout the paper, as follows. Let
$\P\in C_0^\infty(Q_1(0))$, $\P\geq 0$ be real-valued,  $\int \P\, dxdt=1$, where $Q_1(0)$ is the unit parabolic cube in $\mathbb R^{n+1}$ centered at $0$. At instances we will also assume that $\int x_i\P(x,t)\, dxdt=0$ for all $i\in \{1,..,n\}$. At instances we will also assume, which we always may by construction, that $\P$ has a product structure, i.e., $\P(x,t)=\P^x(x)\P^t(t)$ where $\P^x$ and $\P^t$ have the same properties as $\P$ but are defined with respect to $\mathbb R^{n}$ and $\mathbb R$.  We set $\P_\lambda(x,t)=\lambda^{-n-2}\P(\lambda^{-1}x,\lambda^{-2}t)$ whenever $\lambda>0$. Given $\P$ we let
$\P_\lambda$ denote the convolution operator
$$\P_\lambda f(x,t)=\int_{\mathbb R^{n+1}}\P_\lambda(x-y,t-s)f(y,s)\, dyds.$$
Similarly, we will by $\mathcal{Q}_\lambda$ denote a generic approximation to the zero operator, not necessarily the same at each instance, but chosen from a finite set of such
operators depending only on our original choice of $\P_\lambda$. In particular, $\mathcal{Q}_\lambda(x,t)=\lambda^{-n-2}\mathcal{Q}(\lambda^{-1}x,\lambda^{-2}t)$ where
$\mathcal{Q}\in C_0^\infty(Q_1(0))$, $\int \mathcal{Q}\, dxdt=0$. In addition we will, following \cite{HL}, assume that $\mathcal{Q}_\lambda$ satisfies the conditions
\begin{eqnarray} \label{li-}\mathcal{Q}_\lambda(x,t)&\leq&\frac {c\lambda}{(\lambda+||(x,t)||)^{n+3}},\notag\\
|\mathcal{Q}_\lambda(x,t)-\mathcal{Q}_\lambda(y,s)|&\leq&\frac {c||(x-y,t-s)||^\alpha}{(\lambda+||(x,t)||)^{n+2+\alpha}},
\end{eqnarray}
where the latter estimate holds for some $\alpha\in (0,1)$ whenever $2||(x-y,t-s)||\leq ||(x,t)||$. It is well known that
\begin{eqnarray}\label{li}|||\mathcal{Q}_\lambda f|||_+=\biggl (\int_0^\infty\int_{\mathbb R^{n+1}}|\mathcal{Q}_\lambda f|^2\, \frac{dxdtd\lambda}\lambda\biggr )^{1/2}\leq c||f||_2
\end{eqnarray}
for all $f\in L^2(\mathbb R^{n+1},\mathbb C)$. In the following we collect a number of elementary observations to be used in the forthcoming sections.
\begin{lemma} \label{little2} Let $\P_\lambda$ be as above. Then
\begin{eqnarray*}
(i)&&|||\lambda\nabla \P_\lambda f|||_++|||\lambda^2\partial_t\P_\lambda f|||_++|||\lambda\mathbb D \P_\lambda f|||_+\leq c||f||_2,\notag\\
(ii)&&|||\P_\lambda(I-\P_\lambda) f|||_+\leq c||f||_2,\notag\\
(iii)&&|||\lambda^{-1}(I-\P_\lambda) g|||_+\leq c||\mathbb Dg||_2,
\end{eqnarray*}
for all $f\in L^2(\mathbb R^{n+1},\mathbb C)$,  $g\in \mathbb H(\mathbb R^{n+1},\mathbb C)$.
\end{lemma}
\begin{proof} For the proof of $(i)$ we refer to Lemma 2.30 in \cite{N}. For the proof of $(ii)$ we refer to the end of the proof of $(iii)$. To prove $(iii)$, let $\mathbb I_1$ denote the parabolic Riesz operator defined on the Fourier transform side
through $$\widehat{\mathbb I_1 g}(\xi,\tau)=||(\xi,\tau)||^{-1}\hat g(\xi,\tau).$$ Then, using Plancherel's theorem we see that
\begin{eqnarray*}
|||\lambda^{-1}(I-\P_\lambda) g|||_+^2&= & |||\lambda^{-1} {\mathbb I_1}(I-\P_\lambda) \mathbb D g|||_+^2\notag\\
&\leq&c\int_0^\infty\int_{\mathbb R^{n+1}}\bigl|(\lambda||(\xi,\tau)||)^{-1}\bigl(1-\hat \P(\lambda\xi,\lambda^2\tau)\bigr ) \hat h\bigr |^2\, \frac{d\xi d\tau d\lambda}\lambda,
\end{eqnarray*}
where  $h=\mathbb D g$. Let now in addition $\P$ be such that $\int x_i\P(x,t)\, dxdt=0$ for all $i\in \{1,..,n\}$. Then
\begin{eqnarray*}
|(\lambda||(\xi,\tau)||)^{-1}(1-\hat \P(\lambda\xi,\lambda^2\tau))|\leq c\min\bigl\{\lambda||(\xi,\tau)||,1/(\lambda||(\xi,\tau)||)\bigr\}
\end{eqnarray*}
and we deduce $(iii)$.
\end{proof}

Consider a cube $Q\subset\mathbb R^{n+1}$. In the following we let $\mathcal{A}_\lambda^Q$ denote the dyadic averaging operator induced by $Q$, i.e., if $\hat Q_\lambda(x,t)$ is the minimal dyadic cube
      (with respect to the grid induced by $Q$) containing $(x,t)$, with side length at least $\lambda$, then
      \begin{eqnarray}\label{dy}\mathcal{A}_\lambda^Q f(x,t)=\mean{\hat Q_\lambda(x,t)}f\, dyds,
      \end{eqnarray}
      the average of $f$ over $\hat Q_\lambda(x,t)$.

      \begin{lemma} \label{little3} Let $\P_\lambda$ be as above. Then
\begin{eqnarray*}
|||(\mathcal{A}_\lambda^Q-\P_\lambda)f|||_+\leq c||f||_2
\end{eqnarray*}
for all $f\in L^2(\mathbb R^{n+1},\mathbb C)$.
\end{lemma}
\begin{proof} For a proof of this lemma in our context we refer to Lemma 2.19 in \cite{CNS}.
\end{proof}

\subsection{Uniform (in $\lambda$) $L^2$-estimates and off-diagonal estimates: consequences} We here establish a number of results for general linear operators $\Theta_\lambda$ and $\tilde \Theta_\lambda$  satisfying two crucial estimates. First, we assume that
             \begin{eqnarray}\label{stand-}
   \sup_{\lambda>0}\bigl (||{\Theta}_\lambda||_{2\to 2}+||{\tilde\Theta}_\lambda||_{2\to 2}\bigr )\leq \Gamma,
       \end{eqnarray}
       for some constant $\Gamma$. Second, we assume that there exists, for some integer $d\geq 0$, a constant $\tilde\Gamma=\tilde\Gamma_d$ such that
\begin{eqnarray}\label{stand}
||{\Theta}_\lambda (f1_{2^{k+1}Q\setminus 2^kQ})||_{L^2(Q)}^2&\leq& \tilde \Gamma^2 2^{-(n+2)k}(\lambda/(2^kl(Q)))^{2d+2}||{f}||^2_{L^2(2^{k+1}Q\setminus 2^kQ)},\notag\\
\quad||{\tilde\Theta}_\lambda ({\bf f}1_{2^{k+1}Q\setminus 2^kQ})||_{L^2(Q)}^2&\leq& \tilde\Gamma^2 2^{-(n+2)k}(\lambda/(2^kl(Q)))^{2d+2}||{\bf f}||^2_{L^2(2^{k+1}Q\setminus 2^kQ)}, \end{eqnarray}
whenever $0<\lambda\leq cl(Q)$,  $Q\subset\mathbb R^{n+1}$ is a parabolic cube, $k\in\mathbb Z_+$,  and for all $f\in L^2(\mathbb R^{n+1},\mathbb C)$, ${\bf f}\in L^2(\mathbb R^{n+1},\mathbb C^{n+1})$, respectively. In the following we state and prove a number of lemmas  for operators ${\Theta}_\lambda$ satisfying
\eqref{stand-} and \eqref{stand}. The corresponding statements for operators $\tilde {\Theta}_\lambda$ satisfying
\eqref{stand-} and \eqref{stand} are analogous. Throughout the subsection we assume $\lambda>0$.

\begin{lemma}\label{le8} Assume that ${\Theta}_\lambda$ is an operator satisfying  \eqref{stand-} and \eqref{stand} for some $d\geq 0$. Assume
also that
\begin{eqnarray}\label{square}\int_0^\infty\int_{\mathbb R^{n+1}} |\Theta_\lambda f(x,t)|^2\, \frac {dxdtd\lambda}\lambda\leq \hat\Gamma||f||_2^2
\end{eqnarray}
for some constant $\hat\Gamma\geq 1$ and for all $f\in L^2(\mathbb R^{n+1},\mathbb C)$. Then
\begin{eqnarray}\label{square+}
\int_0^{l(Q)}\int_Q|\Theta_\lambda b(x,t)|^2\frac {dxdtd\lambda}{\lambda}\leq c||b||_\infty^2|Q|
\end{eqnarray}
for all parabolic cubes $Q\subset\mathbb R^{n+1}$, whenever $b\in L^\infty(\mathbb R^{n+1},\mathbb C)$, and for a constant $c$ depending only on $n$, $\Gamma$, $\tilde\Gamma$, $\hat\Gamma$.
\end{lemma}
\begin{proof} This can be proved by adapting the corresponding arguments in \cite{FeS}.
\end{proof}

\begin{lemma}\label{le9}   Assume that ${\Theta}_\lambda$ is an operator satisfying  \eqref{stand-} and \eqref{stand} for some $d\geq 0$. Assume also that $\Lambda_\lambda$ is an operator which satisfies \eqref{stand-} and that there exists a constant $c$, $1\leq c<\infty$, such that
\begin{eqnarray}\label{stand+}
       \int_{F}\ |\Lambda_\lambda f(x,t)|^2\, dxdt\leq ce^{-c^{-1}(d_p(E,F)/{\lambda})}\int_{E}\ |f(x,t)|^2\, dxdt,
       \end{eqnarray}
whenever $E$ and $F$ are two closed sets in $\mathbb R^{n+1}$, $f\in L^2(\mathbb R^{n+1},\mathbb C)$, $\mbox{supp }f\subset E$, and where $d_p(E,F)$ denotes the parabolic distance between $E$ and $F$ introduced in Lemma \ref{le8-+}. Then $\Theta_\lambda\Lambda_\lambda$ also satisfies \eqref{stand-} and \eqref{stand} for some integer $d\geq 0$ and for some constants $\Gamma$, $\tilde\Gamma$, depending only on
$n$, the constants $\Gamma$, $\tilde\Gamma$ for ${\Theta}_\lambda$, and the constant $c$ in \eqref{stand+}.
\end{lemma}
\begin{proof} That $\Theta_\lambda\Lambda_\lambda$ satisfies \eqref{stand-} is immediate from the corresponding assumption for $\Theta_\lambda$ and $\Lambda_\lambda$. To verify \eqref{stand}, consider a parabolic cube  $Q\subset\mathbb R^{n+1}$, $\lambda\leq cl(Q)$, $k\in\mathbb Z_+$,  and $f\in L^2(\mathbb R^{n+1},\mathbb C)$. In the following we may without loss of generality assume that $k\geq 4$ as we, otherwise, subdivide $Q$ dyadically to reduce to this case. Given $Q$,  $\lambda\leq cl(Q)$, we let $\hat Q=2^{k-2}Q$ and write
$$\Theta_\lambda\Lambda_\lambda=\Theta_\lambda 1_{\hat Q}\Lambda_\lambda+\Theta_\lambda 1_{\mathbb R^{n+1}\setminus\hat Q}\Lambda_\lambda.$$
Then
\begin{eqnarray}\label{aaqlu}||\Theta_\lambda 1_{\hat Q}\Lambda_\lambda(f1_{2^{k+1}Q\setminus 2^kQ})||_{L^2(Q)}&\leq& c||\Theta_\lambda||_{2\to 2}
||\Lambda_\lambda(f1_{2^{k+1}Q\setminus 2^kQ})||_{L^2(\hat Q)}\notag\\
&\leq& c||\Theta_\lambda||_{2\to 2} \exp(-{c^{-1}2^kl(Q)}/\lambda) ||{f}||_{L^2(2^{k+1}Q\setminus 2^kQ)}.
\end{eqnarray}
Furthermore, using \eqref{stand} for $\Theta_\lambda$,
\begin{eqnarray}\label{aaq+lu}
&&||\Theta_\lambda 1_{\mathbb R^{n+1}\setminus\hat Q}\Lambda_\lambda(f1_{2^{k+1}Q\setminus 2^kQ})||_{L^2(Q)}\notag\\
&\leq&
\sum_{j\geq k-2}||\Theta_\lambda 1_{2^{j+1}Q\setminus 2^jQ}\Lambda_\lambda(f1_{2^{k+1}Q\setminus 2^kQ})||_{L^2(Q)}\notag\\
&\leq&\sum_{j\geq k-2}2^{-(n+2)j}(\lambda/2^jl(Q))^{2d+2}||\Lambda_\lambda(f1_{2^{k+1}Q\setminus 2^kQ})||_{L^2(2^{j+1}Q\setminus 2^jQ)}\notag\\
&\leq&\sum_{j\geq k-2}2^{-(n+2)j}(\lambda/2^jl(Q))^{2d+2}\exp(-{c^{-1}2^jl(Q)}/\lambda) ||{f}||_{L^2(2^{k+1}Q\setminus 2^kQ)}\notag\\
&\leq& c 2^{-(n+2)k}(\lambda/(2^kl(Q)))^{2d+2}||{\bf f}||^2_{L^2(2^{k+1}Q\setminus 2^kQ)},
\end{eqnarray}
as we see by summing a geometric series. The estimates in \eqref{aaqlu} and \eqref{aaq+lu} complete the proof of the lemma.
\end{proof}

\begin{lemma}\label{le11+} Assume that ${\Theta}_\lambda$ is an operator satisfying  \eqref{stand-} and \eqref{stand} for some $d\geq 0$.
Let $b\in L^\infty(\mathbb R^{n+1},\mathbb C)$ and let $\mathcal{A}_\lambda $ denote a self-adjoint averaging operator whose kernel satisfies
$$\mbox{$\phi_\lambda(x,t,y,s)\leq c\lambda^{-n-2}1_{\{|x-y|+|t-s|^{1/2}\leq c\lambda\}}$, $\phi_\lambda\geq 0$},$$
 and $$\int_{\mathbb R^{n+1}}\phi_\lambda(x,t,y,s)dyds=1,$$ whenever $(x,t), (y,s)\in\mathbb R^{n+1}$. Then
$$\sup_{\lambda>0}||({\Theta}_\lambda b)\mathcal{A}_\lambda f||_2\leq c||b||_\infty||f||_2,$$
whenever $f\in L^2(\mathbb R^{n+1},\mathbb C)$, and for a constant $c$ depending only on $n$, $\Gamma$ and $\tilde\Gamma$.
\end{lemma}
\begin{proof} See the proof of Lemma 2.26 in \cite{N}. \end{proof}

\begin{lemma}\label{le11++} Assume that ${\Theta}_\lambda$ is an operator satisfying  \eqref{stand-} and \eqref{stand} for some $d\geq 0$. Assume that
$$\Omega_\lambda=\int_0^\lambda\biggl (\frac \sigma\lambda\biggr )^{\delta}W_{\lambda,\sigma}\Theta_\sigma\frac {d\sigma}\sigma,$$
for some $\delta>0$, and that  $$\sup_{\sigma,\lambda}||W_{\lambda,\sigma}||_{2\to 2}\leq \hat c.$$ Then
$$\int_0^\infty\int_{\mathbb R^{n+1}} |\Omega_\lambda f(x,t)|^2\, \frac {dxdtd\lambda}\lambda\leq
c\int_0^\infty\int_{\mathbb R^{n+1}} |\Theta_\lambda f(x,t)|^2\, \frac {dxdtd\lambda}\lambda$$
for all $f\in L^2(\mathbb R^{n+1},\mathbb C)$ and for a constant $c$ depending only on $n$, $\Gamma$, $\tilde\Gamma$, and $\hat c$.
\end{lemma}
\begin{proof} To proof of Lemma 3.12 in \cite{AAAHK} can be adopted.
\end{proof}

\begin{remark} Assume that $\Theta_\lambda$ is an operator satisfying  \eqref{stand-} and \eqref{stand} for some $d\geq 0$. Then,  for   $\lambda$ fixed,
$\Theta_\lambda 1$ exists as an element in $L^2_{\mbox{loc}}(\mathbb R^{n+1},\mathbb C)$. Indeed, let $Q_R$ be the parabolic cube on $\mathbb R^{n+1}$ with center at $(0,0)$ and with size determined by $R\gg 1$. Writing
$$\Theta_\lambda 1=\Theta_\lambda1_{2Q_R}+ \Theta_\lambda 1_{\mathbb R^{n+1}\setminus 2Q_R}, $$
and using \eqref{stand-}  we see that
$$||(\Theta_\lambda 1_{2Q_R})1_{Q_R}||_2\leq c\Gamma R^{(n+2)/2}.$$
Furthermore, by the off-diagonal estimates in \eqref{stand} it also follows that
$$||(\Theta_\lambda 1_{\mathbb R^{n+1}\setminus 2Q_R})1_{Q_R}||_2\leq c\tilde \Gamma R^{(n+2)/2}.$$
\end{remark}

\begin{lemma}\label{le11-}  Assume that $\mathcal{R}_\lambda$ is an operator satisfying  \eqref{stand-} and \eqref{stand} for some $d\geq 0$. Assume in addition that
$\mathcal{R}_\lambda1=0$. Then
$$||\mathcal{R}_\lambda f||_2\leq c(||\lambda \nabla_{||} f||_2+||\lambda^2 \partial_tf||_2),$$
whenever $f\in C_0^\infty(\mathbb R^{n+1},\mathbb C)$, and for a constant $c$ depending only on $n$, $\Gamma$ and $\tilde\Gamma$.
\end{lemma}
\begin{proof} See the proof of Lemma 2.27 in \cite{N}.\end{proof}

\begin{lemma}\label{le11}  Assume that $\mathcal{R}_\lambda$ is an operator satisfying  \eqref{stand-} and \eqref{stand} for some $d\geq 0$. Assume in addition that $\mathcal{R}_\lambda1=0$ and that
$$\int_0^{l(Q)}\int_Q |\lambda^{-1} R_\lambda\Psi(x,t)|^2\frac {dxdtd\lambda}{\lambda}\leq \hat\Gamma |Q|,$$
whenever $Q\subset\mathbb R^{n+1}$ is a parabolic cube, and where $\Psi(x,t)=x$. Then
$$\biggl (\int_0^\infty\int_{\mathbb R^{n+1}}|\lambda^{-1} R_\lambda f|^2\, \frac {dxdtd\lambda}\lambda\biggr )^{1/2}
\leq c||\mathbb Df||_2,$$
whenever $f\in \mathbb H(\mathbb R^{n+1},\mathbb C)$, and for a constant $c$ depending only on  $n$, $\Gamma$, $\tilde\Gamma$ and $\hat\Gamma$.
\end{lemma}
\begin{proof} In the following we can without loss of generality assume that $f\in C_0^\infty(\mathbb R^{n+1},\mathbb C)$. Let $\mathbb D_j$ denote a dyadic grid of
parabolic cubes on $\mathbb R^{n+1}$ of size $2^{-j}$. Then
\begin{eqnarray}\label{car}
&&\int_0^\infty\int_{\mathbb R^{n+1}}|\lambda^{-1} R_\lambda f|^2\, \frac {dxdtd\lambda}\lambda\notag\\
&=&\sum_{j=-\infty}^\infty\sum_{Q\in \mathbb D_{-j}}\int_{2^j}^{2^{j+1}}
\int_Q|\lambda^{-1}R_\lambda f(y,s)|^2\frac {dydsd\lambda}\lambda\notag\\
&=&\sum_{j=-\infty}^\infty\sum_{Q\in \mathbb D_{-j}}\int_{2^j}^{2^{j+1}}
\int_Q\biggl (\mean{Q}|\lambda^{-1}R_\lambda f(y,s)|^2\, dxdt\biggr )\, \frac {dydsd\lambda}\lambda.
\end{eqnarray}
For $Q\in \mathbb D_{-j}$,  $(x,t)\in Q$, and $\lambda\in ({2^j},{2^{j+1}})$ fixed,  we let
$$G_{(x,t,\lambda)}(y,s)=f(y,s)-f(x,t)-(y-x)\cdot \P_\lambda(\nabla_{||}f)(x,t),$$
where $\P_\lambda$ is a standard parabolic approximation of the identity. Using that $R_\lambda 1=0$ we see that
\begin{eqnarray*}
\lambda^{-1}R_\lambda f(y,s)=\lambda^{-1}R_\lambda(G_{(x,t,\lambda)})(y,s)+\lambda^{-1}R_\lambda\Psi(y,s)\P_\lambda(\nabla_{||}f)(x,t).
\end{eqnarray*}
Hence,
\begin{eqnarray}\label{car+}
\int_0^\infty\int_{\mathbb R^{n+1}}|\lambda^{-1} R_\lambda f|^2\, \frac {dxdtd\lambda}\lambda\leq I+II,
\end{eqnarray}
where
\begin{eqnarray*}\label{car+}
I&=&\sum_{j=-\infty}^\infty\sum_{Q\in \mathbb D_{-j}}\int_{2^j}^{2^{j+1}}
\int_Q\biggl (\mean{Q}|\lambda^{-1}R_\lambda(G_{(x,t,\lambda)})(y,s)|^2\, dxdt\biggr )\, \frac {dydsd\lambda}\lambda,\notag\\
II&=&\sum_{j=-\infty}^\infty\sum_{Q\in \mathbb D_{-j}}\int_{2^j}^{2^{j+1}}
\int_Q\biggl (\mean{Q}|\lambda^{-1}(R_\lambda\Psi(y,s))\P_\lambda(\nabla_{||}f)(x,t)|^2\, dxdt\biggr )\, \frac {dydsd\lambda}\lambda.
\end{eqnarray*}
To estimate $II$ we note that
\begin{eqnarray}\label{car+}
|II|&=&\sum_{j=-\infty}^\infty\sum_{Q\in \mathbb D_{-j}}\int_{2^j}^{2^{j+1}}
\int_Q|\P_\lambda(\nabla_{||}f)(x,t)|^2\biggl (\mean{Q}|\lambda^{-1}R_\lambda\Psi(y,s)|^2{dyds}\biggr )\frac {dxdtd\lambda}\lambda\notag\\
&\leq &\int_0^\infty
\int_{\mathbb R^{n+1}}|\P_\lambda(\nabla_{||}f)(x,t)|^2\biggl (\mean{Q_{c\lambda(x,t)}}|\lambda^{-1}R_\lambda\Psi(y,s)|^2{dyds}\biggr )
\frac {dxdtd\lambda}\lambda\notag\\
&\leq& c||\nabla_{||}f||^2_2\biggl (\sup_{Q}\frac 1 {|Q|}\int_0^{l(Q)}\int_Q |\lambda^{-1} R_\lambda\Psi(x,t)|^2\frac {dxdtd\lambda}{\lambda}\biggr ).
\end{eqnarray}
To estimate $I$ we write, recall that $Q\in \mathbb D_{-j}$,  $(x,t)\in Q$, and $\lambda\in ({2^j},{2^{j+1}})$,
\begin{eqnarray*}
\lambda^{-1}R_\lambda(G_{(x,t,\lambda)})(y,s)&=&R_\lambda(\lambda^{-1}G_{(x,t,\lambda)}1_{2Q})(y,s)\notag\\
&&+\sum_{k=1}^\infty R_\lambda(\lambda^{-1}G_{(x,t,\lambda)}1_{2^{k+1}Q\setminus 2^{k}Q })(y,s)\notag\\
&=:&J_0+\sum_{k=1}^\infty J_k.
\end{eqnarray*}
Using that $\mathcal{R}_\lambda$  satisfies \eqref{stand-}  we see that the contribution to $I$ from the term defined by $J_0$ is bounded by
\begin{eqnarray}\label{rhh}
&&c\sum_{j=-\infty}^\infty\sum_{Q\in \mathbb D_{-j}}\int_{2^j}^{2^{j+1}}
\int_Q\biggl (\mean{2Q}\biggl |\frac {G_{(x,t,\lambda)}(y,s)}{\lambda}\biggr|^2\, dyds\biggr )\, \frac {dxdtd\lambda}\lambda\notag\\
&\leq &c\int_0^\infty
\int_{\mathbb R^{n+1}}|\beta(x,t,\lambda)|^2\frac {dxdtd\lambda}\lambda,
\end{eqnarray}
where
\begin{eqnarray*}|\beta(x,t,\lambda)|^2&=&\mean{B_{c\lambda}(x,t)}\biggl |\frac {G_{(x,t,\lambda)}(y,s)}{\lambda}\biggr|^2\, {dyds}\notag\\
&=&\mean{B_{c\lambda}(x,t)}\biggl |\frac {f(y,s)-f(x,t)-(y-x)\cdot \P_\lambda(\nabla_{||}f)(x,t)}{\lambda}\biggr|^2\, {dyds},
\end{eqnarray*}
and where $B_{c\lambda}(x,t)$ now is a standard parabolic ball centered at $(x,t)$ and of radius $c\lambda$.  To estimate the expression on the last line in \eqref{rhh}  we
change variables $(y,s)=(x,t)+(z,w)$ in the definition of $\beta(x,t,\lambda)$ and apply Plancerel's theorem. Indeed, doing so and letting
$$K_\lambda(z,w,\xi,\tau):=\frac{|e^{i(\xi,\tau)\cdot(z,w)}-1-i(z\cdot\xi)\hat \P(\lambda\xi,\lambda^2\tau)|}
{||(\xi,\tau)||}$$
we see that
\begin{eqnarray*}
&&\int_0^\infty\int_{\mathbb R^{n+1}}|\beta(x,t,\lambda)|^2\frac {dxdtd\lambda}\lambda\notag\\
&\leq&\int_0^\infty\int_{\mathbb R^{n+1}}\lambda^{-2}\mean{B_{c\lambda}(0,0)}(K_\lambda(z,w,\xi,\tau))^2|||(\xi,\tau)||\hat f|^2\, \frac{dzdw d\xi d\tau d\lambda}\lambda\notag\\
&\leq&c\mean{B_{c}(0,0)}\int_{\mathbb R^{n+1}}|||(\xi,\tau)||\hat f|^2\int_0^\infty(K_\lambda(z,w,\lambda \xi,\lambda^2\tau))^2\, \frac{ d\lambda d\xi d\tau dzdw}\lambda.
\end{eqnarray*}
We now argue  as on p. 250 in \cite{H}. Indeed, using that $\P\in C_0^\infty(\mathbb R^{n+1},\mathbb R)$ we have that $\hat \P\in C^\infty$ and
$|\hat \P(\xi,\tau)|\leq (1+||(\xi,\tau)||)^{-1}$. Also $\hat \P(0)=1$. Thus, using Taylor's formula, and that fact that $||(x,t)||^2\approx |x|^2+|t|$, we see that
\begin{eqnarray*}
K_\lambda(z,w,\lambda \xi,\lambda^2\tau)\leq c\min\{\lambda ||(\xi,\tau)||,(\lambda ||(\xi,\tau)||)^{-1}\}.
\end{eqnarray*}
Combining the estimates in the last two displays we see that
\begin{eqnarray*}
\int_0^\infty\int_{\mathbb R^{n+1}}|\beta(x,t,\lambda)|^2\frac {dxdtd\lambda}\lambda\leq c||\mathbb D f||_2^2.
\end{eqnarray*}
By a similar argument, see also Theorem 3.9 in \cite{AAAHK}, using also that $R_\lambda$ satisfies \eqref{stand} for some integer $d\geq 0$, we can conclude that the  contribution to $I$ from the term defined by $\sum_{k=1}^\infty J_k$ also is bounded by $||\mathbb Df||_2^2$. We omit further details.
\end{proof}

\section{Boundedness of single layer potentials}\label{sec3}

We here collect a number of estimates related to the boundedness of (single) layer potentials: off-diagonal estimates, uniform  (in $\lambda$) $L^2$-estimates, estimates of non-tangential maximal functions and square functions. Much of the material in this sections is a summary of the key results established in \cite{CNS}. As mentioned, \cite{CNS} should be seen as a companion to this paper. We will consistently only formulate and prove results for
 $\mathcal{S}_\lambda=\mathcal{S}_\lambda^{\mathcal{H}}$, and for $\lambda>0$. Throughout the section we will consistently assume that $\mathcal{H}=\partial_t-\div A\nabla$ satisfies \eqref{eq3}-\eqref{eq4} as well as \eqref{eq14+}-\eqref{eq14++}. The corresponding results for $\lambda<0$ and for $\mathcal{S}_\lambda^\ast=\mathcal{S}_\lambda^{\mathcal{H}^\ast}$ follow by analogy. Recall the notation $|||\cdot|||_+$, $\Phi_+(f)$, introduced in \eqref{keyestint-ex+-}, \eqref{keyestint-ex+}.  Given $f\in L^2(\mathbb R^{n+1},\mathbb C)$ we let
                \begin{eqnarray}
     (\mathcal{S}_\lambda D_j)f(x,t)&:=&\int_{\mathbb R^{n+1}}\partial_{y_j}\Gamma_\lambda(x,t,y,s)f(y,s)\, dyds,\ 1\leq j\leq n,\notag\\
     (\mathcal{S}_\lambda D_{n+1})f(x,t)&:=&\int_{\mathbb R^{n+1}}\partial_{\sigma}\Gamma(x,t,\lambda,y,s,\sigma)|_{\sigma=0}f(y,s)\, dyds,
     \end{eqnarray}
     recall that  $D_i=\partial_{x_i}$ for
$i\in\{1,...,n+1\}$, and we set
              \begin{eqnarray}
     (\mathcal{S}_\lambda\nabla)&:=&((\mathcal{S}_\lambda D_1),...,(\mathcal{S}_\lambda D_n),(\mathcal{S}_\lambda D_{n+1})),\notag\\
     (\mathcal{S}_\lambda\nabla\cdot){\bf f}&:=&\sum_{j=1}^{n+1}(\mathcal{S}_\lambda D_j)f_j,
     \end{eqnarray}
     whenever ${\bf f}=(f_1,...,f_{n+1})\in L^2(\mathbb R^{n+1},\mathbb C^{n+1})$. Using the notation  $\nabla =(\nabla_{||},\partial_\lambda)$, $\nabla_{||}=(\partial_{x_1},...,\partial_{x_n})$, $\div_{||}=\nabla_{||}\cdot$, we have
     \begin{eqnarray}
     (\mathcal{S}_\lambda\nabla_{||})\cdot{\bf f}_{||}(x,t)=-\mathcal{S}_\lambda(\div_{||}{\bf f}_{||}),\ (\mathcal{S}_\lambda D_{n+1})=-\partial_\lambda\mathcal{S}_\lambda,
     \end{eqnarray}
     whenever ${\bf f}=({\bf f}_{||},f_{n+1}) \in C_0^\infty(\mathbb R^{n+1},\mathbb C^{n+1})$. Furthermore, in line with \cite{AAAHK}, at instances we will find it appropriate to consider smoothed layer potentials in order to make certain otherwise formal manipulations rigorous. In particular, some of the estimates for these smoothed layer potentials  will not be used quantitatively, but will only serve to justify the otherwise formal manipulations. For $\eta>0$ we set
 \begin{eqnarray}\label{sop}\mathcal{S}_\lambda^\eta=\int_{\mathbb R}\varphi_\eta(\lambda-\sigma)\mathcal{S}_\sigma\, d\sigma,
 \end{eqnarray}
where $\varphi_\eta=\tilde\varphi_\eta\ast\tilde\varphi_\eta$,  $\tilde\varphi_\eta(\lambda)=\eta^{-1}\tilde\varphi_\eta(\lambda/\eta)$ and $\tilde\varphi_\eta\in C_0^\infty(-\eta/2,\eta/2)$ is a non-negative and even function satisfying $\int\tilde\varphi_\eta=1$. Note that, by construction,
$\partial_\lambda \mathcal{S}_\lambda^\eta$ exists and is continuous over the boundary $\partial \mathbb R^{n+2}_+=\mathbb R^{n+1}=\{(x,t,\lambda)\in \mathbb R^{n}\times\mathbb R\times\mathbb R:\ \lambda=0\}$. We also note that
    \begin{eqnarray}\label{fsmooth}
\mathcal{H}\mathcal{S}_\lambda^\eta f(x,t)=f_\eta(x,t,\lambda):=f(x,t)\varphi_\eta(\lambda),
     \end{eqnarray}
     whenever $(x,t,\lambda)\in\mathbb R^{n+2}$. In particular, $\mathcal{S}_\lambda^\eta f(x,t)= (\mathcal{H}^{-1}f_\eta)(x,t,\lambda)$. We let
          \begin{eqnarray}\label{keyestint-ex+se}
\Phi^\eta(f):=\sup_{\lambda\neq 0}||\partial_\lambda \mathcal{S}_\lambda^\eta f||_2+|||\lambda\partial_\lambda^2 \mathcal{S}_{\lambda}^\eta f|||.
     \end{eqnarray}

     \subsection{Kernel estimates and consequences} Given a function $f\in L^2(\mathbb R^{n+1},\mathbb C)$, and $h=(h_1,...,h_{n+1})\in \mathbb R^{n+1}$, we let $(\mathbb D^hf)(x,t)=f(x_1+h_1,...,x_n+h_h,t+h_{n+1})-f(x,t)$. Given $m\geq -1$, $l\geq -1$ we let
     \begin{eqnarray}\label{kernel}
      K_{m,\lambda}(x,t,y,s)&=&\partial_\lambda^{m+1}\Gamma_\lambda(x,t,y,s),\notag\\
     K_{m,l,\lambda}(x,t,y,s)&=&\partial_t^{l+1}\partial_\lambda^{m+1}\Gamma_\lambda(x,t,y,s),
     \end{eqnarray}
     and introduce \begin{eqnarray}
     d_\lambda(x,t,y,s):=|x-y|+|t-s|^{1/2} +\lambda.
     \end{eqnarray}
     Below Lemma \ref{le2+}, Lemma \ref{le3}, Lemma \ref{le4} and  Lemma \ref{le5} are Lemma 3.2, Lemma 3.3, Lemma 3.4 and Lemma 3.5 in \cite{CNS}, respectively.

\begin{lemma}\label{le2+} Assume  $m\geq -1$, $l\geq -1$. Then there exists constants $c_{m,l}$ and $\alpha\in (0,1)$, depending at most on $n$, $\Lambda$, the De Giorgi-Moser-Nash constants, $m$, $l$, such that
\begin{eqnarray*}
(i)&&|K_{m,l,\lambda}(x,t,y,s)|\leq c_{m,l}({d_\lambda(x,t,y,s)})^{-n-m-2l-4},\notag\\
(ii)&&|(\mathbb D^hK_{m,l,\lambda}(\cdot,\cdot,y,s))(x,t)|\leq c_{m,l}||h||^\alpha({d_\lambda(x,t,y,s)})^{-n-m-2l-4-\alpha},\notag\\
(iii)&&|(\mathbb D^hK_{m,l,\lambda}(x,t\cdot,\cdot))(y,s)|\notag\leq c_{m,l}||h||^\alpha({d_\lambda(x,t,y,s)})^{-n-m-2l-4-\alpha},
\end{eqnarray*}
whenever $2||h||\leq ||(x-y,t-s)||$ or $||h||\leq 20\lambda$.
\end{lemma}

\begin{lemma}\label{le3} Consider  $m\geq -1$, $l\geq -1$. Then there exists a constant
     $c_{m,l}$, depending at most on $n$, $\Lambda$, the De Giorgi-Moser-Nash constants, $m$, $l$, such that the following holds whenever $Q\subset\mathbb R^{n+1}$ is a parabolic cube, $k\geq 1$ is an integer and $(x,t)\in Q$.
\begin{eqnarray*}
(i)&&\int_{2^{k+1}Q\setminus 2^kQ}|(2^kl(Q))^{m+2l+3}\nabla_yK_{m,l,\lambda}(x,t,y,s)|^2dy ds\leq c_{m,l}(2^kl(Q))^{-n-2},\notag\\
(ii)&&\int_{2Q}|(l(Q))^{m+2l+3}\nabla_yK_{m,l,\lambda}(x,t,y,s)|^2dy ds\leq c_{m,l,\rho}(l(Q))^{-n-2},\notag\\
&&\mbox{whenever } l(Q)/\rho\leq\lambda\leq \rho l(Q).
\end{eqnarray*}
\end{lemma}

\begin{lemma}\label{le4} Assume  $m\geq -1$, $l\geq -1$. Then there exists a constant
     $c_{m,l}$, depending at most on $n$, $\Lambda$, the De Giorgi-Moser-Nash constants, $m$, $l$, such that the following holds whenever $Q\subset\mathbb R^{n+1}$ is a parabolic cube, $k\geq 1$ is an integer. Let ${\bf f}
     \in L^2(\mathbb R^{n+1},\mathbb C^{n})$, ${f}
     \in L^2(\mathbb R^{n+1},\mathbb C)$. Then
     \begin{eqnarray*}
(i)&&||\partial_t^{l+1}\partial_\lambda^{m+1}(\mathcal{S}_\lambda\nabla_{||}\cdot)({\bf f}1_{2^{k+1}Q\setminus 2^kQ})||_{L^2(Q)}^2\leq \ c_{m,l}2^{-(n+2)k}(2^kl(Q))^{-2m-4l-6}||{\bf f}||^2_{L^2(2^{k+1}Q\setminus 2^kQ)},\notag\\
(ii)&&||\partial_t^{l+1}\partial_\lambda^{m+1}(\mathcal{S}_\lambda\nabla_{||}\cdot)({\bf f}1_{2Q})||_{L^2(Q)}^2
	\ \leq \ c_{m,l,\rho}(l(Q))^{-2m-4l-6}||{\bf f}||^2_{L^2(2Q)},\notag\\
&&\mbox{whenever $\rho>0$, $l(Q)/\rho\leq\lambda\leq \rho l(Q)$}.\notag\\
(iii)&&||\partial_t^{l+1}\partial_\lambda^{m+1}(\mathcal{S}_\lambda)({f}1_{2^{k+1}Q\setminus 2^kQ})||_{L^2(Q)}^2\leq \ c_{m,l}2^{-(n+2)k}(2^kl(Q))^{-2m-4l-4}||{f}||^2_{L^2(2^{k+1}Q\setminus 2^kQ)},\notag\\
(iv)&&||\partial_t^{l+1}\partial_\lambda^{m+1}(\mathcal{S}_\lambda)({f}1_{2Q})||_{L^2(Q)}^2
\ \leq \
c_{m,l,\rho}(l(Q))^{-2m-4l-4}||{f}||^2_{L^2(2Q)},\notag\\
&&\mbox{whenever $\rho>0$, $l(Q)/\rho\leq\lambda\leq \rho l(Q)$}.
\end{eqnarray*}
\end{lemma}

\begin{lemma}\label{le5}  Assume $m\geq -1$, $l\geq -1$, $m+2l\geq -2$,  Then there exists a constant
     $c_{m,l}$, depending at most on $n$, $\Lambda$, the De Giorgi-Moser-Nash constants, $m$, $l$, such that the following holds. Let ${\bf f}
     \in L^2(\mathbb R^{n+1},\mathbb C^{n})$ and ${f}
     \in L^2(\mathbb R^{n+1},\mathbb C)$. Then
 \begin{eqnarray*}
(i)&&\sup_{\lambda>0}||\lambda^{m+2l+3}\partial_t^{l+1}\partial_\lambda^{m+1}(\mathcal{S}_\lambda\nabla_{||}\cdot){\bf f}||_2\leq c_{m,l}||{\bf f}||_2,\notag\\
(ii)&&\sup_{\lambda>0} ||\lambda^{m+2l+3}\partial_t^{l+1}\partial_\lambda^{m+1}(\nabla_{||} \mathcal{S}_\lambda f)||_2\leq c_{m,l}||{f}||_2.
\end{eqnarray*}
Furthermore, if $m+2l\geq -1$,  then
 \begin{eqnarray*}
(iii)&&\sup_{\lambda>0} ||\lambda^{m+2l+2}\partial_t^{l+1}\partial_\lambda^{m+1} (\mathcal{S}_\lambda f)||_2\leq c_{m,l}||{f}||_2.
\end{eqnarray*}
\end{lemma}

\begin{lemma}\label{le5ga}  Assume $m\geq -1$, $l\geq -1$, $m+2l\geq -2$,  Then there exists a constant
     $c_{m,l}$, depending at most on $n$, $\Lambda$, the De Giorgi-Moser-Nash constants, $m$, $l$, such that the following holds. Let ${\bf f}
     \in L^2(\mathbb R^{n+1},\mathbb C^{n})$ and ${f}
     \in L^2(\mathbb R^{n+1},\mathbb C)$. Then
 \begin{eqnarray*}
(i)&&\sup_{\lambda>0}||\lambda^{m+2l+4}\nabla\partial_t^{l+1}\partial_\lambda^{m+1}(\mathcal{S}_\lambda\nabla_{||}\cdot){\bf f}||_2\leq c_{m,l}||{\bf f}||_2,\notag\\
(ii)&&\sup_{\lambda>0} ||\lambda^{m+2l+4}\nabla_{||}\partial_t^{l+1}\partial_\lambda^{m+1}(\nabla_{||} \mathcal{S}_\lambda)f||_2\leq c_{m,l}||{f}||_2.
\end{eqnarray*}
Furthermore, if $m+2l\geq -1$,  then
 \begin{eqnarray*}
(iii)&&\sup_{\lambda>0} ||\lambda^{m+2l+3}\nabla_{||}\partial_t^{l+1}\partial_\lambda^{m+1} (\mathcal{S}_\lambda f)||_2\leq c_{m,l}||{f}||_2.
\end{eqnarray*}
\end{lemma}
\begin{proof} The lemma follows immediately from  Lemma \ref{le1--} and Lemma \ref{le5}.
\end{proof}

\begin{lemma}\label{appf} Let $f\in C_0^\infty(\mathbb R^{n+1},\mathbb C)$ and $\lambda>0$. Then $\mathcal{S}_\lambda f\in  {\mathbb H}(\mathbb R^{n+1},\mathbb C)\cap L^2(\mathbb R^{n+1},\mathbb C)$.
\end{lemma}
\begin{proof}
Given $f\in C_0^\infty(\mathbb R^{n+1},\mathbb C)$ we let $Q\subset \mathbb R^{n+1}$ be a parabolic cube, centered at $(0,0)$, such that the support of $f$ is contained in $Q$. Let $\lambda>0$ be fixed. We have to prove that $||\nabla_{||}\mathcal{S}_\lambda f||_2<\infty$,  $||H_tD_{1/2}^t\mathcal{S}_\lambda f||_2<\infty$, and that $||\mathcal{S}_\lambda f||_2<\infty$. To estimate $||\nabla_{||}\mathcal{S}_\lambda f||_2$ we see, by duality, that it suffices to bound
 \begin{eqnarray*}
\int_{Q}|(\mathcal{S}_\lambda^\ast\nabla_{||}\cdot){\bf f}(x,t)|^2\, dxdt&\leq&\int_{Q}|(\mathcal{S}_\lambda^\ast\nabla_{||}\cdot)({\bf f}1_{2Q})(x,t)|^2\, dxdt\notag\\
&&+\sum_{k\geq 1}\int_{Q}|(\mathcal{S}_\lambda^\ast \nabla_{||}\cdot)({\bf f}1_{2^{k+1}Q\setminus 2^kQ})(x,t)|^2\, dxdt,
\end{eqnarray*}
where ${\bf f}\in C_0^\infty(\mathbb R^{n+1},\mathbb C^n)$, $||{\bf f}||_2=1$. However, using  the adjoint version of Lemma \ref{le4} $(i)$ with $l=-1=m$, we immediately see that
 \begin{eqnarray*}
\int_{Q}|(\mathcal{S}_\lambda^\ast\nabla_{||}\cdot){\bf f}(x,t)|^2\, dxdt\leq c(n,\Lambda,\lambda)<\infty,
\end{eqnarray*}
 whenever ${\bf f}\in C_0^\infty(\mathbb R^{n+1},\mathbb C^n)$, $||{\bf f}||_2=1$. To estimate $||H_tD_{1/2}^t\mathcal{S}_\lambda f||_2$ we first note that
 $$||H_tD_{1/2}^t\mathcal{S}_\lambda f||_2^2\leq ||\partial_t\mathcal{S}_\lambda f||_2||\mathcal{S}_\lambda f||_2.$$
 Using Lemma \ref{le5} $(iii)$ we see that $||\partial_t\mathcal{S}_\lambda f||_2\leq c(n,\Lambda,\lambda)||f||_2<\infty$. To estimate $||\mathcal{S}_\lambda f||_2$ we again use duality and note that it suffices to bound
 \begin{eqnarray*}
\int_{Q}|\mathcal{S}_\lambda^\ast g(x,t)|^2\, dxdt&\leq&\int_{Q}|\mathcal{S}_\lambda^\ast (g1_{2Q})(x,t)|^2\, dxdt\notag\\
&&+\sum_{k\geq 1}\int_{Q}|\mathcal{S}_\lambda^\ast (g1_{2^{k+1}Q\setminus 2^kQ})(x,t)|^2\, dxdt,
\end{eqnarray*}
where $g\in C_0^\infty(\mathbb R^{n+1},\mathbb C)$, $||g||_2=1$. Using this  it is easy to see that
 \begin{eqnarray*}
\int_{Q}|\mathcal{S}_\lambda^\ast g(x,t)|^2\, dxdt\leq c(n,\Lambda,\lambda)<\infty,
\end{eqnarray*}
whenever $g\in C_0^\infty(\mathbb R^{n+1},\mathbb C)$, $||g||_2=1$. This completes the proof of the lemma. \end{proof}

\begin{lemma}\label{smooth1}  Let $\mathcal{S}_\lambda$ denote the single layer associated to $\mathcal{H}$, consider $\eta\in (0,1/10)$ and let $\mathcal{S}_\lambda^\eta$ be the smoothed
single layer associated to $\mathcal{H}$ introduced in \eqref{sop}.  Then
         \begin{eqnarray}
     (i)&&||\partial_\lambda\mathcal{S}_\lambda^\eta f||_2\leq c_{\beta,\eta}||f||_{2(n+2)/(n+2+2\beta)},\ 0<\beta<1,\notag\\
     (ii)&&||\nabla_{||}\mathcal{S}_\lambda^\eta f||_2\leq c_{\eta}||f||_{2(n+3)/(n+5)},\notag\\
     (iii)&&||H_tD^t_{1/2}\mathcal{S}_\lambda^\eta f||_2\leq c_{\eta}||f||_{2(n+3)/(n+5)},\notag\\
     (iv)&&|||\lambda\partial_\lambda^2\mathcal{S}_\lambda^\eta f|||_+\leq c_{\beta,\eta}||f||_{2(n+1)/(n+1+2\beta)},\ 0<\beta<1,\notag\\
     (v)&&||\nabla(\mathcal{S}_\lambda^\eta-\mathcal{S}_\lambda)f||_2\leq c\eta||f||_2/\lambda,\ \eta<\lambda/2, \notag\\
     (vi)&&||H_tD^t_{1/2}(\mathcal{S}_\lambda^\eta-\mathcal{S}_\lambda)f||_2\leq c\eta ||f||_2^{1/2}\Phi_+(f)^{1/2}/\lambda,\ \eta<\lambda/2, \notag\\
     (vii)&&\lim_{\eta\to 0}\int_\epsilon^\infty\int_{\mathbb R^n}|\lambda\nabla\partial_\lambda(\mathcal{S}_\lambda^\eta-\mathcal{S}_\lambda)f|^2\,
     \frac{dxdtd\lambda}\lambda=0,\ 0<\epsilon<1,\notag\\
     (viii)&&\mbox{for each cube $Q\subset\mathbb R^{n+1}$}, ||\partial_\lambda\mathcal{S}_\lambda^\eta||_{L^2(Q)\to L^2(\mathbb R^{n+1})}\leq
     c_{\eta,l(Q)},
     \end{eqnarray}
     whenever $f\in L^2(\mathbb R^{n+1},\mathbb C)$ has compact support. In $(v)-(vii)$ the constant $c$ depends at most
     on $n$, $\Lambda$, and the De Giorgi-Moser-Nash constants. In $(viii)$ the constant $c_{\eta,l(Q)}$ depends at most
     on $n$, $\Lambda$, the De Giorgi-Moser-Nash constants, $l(Q)$ and $l(Q)$
     \end{lemma}
     \begin{proof} To prove $(i)$ we note that
     $$\partial_\lambda\mathcal{S}_\lambda^\eta f=\int_{\mathbb R^{n+1}}K_{0,\lambda}^\eta(x,t,y,s)f(y,s)\, dyds,$$
     where $K_{0,\lambda}^\eta(x,t,y,s)=\partial_\lambda(\varphi_\eta\ast(\Gamma_\cdot(x,t,y,s))(\lambda))$. Using Lemma \ref{le2+} we see that
     \begin{eqnarray}|K_{0,\lambda}^\eta(x,t,y,s)|&\leq& c\biggl (\frac {1_{d_\lambda(x,t,y,s)>40\eta}}{(d_\lambda(x,t,y,s))^{n+2}}+\frac {1_{d_\lambda(x,t,y,s)<40\eta}}{\eta(|x-y|+|t-s|^{1/2})^{n+1}}\biggr )\notag\\
     &\leq&c\eta^{-\beta}(|x-y|+|t-s|^{1/2})^{\beta-n-2},
     \end{eqnarray}
     for $0<\beta<1$. $(i)$ now follows by the parabolic version of the Hardy-Littlewood-Sobolev theorem for fractional integration (see \cite{St} for the corresponding proof in the elliptic case) . To prove $(ii)$ and $(iii)$ we first note that
        \begin{eqnarray}
        \mathcal{S}_\lambda^\eta f(x,t)&=&\int_{\mathbb R}\int_{\mathbb R}\int_{\mathbb R^{n+1}} \Gamma_{\lambda-\sigma_1-\sigma_2}(x,t,y,s)f(y,s)\tilde\varphi_\eta(\sigma_1)\tilde\varphi_\eta(\sigma_2)\, dydsd\sigma_1d\sigma_2\notag\\
        &=&\int_{\mathbb R}(\mathcal{H}^{-1}f_\eta)(x,t,\lambda-\sigma)\tilde\varphi_\eta(\sigma)\, d\sigma,
     \end{eqnarray}
     where $f_\eta(y,s,\sigma_1)=f(y,s)\tilde\varphi_\eta(\sigma_1)$. To prove $(ii)$, let ${\bf g}\in C_0^\infty(\mathbb R^{n+1},\mathbb C^{n})$, $||{\bf g}||_2=1$, and set
     ${\bf g}_\eta(x,t,\sigma)={\bf g}(x,t)\tilde\varphi_\eta(\sigma)$. Then
          \begin{eqnarray}
       \biggl | \int_{\mathbb R^{n+1}}{\bf g}\cdot\nabla_{||}\overline{\mathcal{S}_\lambda^\eta f}\, dxdt\biggr |&=&
       \biggl | \int_{\mathbb R}\int_{\mathbb R^{n+1}}\div_{||}{\bf g}_\eta(x,t,\sigma)\overline{(\mathcal{H}^{-1}f_\eta)(x,t,\lambda-\sigma)}\, dxdtd\sigma\biggr |\notag\\
       &\leq &c||{\bf g}_\eta||_{L^2(\mathbb R^{n+2})}||\nabla_{||}(\mathcal{H}^{-1}f_\eta)||_{L^2(\mathbb R^{n+2})}\notag\\
       &\leq& c\eta^{-1/2}||\nabla_{||}(\mathcal{H}^{-1}f_\eta)||_{L^2(\mathbb R^{n+2})}.
     \end{eqnarray}
     Hence, using Lemma \ref{gara} and the parabolic version of the Hardy-Littlewood-Sobolev theorem, now in $\mathbb R^{n+2}$, we see that
              \begin{eqnarray}\label{raesy1}
       \biggl | \int_{\mathbb R^{n+1}}{\bf g}\cdot\nabla_{||}\overline{\mathcal{S}_\lambda^\eta f}\, dxdt\biggr |&\leq& c\eta^{-1/2}
       ||\varphi_\eta||_{2(n+3)/(n+5)}||f||_{2(n+3)/(n+5)},
     \end{eqnarray}
     and this proves $(ii)$. To prove $(iii)$, let ${g}\in C_0^\infty(\mathbb R^{n+1},\mathbb C)$, $||{ g}||_2=1$, and set
     ${g}_\eta(x,t,\sigma)={g}(x,t)\tilde\varphi_\eta(\sigma)$. Then, arguing as above we see that
         \begin{eqnarray}
       \biggl | \int_{\mathbb R^{n+1}}{g}H_tD_{1/2}^t{\mathcal{S}_\lambda^\eta f}\, dxdt\biggr |\leq c\eta^{-1/2}||H_tD_{1/2}^t(\mathcal{H}^{-1}f_\eta)||_{L^2(\mathbb R^{n+2})}.
     \end{eqnarray}
     Furthermore, again using  Lemma \ref{gara}  and arguing as in the proof of $(ii)$ we have
                \begin{eqnarray}\label{raesy1again}
\biggl | \int_{\mathbb R^{n+1}}{g}H_tD_{1/2}^t{\mathcal{S}_\lambda^\eta f}\, dxdt\biggr |&\leq& c\eta^{-1/2}
       ||\varphi_\eta||_{2(n+3)/(n+5)}||f||_{2(n+3)/(n+5)},
     \end{eqnarray}
     and this proves $(iii)$. To prove $(iv)$ we proceed as in the proof of $(i)$ and we first note that
     $$\lambda\partial_\lambda^2\mathcal{S}_\lambda^\eta f=\int_{\mathbb R^{n+1}}\lambda K_{1,\lambda}^\eta(x,t,y,s)f(y,s)\, dyds,$$
     where $K_{1,\lambda}^\eta(x,t,y,s)=\partial_\lambda^2(\varphi_\eta\ast(\Gamma_\cdot(x,t,y,s))(\lambda))$.  Using Lemma \ref{le2+} we see that
     \begin{eqnarray}\lambda|K_{1,\lambda}^\eta(x,t,y,s)|&\leq& c\lambda \biggl (\frac {1_{d_\lambda(x,t,y,s)>40\eta}}{(d_\lambda(x,t,y,s))^{n+3}}+\frac {1_{d_\lambda(x,t,y,s)<40\eta}}{\eta(|x-y|+|t-s|^{1/2})^{n+2}}\biggr )\notag\\
     &\leq&c\lambda\eta^{-1-\beta}(|x-y|+|t-s|^{1/2})^{\beta-n-2},
     \end{eqnarray}
     for $0<\beta<1$. Moreover, if $\lambda>2\eta$ then
       \begin{eqnarray}\lambda|K_{1,\lambda}^\eta(x,t,y,s)|&\leq& c\lambda d_\lambda(x,t,y,s)^{-n-3}\notag\\
       &\leq& c\lambda^{-\beta}(|x-y|+|t-s|^{1/2})^{\beta-n-2},
     \end{eqnarray}
     for $0<\beta<1$. Hence, arguing as in the proof of $(i)$ we see that
              \begin{eqnarray}
  |||\lambda\partial_\lambda^2\mathcal{S}_\lambda^\eta f|||_+^2&=&\int_0^{2\eta}\int_{\mathbb R^{n+1}} |\lambda\partial_\lambda^2\mathcal{S}_\lambda^\eta f(x,t)|^2\, \frac{dxdtd\lambda}\lambda\notag\\
  &&+\int_{2\eta}^\infty\int_{\mathbb R^{n+1}} |\partial_\lambda^2\mathcal{S}_\lambda^\eta f(x,t)|^2\, \frac{dxdtd\lambda}\lambda\notag\\
  &\leq &c\biggl (\int_0^{2\eta}\eta^{-2-2\beta}\,\lambda d\lambda\biggr)||f||_{2(n+1)/(n+1+2\beta)}\notag\\
  &&+c\biggl (\int_{2\eta}^\infty\lambda^{-1-2\beta}\, d\lambda\biggr)||f||_{2(n+1)/(n+1+2\beta)}.
  \end{eqnarray}
  This proves $(iv)$. To prove $(v)$, let $\eta<\lambda/2$ and note that
     $$||\nabla(\mathcal{S}_\lambda^\eta-\mathcal{S}_\lambda)f||_2\leq \varphi_\eta\ast||\nabla(\mathcal{S}_\cdot-\mathcal{S}_\lambda)f||_2.$$
     Furthermore, for $|\sigma-\lambda|<\lambda/2$ we see, using the mean value theorem, that
       \begin{eqnarray}
 ||\nabla(\mathcal{S}_\sigma-\mathcal{S}_\lambda)f||_2\leq\frac {\eta}\lambda\sup_{|\tilde\sigma-\lambda|<\lambda/2}||\tilde\sigma\nabla\partial_{\tilde\sigma}
 \mathcal{S}_{\tilde\sigma}f||_2.
     \end{eqnarray}
     Hence, using Lemma \ref{le5} we can therefore conclude that
           \begin{eqnarray}
 ||\nabla(\mathcal{S}_\sigma-\mathcal{S}_\lambda)f||_2\leq c\frac {\eta} \lambda||f||_2
     \end{eqnarray}
     whenever $|\sigma-\lambda|<\lambda/2$ and this completes the proof of $(v)$. To prove $(vi)$, let $\eta<\lambda/2$ and note that
     $$||H_tD_{1/2}^t(\mathcal{S}_\lambda^\eta-\mathcal{S}_\lambda)f||_2\leq \varphi_\eta\ast||H_tD_{1/2}^t(\mathcal{S}_\cdot-\mathcal{S}_\lambda)f||_2.$$
     However, for $|\sigma-\lambda|<\lambda/2$ and again using the mean value theorem we see that
       \begin{eqnarray}
 ||H_tD_{1/2}^t(\mathcal{S}_\sigma-\mathcal{S}_\lambda)f||_2\leq\frac {\eta}\lambda\sup_{|\tilde\sigma-\lambda|<\lambda/2}||\tilde\sigma H_tD_{1/2}^t\partial_{\tilde\sigma}
 \mathcal{S}_{\tilde\sigma}f||_2.
     \end{eqnarray}
     Furthermore,
            \begin{eqnarray}
||\tilde\sigma H_tD_{1/2}^t\partial_{\tilde\sigma}
 \mathcal{S}_{\tilde\sigma}f||_2^2&\leq& c ||\tilde\sigma^2 \partial_t\partial_{\tilde\sigma}
 \mathcal{S}_{\tilde\sigma}f||_2||\partial_{\tilde\sigma}
 \mathcal{S}_{\tilde\sigma}f||_2\notag\\
 &\leq& c||f||_2||\partial_{\tilde\sigma}
 \mathcal{S}_{\tilde\sigma}f||_2
     \end{eqnarray}
     where we again have used Lemma \ref{le5}. Hence,
            \begin{eqnarray}
            ||H_tD_{1/2}^t(\mathcal{S}_\lambda^\eta-\mathcal{S}_\lambda)f||_2\leq c\frac {\eta} \lambda||f||_2^{1/2}\Phi_+(f)^{1/2}
     \end{eqnarray}
     and this completes the proof of $(vi)$. To prove $(vii)$,  we let $\eta<\epsilon/2$ and write
     \begin{eqnarray}
&&\int_\epsilon^\infty\int_{\mathbb R^{n+1}}|\lambda\nabla\partial_\lambda(\mathcal{S}_\lambda^\eta-\mathcal{S}_\lambda)f|^2\,
     \frac{dxdtd\lambda}\lambda\notag\\
     &=&\int_\epsilon^\infty\int_{\mathbb R^{n+1}}|\varphi_\eta\ast\lambda\nabla D_{n+1}(\mathcal{S}_\cdot-\mathcal{S}_\lambda)f|^2\,
     \frac{dxdtd\lambda}\lambda\notag\\
     &\leq&\int_\epsilon^\infty\varphi_\eta\ast||\lambda\nabla D_{n+1}(\mathcal{S}_\cdot-\mathcal{S}_\lambda)f||_2^2\,
     \frac{d\lambda}\lambda.
     \end{eqnarray}
     We claim that the expression on the last line in the last display converges to 0 as $\eta\to 0$. Indeed, for $|\sigma-\lambda|<\eta<\lambda/2$, we have, arguing as above using
     Lemma \ref{le5}, that
          \begin{eqnarray}
||\lambda\nabla D_{n+1}(\mathcal{S}_\sigma-\mathcal{S}_\lambda)f||_2&\leq& c\frac {\eta}\lambda \sup_{|\tilde\sigma-\lambda|<\lambda/2}||\tilde\sigma^2\nabla\partial_{\tilde\sigma}^2
 \mathcal{S}_{\tilde\sigma}f||_2\notag\\
 &\leq &c\frac\eta\lambda ||f||_2.\end{eqnarray}
 Hence, if $\eta<\epsilon/2$, then
      \begin{eqnarray}
\int_\epsilon^\infty\int_{\mathbb R^{n+1}}|\lambda\nabla\partial_\lambda(\mathcal{S}_\lambda^\eta-\mathcal{S}_\lambda)f|^2\,
     \frac{dxdtd\lambda}\lambda\leq c\eta^2\epsilon^{-2}||f||_2^2.
     \end{eqnarray}
     This proves $(vii)$. $(viii)$ follows from Lemma \ref{smooth1} $(i)$ and H{\"o}lder's inequality. This completes the proof of the lemma.\end{proof}

\subsection{Maximal functions, square functions and parabolic Sobolev spaces}

\begin{lemma}\label{lemsl1++}  Let $\mathcal{S}_\lambda$ denote the single layer associated to $\mathcal{H}$, consider $\eta\in (0,1/10)$ and let $\mathcal{S}_\lambda^\eta$ be the smoothed
single layer associated to $\mathcal{H}$ introduced in \eqref{sop}. Then there exists a constant $c$, depending at most
     on $n$, $\Lambda$, and the De Giorgi-Moser-Nash constants, such that
\begin{eqnarray*}
(i)&& ||N_\ast(\partial_\lambda \mathcal{S}_\lambda f)||_2\leq c(\sup_{\lambda>0}||\partial_\lambda \mathcal{S}_\lambda||_{2\to 2}+1)||f||_2,\notag\\
(ii)&&||\tilde N_\ast(\nabla_{||}\mathcal{S}_\lambda f)||_2\leq c \biggl (||f||_2+\sup_{\lambda>0}||\nabla_{||}\mathcal{S}_\lambda f||_2+||N_{\ast\ast}(\partial_\lambda \mathcal{S}_\lambda f)||_2\biggl ),\notag\\
(iii)&&||\tilde N_\ast(H_tD^t_{1/2}\mathcal{S}_\lambda f)||_2\leq c\biggl (||f||_2+\sup_{\lambda>0}||H_tD^t_{1/2}\mathcal{S}_\lambda f||_2\biggr)\notag\\
&&\quad\quad\quad\quad\quad\quad\quad\quad+c\biggl (||\tilde N_{\ast\ast}(\nabla_{||}
\mathcal{S}_\lambda f)||_2+||N_{\ast\ast}(\partial_\lambda \mathcal{S}_\lambda f)||_2\biggr ),\notag\\
(iv)&& \sup_{\lambda_0\geq 0}||N_\ast(\P_\lambda(\partial_\lambda\mathcal{S}_{\lambda+\lambda_0}^\eta f))||_2\leq c(
    \sup_{\lambda>0} ||\partial_\lambda\mathcal{S}_\lambda^\eta||_{L^2(Q)\to L^2(\mathbb R^{n+1})}+1)||f||_2\notag\\
    &&\mbox{whenever $Q\subset\mathbb R^{n+1}$ and the support of $f$ is contained in $Q$},\notag\\
(v)&&|| N_\ast(\P_\lambda(\nabla\mathcal{S}_\lambda f))||_2\leq c \biggl (\sup_{\lambda>0}||\nabla_{||}\mathcal{S}_\lambda f||_2+||N_{\ast\ast}(\partial_\lambda \mathcal{S}_\lambda f)||_2\biggl ),\notag\\
(vi)&&|| N_\ast(\P_\lambda(H_tD^t_{1/2}\mathcal{S}_\lambda f))||_2\leq c\biggl (||f||_2+\sup_{\lambda>0}||H_tD^t_{1/2}\mathcal{S}_\lambda f||_2\biggr)\notag\\
&&\quad\quad\quad\quad\quad\quad\quad\quad\quad\quad+c\biggl (||\tilde N_{\ast\ast}(\nabla_{||}
\mathcal{S}_\lambda f)||_2+||N_{\ast\ast}(\partial_\lambda \mathcal{S}_\lambda f)||_2\biggr ),\notag\\
  (vii)&&||N_\ast((\mathcal{S}_\lambda\nabla)\cdot{\bf f})||_{{2,\infty}}\leq c\bigl (1+\sup_{\lambda>0}||\partial_\lambda \mathcal{S}_\lambda||_{2\to 2}+\sup_{\lambda>0}||\mathcal{S}_\lambda\nabla_{||}||_{2\to 2}\bigr)||{\bf f}||_2,\notag\\
  (viii)&&||N_\ast(\mathcal{D}_\lambda f)||_{{2,\infty}}\leq c\bigl (1+\sup_{\lambda>0}||\partial_\lambda \mathcal{S}_\lambda||_{2\to 2}+\sup_{\lambda>0}||\mathcal{S}_\lambda\nabla_{||}||_{2\to 2}\bigr)||{\bf f}||_2,
\end{eqnarray*}
whenever $f\in L^2(\mathbb R^{n+1},\mathbb C)$, ${\bf f}\in L^2(\mathbb R^{n+1},\mathbb C^{n+1})$.
\end{lemma}
\begin{proof}[Proof of Lemma \ref{lemsl1++} $(i)-(iii)$] $(i)-(iii)$ are proved in Lemma 4.1 in \cite{CNS}.
\end{proof}

\begin{proof}[Proof of Lemma \ref{lemsl1++} $(iv)$] The proof of $(iv)$ is very similar to the proof of $(i)$ i.e., to the proof of Lemma 4.1 $(i)$ in \cite{CNS}. Indeed, let $K_{0,\lambda}^\eta(x,t,y,s)$ denote the kernel of $\partial_\lambda\mathcal{S}_\lambda^\eta$  and note again that
$$K_{0,\lambda}^\eta(x,t,y,s)=\partial_\lambda\bigl(\varphi_\eta\ast\Gamma(x,t,\cdot,y,s,0)\bigr)(\lambda).$$
Then by the Calderon-Zygmund type estimates stated in Lemma \ref{le2+} we have, for all $\lambda\geq 0$, and uniformly in $\lambda_0\geq 0$, that
\begin{eqnarray}\label{smoothker}
\quad\quad |K_{0,\lambda+\lambda_0}^\eta(x,t,y,s)|\leq c\biggl (\frac {1_{d_\lambda(x,t,y,s)>40\eta}}{(d_\lambda(x,t,y,s))^{n+2}}+\frac {1_{d_\lambda(x,t,y,s)<40\eta}}{\eta(|x-y|+|t-s|^{1/2})^{n+1}}\biggr ),
\end{eqnarray}
and
\begin{eqnarray}\label{smoothker+}
\quad |(\mathbb D^hK^\eta_{0,\lambda}(\cdot,\cdot,y,s))(x,t)|\leq c\frac {||h||^\alpha}{({d_\lambda(x,t,y,s)})^{n+2+\alpha}},\ d_\lambda(x,t,y,s)>10\eta,
\end{eqnarray}
whenever $2||h||\leq ||(x-y,t-s)||$ or $||h||\leq 2\lambda$. Of course we have a similar estimate concerning the parabolic H{\"o}lder continuity in the $(y,s)$ variables. In particular, $K_{0,\lambda}^\eta(x,t,y,s)$ is a standard (parabolic) Calderon-Zygmund kernel uniformly in $\lambda$, $\lambda_0$ and $\eta$. Hence, given $(x_0,t_0)\in\mathbb R^{n+1}$, and using that the support of $f$ is contained in $Q$, we can argue as in displays (4.4)-(4.10) \cite{CNS}, see also display (4.12) in \cite{AAAHK},  to conclude that
\begin{eqnarray}N_\ast(\P_\lambda(\partial_\lambda {\mathcal{S}}_{\lambda+\lambda_0}^\eta f))(x_0,t_0)&\leq&  \mathcal{T}_\ast^{l(Q)} f(x_0,t_0)+cM(f)(x_0,t_0)\notag\\
&&+cM(M(f))(x_0,t_0),
\end{eqnarray}
where
\begin{eqnarray}
 \mathcal{T}_\ast^{l(Q)} f(x_0,t_0)=\sup_{0<\epsilon<l(Q)} |\mathcal{T}_\epsilon^\delta f(x_0,t_0)|
\end{eqnarray}
and
\begin{eqnarray}
 \mathcal{T}_\epsilon^\delta f(x_0,t_0)=\int_{||(x_0-y,t_0-s)||>\epsilon}K_{0,\delta}^\eta(x_0,t_0,y,s)f(y,s)\, dyds.
\end{eqnarray}
$M$ is the standard parabolic the Hardy-Littlewood maximal. $(iv)$ now follows from these deductions and by proceeding as in the rest of the proof of Lemma 4.1 $(i)$ in \cite{CNS}. We refer the interested reader to \cite{CNS} for details.
\end{proof}

\begin{proof}[Proof of Lemma \ref{lemsl1++} $(v)$] To prove $(v)$  we first note that $N_\ast(\P_\lambda(\partial_\lambda\mathcal{S}_\lambda f))(x_0,t_0)\leq cM(N_\ast(\partial_\lambda\mathcal{S}_\lambda f))(x_0,t_0)$ and hence we only have to estimate $N_\ast(\P_\lambda(\nabla_{||}\mathcal{S}_\lambda f))$. Fix $(x_0,t_0)\in\mathbb R^{n+1}$ and consider $(x,t,\lambda)\in \Gamma(x_0,t_0)$. We now let, as we may, $\P_\lambda$ have a product structure, i.e., $\P_\lambda(x,t)=\P_\lambda^x(x)\P_\lambda^t(t)$. In the following we let $M^x$ and $M^t$ denote, respectively, the Hardy-Littlewood maximal operators acting in the
$x$ and $t$ variables only. To proceed we note, for $k\in \{1,...,n\}$, that
\begin{eqnarray}\label{apa1}
\P_\lambda(\partial_{x_k}\mathcal{S}_\lambda f)(x,t)&=&\P_\lambda^t\bigl(\P_\lambda^x(\partial_{x_k}\mathcal{S}_\lambda f)(x,\cdot)\bigr )(t)
\end{eqnarray}
and that
$$\P_\lambda^x(\partial_{x_k}\mathcal{S}_\lambda f)(x,\cdot)=\lambda^{-1}\mathcal{Q}_\lambda^x(\mathcal{S}_\lambda f)(x,\cdot)$$
where $Q_\lambda^x$ is an approximation of the zero operator, in $x$ only.  As  $Q_\lambda^x$ annihilates constants we have
\begin{eqnarray}\label{apa2}
\P_\lambda^x(\partial_{x_k}\mathcal{S}_\lambda f)(x,\cdot)&=&\lambda^{-1}\mathcal{Q}_\lambda^x\biggl (\int_\delta^\lambda\partial_\sigma\mathcal{S}_\sigma f\, d\sigma\biggr )(x,\cdot)\notag\\
&&+\lambda^{-1}\mathcal{Q}_\lambda^x\biggl (\mathcal{S}_\delta f-\mean{Q_{2\lambda}^x(x_0)}\mathcal{S}_\delta f\biggr )(x,\cdot),
\end{eqnarray}
for $\delta>0$ small and where $Q_{2\lambda}^x(x_0)$ now denotes the cube in $\mathbb R^{n}$, and in the spatial variables only, which is centered at $x_0$ and has size $2\lambda$.  But
\begin{eqnarray}\label{apa3}
\mathcal{Q}_\lambda^x\biggl (\lambda^{-1}\int_\delta^\lambda\partial_\sigma\mathcal{S}_\sigma f\, d\sigma\biggr )(x,\cdot)\leq cM^x(N_\ast(\partial_\lambda\mathcal{S}_\lambda f))(x_0,\cdot)
\end{eqnarray}
and by Poincare's inequality
\begin{eqnarray}\label{apa4}
\lambda^{-1}\mathcal{Q}_\lambda^x\biggl (
\mathcal{S}_\delta f-\mean{Q_{2\lambda}^x(x_0)}\mathcal{S}_\delta f\biggr )(x,\cdot)\leq cM^x(\nabla_{||}\mathcal{S}_\delta f)(x_0,\cdot).
\end{eqnarray}
Combining \eqref{apa1}-\eqref{apa4} we see that
\begin{eqnarray}
\P_\lambda(\partial_{x_k}\mathcal{S}_\lambda f)(x,t)&\leq& cM^t(M^x(N_\ast(\partial_\lambda\mathcal{S}_\lambda f))(x_0,\cdot))(t_0)\notag\\
&&+cM^t(M^x(\nabla_{||}\mathcal{S}_\delta f)(x_0,\cdot))(t_0),
\end{eqnarray}
whenever $(x,t,\lambda)\in \Gamma(x_0,t_0)$. Hence
\begin{eqnarray}
|| N_\ast(\P_\lambda(\nabla\mathcal{S}_\lambda f))||_2\leq c\bigl (||N_\ast(\partial_\lambda\mathcal{S}_\lambda f)||_2+
||\nabla_{||}\mathcal{S}_\delta f||_2\bigr ).
\end{eqnarray}
This completes the proof of $(v)$.
 \end{proof}

 \begin{proof}[Proof of Lemma \ref{lemsl1++} $(vi)$] To prove $(vi)$ we again let $(x_0,t_0)\in\mathbb R^{n+1}$ and we consider $(x,t,\lambda)\in \Gamma(x_0,t_0)$. We want to bound $\P_\lambda(H_tD_{1/2}^t\mathcal{S}_\lambda f)(x,t)$. Recall that $\P_\lambda$ has support in a parabolic cube centered at $(0,0)$ and with size  $\lambda$. Consider $(y,s)\in\mathbb R^{n+1}$ such that $||(y-x_0,s-t_0)||<8\lambda$ and let $K\gg 1$ be a degree of freedom to be chosen. Then
\begin{eqnarray*}
H_tD_{1/2}^t(\mathcal{S}_\lambda f)(y,s)&=&\lim_{\epsilon\to 0}\int_{\epsilon\leq |s-\tilde t|<1/\epsilon}\frac {\mbox{sgn}(s-\tilde t)}{|s-\tilde t|^{3/2}}(\mathcal{S}_\lambda f)(y,\tilde t)\, d\tilde t\notag\\
&=&\lim_{\epsilon\to 0}\int_{\epsilon\leq |s-\tilde t|<(K\lambda)^2}\frac {\mbox{sgn}(s-\tilde t)}{|s-\tilde t|^{3/2}}(\mathcal{S}_\lambda f)(y,\tilde t)\, d\tilde t\notag\\
&&+\lim_{\epsilon\to 0}\int_{(K\lambda)^2\leq |s-\tilde t|<1/\epsilon}\frac {\mbox{sgn}(s-\tilde t)}{|s-\tilde t|^{3/2}}(\mathcal{S}_\lambda f)(y,\tilde t)\, d\tilde t\notag\\
&=:&g_1(y,s,\lambda)+g_2(y,s,\lambda).
\end{eqnarray*}
Let
$$g_3(x_0,t_0,\lambda):=\sup_{\{y:\ |y-x_0|\leq 8\lambda\}}\sup_{\{\tau:\ |\tau-t_0|\leq (4K\lambda)^2\}}|\partial_\tau(\mathcal{S}_\lambda f)(y,\tau)|.$$
Then
\begin{eqnarray*}
|g_1(y,s,\lambda)|\leq cK\lambda g_3(x_0,t_0,\lambda),
\end{eqnarray*}
whenever $||(y-x_0,s-t_0)||<8\lambda$. Using this and arguing as in the argument leading up to the estimate in display (4.3) in \cite{CNS} we see that
\begin{eqnarray}\label{citama}
\P_\lambda(|g_1|)(x,t)\leq cM(f)(x_0,t_0),
\end{eqnarray}
where, as usual, $M$ is the standard parabolic the Hardy-Littlewood maximal. To estimate $g_2(y,s,\lambda)$, for $(y,s)$  as above, we introduce the function
\begin{eqnarray*}
g_4(\bar y,\bar s,\lambda)=\lim_{\epsilon\to 0}\int_{(K\lambda)^2\leq |\tilde s-\bar s|<1/\epsilon}\frac {\mbox{sgn}
(\bar s-\tilde s)}{|\bar s-\tilde  s|^{3/2}}({\mathcal{S}}_\delta f)(\bar y,\tilde  s)\, d\tilde  s,
\end{eqnarray*}
for $\delta$ small. Now
\begin{eqnarray*}
|g_2(y,s,\lambda)-g_4(x_0,t_0,\lambda)|&\leq & |g_2(y,s,\lambda)-g_2(x_0,s,\lambda)|\notag\\
&&+|g_2(x_0,s,\lambda)-g_2(x_0,t_0,\lambda)|\notag\\
&&+|g_2(x_0,t_0,\lambda)-g_4(x_0,t_0,\lambda)|.
\end{eqnarray*}
In particular,
\begin{eqnarray*}
|g_2(y,s,\lambda)-g_4(x_0,t_0,\lambda)|&\leq & \int_{(K\lambda)^2\leq |s-\tilde t|}\frac {|{\mathcal{S}}_\lambda f(y,\tilde  t)-{\mathcal{S}}_\lambda f(x_0,\tilde  t)|}{|\tilde  t-s|^{3/2}}\, d\tilde  t\notag\\
&&+\int_{(K\lambda)^2\leq |\xi|}\frac {|{\mathcal{S}}_\lambda f(x_0,\xi+s)-{\mathcal{S}}_\lambda f(x_0,\xi+t_0)|}{|\xi|^{3/2}}\, d\xi\notag\\
&&+\int_{(K\lambda)^2\leq |\tilde  t-t_0|}\frac {|{\mathcal{S}}_\lambda f(x_0,\tilde  t)-{\mathcal{S}}_\delta f(x_0,\tilde t)|}{|t_0-\tilde t|^{3/2}}\, d\tilde t\notag\\
&=:&h_1(y,s,\lambda)+h_2(y,s,\lambda)+h_3(x_0,t_0,\lambda).
\end{eqnarray*}
We note that
\begin{eqnarray*}
h_2(y,s,\lambda)&\leq&c\lambda^2\int_{(K\lambda)^2\leq |\xi|}\frac {N_\ast(\partial_t{\mathcal{S}}_\lambda f)(x_0,\xi+t_0)}{|\xi|^{3/2}}\, d\xi\notag\\
&\leq&c\lambda\int_{(K\lambda)^2\leq |\xi|}\frac {M(f)(x_0,\xi+t_0)}{|\xi|^{3/2}}\, d\xi\leq c
M^t(M(f)(x_0,\cdot))(t_0),
\end{eqnarray*}
where $M^t$ is the Hardy-Littlewood maximal operator in the $t$-variable, as we see by arguing as in the proof of \eqref{citama} above. Similarly,
\begin{eqnarray*}
h_3(y,s,\lambda)\leq c M^t(N_\ast(\partial_\lambda {\mathcal{S}}_\lambda f)(x_0,\cdot))(t_0).
\end{eqnarray*}
We therefore focus on $h_1(y,s,\lambda)$. Let
$$\tilde h_1(y)=\int_{\lambda^2\leq |\tilde  t-t_0|}\frac {|{\mathcal{S}}_\lambda f(y,\tilde t)-{\mathcal{S}}_\lambda f(x_0,\tilde t)|}{|\tilde t-t_0|^{3/2}}\, d\tilde t.$$
If  $K$ is large enough, then $h_1(y,s,\lambda)\leq c\tilde h_1(y)$, whenever $||(y-x_0,s-t_0)||<8\lambda$. To estimate $\tilde h_1(y)$ is a bit tricky. However, fortunately we can reuse the corresponding arguments in \cite{CNS}. Indeed, basically arguing as is done below display (4.4) in \cite{CNS} it follows that
\begin{eqnarray}\label{citama+}
\P_\lambda(h_1)(x,t)\leq c\P_\lambda(\tilde h_1)(x,t)\leq cM^t(\tilde N_{\ast\ast}(\nabla_{||}{\mathcal{S}}_\lambda f)(x_0,\cdot))(t_0).
\end{eqnarray}
Putting the estimates together we can conclude that
\begin{eqnarray*}
&& \P_\lambda(h_1)(x_0,t_0)+\P_\lambda(h_2)(x_0,t_0)+\P_\lambda(h_3)(x_0,t_0)\notag\\
&& \qquad \leq \ cM^t(\tilde N_{\ast\ast}(\nabla_{||}{\mathcal{S}}_\lambda f)(x_0,\cdot))(t_0)+c M^t(M(f)(x_0,\cdot))(t_0)\notag\\
&& \qquad \qquad + \ cM^t(N_\ast(\partial_\lambda {\mathcal{S}}_\lambda f)(x_0,\cdot))(t_0),
\end{eqnarray*}
where $M^t$ is the Hardy-Littlewood maximal operator in the $t$-variable and $M$ is the standard parabolic Hardy Littlewood maximal function. To complete the proof of $(vi)$ we let
 $$\psi_\delta(x_0,t_0):=\sup_{\lambda>\delta}|g_4(x_0,t_0,\lambda)|$$ and we note that it suffices to estimate  $||\psi_\delta||_2$. To do this we note that
$${\mathcal{S}}_\delta f(x,t)=cI_{1/2}^t(D_{1/2}^t{\mathcal{S}}_\delta f)(x,t)=c I_{1/2}^th_\delta(x,t), $$
where $ I_{1/2}^t$ is the (fractional)  Riesz operator in $t$ defined on the Fourier transform side through the multiplier $|\tau|^{-1/2}$ and
$h_\delta(x,t):=(D_{1/2}^t{\mathcal{S}}_\delta f)(x,t)$. Using this we see that
$$\psi_\delta(x_0,t_0)\leq c\sup_{\epsilon>0}|\tilde V_\epsilon h_\delta(x_0,t_0)|=:c\tilde V_\ast h(x_0,t_0),$$
 $\tilde V_\epsilon h(x,t)=V_\epsilon h(x,\cdot)$ evaluated at $t$, where $V_\epsilon$ is defined on functions $k\in L^2(\mathbb R,\mathbb C)$ by
$$V_\epsilon k(t)=\int_{\{|s-t|>\epsilon\}}\frac{\mbox{sgn}(t-s) I_{1/2}^tk(s)}{|s-t|^{3/2}}\, ds.$$
However, using this notation we can now apply Lemma 2.27 in \cite{HL} and conclude that
$$||\psi_\delta||_2\leq c||h_\delta||_2=c||D_{1/2}^t{\mathcal{S}}_\delta f||_2.$$
This completes the proof of $(vi)$.
\end{proof}

\begin{proof}[Proof of Lemma \ref{lemsl1++} $(vii)$- $(viii)$]
To start the proof of $(vii)$ and  $(viii)$ we note, using \eqref{eq11edsea+}, that $(vii)$ implies $(viii)$. Hence we only have to prove
$(vii)$. To start the proof, we let ${\bf f}=({\bf f}_{||},{\bf f}_{n+1})\in  L^2(\mathbb R^{n+1},\mathbb C^{n+1})$  and we again note that we only have to estimate $N_\ast((\mathcal{S}_\lambda\nabla_{||})\cdot{\bf f}_{||})$. Indeed, $N_\ast((\mathcal{S}_\lambda D_{n+1}){\bf f}_{n+1})=
N_\ast(\partial_\lambda(\mathcal{S}_\lambda {\bf f}_{n+1}))$ and using that
$$||N_\ast(\partial_\lambda(\mathcal{S}_\lambda {\bf f}_{n+1}))||_{{2,\infty}}\leq ||N_\ast(\partial_\lambda(\mathcal{S}_\lambda {\bf f}_{n+1}))||_{2}$$
we see that the estimate of $||N_\ast((\mathcal{S}_\lambda D_{n+1}){\bf f}_{n+1})||_{{2,\infty}}$ follows from $(i)$. To proceed we will estimate
$N_\ast((\mathcal{S}_\lambda\nabla_{||})\cdot{\bf g})$ where we have put ${\bf g}={\bf f}_{||}$.  Fix
$(x_0,t_0)\in\mathbb R^{n+1}$,  consider $(x,t,\lambda)\in \Gamma(x_0,t_0)$ and let $\sigma\in (-\lambda,\lambda)$. Given $(x,t,\lambda)$ we let
\begin{eqnarray}
E&=&\{(y,s): |y-x|+|s-t|^{1/2}<16\lambda\},\notag\\
E_k&=&\{(y,s): 2^k\lambda\leq |y-x|+|s-t|^{1/2}<2^{k+1}\lambda\},\ k=4,....,
\end{eqnarray}
and \begin{eqnarray}
\bar {\bf g}={\bf g}1_{E},\ {\bf g}_k={\bf g}1_{E_k},\ k=4,....
\end{eqnarray}
Using this notation we set $u(x,t,\lambda)=(\mathcal{S}_\lambda\nabla_{||})\cdot{\bf g}(x,t)$ and we split
$$\mbox{$u=\bar u+\tilde u$ where $\tilde u=\sum_{k=4}^\infty u_k$}$$
 and $$\bar u(x,t,\lambda)=(\mathcal{S}_\lambda\nabla_{||})\cdot\bar {\bf g}(x,t),\ u_k(x,t,\lambda)=(\mathcal{S}_\lambda\nabla_{||})\cdot{\bf g}_k(x,t).$$
We first estimate $u_k(x,t,\sigma)-u_k(x_0,t_0,0)$ for $(x,t,\sigma)$ as above  and for $k=4,....$. We write
\begin{eqnarray}
&&|u_k(x,t,\sigma)-u_k(x_0,t_0,0)|\notag\\
&\leq&\int_{E_k}|\nabla_{||,y} \bigl (\Gamma_\sigma(x,t,y,s)-\Gamma_0(x_0,t_0,y,s)\bigr )\cdot{\bf g}|\, dyds\notag\\
&\leq& \int_{E_k}|\nabla_{||,y} \bigl (\Gamma_\sigma(x,t,y,s)-\Gamma_\sigma(x_0,t_0,y,s)\bigr )\cdot{\bf g}|\, dyds\notag\\
&&+\int_{E_k}|\nabla_{||,y} \bigl (\Gamma_\sigma(x_0,t_0,y,s)-\Gamma_0(x_0,t_0,y,s)\bigr )\cdot{\bf g}|\, dyds.
\end{eqnarray}
We now note that
\begin{eqnarray}\label{cest}
\int_{E_k}|\nabla_{||,y} \bigl (\Gamma_\sigma(x,t,y,s)-\Gamma_\sigma(x_0,t_0,y,s)\bigr )|^2\, dyds\leq c2^{-k\alpha}(2^k\lambda)^{-n-2},
\end{eqnarray}
where $\alpha>0$ is as in Lemma \ref{le2+}. Indeed, \eqref{cest} follows from Lemma \ref{le1} $(i)$ and Lemma \ref{le2+}. Similarly, writing
\begin{eqnarray}
\Gamma_\sigma(x_0,t_0,y,s)-\Gamma_0(x_0,t_0,y,s)=\int_0^\sigma\partial_\tau \Gamma_\tau(x_0,t_0,y,s)\, d\tau
\end{eqnarray}
we see that we can use Lemma \ref{le3} to conclude that
\begin{eqnarray}\label{cest+}
\int_{E_k}|\nabla_{||,y} \bigl (\Gamma_\sigma(x_0,t_0,y,s)-\Gamma_0(x_0,t_0,y,s)\bigr )|^2\, dyds\leq c2^{-k\alpha}(2^k\lambda)^{-n-2}.
\end{eqnarray}
Using \eqref{cest} and \eqref{cest+} we first see that
\begin{eqnarray}
|u_k(x,t,\sigma)-u_k(x_0,t_0,0)|&\leq& c2^{-k\alpha/2}\bigl (\mean{E_k}|{\bf g}|^2\bigr )^{1/2}\notag\\
&\leq& c2^{-k\alpha/2}\bigl (M(|{\bf g}|^2)\bigr )^{1/2}(x_0,t_0),
\end{eqnarray}
where again $M$ is the standard parabolic the Hardy-Littlewood maximal, and then, by summing, that
\begin{eqnarray}\label{far}
|\tilde u(x,t,\sigma)-\tilde u(x_0,t_0,0)|\leq c\bigl (M(|{\bf g}|^2)\bigr )^{1/2}(x_0,t_0),
\end{eqnarray}
whenever $(x,t,\lambda)\in \Gamma(x_0,t_0)$ and  $\sigma\in (-\lambda,\lambda)$. Furthermore, using \eqref{eq14+}
\begin{eqnarray}\label{ele1}
|\bar u(x,t,\lambda)|&\leq& c\lambda^{-(n+2)/2}\sup_{\lambda>0}||(\mathcal{S}_\lambda\nabla_{||})\cdot{\bf g}||\notag\\
&\leq &c(\sup_{\lambda>0}||\mathcal{S}\nabla_{||}||_{2\to 2})\bigl (M(|{\bf g}|^2)\bigr )^{1/2}(x_0,t_0).
\end{eqnarray}
Put together we see that
\begin{eqnarray}\label{yra}
|u(x,t,\lambda)|&\leq& c(\sup_{\lambda>0}||\mathcal{S}\nabla_{||}||_{2\to 2})\bigl (M(|{\bf g}|^2)\bigr )^{1/2}(x_0,t_0)+c\bigl (M(|{\bf g}|^2)\bigr )^{1/2}(x_0,t_0)\notag\\
&&+|\tilde u(x_0,t_0,0)|
\end{eqnarray}
whenever $(x,t,\lambda)\in \Gamma(x_0,t_0)$ and $\sigma\in (-\lambda,\lambda)$. To estimate $\tilde u(x_0,t_0,0)$, consider $(x,t,\lambda)\in \Gamma(x_0,t_0)$. Then
\begin{eqnarray}\label{far+}
|\tilde u(x_0,t_0,0)|&\leq& |\tilde u(x,t,\delta)-\tilde u(x_0,t_0,0)|+|\bar u(x,t,\delta)|+|u(x,t,\delta)|\notag\\
&\leq &c\bigl (M(|{\bf g}|^2)\bigr )^{1/2}(x_0,t_0)+|\bar u(x,t,\delta)|+|u(x,t,\delta)|,
\end{eqnarray}
by \eqref{far} and whenever $0<\delta\ll \lambda$. Let $\Delta_\lambda(x_0,t_0)$ be the set of all points $(x,t)$ such that $|x-x_0|+|t-t_0|^{1/2}<\lambda$. Taking the average over $\Delta_\lambda(x_0,t_0)$ in
\eqref{far+} we see that
\begin{eqnarray}\label{far++}
|\tilde u(x_0,t_0,0)|&\leq& c\bigl (M(|{\bf g}|^2)\bigr )^{1/2}(x_0,t_0)\notag\\
&&+\mean{\Delta_\lambda(x_0,t_0)}|\bar u(x,t,\delta)|\, dxdt+M(u(\cdot,\cdot,\delta))(x_0,t_0)\notag\\
&\leq& c\bigl (M(|{\bf g}|^2)\bigr )^{1/2}(x_0,t_0)\notag\\
&&+c(\sup_{\lambda>0}||\mathcal{S}_\lambda\nabla_{||}||_{2\to 2})\bigl (M(|{\bf g}|^2)\bigr )^{1/2}(x_0,t_0)\notag\\
&&+M((\mathcal{S}_\lambda|_{\lambda=\delta}\nabla_{||})\cdot{\bf g})(x_0,t_0),
\end{eqnarray}
where we have also used \eqref{ele1}. In particular, using \eqref{yra} and \eqref{far++} we can conclude that
\begin{eqnarray}\label{far+++}
N_\ast((\mathcal{S}_\lambda\nabla_{||})\cdot{\bf g})(x_0,t_0)&\leq& c\bigl (1+\sup_{\lambda>0}||\mathcal{S}_\lambda\nabla_{||}||_{2\to 2}\bigr)\bigl (M(|{\bf g}|^2)\bigr )^{1/2}(x_0,t_0)\notag\\
&&+M((\mathcal{S}_\lambda|_{\lambda=\delta}\nabla_{||})\cdot{\bf g})(x_0,t_0).
\end{eqnarray}
This completes the proof of $(vii)$.\end{proof}

\begin{lemma}\label{lemsl1}  Assume $m\geq -1$, $l\geq -1$. Let $\Phi_+(f)$ be defined as in \eqref{keyestint-ex+}.  Then there exists a constant $c$, depending at most
     on $n$, $\Lambda$, the De Giorgi-Moser-Nash constants, $m$, $l$, such that
         \begin{eqnarray*}
              (i)&&|||\lambda^{m+2l+4}\nabla\partial_\lambda\partial_t^{l+1}\partial_\lambda^{m+1}\mathcal{S}_{\lambda}f|||_+\leq c\Phi_+(f)+c||f||_2,\notag\\
     (ii)&&|||\lambda^{m+2l+4}\partial_t\partial_t^{l+1}\partial_\lambda^{m+1}\mathcal{S}_{\lambda}f|||_+\leq c\Phi_+(f)+c||f||_2,
     \end{eqnarray*}
whenever $f\in L^2(\mathbb R^{n+1},\mathbb C)$. Furthermore,  assume $m\geq -1$, let $\Phi^\eta(f)$ be defined as in \eqref{keyestint-ex+se} and let $\eta\in (0,1/10)$. Then there exists a constant $c$, depending at most
     on $n$, $\Lambda$, the De Giorgi-Moser-Nash constants, $m$, such that
         \begin{eqnarray*}
     (iii)&&|||\lambda^{m+2}\nabla\partial_\lambda\partial_\lambda^{m+1}\mathcal{S}_{\lambda}^\eta f|||_+\leq c\Phi^\eta(f)+c||f|_2,\notag\\
     (iv)&&|||\lambda^{m+2}\partial_t\partial_\lambda^{m+1}\mathcal{S}_{\lambda}^\eta f|||_+\leq c\Phi^\eta(f)+c||f||_2,
     \end{eqnarray*}
whenever $f\in L^2(\mathbb R^{n+1},\mathbb C)$.
\end{lemma}
\begin{proof} $(i)$-$(ii)$ are proved in Lemma 4.3 in \cite{CNS}. To prove prove $(iii)$-$(iv)$ we have to be slightly more careful as we in this case only have      \begin{eqnarray}\label{id-}
\mathcal{H}\mathcal{S}_{\lambda}^\eta f(x,t)=f_\eta(x,t,\lambda)=f(x,t)\varphi_\eta(\lambda),
\end{eqnarray}
i.e., we have an inhomogeneous right hand side. Note that \begin{eqnarray}\label{id}
\mbox{$\mathcal{H}\mathcal{S}_{\lambda}^\eta f(x,t)=0$ whenever $\lambda>\eta$.}
\end{eqnarray}
To prove $(iii)$ we write
  \begin{eqnarray*}
              |||\lambda^{m+2}\nabla\partial_\lambda\partial_\lambda^{m+1}\mathcal{S}_{\lambda}^\eta f|||_+^2=I_1+I_2.
              \end{eqnarray*}
              where
     \begin{eqnarray*}
              I_1&=&\int_0^{2\eta}\int_{\mathbb R^{n+1}}|\lambda^{m+2}\nabla\partial_\lambda\partial_\lambda^{m+1}\mathcal{S}_{\lambda}^\eta f|^2\, \frac {dxdtd\lambda}\lambda,\notag\\
              I_2&=&\int_{2\eta}^\infty\int_{\mathbb R^{n+1}}|\lambda^{m+2}\nabla\partial_\lambda\partial_\lambda^{m+1}\mathcal{S}_{\lambda}^\eta f|^2\, \frac {dxdtd\lambda}\lambda.
              \end{eqnarray*}
              To estimate $I_2$  we first note, using \eqref{id}, Lemma \ref{le1--}, induction, and the definition of $\Phi^\eta(f)$, that it suffices to prove the estimate
\begin{eqnarray}\label{impesta}
I_2':=\int_{3\eta/2}^\infty\int_{\mathbb R^{n+1}}|\lambda\nabla_{||}\partial_\lambda\mathcal{S}_{\lambda}^\eta f|^2\, \frac {dxdtd\lambda}\lambda\leq c\Phi^\eta(f)^2+c||f||_2^2,
     \end{eqnarray}
whenever $f\in L^2(\mathbb R^{n+1},\mathbb C)$. To prove \eqref{impesta}  we first integrate by parts with respect to $\lambda$ to see that
\begin{eqnarray*}
I_2'=&=&\lim_{\epsilon\to 0}\int_{3\eta/2}^{1/\epsilon}\int_{\mathbb R^{n+1}}\nabla_{||}\partial_\lambda \mathcal{S}_{\lambda}^\eta f\cdot
\overline{\nabla_{||}\partial_\lambda \mathcal{S}_{\lambda}^\eta f}\,  \lambda{dxdtd\lambda}\notag\\
&=&-\frac 12\lim_{\epsilon\to 0}\int_{3\eta/2}^{1/\epsilon}\int_{\mathbb R^{n+1}}\nabla_{||}\partial_\lambda^2 \mathcal{S}_{\lambda}^\eta f\cdot
\overline{\nabla_{||}\partial_\lambda \mathcal{S}_{\lambda}^\eta f}\,  \lambda^2{dxdtd\lambda}\notag\\
&&-\frac 12\lim_{\epsilon\to 0}\int_{3\eta/2}^{1/\epsilon}\int_{\mathbb R^{n+1}}\nabla_{||}\partial_\lambda \mathcal{S}_{\lambda}^\eta f\cdot
\overline{\nabla_{||}\partial_\lambda^2 \mathcal{S}_{\lambda}^\eta f}\,  \lambda^2{dxdtd\lambda}\notag\\
&&+\lim_{\epsilon\to 0}\int_{\mathbb R^{n+1}}\nabla_{||}\partial_\lambda \mathcal{S}_{\lambda}^\eta f\cdot
\overline{\nabla_{||}\partial_\lambda \mathcal{S}_{\lambda}^\eta f}\,  \lambda^2{dxdt}\biggl |_{\lambda=1/\epsilon}\notag\\
&&-\lim_{\epsilon\to 0}\int_{\mathbb R^{n+1}}\nabla_{||}\partial_\lambda \mathcal{S}_{\lambda}^\eta f\cdot
\overline{\nabla_{||}\partial_\lambda \mathcal{S}_{\lambda}^\eta f}\,  \lambda^2{dxdt}\biggl |_{\lambda={3\eta/2}}.
\end{eqnarray*}
Hence using Lemma \ref{le5} $(ii)$ we see that
\begin{eqnarray}\label{impesta+}
I_2'\leq c\Phi^\eta(f)^2+c||f||_2^2+c\int_{3\eta/2}^\infty\int_{\mathbb R^{n+1}}|\lambda^2\nabla_{||}\partial_\lambda^2\mathcal{S}_{\lambda}^\eta f|^2\, \frac {dxdtd\lambda}\lambda.
     \end{eqnarray}
 \eqref{impesta}  now follows from an application of  Lemma \ref{le1--}.  To estimate $I_1$ we have to use \eqref{id-} and we see that
      \begin{eqnarray*}
              I_1&=&\int_0^{2\eta}\int_{\mathbb R^{n+1}}|\nabla\mathcal{H}^{-1}(\partial_\lambda\partial_\lambda^{m+1}f_\eta)|^2\, \lambda^{2m+3}{dxdtd\lambda}\notag\\
              &\leq &c\eta^{2m+3}\int_{\mathbb R^{n+2}}|\partial_\lambda^{m+1}f_\eta|^2\, {dxdtd\lambda}
            \end{eqnarray*}
               where the estimate on the second line in this display follows from Lemma \ref{gara} applied to the operator $\nabla\mathcal{H}^{-1}\div$. Hence,
                 \begin{eqnarray*}
              I_1&\leq& c\eta^{2m+3}||f||_2^2\biggl (\int_{-\infty}^\infty|\partial_\lambda^{m+1}\varphi_\eta(\lambda)|^2\, d\lambda\biggr )\leq c||f||_2^2.
            \end{eqnarray*}
            This proves $(iii)$.  To prove $(iv)$ we write
  \begin{eqnarray*}
              |||\lambda^{m+2}\partial_t\partial_\lambda^{m+1}\mathcal{S}_{\lambda}^\eta f|||_+^2=\tilde I_1+\tilde I_2.
              \end{eqnarray*}
              where
     \begin{eqnarray*}
              \tilde I_1&=&\int_0^{2\eta}\int_{\mathbb R^{n+1}}|\lambda^{m+2}\partial_t\partial_\lambda^{m+1}\mathcal{S}_{\lambda}^\eta f|^2\, \frac {dxdtd\lambda}\lambda,\notag\\
              \tilde I_2&=&\int_{2\eta}^\infty\int_{\mathbb R^{n+1}}|\lambda^{m+2}\partial_t\partial_\lambda^{m+1}\mathcal{S}_{\lambda}^\eta f|^2\, \frac {dxdtd\lambda}\lambda.
              \end{eqnarray*}
              Again using \eqref{id}, Lemma \ref{le1--}, Lemma \ref{le1a} and induction, we see that it suffices to prove that
              \begin{eqnarray}\label{impestaapa}
\tilde I_2':=\int_{3\eta/2}^\infty\int_{\mathbb R^{n+1}}|\lambda \partial_t\mathcal{S}_{\lambda}^\eta f|^2\, \frac {dxdtd\lambda}\lambda\leq c\Phi^\eta(f)^2+c||f||_2^2.
     \end{eqnarray}
     To prove \eqref{impestaapa}  we first integrate by parts with respect to $\lambda$,
\begin{eqnarray}\label{ua1}
 \tilde I_2'&=&\lim_{\epsilon\to 0}\int_{3\eta/2}^{1/\epsilon}\int_{\mathbb R^{n+1}}\partial_t \mathcal{S}_{\lambda}^\eta f
\overline{\partial_t \mathcal{S}_{\lambda}^\eta f}\,  \lambda{dxdtd\lambda}\notag\\
&=&-\frac 1 2\lim_{\epsilon\to 0}\int_{3\eta/2}^{1/\epsilon}\int_{\mathbb R^{n+1}}\partial_t\partial_\lambda \mathcal{S}_{\lambda}^\eta f
\overline{\partial_t \mathcal{S}_{\lambda}^\eta f}\,  \lambda^2{dxdtd\lambda}\notag\\
&&-\frac 1 2\lim_{\epsilon\to 0}\int_{3\eta/2}^{1/\epsilon}\int_{\mathbb R^{n+1}}\partial_t\mathcal{S}_{\lambda}^\eta f
\overline{\partial_t \partial_\lambda \mathcal{S}_{\lambda}^\eta f}\,  \lambda^2{dxdtd\lambda}\notag\\
&&+\int_{\mathbb R^{n+1}}\partial_t \mathcal{S}_{\lambda}^\eta f
\overline{\partial_t \mathcal{S}_{\lambda}^\eta f}\,  \lambda^2{dxdt}\biggl |_{\lambda=1/\epsilon}\notag\\
&&-\int_{\mathbb R^{n+1}}\partial_t \mathcal{S}_{\lambda}^\eta f
\overline{\partial_t \mathcal{S}_{\lambda}^\eta f}\,  \lambda^2{dxdt}\biggl |_{\lambda=3\eta/2}.
\end{eqnarray}
Hence
\begin{eqnarray}\label{ua1b}
 \tilde I_2'&\leq & c\int_{3\eta/2}^{1/\epsilon}\int_{\mathbb R^{n+1}}|\partial_t \partial_\lambda \mathcal{S}_{\lambda}^\eta f|^2\, \lambda^3 {dxdtd\lambda}\notag\\
 &&+c\sup_{\lambda\geq 3\eta/2}\int_{\mathbb R^{n+1}}|\partial_t \mathcal{S}_{\lambda}^\eta f|^2\lambda^2{dxdt}.
\end{eqnarray}
However, using  Lemma \ref{le5} $(ii)$, \eqref{id}, Lemma \ref{le1--} and basically \eqref{impesta}, we see that
\begin{eqnarray}\label{ua1bo}
 \tilde I_2'&\leq & \int_{3\eta/2}^{1/\epsilon}\int_{\mathbb R^{n+1}}|\partial_t \partial_\lambda \mathcal{S}_{\lambda}^\eta f|^2\, \lambda^3 {dxdtd\lambda}
 \leq c\Phi^\eta(f)^2+c||f||_2^2,
     \end{eqnarray}
whenever $f\in L^2(\mathbb R^{n+1},\mathbb C)$. To estimate $\tilde I_1$ we  use \eqref{id-} and we see that
      \begin{eqnarray*}
              \tilde I_1&=&\int_0^{2\eta}\int_{\mathbb R^{n+1}}|H_tD_{1/2}^t\mathcal{H}^{-1}(D_{1/2}^t\partial_\lambda\partial_\lambda^{m+1}f_\eta)|^2\, \lambda^{2m+3}{dxdtd\lambda}\notag\\
              &\leq &c\eta^{2m+3}\int_{\mathbb R^{n+2}}|\partial_\lambda^{m+1}f_\eta|^2\, {dxdtd\lambda}
            \end{eqnarray*}
               where the estimate on the second line in this display follows from Lemma \ref{gara} applied to the operator $D_{1/2}^t\mathcal{H}^{-1} D_{1/2}^t$. Hence,
                 \begin{eqnarray*}
              \tilde I_1&\leq& c\eta^{2m+3}||f||_2^2\biggl (\int_{-\infty}^\infty|\partial_\lambda^{m+1}\varphi_\eta(\lambda)|^2\, d\lambda\biggr )\leq c||f||_2^2.
            \end{eqnarray*}
            This proves $(iv)$ and the lemma.
     \end{proof}

\begin{lemma}\label{lemsl1c}  Let $\Phi_+(f)$ be defined as in \eqref{keyestint-ex+}, let $\Phi^\eta(f)$ be defined as in \eqref{keyestint-ex+se} and let $\eta\in (0,1/10)$.  Assume that $\Phi_+(f)<\infty$, $\Phi^\eta(f)<\infty$. Then there exists a constant $c$, depending at most
     on $n$, $\Lambda$, and the De Giorgi-Moser-Nash constants, such that
\begin{eqnarray}
(i)\ ||\mathbb D\mathcal{S}_{\lambda _0}f||_{2}&\leq& c\bigl(\Phi_+(f)+||f||_2+||N_{\ast\ast}(\partial_\lambda \mathcal{S}_{\lambda}f)||_2\bigr),\notag\\
(ii)\ ||\mathbb D\mathcal{S}_{\lambda _0}^\eta f||_{2}&\leq& c\bigl(\Phi^\eta (f)+||f||_2+||N_{\ast\ast}(\P_\lambda (\partial_\lambda \mathcal{S}_{\lambda+\lambda_0}^\eta f))||_2\bigr),
\end{eqnarray}
whenever   $f\in L^2(\mathbb R^{n+1},\mathbb C)$, $\lambda_0> 0$.
\end{lemma}
\begin{proof}  $(i)$ follows from Lemma 6.1, Lemma 6.2 and Lemma 6.3 in \cite{CNS}. The proof of $(ii)$ is a modification of the proof of $(i)$ and we here only include the proof of some of the core estimates. Indeed, we first note that it follows from the proof of Lemma 6.2 and Lemma 6.3 in \cite{CNS}, using Lemma \ref{lemsl1},  that
\begin{eqnarray}
||\mathbb D\mathcal{S}_{\lambda_0}^\eta f||_{2}\leq c(\Phi^\eta (f)+||f||_2+||\nabla_{||}\mathcal{S}_{\lambda_0}^\eta f||_{2}).
\end{eqnarray}
We will prove that
\begin{eqnarray}\label{core}
||\nabla_{||}\mathcal{S}_{\lambda_0}^\eta f||_{2}&\leq& c\bigl(\Phi^\eta (f)+||f||_2+||N_{\ast\ast}(\P_\lambda (\partial_\lambda \mathcal{S}_{\lambda+\lambda_0}^\eta f))||_2\bigr),
\end{eqnarray}
whenever $f\in L^2(\mathbb R^{n+1},\mathbb C)$,  $\lambda_0>0$.  We can without loss of generality assume that
$f\in C_0^\infty(\mathbb R^{n+1},\mathbb C)$ and to prove the lemma it suffices to estimate
 \begin{eqnarray}
I:=\int_{\mathbb R^{n+1}}{\bf g}\cdot\nabla_{||} \overline{\mathcal{S}^\eta_{\lambda_0}f}\, dxdt,
\end{eqnarray}
where ${\bf \bar g}:C_0^\infty(\mathbb R^{n+1},\mathbb C^n)$ and $||{\bf \bar g}||_2=1$. Given $f\in C_0^\infty(\mathbb R^{n+1},\mathbb C)$, we note, see Lemma \ref{smooth1} $(i)$-$(iii)$, that $\mathcal{S}^\eta_{\lambda_0}f\in {\mathbb H}(\mathbb R^{n+1},\mathbb C)\cap L^2(\mathbb R^{n+1},\mathbb C)$. Hence, using Lemma \ref{parahodge} we see that
 \begin{eqnarray}
I&=&\int_{\mathbb R^{n+1}}A_{||}^\ast\nabla_{||}v\cdot \nabla_{||} \overline{\mathcal{S}^\eta_{\lambda_0}f}\, dxdt+\int_{\mathbb R^{n+1}}D_{1/2}^t(v)\overline{H_tD_{1/2}^t(\mathcal{S}^\eta_{\lambda_0}f)}\, dxdt\notag\\
&=&\int_{\mathbb R^{n+1}}A_{||}\nabla_{||} \mathcal{S}^\eta_{\lambda_0}f\cdot\overline{\nabla_{||}v}\, dxdt+\int_{\mathbb R^{n+1}}H_tD_{1/2}^t(\mathcal{S}^\eta_{\lambda_0}f)\overline{D_{1/2}^t(v)}\, dxdt,
\end{eqnarray}
for a function $v\in\mathbb H=\mathbb H(\mathbb R^{n+1},\mathbb C)$ which satisfies
 \begin{eqnarray}\label{hodest}||v||_{\mathbb H}\leq c||{\bf g}||_2,
 \end{eqnarray}
for some constant $c$ depending only on $n$ and $\Lambda$. In the following we let
 \begin{eqnarray}
I_1&:=&\int_{\mathbb R^{n+1}}A_{||}\nabla_{||} \mathcal{S}^\eta_{\lambda_0}f\cdot\overline{\nabla_{||}v}\, dxdt,\notag\\
I_2&:=&\int_{\mathbb R^{n+1}}H_tD_{1/2}^t(\mathcal{S}^\eta_{\lambda_0}f)\overline{D_{1/2}^t(v)}\, dxdt.
\end{eqnarray}
Using that $C_0^\infty(\mathbb R^{n+1},\mathbb C)$ is dense in $\mathbb H(\mathbb R^{n+1},\mathbb C)$ we see that we in the following we can without loss of generality assume that $v\in C_0^\infty(\mathbb R^{n+1},\mathbb C)$.

We first estimate $I_1$. Recall the resolvents, $\mathcal{E}_\lambda=(I+\lambda^2\mathcal{H}_{||})^{-1}$ and $\mathcal{E}_\lambda^\ast=(I+\lambda^2\mathcal{H}_{||}^\ast)^{-1}$,  introduced in Section \ref{sec5}. To start the estimate of $I_1$ we first note, using that $f, v\in  C_0^\infty(\mathbb R^{n+1},\mathbb C)$, and by applying Lemma \ref{le8-}, that
\begin{eqnarray}\label{pint1-}
&&\biggl |\int_{\mathbb R^{n+1}}A_{||}\nabla_{||} \mathcal{E}_\lambda \mathcal{S}_{\lambda+\lambda_0}^\eta f\cdot\overline{\nabla_{||}\mathcal{E}_\lambda^\ast v}\, dxdt\biggr |\leq \frac c{\lambda^2}||\mathcal{S}_{\lambda+\lambda_0}^\eta f||_2||v||_2.
\end{eqnarray}
Hence, using that
\begin{eqnarray}\label{pint1a}
	\mathcal{S}_{\lambda+\lambda_0}^\eta f-\mathcal{S}_{\lambda_0}^\eta f
		=\int_{\lambda_0}^{\lambda + \lambda_0}\partial_\sigma\mathcal{S}_{\sigma}^\eta f\, d\sigma,
\end{eqnarray}
the fact that $\Phi^\eta(f)<\infty$, Lemma \ref{appf} and that $f, v\in  C_0^\infty(\mathbb R^{n+1},\mathbb C)$, we can use
\eqref{pint1-} to conclude that
\begin{eqnarray}\label{pint1}
&&\biggl |\int_{\mathbb R^{n+1}}A_{||}\nabla_{||} \mathcal{E}_\lambda \mathcal{S}_{\lambda+\lambda_0}^\eta f\cdot\overline{\nabla_{||}\mathcal{E}_\lambda^\ast v}\, dxdt\biggr | \longrightarrow 0 \quad \mbox{ as $\lambda\to \infty$}.
\end{eqnarray}
Hence,
\begin{eqnarray}\label{pint2}
I_1&=&-\int_{0}^\infty\int_{\mathbb R^{n+1}}\partial_\lambda \bigl (A_{||}\nabla_{||} \mathcal{E}_\lambda \mathcal{S}^\eta_{\lambda+\lambda_0}f\cdot\overline{\nabla_{||}\mathcal{E}_\lambda^\ast v}\bigr )\, dxdtd\lambda.
\end{eqnarray}
We here note, once and for all, that all (formal) integration by parts carried out below can be made rigorous by considerations similar to those in \eqref{pint1} and \eqref{pint2}. In the following we will in general omit the details of those manipulations. Using \eqref{pint2} we see that
\begin{eqnarray}
I_1&=&2\int_{0}^\infty\int_{\mathbb R^{n+1}}\bigl ((A_{||}\nabla_{||} (\mathcal{E}_\lambda)^2\lambda \mathcal{H}_{||} \mathcal{S}^\eta_{\lambda+\lambda_0}f)\cdot\overline{\nabla_{||}\mathcal{E}_\lambda^\ast v}\bigr )\, dxdtd\lambda\notag\\
&&+2\int_{0}^\infty\int_{\mathbb R^{n+1}} \bigl ((A_{||}\nabla_{||} \mathcal{E}_\lambda \mathcal{S}^\eta_{\lambda+\lambda_0}f)\cdot\overline{\nabla_{||}(\mathcal{E}_\lambda^\ast)^2\lambda \mathcal{H}_{||}^\ast v}\bigr )\, dxdtd\lambda\notag\\
&&-\int_{0}^\infty\int_{\mathbb R^{n+1}}\bigl ((A_{||}\nabla_{||} \mathcal{E}_\lambda\partial_\lambda \mathcal{S}^\eta_{\lambda+\lambda_0}f)\cdot\overline{\nabla_{||}\mathcal{E}_\lambda^\ast v}\bigr )\, dxdtd\lambda\notag\\
&=:&I_{11}+I_{12}-I_{13},
\end{eqnarray}
where we have used the identities
$$\partial_\lambda\mathcal{E}_\lambda=(\mathcal{E}_\lambda)^2\lambda \mathcal{H}_{||},\ \partial_\lambda\mathcal{E}_\lambda^\ast=(\mathcal{E}_\lambda^\ast)^2\lambda \mathcal{H}_{||}^\ast.$$
Integrating by parts in $I_{11}$, $I_{12}$, we see that
\begin{eqnarray}
I_{11}+I_{12}&=&-2\int_{0}^\infty\int_{\mathbb R^{n+1}}((\mathcal{E}_\lambda)^2 \mathcal{H}_{||} \mathcal{S}^\eta_{\lambda+\lambda_0}f)\overline{\mathcal{L}_{||}^\ast\mathcal{E}_\lambda^\ast v}\, \lambda dxdtd\lambda\notag\\
&&-2\int_{0}^\infty\int_{\mathbb R^{n+1}} (\mathcal{L}_{||} \mathcal{E}_\lambda \mathcal{S}^\eta_{\lambda+\lambda_0}f)\overline{(\mathcal{E}_\lambda^\ast)^2 \mathcal{H}_{||}^\ast v}\, \lambda dxdtd\lambda.
\end{eqnarray}
Using that $\mathcal{L}_{||}^\ast$ and $\mathcal{E}_\lambda^\ast$, and $\mathcal{L}_{||}$ and $\mathcal{E}_\lambda$, commute, we see that
\begin{eqnarray}\label{agaa}
I_{11}+I_{12}&=&-2\int_{0}^\infty\int_{\mathbb R^{n+1}}((\mathcal{E}_\lambda)^2 \mathcal{H}_{||} \mathcal{S}^\eta_{\lambda+\lambda_0}f)\overline{\mathcal{E}_\lambda^\ast \mathcal{L}_{||}^\ast v}\, \lambda dxdtd\lambda\notag\\
&&-2\int_{0}^\infty\int_{\mathbb R^{n+1}} (\mathcal{E}_\lambda \mathcal{L}_{||}  \mathcal{S}^\eta_{\lambda+\lambda_0}f)\overline{(\mathcal{E}_\lambda^\ast)^2 \mathcal{H}_{||}^\ast v}\, \lambda dxdtd\lambda.
\end{eqnarray}
Let
$$J:=\int_{0}^\infty\int_{\mathbb R^{n+1}} |\mathcal{E}_\lambda \mathcal{L}_{||}  \mathcal{S}^\eta_{\lambda+\lambda_0}f|^2\, \lambda dxdtd\lambda.$$
Then, using \eqref{agaa}, the $L^2$-boundedness of $\mathcal{E}_\lambda$ and $\mathcal{E}_\lambda^\ast$, Lemma \ref{le8-}, and the square function estimates for $\mathcal{E}_\lambda^\ast \mathcal{L}_{||}^\ast$ and $(\mathcal{E}_\lambda^\ast) \mathcal{H}_{||}^\ast$, Theorem \ref{thm1}, we see that
\begin{eqnarray}
|I_{11}+I_{12}|&\leq& c|||\lambda\partial_t\mathcal{S}^\eta_{\lambda+\lambda_0}f|||||v||_{\mathbb H}+J^{1/2}||v||_{\mathbb H}\notag\\
&\leq&c(\Phi^\eta(f)+||f||_2+J^{1/2})||v||_{\mathbb H},
\end{eqnarray}
by Lemma \ref{lemsl1}. To estimate $J$ we note that formally
\begin{eqnarray*}
\mathcal{L}_{||}\mathcal{S}^\eta_{\lambda+\lambda_0}f&=&\sum_{j=1}^{n+1}A_{n+1,j}D_{n+1}D_j\mathcal{S}^\eta_{\lambda+\lambda_0}f\notag\\
&&+\sum_{i=1}^{n}D_iA_{i,n+1}D_{n+1}\mathcal{S}^\eta_{\lambda+\lambda_0}f+\partial_t\mathcal{S}^\eta_{\lambda+\lambda_0}f+f_\eta.
\end{eqnarray*}
Using this, and the $L^2$-boundedness of $\mathcal{E}_\lambda$, Lemma \ref{le8-}, we see that
\begin{eqnarray}
J\leq c(|||\lambda\nabla\partial_\lambda \mathcal{S}^\eta_{\lambda+\lambda_0}f|||^2+|||\lambda\partial_t \mathcal{S}^\eta_{\lambda+\lambda_0}f|||^2+\tilde J+||f||_2^2),
\end{eqnarray}
where
\begin{eqnarray}
\tilde J=\int_{0}^\infty\int_{\mathbb R^{n+1}} |\mathcal{E}_\lambda\sum_{i=1}^{n}D_iA_{i,n+1}\partial_\lambda \mathcal{S}^\eta_{\lambda+\lambda_0}f|^2\, \lambda dxdtd\lambda.
\end{eqnarray}
In particular, again using Lemma \ref{lemsl1} we see that
\begin{eqnarray}
J\leq c(\Phi^\eta(f)^2+||f||_2^2+\tilde J).
\end{eqnarray}
To estimate $\tilde J$, let $A_{n+1}^{||}:=(A_{1,n+1},...,A_{n,n+1})$. Then
\begin{eqnarray*}
\tilde J&=&\int_{0}^\infty\int_{\mathbb R^{n+1}} |\mathcal{E}_\lambda (\div_{||}(A_{n+1}^{||})\partial_\lambda \mathcal{S}^\eta_{\lambda+\lambda_0}f)|^2\, \lambda{dxdtd\lambda}\notag\\
&\leq& c(\tilde J_1+\tilde J_2),
\end{eqnarray*}
where
\begin{eqnarray*}
\tilde J_1&=&\int_{0}^\infty\int_{\mathbb R^{n+1}} |(\lambda\mathcal{E}_\lambda\div_{||}) A_{n+1}^{||}\partial_\lambda \mathcal{S}^\eta_{\lambda+\lambda_0}f|^2\, \frac{dxdtd\lambda}{\lambda},\notag\\
\tilde J_2&=&\int_{0}^\infty\int_{\mathbb R^{n+1}} |\mathcal{E}_\lambda  (A_{n+1}^{||}\cdot\nabla_{||}(\partial_\lambda \mathcal{S}^\eta_{\lambda+\lambda_0}f))\bigr|^2\, \lambda{dxdtd\lambda}.
\end{eqnarray*}
Obviously, and by familiar arguments
\begin{eqnarray*}
\tilde J_2\leq c|||\lambda\nabla\partial_\lambda \mathcal{S}^\eta_{\lambda+\lambda_0}f|||^2\leq c(\Phi^\eta(f)^2+||f||_2^2),
\end{eqnarray*}
and we are left with $\tilde J_1$. We write
\begin{eqnarray*}
(\lambda\mathcal{E}_\lambda\div_{||}) A_{n+1}^{||}=\mathcal{R}_\lambda+((\lambda\mathcal{E}_\lambda\div_{||}) A_{n+1}^{||})\P_\lambda.
\end{eqnarray*}
where
\begin{eqnarray*}
\mathcal{R}_\lambda=(\lambda\mathcal{E}_\lambda\div_{||}) A_{n+1}^{||}-((\lambda\mathcal{E}_\lambda\div_{||}) A_{n+1}^{||})\P_\lambda.
\end{eqnarray*}
Then
\begin{eqnarray}
\tilde J_1\leq \tilde J_{11}+\tilde J_{12},
\end{eqnarray}
where
\begin{eqnarray}
\tilde J_{11}&=&\int_{0}^\infty\int_{\mathbb R^{n+1}} |R_\lambda\partial_\lambda \mathcal{S}^\eta_{\lambda+\lambda_0}f|^2\, \frac{dxdtd\lambda}{\lambda},\notag\\
\tilde J_{12}&=&\int_{0}^\infty\int_{\mathbb R^{n+1}} |((\lambda\mathcal{E}_\lambda\div_{||}) A_{n+1}^{||})\P_\lambda\partial_\lambda \mathcal{S}^\eta_{\lambda+\lambda_0}f|^2\, \frac{dxdtd\lambda}{\lambda}.
\end{eqnarray}
 Using Lemma \ref{le8-}, Lemma \ref{le8-+} and Lemma \ref{le11-} we see that
\begin{eqnarray}
\tilde J_{11}&\leq& c\int_{0}^\infty\int_{\mathbb R^{n+1}} |\nabla\partial_\lambda \mathcal{S}^\eta_{\lambda+\lambda_0}f|^2\, \lambda{dxdtd\lambda}\notag\\
&&+c\int_{0}^\infty\int_{\mathbb R^{n+1}} |\partial_t\partial_\lambda \mathcal{S}^\eta_{\lambda+\lambda_0}f|^2\, \lambda^3{dxdtd\lambda}\notag\\
&\leq&c(\Phi^\eta(f)^2+||f||_2^2),
\end{eqnarray}
by Lemma \ref{lemsl1}. Furthermore, using Lemma 3.1 in \cite{N} we see that there exists a constant $c$, depending only on $n$, $\Lambda$, such that
\begin{eqnarray*}\label{crucacar+}\int_0^{l(Q)}\int_Q|(\lambda\mathcal{E}_\lambda\div_{||}) A_{n+1}^{||}|^2\frac {dxdtd\lambda}\lambda\leq c|Q|
\end{eqnarray*}
for all cubes $Q\subset\mathbb R^{n+1}$. In particular, $|(\lambda\mathcal{E}_\lambda\div_{||}) A_{n+1}^{||}|^2\lambda^{-1}dxdtd\lambda$ defines a Carleson measure on $\mathbb R^{n+2}_+$. Using this we see that
\begin{eqnarray}
\tilde J_{12}\leq c||N_{\ast\ast}(\P_\lambda(\partial_\lambda \mathcal{S}^\eta_{\lambda+\lambda_0}f))||^2.
\end{eqnarray}
Putting all the estimates together we can conclude that
\begin{eqnarray}
|I_{11}+I_{12}|\leq (\Phi^\eta(f)+||f||_2+||N_{\ast\ast}(\P_\lambda(\partial_\lambda \mathcal{S}^\eta_{\lambda+\lambda_0}f))||)||v||_{\mathbb H},
\end{eqnarray}
which completes the estimate of $|I_{11}+I_{12}|$. We next estimate $I_{13}$. Integrating by parts with respect to $\lambda$ we see that
 \begin{eqnarray}
I_{13}&=&\int_{0}^\infty\int_{\mathbb R^{n+1}}\bigl (A_{||}\nabla_{||} \mathcal{E}_\lambda\partial_\lambda \mathcal{S}^\eta_{\lambda+\lambda_0}f\cdot\overline{\nabla_{||}\mathcal{E}_\lambda^\ast v}\bigr )\, dxdtd\lambda\notag\\
&=&-\int_{0}^\infty\int_{\mathbb R^{n+1}}\partial_\lambda\bigl (A_{||}\nabla_{||} \mathcal{E}_\lambda\partial_\lambda \mathcal{S}^\eta_{\lambda+\lambda_0}f\cdot\overline{\nabla_{||}\mathcal{E}_\lambda^\ast v}\bigr )\, \lambda dxdtd\lambda\notag\\
&=&2\int_{0}^\infty\int_{\mathbb R^{n+1}}\bigl (A_{||}\nabla_{||} (\mathcal{E}_\lambda)^2\lambda \mathcal{H}_{||} \partial_\lambda\mathcal{S}^\eta_{\lambda+\lambda_0}f\cdot\overline{\nabla_{||}\mathcal{E}_\lambda^\ast v}\bigr )\, \lambda dxdtd\lambda\notag\\
&&+2\int_{0}^\infty\int_{\mathbb R^{n+1}} \bigl (A_{||}\nabla_{||} \mathcal{E}_\lambda \partial_\lambda\mathcal{S}^\eta_{\lambda+\lambda_0}f\cdot\overline{\nabla_{||}(\mathcal{E}_\lambda^\ast)^2\lambda \mathcal{H}_{||}^\ast v}\bigr )\, \lambda dxdtd\lambda\notag\\
&&-\int_{0}^\infty\int_{\mathbb R^{n+1}}\bigl (A_{||}\nabla_{||} \mathcal{E}_\lambda\partial_\lambda^2 \mathcal{S}^\eta_{\lambda+\lambda_0}f\cdot\overline{\nabla_{||}\mathcal{E}_\lambda^\ast v}\bigr )\, \lambda dxdtd\lambda\notag\\
&=&I_{131}+I_{132}-I_{133}.
\end{eqnarray}
As above we see that \begin{eqnarray}\label{agaallu}
I_{131}+I_{132}&=&2\int_{0}^\infty\int_{\mathbb R^{n+1}}((\mathcal{E}_\lambda)^2 \mathcal{H}_{||} \partial_\lambda\mathcal{S}^\eta_{\lambda+\lambda_0}f)\overline{\mathcal{E}_\lambda^\ast \mathcal{L}_{||}^\ast v}\, \lambda^2 dxdtd\lambda\notag\\
&&+2\int_{0}^\infty\int_{\mathbb R^{n+1}} (\mathcal{E}_\lambda \mathcal{L}_{||}  \partial_\lambda\mathcal{S}^\eta_{\lambda+\lambda_0}f)\overline{(\mathcal{E}_\lambda^\ast)^2 \mathcal{H}_{||}^\ast v}\, \lambda^2 dxdtd\lambda.
\end{eqnarray}
Then, by the above argument,
\begin{eqnarray}
|I_{131}+I_{132}|^2&\leq &c\int_{0}^\infty\int_{\mathbb R^{n+1}} |\partial_t \partial_\lambda\mathcal{S}^\eta_{\lambda+\lambda_0}f|^2\, \lambda^3 dxdtd\lambda\notag\\
&&+c\int_{0}^\infty\int_{\mathbb R^{n+1}} |\mathcal{E}_\lambda \mathcal{L}_{||}  \partial_\lambda\mathcal{S}^\eta_{\lambda+\lambda_0}f|^2\, \lambda^3 dxdtd\lambda.
\end{eqnarray}
Recall that $\mathcal{E}_\lambda \mathcal{L}_{||}=\mathcal{E}_\lambda\div_{||}(A_{||}\nabla_{||}\cdot)$. Hence, using the $L^2$-boundedness of $\lambda \mathcal{E}_\lambda\div_{||}$  we see that
\begin{eqnarray}
|I_{131}+I_{132}|^2&\leq &c\int_{0}^\infty\int_{\mathbb R^{n+1}} |\partial_t \partial_\lambda\mathcal{S}^\eta_{\lambda+\lambda_0}f|^2\, \lambda^3 dxdtd\lambda\notag\\
&&+c\int_{0}^\infty\int_{\mathbb R^{n+1}} |\nabla_{||} \partial_\lambda\mathcal{S}^\eta_{\lambda+\lambda_0}f|^2\, \lambda dxdtd\lambda.
\end{eqnarray}
Furthermore,
 \begin{eqnarray}
-I_{133}&=&\int_{0}^\infty\int_{\mathbb R^{n+1}}\bigl (A_{||}\nabla_{||} \mathcal{E}_\lambda\partial_\lambda^2 \mathcal{S}^\eta_{\lambda+\lambda_0}f\cdot\overline{\nabla_{||}\mathcal{E}_\lambda^\ast v}\bigr )\, \lambda dxdtd\lambda\notag\\
&=&-\int_{0}^\infty\int_{\mathbb R^{n+1}}\mathcal{E}_\lambda\partial_\lambda^2 \mathcal{S}^\eta_{\lambda+\lambda_0}f\overline{\mathcal{E}_\lambda^\ast \mathcal{L}_{||}^\ast v}\, \lambda dxdtd\lambda,\end{eqnarray}
by previous arguments. Using the uniform $L^2$-boundedness of $\mathcal{E}_\lambda$, Lemma \ref{le8-} and the square function estimate for $\mathcal{E}_\lambda^\ast \mathcal{L}_{||}^\ast$, Theorem \ref{thm1}, we can conclude that
 \begin{eqnarray}
|I_{133}| &\leq &c\biggl (\int_{0}^\infty\int_{\mathbb R^{n+1}} |\partial_\lambda^2 \mathcal{S}^\eta_{\lambda+\lambda_0}f|^2\, \lambda dxdtd\lambda\biggr )^{1/2}||v||_{\mathbb H}.\end{eqnarray}
Hence, again using Lemma \ref{lemsl1} we have
\begin{eqnarray}
|I_{13}|\leq c\bigl(\Phi^\eta(f)+||f||_2\bigr)||v||_{\mathbb H},
\end{eqnarray}
 This completes the proof of $I_1$.

We next estimate $I_2$. To estimate $I_2$ we note that
\begin{eqnarray}
I_2&=&-\int_{0}^\infty\int_{\mathbb R^{n+1}}\partial_\lambda \bigl (H_tD^t_{1/2}\mathcal{E}_\lambda \mathcal{S}^\eta_{\lambda+\lambda_0}f\cdot\overline{D_{1/2}^t\mathcal{E}_\lambda^\ast v}\bigr )\, dxdtd\lambda\notag\\
&=&2\int_{0}^\infty\int_{\mathbb R^{n+1}}\bigl ((H_tD^t_{1/2} (\mathcal{E}_\lambda)^2\lambda \mathcal{H}_{||} \mathcal{S}^\eta_{\lambda+\lambda_0}f)\cdot\overline{D_{1/2}^t\mathcal{E}_\lambda^\ast v}\bigr )\, dxdtd\lambda\notag\\
&&+2\int_{0}^\infty\int_{\mathbb R^{n+1}} \bigl ((H_tD^t_{1/2} \mathcal{E}_\lambda \mathcal{S}^\eta_{\lambda+\lambda_0}f)\cdot\overline{D_{1/2}^t(\mathcal{E}_\lambda^\ast)^2\lambda \mathcal{H}_{||}^\ast v}\bigr )\, dxdtd\lambda\notag\\
&&-\int_{0}^\infty\int_{\mathbb R^{n+1}}\bigl ((H_tD^t_{1/2}\mathcal{E}_\lambda\partial_\lambda \mathcal{S}^\eta_{\lambda+\lambda_0}f)\cdot\overline{D_{1/2}^t\mathcal{E}_\lambda^\ast v}\bigr )\, dxdtd\lambda\notag\\
&=&I_{21}+I_{22}-I_{23}.
\end{eqnarray}
Using the $L^2$-boundedness of $\mathcal{E}_\lambda$ and $\mathcal{E}_\lambda^\ast$, Lemma \ref{le8-}, and the square function estimates for $(\mathcal{E}_\lambda^\ast) \mathcal{H}_{||}^\ast$, Theorem \ref{thm1}, we immediately see that
\begin{eqnarray}\label{kau}
|I_{22}|&\leq& c|||\lambda\partial_t\mathcal{S}^\eta_{\lambda+\lambda_0}f|||||v||_{\mathbb H}\leq c(\Phi^\eta(f)+||f||_2)||v||_{\mathbb H},
\end{eqnarray}
by Lemma \ref{lemsl1}. As $\mathcal{H}_{||}$ commutes with $\mathcal{E}_\lambda$, $D_{1/2}^t$, and $H_tD^t_{1/2}$,  and as
$\mathcal{H}_{||}^\ast$ commutes with $\mathcal{E}_\lambda^\ast$, $D_{1/2}^t$, and $H_tD^t_{1/2}$, we can integrate by parts in $I_{21}$, moving
$\mathcal{H}_{||}$ from the left to the right, and use the same argument as in the estimate of $|I_{22}|$ to conclude that \eqref{kau} holds with
$I_{22}$ replaced by $I_{21}$. Integrating by parts with respect to $\lambda$ in $I_{23}$, and repeating the arguments used in the estimates  of
$|I_{21}|$ and $|I_{22}|$, it is easily seen, using Lemma \ref{lemsl1}, that
\begin{eqnarray}\label{kau+}
|I_{23}|&\leq& c(\Phi^\eta(f)+||f||_2)||v||_{\mathbb H}+|\tilde I_{23}|,
\end{eqnarray}
where
\begin{eqnarray}\label{kau++}
\tilde I_{23}=\int_{0}^\infty\int_{\mathbb R^{n+1}}\bigl ((H_tD^t_{1/2}\mathcal{E}_\lambda\partial_\lambda^2 \mathcal{S}^\eta_{\lambda+\lambda_0}f)\cdot\overline{D_{1/2}^t\mathcal{E}_\lambda^\ast v}\bigr )\, \lambda dxdtd\lambda.
\end{eqnarray}
However,
\begin{eqnarray}\label{kau++}
|\tilde I_{23}|\leq |||\lambda \partial_\lambda^2 \mathcal{S}^\eta_{\lambda+\lambda_0}f||||||\lambda\partial_t\mathcal{E}_\lambda^\ast f|||\leq c\Phi^\eta(f)||v||_{\mathbb H},
\end{eqnarray}
by Theorem \ref{thm1}. This completes the proof of \eqref{core} and the lemma.
\end{proof}

\begin{theorem}\label{th0} Assume that $\mathcal{H}$,  $\mathcal{H}^\ast$ satisfy \eqref{eq3}-\eqref{eq4} as well as \eqref{eq14+}-\eqref{eq14++}. Let $\Phi_+(f)$ be defined as in \eqref{keyestint-ex+}. Then there exists a constant $c$, depending at most
     on $n$, $\Lambda$, and the De Giorgi-Moser-Nash constants such that
     \begin{eqnarray}\label{keyestint+a}
(i)&& ||N_\ast(\partial_\lambda \mathcal{S}_\lambda f)||_2\leq c\Phi_+(f)+c||f||_2,\notag\\
  (ii)&&\sup_{\lambda>0}||\mathbb D\mathcal{S}_{\lambda}f||_{2}\leq c\Phi_+(f)+c||f||_2,\notag\\
(iii)&&||\tilde N_\ast(\nabla_{||}\mathcal{S}_\lambda f)||_2\leq c\Phi_+(f)+c||f||_2,\notag\\
(iv)&&||\tilde N_\ast(H_tD^t_{1/2}\mathcal{S}_\lambda f)||_2\leq  c\Phi_+(f)+c||f||_2,
\end{eqnarray}
whenever  $f\in L^2(\mathbb R^{n+1},\mathbb C)$.
               \end{theorem}
               \begin{proof} \eqref{keyestint+a} $(i)$-$(iv)$ is Theorem 2.18 in \cite{CNS}. Indeed, \eqref{keyestint+a} $(i)$ is an immediate consequence of Lemma \ref{lemsl1++} $(i)$. Using Lemma
     \ref{lemsl1c} and Lemma \ref{lemsl1}, we see that \eqref{keyestint+a} $(i)$ imply \eqref{keyestint+a} $(ii)$. \eqref{keyestint+a} $(iii)$, $(iv)$, now follows immediately from these estimates and Lemma \ref{lemsl1++}.  \end{proof}

   \section{Traces, boundary layer potentials and weak limits}\label{sec4}

    In this section  we are concerned with boundary traces theorems for weak solutions, weak solutions for which the appropriate non-tangential maximal functions are controlled,  and the existence of boundary layer potentials.

\subsection{Boundary traces of weak solutions}\label{sec6}
\noindent

\begin{lemma}\label{trace1} Assume that $\mathcal{H}$,  $\mathcal{H}^\ast$ satisfy \eqref{eq3}-\eqref{eq4} as well as \eqref{eq14+}-\eqref{eq14++}. Assume that  $\mathcal{H}u=0$ in $\mathbb R^{n+2}_+$ and that
       \begin{eqnarray*}\label{assump}
     \tilde N_\ast(\nabla u),\ \tilde N_\ast(H_tD^t_{1/2} u)\in L^2(\mathbb R^{n+1}).
     \end{eqnarray*}
      Then there exists a constant $c$, depending at most
     on $n$, $\Lambda$, and the De Giorgi-Moser-Nash constants, such that
     \begin{eqnarray*}
     \sup_{\lambda>0}||\nabla u(\cdot,\cdot,\lambda)||_2&\leq& c||\tilde N_\ast(\nabla u)||_2,\notag\\
     \sup_{\lambda>0}||H_tD^t_{1/2} u(\cdot,\cdot,\lambda)||_2
     &\leq& c\biggl (||\tilde N_\ast(\nabla u)||_2+||\tilde N_\ast(H_tD^t_{1/2}u)||_2\biggr ).
     \end{eqnarray*}
     \end{lemma}
     \begin{proof} Using the  $\lambda$-independence of $A$, and  \eqref{eq14+}, we see that to prove the lemma it suffices to estimate
     $||\nabla_{||} u(\cdot,\cdot,\lambda)||_2$ and $||H_tD^t_{1/2} u(\cdot,\cdot,\lambda)||_2$. To start the estimate of $||\nabla_{||} u(\cdot,\cdot,\lambda)||_2$, let ${\bf \psi}\in L^2(\mathbb R^{n+1},\mathbb C^n)$ with $||{\bf \psi}||_2=1$. Considering $\lambda$ as fixed
     we see that it is enough to establish the bound
     \begin{eqnarray}
     \biggl |\int_{\mathbb R^{n+1}} u(x,t,\lambda)\div_{||}\bar{\bf \psi}\, dxdt\biggr |\leq c||\tilde N_\ast(\nabla u)||_2.
     \end{eqnarray}
     We write
           \begin{eqnarray}
    \int_{\mathbb R^{n+1}} u(x,t,\lambda)\div_{||}\bar{\bf \psi}\, dxdt=I+II,
     \end{eqnarray}
     where
         \begin{eqnarray}
    I&=& \int_{\mathbb R^{n+1}} \biggl (u(x,t,\lambda)-\frac 2\lambda
    \int_{3\lambda/4}^{5\lambda/4}  u(x,t,\sigma)\, d\sigma\biggr )\div_{||}\bar{\bf \psi}\, dxdt\notag\\
    II&=&\frac 2\lambda\int_{3\lambda/4}^{5\lambda/4} \int_{\mathbb R^{n+1}} u(x,t,\sigma)\div_{||}\bar{\bf \psi}\, dxdtd\sigma.
     \end{eqnarray}
     Using Cauchy-Schwarz and Fubini's theorem we see that
            \begin{eqnarray}
|II|\leq  c||\tilde N_\ast(\nabla u)||_2.
     \end{eqnarray}
     To estimate $I$ we write
              \begin{eqnarray}
    I&=& \frac 2\lambda
    \int_{3\lambda/4}^{5\lambda/4}\int_{\mathbb R^{n+1}} (u(x,t,\lambda)- u(x,t,\sigma))\div_{||}\bar{\bf \psi}\, dxdtd\sigma\notag\\
   &=&\frac 2\lambda
    \int_{3\lambda/4}^{5\lambda/4}\int_{\mathbb R^{n+1}} \biggl (\int_\sigma^\lambda \partial_{\tilde\sigma}u(x,t,\tilde\sigma)\, d\tilde\sigma\biggr )\div_{||}\bar{\bf \psi}\, dxdtd\sigma\notag\\
    &=&\frac 2\lambda
    \int_{3\lambda/4}^{5\lambda/4}\int_\sigma^\lambda\int_{\mathbb R^{n+1}} \nabla_{||}\partial_{\tilde\sigma}u(x,t,\tilde\sigma)\cdot \bar{\bf \psi}\, dxdtd\tilde\sigma d\sigma.
     \end{eqnarray}
     Hence,
                  \begin{eqnarray}
    |I|&\leq&c\biggl (\int_{3\lambda/4}^{5\lambda/4} \int_{\mathbb R^{n+1}} \lambda |\nabla_{||}\partial_\sigma u(x,t,\sigma)|^2\, dxdtd\sigma\biggr )^{1/2}\notag\\
    &\leq&c\biggl (\frac 1\lambda\int_{\lambda/2}^{3\lambda/2} \int_{\mathbb R^{n+1}} |\partial_\sigma u(x,t,\sigma)|^2\, dxdtd\sigma\biggr )^{1/2}\notag\\
    &\leq& c||\tilde N_\ast(\nabla u)||_2,
     \end{eqnarray}
     by elementary manipulations and Lemma \ref{le1--}.
     To bound $||H_tD^t_{1/2} u(\cdot,\cdot,\lambda)||_2$ we let ${\psi}\in L^2(\mathbb R^{n+1},\mathbb C)$ with $||{\psi}||_2=1$ and  write
           \begin{eqnarray}
    \int_{\mathbb R^{n+1}} u(x,t,\lambda)H_tD^t_{1/2}\bar{\psi}\, dxdt=\tilde I+\widetilde{II},
     \end{eqnarray}
     where
         \begin{eqnarray}
    \tilde I&=& \int_{\mathbb R^{n+1}} \biggl (u(x,t,\lambda)-\frac 2\lambda
    \int_{3\lambda/4}^{5\lambda/4} u(x,t,\sigma)\, d\sigma\biggr )H_tD^t_{1/2}\bar{\psi}\, dxdt,\notag\\
    \widetilde{II}&=&\frac 2\lambda\int_{3\lambda/4}^{5\lambda/4}\int_{\mathbb R^{n+1}} u(x,t,\sigma)H_tD^t_{1/2}\bar{\psi}\, dxdtd\sigma.
     \end{eqnarray}
     Arguing as above we see that $|\widetilde{II}|\leq c||\tilde N_\ast(H_tD^t_{1/2}u)||_2$ and that
              \begin{eqnarray}
    |\tilde I|^2&\leq& c\lambda\int_{3\lambda/4}^{5\lambda/4} \int_{\mathbb R^{n+1}}  |H_tD^t_{1/2}\partial_\sigma u(x,t,\sigma)|^2\, dxdtd\sigma\notag\\
    &\leq& c\tilde I_1^{1/2}\tilde I_2^{1/2},
     \end{eqnarray}
     where
             \begin{eqnarray}
             \tilde I_1&=&\frac 1\lambda\int_{3\lambda/4}^{5\lambda/4} \int_{\mathbb R^{n+1}} |\partial_\sigma u(x,t,\sigma)|^2\, dxdtd\sigma,\notag\\
    \tilde I_2&=&\lambda^3\int_{3\lambda/4}^{5\lambda/4} \int_{\mathbb R^{n+1}} |\partial_t\partial_\sigma u(x,t,\sigma)|^2\, dxdtd\sigma.
     \end{eqnarray}
     Again, $\tilde I_1\leq c||\tilde N_\ast(\nabla u)||_2^2$. Furthermore, first using Lemma \ref{le1a} and then Lemma \ref{le1--}, we see that
              \begin{eqnarray}
             \tilde I_2&\leq&c\lambda\int_{5\lambda/8}^{11\lambda/8} \int_{\mathbb R^{n+1}} |\nabla\partial_\sigma u(x,t,\sigma)|^2\, dxdtd\sigma\notag\\
             &\leq &\frac c\lambda\int_{\lambda/2}^{3\lambda/2} \int_{\mathbb R^{n+1}} |\partial_\sigma u(x,t,\sigma)|^2\, dxdtd\sigma\notag\\
             &\leq&c||\tilde N_\ast(\nabla u)||_2^2.
     \end{eqnarray}
     This completes the proof of the lemma. \end{proof}

     \begin{lemma}\label{trace2} Assume that $\mathcal{H}$,  $\mathcal{H}^\ast$ satisfy \eqref{eq3}-\eqref{eq4} as well as \eqref{eq14+}-\eqref{eq14++}. Assume that  $\mathcal{H}u=0$ in $\mathbb R^{n+2}_+$ and that
       \begin{eqnarray*}\label{assump}
     \tilde N_\ast(\nabla u)\in L^2(\mathbb R^{n+1})\mbox{ and  }\sup_{\lambda>0}||H_tD^t_{1/2} u(\cdot,\cdot,\lambda)||_2<\infty.
     \end{eqnarray*}
     Then there exists a constant $c$, depending at most
     on $n$, $\Lambda$, and the De Giorgi-Moser-Nash constants, and $f\in{\mathbb H}={\mathbb H}(\mathbb R^{n+1},\mathbb C)$ such that
     \begin{eqnarray*}
     (i)&& \mbox{$u\to f$ n.t.},\notag\\
     (ii)&&\mbox{$|u(x,t,\lambda)-f(x_0,t_0)|\leq c\lambda\tilde N_\ast(\nabla u)(x_0,t_0)$}\mbox{ when $(x,t,\lambda)\in\Gamma(x_0,t_0)$,}\notag\\
     (iii)&&||f||_{\mathbb H}\leq c\bigl(||\tilde N_\ast(\nabla u)||_2+\sup_{\lambda>0}||H_tD^t_{1/2} u(\cdot,\cdot,\lambda)||_2\bigr ).
     \end{eqnarray*}
     Furthermore,
        \begin{eqnarray}
     (iv)&&\mbox{$\nabla_{||}u(\cdot,\cdot,\lambda)\to \nabla_{||}f(\cdot,\cdot)$,}\notag\\
     (v)&&\mbox{$H_tD^t_{1/2}u(\cdot,\cdot,\lambda)\to H_tD^t_{1/2}f(\cdot,\cdot)$},
     \end{eqnarray}
     weakly in $L^2(\mathbb R^{n+1},\mathbb C)$ as $\lambda\to 0$.
     \end{lemma}
     \begin{proof} Let $(x_0,t_0)\in\mathbb R^{n+1}$ be such that $\tilde N_\ast(\nabla u)(x_0,t_0)<\infty$ and let $\epsilon>0$. Consider $(x,t,\lambda),\
     (\tilde x,\tilde t,\tilde\lambda)\in\Gamma(x_0,t_0)$ with $0<\lambda\leq\epsilon$, $0<\tilde\lambda\leq\epsilon$. Arguing as on p.461-462 in \cite{KP}, using \eqref{eq14+}-\eqref{eq14++} and using parabolic balls instead  of the standard (elliptic) balls, and applying Lemma \ref{le1a}, we can conclude that
     \begin{eqnarray}\label{iosced}|u(x,t,\lambda)-u(\tilde x,\tilde t,\tilde\lambda)|\leq c\epsilon\tilde N_\ast(\nabla u)(x_0,t_0).
     \end{eqnarray}
     \eqref{iosced} implies $(i)$ and $(ii)$. To prove $(iii)$ we consider ${\bf\psi}\in C_0^\infty(\mathbb R^{n+1},\mathbb C^n)$, $\epsilon>0$, and note that
               \begin{eqnarray}
             \biggl |\int_{\mathbb R^{n+1}} f\div_{||}{\bf\psi}\, dxdt\biggr |
             &\leq&
             \biggl |\int_{\mathbb R^{n+1}} u(x,t,\epsilon)\div_{||}{\bf\psi}\, dxdt\biggr |\notag\\
             &&+\int_{\mathbb R^{n+1}} |u(x,t,\epsilon)-f(x,t)||\div_{||}{\bf\psi}|\, dxdt.\end{eqnarray}
             Hence,
             \begin{eqnarray}
             \biggl |\int_{\mathbb R^{n+1}} f\div_{||}{\bf\psi}\, dxdt\biggr |&\leq&||\nabla u(\cdot,\epsilon)||_2||{\bf\psi}||_2+c\epsilon||\tilde N_\ast(\nabla u)||_2||\div_{||}{\bf\psi}||_2\notag\\
             &\leq&c||\tilde N_\ast(\nabla u)||_2||{\bf\psi}||_2+c\epsilon||\tilde N_\ast(\nabla u)||_2||\div_{||}{\bf\psi}||_2,\end{eqnarray}
             by $(ii)$ and Lemma \ref{trace1}. In particular, letting $\epsilon\to 0$ we see that
             \begin{eqnarray}
     \biggl |\int_{\mathbb R^{n+1}} f\div_{||}{\bf\psi}\, dxdt\biggr |
             \leq c||\tilde N_\ast(\nabla u)||_2||{\bf\psi}||_2
     \end{eqnarray}
     which proves that $\nabla_{||}f\in L^2(\mathbb R^{n+1},\mathbb C)$ and that $||\nabla_{||}f||_2\leq c||\tilde N_\ast(\nabla u)||_2$. Similarly,
     \begin{eqnarray}
             \biggl |\int_{\mathbb R^{n+1}} fH_tD^t_{1/2}\psi\, dxdt\biggr |
             \leq c\bigl(\sup_{\lambda>0}||H_tD^t_{1/2} u(\cdot,\cdot,\lambda)||_2\bigr)||\psi||_2,
             \end{eqnarray}
     whenever $\psi\in C_0^\infty(\mathbb R^{n+1},\mathbb C)$, proving that $H_tD^t_{1/2}f\in L^2(\mathbb R^{n+1},\mathbb C)$ and that
     $$||H_tD^t_{1/2}f||_2\leq c\sup_{\lambda>0}||H_tD^t_{1/2} u(\cdot,\cdot,\lambda)||_2.$$  This completes the proof of
     $(iii)$. $(iv)-(v)$ follows by similar considerations. We omit further details.
     \end{proof}

       \begin{lemma}\label{trace3} Assume that $\mathcal{H}$,  $\mathcal{H}^\ast$ satisfy \eqref{eq3}-\eqref{eq4} as well as \eqref{eq14+}-\eqref{eq14++}. Assume that  $\mathcal{H}u=0$ in $\mathbb R^{n+2}_+$ and that
     \begin{eqnarray}\label{assump+a}\sup_{\lambda>0}||\nabla u(\cdot,\cdot,\lambda)||_2+\sup_{\lambda>0}||H_tD^t_{1/2}u(\cdot,\cdot,\lambda)||_2<\infty.
     \end{eqnarray} Then there exists
     $g\in L^2(\mathbb R^{n+1},\mathbb C)$ such that $g=\partial u/\partial\nu$ in the sense that
        \begin{eqnarray}\label{aaq}
     (i)&& \int_{\mathbb R^{n+2}_+}\biggl (A\nabla u\cdot\overline{\nabla\phi}-D_t^{1/2}u\overline{H_tD^{1/2}_t{\phi}}\biggr )\, dxdtd\lambda
     =\int_{\mathbb R^{n+1}}g\overline{\phi}\, dxdt,
    \end{eqnarray}
     whenever $\phi\in \tilde{\mathbb H}(\mathbb R^{n+2}, \mathbb C)$ has compact support, and such that
     \begin{eqnarray}\label{aaq+}
     (ii)&& \mbox{$-\sum_{j=1}^{n+1}A_{n+1,j}(\cdot)\partial_{x_j}u(\cdot,\cdot,\lambda)\to g$}
     \end{eqnarray}
      weakly in $L^2(\mathbb R^{n+1},\mathbb C)$ as
     $\lambda\to 0$.
     \end{lemma}
     \begin{proof}  Consider $R$, $0<R<\infty$, fixed and let $\tilde Q_R$ be the standard parabolic space-time cube in $\mathbb R^{n+2}$ with center at the origin and with side length defined by $R$. We denote by $\tilde{\mathbb H}_R(\mathbb R^{n+2},\mathbb C)$ the set of all $\Psi\in \tilde{\mathbb H}(\mathbb R^{n+2},\mathbb C)$ which have support contained in $\tilde Q_{R/2}$.  For $\Psi\in\tilde{\mathbb H}_R(\mathbb R^{n+2},\mathbb C)$ we let
       \begin{eqnarray}\label{tra1}\tilde\Lambda_R(\Psi):=\int_{\tilde Q_R\cap \mathbb R^{n+2}_+} \biggl (A\nabla u\cdot
     \overline{\nabla\Psi}-D_t^{1/2}u\overline{H_tD^{1/2}_t{\Psi}}\biggr )\, dxdtd\lambda.\end{eqnarray}
     Then $\tilde\Lambda_R$ is a  linear functional on $\tilde{\mathbb H}_R(\mathbb R^{n+2},\mathbb C)$ and the operator norm of $\tilde\Lambda_R$ satisfies
      \begin{eqnarray}\label{tra2}\quad\quad||\tilde\Lambda_R||_{\bar {\mathbb H}_R(\mathbb R^{n+2},\mathbb C)}\leq cR^{1/2}\biggl (\sup_{\lambda>0}||\nabla u(\cdot,\cdot,\lambda)||_2+\sup_{\lambda>0}||H_tD^t_{1/2} u(\cdot,\cdot,\lambda)||_2\biggr ).\end{eqnarray}
      Using \eqref{trcc1}, \eqref{trcc1+} and
           \eqref{trcc2} we see that the trace space of $\tilde {\mathbb H}_R(\mathbb R^{n+2},\mathbb C)$ onto $\mathbb R^{n+1}$ equals
           ${\mathbb H}_R^{1/2}(\mathbb R^{n+1},\mathbb C)$, i.e., the set of all functions in ${\mathbb H}^{1/2}(\mathbb R^{n+1},\mathbb C)$ which have compact support in $Q_{R/2}$,  the standard parabolic space-time cube in $\mathbb R^{n+1}$ with center at the origin and with side length defined by $R$.  We let $T:\tilde {\mathbb H}_R(\mathbb R^{n+2},\mathbb C)\to{\mathbb H}^{1/2}_R(\mathbb R^{n+1},\mathbb C)$ denote the trace operator and we let
      $$E:{\mathbb H}^{1/2}_R(\mathbb R^{n+1},\mathbb C)\to \tilde {\mathbb H}_R(\mathbb R^{n+2},\mathbb C),$$
      denote a linear extension operator, see \eqref{trcc1+}, such that
           \begin{eqnarray}\label{tra3}||E(\psi)||_{{\mathbb H}_R(\mathbb R^{n+2},\mathbb C)}\leq c||\psi||_{{\mathbb H}_{R}^{1/2}(\mathbb R^{n+1},\mathbb C)},
          \end{eqnarray}
          whenever $\psi\in {\mathbb H}_{R}^{1/2}(\mathbb R^{n+1},\mathbb C)$ and for a constant $c$. In particular, there is a 1-1 correspondence between $\tilde {\mathbb H}_R(\mathbb R^{n+2},\mathbb C)$ and
      ${\mathbb H}^{1/2}_R(\mathbb R^{n+1},\mathbb C)$. Using this we let, given $\psi\in {\mathbb H}_{R}^{1/2}(\mathbb R^{n+1},\mathbb C)$,
          $$ \Lambda_R(\psi)=\tilde\Lambda_R(E(\psi)).$$
          Then,  using \eqref{tra2} and \eqref{tra3} we see that the operator norm of $\Lambda_R$ satisfies
           \begin{eqnarray}\label{tra4}\quad\quad||\Lambda_R||_{{\mathbb H}^{1/2}_R(\mathbb R^{n+1},\mathbb C)}\leq cR^{1/2}\biggl (\sup_{\lambda>0}||\nabla u(\cdot,\cdot,\lambda)||_2+\sup_{\lambda>0}||H_tD^t_{1/2} u(\cdot,\cdot,\lambda)||_2\biggr ).\end{eqnarray}
           In particular, $\Lambda_R$ is a bounded linear functional on ${\mathbb H}^{1/2}_{R}(\mathbb R^{n+1},\mathbb C)$. As the dual of
          ${\mathbb H}^{1/2}(\mathbb R^{n+1},\mathbb C)$ can be identified with ${\mathbb H}^{-1/2}(\mathbb R^{n+1},\mathbb C)$ we see that
           $\Lambda_R$ can be identified with an element $g_R\in{\mathbb H}^{-1/2}(\mathbb R^{n+1},\mathbb C)$. Combining all these facts we see that
                 \begin{eqnarray}\label{tra5}
     \int_{\tilde Q_R\cap \mathbb R^{n+2}_+} \biggl (A\nabla u\cdot
     \overline{\nabla\Psi}-D_t^{1/2}u\overline{H_tD^{1/2}_t{\Psi}}\biggr )\, dxdtd\lambda =\langle g_R,T(\Psi)\rangle,
     \end{eqnarray}
     whenever $\Psi\in\tilde {\mathbb H}_R(\mathbb R^{n+2},\mathbb C)$ and where $\langle \cdot,\cdot\rangle$ is the duality pairing on ${\mathbb H}^{1/2}_{R}(\mathbb R^{n+1},\mathbb C)$. By standard arguments we see that  $g:=\lim_{R\to \infty}g_R$ exists in the sense of distributions and that
            \begin{eqnarray}\label{tra6}
     \int_{\mathbb R^{n+2}_+} \biggl (A\nabla u\cdot
     \overline{\nabla\Psi}-D_t^{1/2}u\overline{H_tD^{1/2}_t{\Psi}}\biggr )\, dxdtd\lambda =\langle g,T(\Psi)\rangle,
     \end{eqnarray}
     whenever $\Psi\in\tilde{\mathbb H}(\mathbb R^{n+2},\mathbb C)$.  It now only remains to prove that $g\in L^2(\mathbb R^{n+1},\mathbb C)$ and that $(ii)$ holds. We will prove these statements jointly. We intend to prove that
         \begin{eqnarray}\label{tra7}
    &&\int_{\mathbb R^{n+1}} -e_{n+1}\cdot A\nabla u(\cdot,\cdot,\lambda)\overline{T(\Psi)}\, dxdt\notag\\
    &&\to \int_{\mathbb R^{n+2}_+} \biggl (A\nabla u\cdot
     \overline{\nabla\Psi}-D_t^{1/2}u\overline{H_tD^{1/2}_t{\Psi}}\biggr )\, dxdtd\lambda,
     \end{eqnarray}
     as $\lambda\to 0$, whenever $\Psi\in C_0^\infty(\mathbb R^{n+2}, \mathbb C)$. Indeed, assuming
     \eqref{tra7} we see that
            \begin{eqnarray}\label{tra8}
    \int_{\mathbb R^{n+1}} -e_{n+1}\cdot A\nabla u(\cdot,\cdot,\lambda)\overline{T(\Psi)}\, dxdt\to \langle g,T(\Psi)\rangle,
     \end{eqnarray}
     as $\lambda\to 0$ and  whenever $\Psi\in C_0^\infty(\mathbb R^{n+2}, \mathbb C)$, and hence
               \begin{eqnarray}\label{tra9}
               ||g||_2\leq c\sup_{\lambda>0}||\nabla u(\cdot,\cdot,\lambda)||_2<\infty.
     \end{eqnarray}
     To prove \eqref{tra7}, fix $\lambda$, consider $0<\epsilon\ll \lambda$, and let  $\P_\epsilon$ be a standard approximation of the identity acting only in the $\lambda$-variable. Then, integrating by parts and using the equation, we see that
            \begin{eqnarray}\label{tra10}
    &&\int_{\mathbb R^{n+1}}-e_{n+1}\cdot \P_\epsilon(A\nabla u(\cdot,\cdot,\lambda))\overline{T(\Psi)(x,t)}\, dxdt\notag\\
    &=&\int_0^\infty\int_{\mathbb R^{n+1}}\div \bigl(\P_\epsilon(A\nabla u(\cdot,\cdot,\lambda+\sigma))\overline{\Psi(x,t,\sigma)}\bigr )\, dxdtd\sigma\notag\\
    &=&\int_0^\infty\int_{\mathbb R^{n+1}} \P_\epsilon(A\nabla u(\cdot,\cdot,\lambda+\sigma))(x,t)\cdot
     {\nabla\overline{\Psi(x,t,\sigma)}}\, dxdtd\sigma\notag\\
     &&-\int_0^\infty\int_{\mathbb R^{n+1}} \P_\epsilon(D_{1/2}^tu(\cdot,\lambda+\sigma))(x,t)
     \overline{H_tD^t_{1/2}{{\Psi}}(x,t,\sigma)}\, dxdtd\sigma.
     \end{eqnarray}
    The deduction in \eqref{tra10} uses \eqref{eq4}. Now, letting $\epsilon\to 0$ in \eqref{tra10} it follows from \eqref{assump+a} and by dominated convergence that
               \begin{eqnarray}
    &&\int_{\mathbb R^{n+1}}-e_{n+1}\cdot \P_\epsilon(A\nabla u(\cdot,\cdot,\lambda))\overline{T(\Psi)(x,t)}\, dxdt\notag\\
    &=&
    \int_{\mathbb R^{n+2}_+} A\nabla u(x,t,\lambda+\sigma)\cdot
     {\nabla\overline{\Psi(x,t,\sigma)}}\, dxdtd\sigma\notag\\
     &&-\int_{\mathbb R^{n+2}_+} D_{1/2}^tu(x,t,\lambda+\sigma)
     \overline{ H_tD^t_{1/2}\Psi(x,t,\sigma)}\, dxdtd\sigma.
     \end{eqnarray}
     Hence, to prove \eqref{tra7} we only have to prove that
                     \begin{eqnarray}\label{tra12}
I_1(\lambda)+I_2(\lambda)\to 0\mbox{ as $\lambda\to 0$},
     \end{eqnarray}
     where
                \begin{eqnarray*}
    I_1(\lambda)&=&\biggl |\int_{\mathbb R^{n+2}_+} \biggl(A\nabla u(x,t,\lambda+\sigma)-A\nabla u(x,t,\sigma)\biggr )\cdot
     {\nabla\overline{\Psi(x,t,\sigma)}}\, dxdtd\sigma\biggr|\notag\\
     I_2(\lambda)&=&\biggl |\int_{\mathbb R^{n+2}_+} \biggl(D_{1/2}^tu(x,t,\lambda+\sigma)-D_{1/2}^tu(x,t,\sigma)\biggr)\overline{ H_tD^t_{1/2}\Psi(x,t,\sigma)}\, dxdtd\sigma\biggr|.
     \end{eqnarray*}
Choose $R$ large enough to ensure that the support of $\Psi$ is contained in $\tilde Q_R=Q_R\times(-R,R)$ where $Q_R\subset\mathbb R^{n+1}$. Using this and writing
                \begin{eqnarray}
    I_1(\lambda)&\leq &\int_0^{2\lambda}\int_{Q_R} |\bigl(A\nabla u(x,t,\lambda+\sigma)-A\nabla u(x,t,\sigma)\bigr )\cdot
     {\nabla\overline{\Psi(x,t,\sigma)}}|\, dxdtd\sigma\notag\\
     &&+\int_{2\lambda}^R\int_{Q_R} |\bigl(A\nabla u(x,t,\lambda+\sigma)-A\nabla u(x,t,\sigma)\bigr )\cdot
     {\nabla\overline{\Psi(x,t,\sigma)}}|\, dxdtd\sigma,
     \end{eqnarray}
     we see that
\begin{eqnarray*}
    I_1(\lambda)\leq c_{\Psi}\lambda^{1/2}\biggl (\sup_{\lambda>0}||\nabla u(\cdot,\cdot,\lambda)||_2+\lambda^{1/2}\biggl (\int_\lambda^R||\nabla\partial_\sigma u(\cdot,\cdot,\sigma)||_2^2\, d\sigma\biggr )^{1/2}\biggr ).
     \end{eqnarray*}
     By a similar argument
     \begin{eqnarray*}
    \quad I_2(\lambda)\leq c_{\Psi}\lambda^{1/2}\biggl (\sup_{\lambda>0}||D_{1/2}^tu(\cdot,\cdot,\lambda)||_2+\lambda^{1/2}\biggl (\int_\lambda^R||D_{1/2}^t\partial_\sigma u(\cdot,\cdot,\sigma)||_2^2\, d\sigma\biggr )^{1/2}\biggr ).
     \end{eqnarray*}
     Using Lemma \ref{le1--} we see that
     \begin{eqnarray*}
\biggl (\int_\lambda^R||\nabla\partial_\sigma u(\cdot,\sigma)||_2^2\, d\sigma\biggr )^{1/2}&\leq& c\lambda^{-1/2}\sup_{\lambda>0}||\nabla u(\cdot,\cdot,\lambda)||_2.
     \end{eqnarray*}
     Using both Lemma \ref{le1a} and Lemma \ref{le1--} we see that
          \begin{eqnarray*}
\biggl (\int_\lambda^R||D_{1/2}^t\partial_\sigma u(\cdot,\sigma)||_2^2\, d\sigma\biggr )^{1/2}&\leq& c\lambda^{-1/2}\sup_{\lambda>0}||\nabla u(\cdot,\cdot,\lambda)||_2\notag\\
&&+c\sup_{\lambda>0}||H_tD^t_{1/2}u(\cdot,\cdot,\lambda)||_2.
     \end{eqnarray*}
     Combining these estimates  we see that \eqref{tra12} follows. Hence we can conclude that $g\in L^2(\mathbb R^{n+1},\mathbb C)$ and that $(ii)$ holds. This completes the proof of the lemma.\end{proof}

     \subsection{Boundary layer potentials}
          \begin{lemma}\label{trace4} Assume that $\mathcal{H}$,  $\mathcal{H}^\ast$ satisfy \eqref{eq3}-\eqref{eq4} as well as \eqref{eq14+}-\eqref{eq14++}. Assume, in addition, that
     \begin{eqnarray}\label{asa1}
     \bar\Gamma&:=&\sup_{\lambda\neq 0}||\nabla \mathcal{S}_\lambda^{\mathcal{H}}||_{2\to 2}+\sup_{\lambda\neq 0}||\nabla \mathcal{S}_\lambda^{\mathcal{H}^\ast}||_{2\to 2}\notag\\
   &+&\sup_{\lambda\neq 0}||H_tD^t_{1/2}\mathcal{S}_\lambda^{\mathcal{H}}||_{2\to 2}+\sup_{\lambda\neq 0}||H_tD^t_{1/2}\mathcal{S}_\lambda^{\mathcal{H}^\ast}||_{2\to 2}<\infty.
   \end{eqnarray}
     Then there exist  operators $\mathcal{K}^{\mathcal{H}}$, $\tilde {\mathcal{K}}^{\mathcal{H}}$,
     $\nabla_{||} \mathcal{S}_\lambda^{\mathcal{H}}|_{\lambda=0}$, $H_tD^t_{1/2}\mathcal{S}_\lambda^{\mathcal{H}}|_{\lambda=0}$, and a constant $c$, depending only on $n$, $\Lambda$, the De Giorgi-Moser-Nash constants, and $\bar\Gamma$, such that the following hold.

     \noindent
     First,
       \begin{eqnarray*}
       (i)&&\mbox{$(\pm\frac 1 2I+\tilde{\mathcal{K}}^{\mathcal{H}})f=\partial_\nu \mathcal{S}_{\pm\lambda}^{\mathcal{H}}f$ in the sense of \eqref{aaq}, and}\notag\\
     &&\mbox{$-e_{n+1}\cdot A\nabla \mathcal{S}_{\pm\lambda}^{\mathcal{H}}f\to (\pm\frac 1 2I+\tilde{\mathcal{K}}^{\mathcal{H}})f$, in the sense of \eqref{aaq+}}.
     \end{eqnarray*}
     Second,
     \begin{eqnarray*}
     (ii)&&\mathcal{D}_{\pm\lambda}^{\mathcal{H}}f\to (\mp\frac 1 2I+ \mathcal{K}^{\mathcal{H}})f,\notag\\
     (iii)&&\nabla_{||}\mathcal{S}_{\pm\lambda}^{\mathcal{H}} f\to \nabla_{||} \mathcal{S}_\lambda^{\mathcal{H}}|_{\lambda=0}f,\notag\\
     (iv)&&H_tD^t_{1/2}(\mathcal{S}_{\pm\lambda}^{\mathcal{H}} f)\to H_tD^t_{1/2}\mathcal{S}_\lambda^{\mathcal{H}}|_{\lambda=0}f,
     \end{eqnarray*}
     weakly in $L^2(\mathbb R^{n+1},\mathbb C)$ as $\lambda\to 0$.

     \noindent
     Third,
       \begin{eqnarray}\label{normest}
       ||(\pm\frac 1 2I+\tilde{\mathcal{K}}^{\mathcal{H}})f||_2+||(\mp\frac 1 2I+ \mathcal{K}^{\mathcal{H}})f||_2&\leq& c||f||_2,\notag\\
       ||\nabla_{||} \mathcal{S}_\lambda^{\mathcal{H}}|_{\lambda=0}f||_2+||H_tD^t_{1/2}\mathcal{S}_\lambda^{\mathcal{H}}|_{\lambda=0}f||_2&\leq& c||f||_2,
     \end{eqnarray}
     whenever $f\in L^2(\mathbb R^{n+1},\mathbb C)$.

     \noindent
     Fourth, there exists an operator $\mathcal{T}_\perp^{\mathcal{H}}$ such that
   \begin{eqnarray*}
     (v)&&\partial_\lambda\mathcal{S}_{\pm\lambda}^{\mathcal{H}} f\to\mp\frac 12\cdot\frac{f(x,t)}{A_{n+1,n+1}(x,t)}e_{n+1}+\mathcal{T}_\perp^{\mathcal{H}}f,
     \end{eqnarray*}
     weakly in $L^2(\mathbb R^{n+1},\mathbb C)$ as $\lambda\to 0$, and such that
        \begin{eqnarray}\label{normest+}
       ||\mathcal{T}_\perp^{\mathcal{H}}f||_2\leq c||f||_2,
     \end{eqnarray}
      whenever $f\in L^2(\mathbb R^{n+1},\mathbb C)$.

      \noindent
      Fifth, the same conclusions hold  with ${\mathcal{H}}$ replaced by ${\mathcal{H}}^\ast$.
     \end{lemma}
     \begin{proof} We first note that to prove the lemma it suffices to prove $(i)$ and that
        \begin{eqnarray}\label{normestagain}
       ||(\pm\frac 1 2I+\tilde{\mathcal{K}}^{\mathcal{H}})f||_2\leq c||f||_2,
     \end{eqnarray}
      whenever $f\in L^2(\mathbb R^{n+1},\mathbb C)$. Indeed, let $\mathcal{K}^{\mathcal{H}}$ be the operator which is the hermitian adjoint to $\tilde{\mathcal{K}}^{\mathcal{H}^\ast}$. Then $(ii)$ follows from $(i)$ and the observation that
     $\mathcal{D}_\lambda^{\mathcal{H}}$ equals the hermitian adjoint to $-e_{n+1}\cdot A^\ast\nabla \mathcal{S}_\sigma^{\mathcal{H}^\ast}|_{\sigma=-\lambda}$, see \eqref{eq11edsea+}. To prove $(iii)$ and $(iv)$ we simply have to verify, based on
     Lemma \ref{trace2} that
        \begin{eqnarray}\label{assumph}
     \tilde N_\ast^\pm(\nabla  \mathcal{S}_{\lambda}^{\mathcal{H}}f)\in L^2(\mathbb R^{n+1})\mbox{ and }\ \sup_{\lambda\neq 0}||H_tD^t_{1/2}  \mathcal{S}_{\lambda}^{\mathcal{H}}f||_2<\infty,
     \end{eqnarray}
     where $\tilde N_\ast^\pm$ are the non-tangential maximal functions defined in $\mathbb R_\pm^{n+2}$. However, \eqref{assumph} follows immediately from Lemma \ref{lemsl1++} $(i)-(ii)$ and the  definition of $\bar\Gamma$ in \eqref{asa1}. To obtain $(v)$ we  note that
     \begin{eqnarray}\label{assumphq}-A_{n+1,n+1}D_{n+1}\mathcal{S}_\lambda^{\mathcal{H}}=-e_{n+1}\cdot A\nabla \mathcal{S}_\lambda^{\mathcal{H}}+\sum_{j=1}^nA_{n+1,j}D_j\mathcal{S}_\lambda^{\mathcal{H}}.\end{eqnarray}
     $(v)$ now follows from $(i)$ and $(iii)$. Concerning the quantitative estimates, using Lemma \ref{trace2} $(iii)$, the definition of $\mathcal{K}^{\mathcal{H}}$ and duality, we see that  \eqref{normest} and \eqref{normest+} follows once we have established \eqref{normestagain}.

     To start the proof of $(i)$ and \eqref{normestagain}, we let $u^+(x,t,\lambda)= \mathcal{S}_{\lambda}^{\mathcal{H}}f(x,t)$ be defined in
     $\mathbb R_+^{n+2}$ and we let $u^-(x,t,\lambda)= \mathcal{S}_{\lambda}^{\mathcal{H}}f(x,t)$ be defined in
     $\mathbb R_-^{n+2}$. Again using Lemma \ref{lemsl1++}, \eqref{asa1} and Lemma \ref{trace1}, we see that
     Lemma \ref{trace3} applies to $u^+$ and $u^-$. Hence, applying Lemma \ref{trace3} we obtain $g^\pm\in L^2(\mathbb R^{n+1},\mathbb C)$ such that
        \begin{eqnarray}\label{aaqed}
     \int_{\mathbb R^{n+2}_+}\biggl (A\nabla u^+\cdot\overline{\nabla\phi}-D_t^{1/2}u^+\overline{H_tD^{1/2}_t{\phi}}\biggr )\, dxdtd\lambda
     &=&\int_{\mathbb R^{n+1}}g^+\overline{\phi}\, dxdt,\notag\\
     \int_{\mathbb R^{n+2}_-}\biggl (A\nabla u^-\cdot\overline{\nabla\phi}-D_t^{1/2}u^-\overline{H_tD^{1/2}_t{\phi}}\biggr )\, dxdtd\lambda
     &=&\int_{\mathbb R^{n+1}}g^-\overline{\phi}\, dxdt,
    \end{eqnarray}
     whenever $\phi\in \tilde{\mathbb H}(\mathbb R^{n+2}, \mathbb C)$ has compact support, and
     \begin{eqnarray}\label{aaqed+}
     &&\mbox{$-\sum_{j=1}^{n+1}A_{n+1,j}(\cdot)\partial_{x_j}u^+(\cdot,\cdot,\lambda)\to g^+(\cdot,\cdot)$},\notag\\
     &&\mbox{$-\sum_{j=1}^{n+1}A_{n+1,j}(\cdot)\partial_{x_j}u^-(\cdot,\cdot,-\lambda)\to g^-(\cdot,\cdot)$},
     \end{eqnarray}
      weakly in $L^2(\mathbb R^{n+1},\mathbb C)$ as $\lambda\to 0^+$. We now define $(\pm\frac 1 2I+\tilde{\mathcal{K}}^{\mathcal{H}})$ on $L^2(\mathbb R^{n+1},\mathbb C)$ through the relation
     \begin{eqnarray}
     (\pm\frac 1 2I+\tilde{\mathcal{K}}^{\mathcal{H}})f=g^\pm.
     \end{eqnarray}
     To show that this operator is well defined we only have to prove that $g^+-g^-=f$. In particular, it suffices to prove that
            \begin{eqnarray}\label{conc}
    \int_{\mathbb R^{n+1}}f\Psi\, dxdt&=&\int_{\mathbb R^{n+2}_+}\biggl (A\nabla u^+\cdot\overline{\nabla\Psi}-D_t^{1/2}u^+\overline{H_tD^{1/2}_t{\Psi}}\biggr )\, dxdtd\lambda\notag\\
    &&+\int_{\mathbb R^{n+2}_-}\biggl (A\nabla u^-\cdot\overline{\nabla\Psi}-D_t^{1/2}u^-\overline{H_tD^{1/2}_t{\Psi}}\biggr )\, dxdtd\lambda,
    \end{eqnarray}
    whenever $\Psi\in C_0^\infty(\mathbb R^{n+2},\mathbb C)$. Let $\eta>0$ and recall the smoothed single layer potential operator
    $\mathcal{S}_\lambda^{\mathcal{H},\eta}$ introduced in \eqref{sop}. We let $u_\eta^+(x,t,\lambda)=\mathcal{S}_\lambda^{\mathcal{H},\eta} f(x,t)$ be defined in $\mathbb R^{n+2}_+$ and we let $u_\eta^-(x,t,\lambda)=\mathcal{S}_\lambda^{\mathcal{H},\eta} f(x,t)$ be defined in $\mathbb R^{n+2}_-$. Then
    $$u_\eta^\pm(x,t,\lambda)=\int_{\mathbb R^{n+2}}\Gamma(x,t,\lambda,y,s,\sigma)f_\eta(y,s,\sigma)\, dydsd\sigma,\ \lambda\in \mathbb R_\pm,$$
    where $f_\eta(y,s,\sigma)=f(y,s)\varphi_\eta(\sigma)$ and $\varphi_\eta$ is the kernel of the smooth approximation of the identity acting in the $\lambda$-dimension. Let $U_\eta=u_\eta^+1_{\mathbb R^{n+2}_+}+u_\eta^-1_{\mathbb R^{n+2}_-}$. Then
             \begin{eqnarray}
   &&\int_{\mathbb R^{n+2}_+}\biggl (A\nabla u_\eta^+\cdot\overline{\nabla\Psi}-D_t^{1/2}u_\eta^+\overline{H_tD^{1/2}_t{\Psi}}\biggr )\, dxdtd\lambda\notag\\
    &&+\int_{\mathbb R^{n+2}_-}\biggl (A\nabla u_\eta^-\cdot\overline{\nabla\Psi}-D_t^{1/2}u_\eta^-\overline{H_tD^{1/2}_t{\Psi}}\biggr )\, dxdtd\lambda\notag\\
    &=&\int_{\mathbb R^{n+2}}\biggl (A\nabla U_\eta\cdot\overline{\nabla\Psi}-D_t^{1/2}U_\eta\overline{H_tD^{1/2}_t{\Psi}}\biggr )\, dxdtd\lambda.
    \end{eqnarray}
    Using that  $\Gamma$ is a fundamental solution to $\mathcal{H}$ we see that
            \begin{eqnarray}
 &&\int_{\mathbb R^{n+2}}\biggl (A\nabla U_\eta\cdot\overline{\nabla\Psi}-D_t^{1/2}U_\eta\overline{H_tD^{1/2}_t{\Psi}}\biggr )\, dxdtd\lambda\notag\\
   &=&\int_{\mathbb R^{n+2}}f_\eta\Psi\, dxdt\to \int_{\mathbb R^{n+1}}f\Psi\, dxdt,
    \end{eqnarray}
    as $\eta\to 0$. Given $\epsilon>0$ small we write
                 \begin{eqnarray}
    \int_{\mathbb R^{n+2}_+}A\nabla (u_\eta^+-u^+)\cdot\nabla\overline{\Psi}\, dxdtd\lambda&=&I_\epsilon+II_\epsilon,\notag\\
     \int_{\mathbb R^{n+2}_+}D^{1/2}_t(u_\eta^+-u^+)\overline{H_tD^{1/2}_t{\Psi}}\, dxdtd\lambda&=&\tilde I_\epsilon+\widetilde {II}_\epsilon,
    \end{eqnarray}
    where
                     \begin{eqnarray}
I_\epsilon&=&\int_\epsilon^\infty\int_{\mathbb R^{n+1}}A\nabla (u_\eta^+-u^+)\cdot\nabla\overline{\Psi}\, dxdtd\lambda,\notag\\
II_\epsilon&=&\int_0^\epsilon\int_{\mathbb R^{n+1}}A\nabla (u_\eta^+-u^+)\cdot\nabla\overline{\Psi}\, dxdtd\lambda,\notag\\
\tilde I_\epsilon&=&\int_\epsilon^\infty\int_{\mathbb R^{n+1}}D^{1/2}_t(u_\eta^+-u^+)\overline{H_tD^{1/2}_t{\Psi}}\, dxdtd\lambda,\notag\\
\widetilde {II}_\epsilon&=&\int_0^\epsilon\int_{\mathbb R^{n+1}} D^{1/2}_t(u_\eta^+-u^+)\overline{H_tD^{1/2}_t{\Psi}}\, dxdtd\lambda.
    \end{eqnarray}
    Choose $R$ so large  that the support of $\Psi$ is contained in $\tilde Q_R=Q_R\times(-R,R)$ where $Q_R\subset\mathbb R^{n+1}$.  Then, using
    Lemma \ref{smooth1} $(v)$ we have that
    $$|I_\epsilon|\leq c_\Psi\int_\epsilon^R\sup_{\epsilon<\lambda<R}||\nabla (\mathcal{S}_\lambda^{\mathcal{H},\eta}-\mathcal{S}_\lambda^\mathcal{H})f||_2\, d\lambda\to 0$$
    as $\eta\to 0$. Also, using \eqref{asa1} we see that
    $$\sup_{\eta>0}|II_\epsilon|\leq c_\Psi\epsilon\sup_{\eta>0}||\nabla \mathcal{S}_\lambda^{\mathcal{H},\eta} f||_2\leq
    c_\Psi\epsilon\sup_{\lambda\neq 0}||\nabla \mathcal{S}_\lambda^{\mathcal{H}} f||_2\leq c_\Psi\epsilon\bar\Gamma\to 0,$$
     as $\epsilon\to 0$. Similarly,  using Lemma \ref{smooth1} $(vi)$,
    \begin{eqnarray}
    |\tilde I_\epsilon|+|\widetilde {II}_\epsilon|&\leq& c_\Psi\int_\epsilon^R\sup_{\epsilon<\lambda<R}
    ||H_tD^t_{1/2}(\mathcal{S}_\lambda^\eta-\mathcal{S}_\lambda)f||_2\, d\lambda\notag\\
    &&+c_\Psi\epsilon\sup_{\eta>0}||H_tD^t_{1/2}\mathcal{S}_\lambda^\eta f||_2\to 0,
    \end{eqnarray}
    if we first let $\eta\to 0$ and then $\epsilon\to 0$. Arguing analogously in $\mathbb R^{n+2}_-$ we can combine the above and conclude that \eqref{conc} holds.
    Thus $(\pm\frac 1 2I+\tilde{\mathcal{K}}^{\mathcal{H}})$ is well-defined. An application of Lemma \ref{trace3} $(ii)$ now completes the proof of $(i)$. \eqref{normestagain} follows from \eqref{tra9}. This completes the proof of the lemma.
     \end{proof}

     \section{Uniqueness}\label{sec4+}
      In this section we establish the uniqueness of solutions to $(D2)$, $(N2)$ and $(R2)$. The proofs of uniqueness for $(D2)$ and $(R2)$ are fairly standard and rely on the introduction of the Green function and appropriate estimates thereof. Our proofs of uniqueness for $(D2)$ and $(R2)$ are similar to the corresponding arguments in \cite{AAAHK} and we will therefore not include all details. However, to prove uniqueness for $(N2)$ we have to work harder compared to
      \cite{AAAHK} and in this case we give all the details of the proof. In the case of $(N2)$ our proof is inspired by arguments in \cite{HL}.

              \begin{lemma}\label{trace7} Assume that $\mathcal{H}$,  $\mathcal{H}^\ast$ satisfy \eqref{eq3}-\eqref{eq4} as well as \eqref{eq14+}-\eqref{eq14++}. Assume the existence of solutions to $(D2)$ and  $(R2)$. Then the solutions are unique  in the sense that
        \begin{eqnarray}
     (i)&&\mbox{if $u$ solves $(D2)$, and $u(\cdot,\cdot,\lambda)\to 0$ in $L^2(\mathbb R^{n+1},\mathbb C)$ as $\lambda\to 0$,}\notag\\
     &&\mbox{then $u\equiv 0$},\notag\\
       (ii)&&\mbox{if $u$ solves $(R2)$, and $u(\cdot,\cdot,\lambda)\to 0$ n.t. in $\mathbb H(\mathbb R^{n+1},\mathbb C)$ as $\lambda\to 0$,}\notag\\
       &&\mbox{then $u\equiv 0$ modulo a constant.}
     \end{eqnarray}
     \end{lemma}
     \begin{proof} We first prove $(i)$. Consider, for $(x,t,\lambda)\in\mathbb R^{n+2}_+$ fixed,  the fundamental solution $\Gamma(x,t,\lambda,y,s,\sigma)$. Using Lemma \ref{le3} we see that
             \begin{eqnarray}\label{esa1}
||\nabla_{||} \Gamma(x,t,\lambda,\cdot,\cdot,\cdot)||_2\leq c\lambda^{-(n+2)/2}.
     \end{eqnarray}
     Furthermore,
               \begin{eqnarray}\label{esa2}
||H_tD^t_{1/2}\Gamma(x,t,\lambda,\cdot,\cdot,\cdot)||_2^2&\leq& c||\partial_t\Gamma(x,t,\lambda,\cdot,\cdot,\cdot)||_2||\Gamma(x,t,\lambda,\cdot,\cdot,\cdot)||_2\notag\\
&\leq& c\lambda^{-(n+2)/2},
     \end{eqnarray}
     by \eqref{eq14+}, Lemma \ref{le1a} and Lemma \ref{le2+}. In particular,  $\Gamma(x,t,\lambda,\cdot,\cdot,\cdot)\in \mathbb H(\mathbb R^{n+1},\mathbb C)$. Hence, using the existence for $(R2)$ we can conclude that that there exists $w=w_{(x,t,\lambda)}$ such that
     \begin{eqnarray}\label{eq9mo}
               &&\mathcal{H}w = 0\mbox{ in $\mathbb R^{n+2}_+$},\notag\\
               &&\lim_{\lambda\to 0}w(\cdot,\cdot,\lambda)=\Gamma(x,t,\lambda,\cdot,\cdot,\cdot)\mbox{ n.t},
               \end{eqnarray}
               and such that, see \eqref{esa1} and \eqref{esa2},
                    \begin{eqnarray}\label{eq9+mo}
                   ||\tilde N_\ast(\nabla w)||_2+||\tilde N_\ast(H_tD^t_{1/2}w)||_2\leq c\lambda^{-(n+2)/2}.
    \end{eqnarray}
    We now let  $$G(x,t,\lambda,y,s,\sigma)=\Gamma(x,t,\lambda,y,s,\sigma)-w_{(x,t,\lambda)}(y,s,\sigma),$$
and note that
                    \begin{eqnarray}\label{eq9+mol}
 \sup_{\sigma:\ |\sigma-\lambda|>\lambda/8}||\nabla G(x,t,\lambda,\cdot,\cdot,\sigma)||_2\leq c\lambda^{-(n+2)/2}.
    \end{eqnarray}
    Let $\theta\in C_0^\infty(\mathbb R^{n+2}_+)$ with $\theta=1$ is a neighborhood of $(x,t,\lambda)$. Then
    \begin{eqnarray}\label{eq9+moa}
    u(x,t,\lambda)=(u\theta)(x,t,\lambda)&=&\int\overline{A^\ast\nabla_{y,\sigma}G(x,t,\lambda,y,s,\sigma)}\cdot \nabla (u\theta)(y,s,\sigma)\, dydsd\sigma\notag\\
    &&-\int\overline{\partial_sG(x,t,\lambda,y,s,\sigma)}(u\theta)(y,s,\sigma)\, dydsd\sigma.
    \end{eqnarray}
    Hence, using that $\mathcal{H}u=0$ we see that
       \begin{eqnarray}\label{eq9+moab}
    |u(x,t,\lambda)|\leq c(I+II+III),
    \end{eqnarray}
    where
        \begin{eqnarray}\label{eq9+moabc}
            I&=&\int|G(x,t,\lambda,y,s,\sigma)|\nabla u(y,s,\sigma)||\nabla\theta(y,s,\sigma)|\, dydsd\sigma,\notag\\
    II&=&\int|\nabla_{y,\sigma}G(x,t,\lambda,y,s,\sigma)|u(y,s,\sigma)||\nabla\theta(y,s,\sigma)|\, dydsd\sigma,\notag\\
    III&=&\int|G(x,t,\lambda,y,s,\sigma)||u(y,s,\sigma)||\partial_s\theta(y,s,\sigma)|\, dydsd\sigma.
    \end{eqnarray}
    Let $\epsilon<\lambda/8$ and let $R>8\lambda$. Let $\phi\in C_0^\infty(-2,2)$ with $\phi\geq 0$, $\phi\equiv 1$ on $(-1,1)$ and let
    $\tilde \phi$ be a standard cut-off for $Q_R(x,t)$  such that $\tilde\phi\in C_0^\infty(2Q_R(x,t))$, $\tilde\phi\geq 0$, $\tilde\phi\equiv 1$ on
     $Q_R(x,t)$. We let
    $$\theta(y,s,\sigma)=\tilde\phi(y,s)(1-\phi(\sigma/\epsilon))\phi(\sigma/(100R)).$$
    Note that
    $$\theta(y,s,\sigma)=1\mbox{ whenever } (y,s,\sigma)\in Q_R(x,t)\times \{2\epsilon\leq \sigma\leq 100 R\}.$$
    The domains where the integrands in $I-III$ are non-zero are contained in
    the union $D_1\cup D_2\cup D_3$ where
                \begin{eqnarray*}\label{eq9+moabcd}
(i)&&D_1\subset 2Q_R(x,t)\times \{\epsilon<\sigma<2\epsilon\},\notag\\
(ii)&&D_2\subset 2Q_R(x,t)\times \{100R<\sigma<200R\},\notag\\
(iii)&&D_3\subset (2Q_R(x,t)\setminus Q_R(x,t))\times \{0<\sigma<200R\},
    \end{eqnarray*}
    and
                \begin{eqnarray*}\label{eq9+moabcd}
(i')&& ||\epsilon\nabla\theta||_{L^\infty(D_1)}+||\epsilon^2\partial_s\theta||_{L^\infty(D_1)}\leq c,\notag\\
(ii')&& ||R\nabla\theta||_{L^\infty(D_2)}+||R^2\partial_s\theta||_{L^\infty(D_2)}\leq c,\notag\\
(iii')&& ||R\nabla\theta||_{L^\infty(D_3)}+||R^2\partial_s\theta||_{L^\infty(D_3)}\leq c.
    \end{eqnarray*}
Using this we see that
          \begin{eqnarray}\label{eq9+moabcg}
            I=I_1+I_2+I_3\notag\\
    \end{eqnarray}
    where
    \begin{eqnarray}\label{eq9+moabcl}
            I_1&=&\frac c\epsilon\int_{D_1}|G||\nabla u|\, dydsd\sigma,\notag\\
            I_2&=&\frac cR\int_{(D_2\cup D_3)\cap \Omega_{\lambda/4}}|G||\nabla u|\, dydsd\sigma,\notag\\
            I_3&=&\frac cR\int_{(D_2\cup D_3)\setminus \Omega_{\lambda/4}}|G||\nabla u|\, dydsd\sigma,
    \end{eqnarray}
    and where $\Omega_\rho=\mathbb R^{n+2}_+\cap\{(y,s,\sigma): \sigma\geq\rho\}$, for $\rho>0$. By the construction it is easily seen that
        \begin{eqnarray}\label{eq9+moabcl}
       (i)&& \biggl (\int_0^a\int_{\mathbb R^{n+1}}|G(x,t,\lambda,y,s,\sigma)|^2\, dydsd\sigma\biggr )^{1/2}\leq ca^{3/2}\lambda^{-(n+2)/2},\notag\\
       (ii)&&\biggl (\int_0^a\int_{\mathbb R^{n+1}}|G(x,t,\lambda,y,s,\sigma)|^2\, \frac {dydsd\sigma}{\sigma}\biggr )^{1/2}\leq ca^{1/2}\lambda^{-(n+2)/2},
    \end{eqnarray}
    whenever $a\in (0,\lambda/2)$. Using this, and by now standard energy estimates for $u$, we see that
      \begin{eqnarray}\label{eq9+moabcl}
            I_1&\leq&c\epsilon^{-3/2}\sup_{0<\sigma<3\epsilon}||u(\cdot,\cdot,\sigma)||_2\epsilon^{3/2}\lambda^{-(n+2)/2}\notag\\
            &=&c\lambda^{-(n+2)/2}\sup_{0<\sigma<3\epsilon}||u(\cdot,\cdot,\sigma)||_2.
    \end{eqnarray}
    Hence, as, by assumption, $u(\cdot,\cdot,\sigma)\to 0$ in $L^2(\mathbb R^{n+1},\mathbb C)$ as $\sigma\to 0$ we can conclude that $I_1\to 0$ as
    $\epsilon\to 0$. To estimate $I_2$ we first note, by the solvability of $(D2)$, that
        \begin{eqnarray}\label{eq9+moabcl}
            I_2&\leq &\frac cR\biggl(\int_{(D_2\cup D_3)\cap \Omega_{\lambda/4}}|G|^2\, \frac {dydsd\sigma}{\sigma}\biggr )^{1/2}|||\sigma\nabla u|||_+,\notag\\
    &\leq &\frac cR\biggl(\int_{(D_2\cup D_3)\cap \Omega_{\lambda/4}}|G|^2\, \frac {dydsd\sigma}{\sigma}\biggr )^{1/2},
    \end{eqnarray}
    for some constant $c<\infty$ now also depending on $u$. To proceed we now need, in analogy with  \cite{AAAHK}, a H{\"o}lder type estimate for $G$ close to  $\mathbb R^{n+1}=\partial \mathbb R^{n+2}$. Fortunately there are several recent papers dealing with the construction and estimates of Green's functions for parabolic equations and systems. We here choose to quote some results from \cite{DK}. Indeed, let $\alpha$ be the De Giorgi-Nash exponent in
    \eqref{eq14+}-\eqref{eq14++} in the case $p=2$. Theorem 3.16 in \cite{DK} gives the existence of positive constants $c$ and $\kappa$ such that
    \begin{eqnarray}\label{eq9+moabclse}
    &&|G(x,t,\lambda,y,s,\sigma)|\notag\\
    &\leq& c(\delta(x,t,\lambda,y,s,\sigma))^\alpha(t-s)^{-(n+1)/2}\exp\biggl (-\kappa\frac {(|\lambda-\sigma|+|x-y|)^2}{t-s}\biggr )
    \end{eqnarray}
    whenever $(x,t,\lambda)$, $(y,s,\sigma)\in \mathbb R^{n+2}_+$, $t>s$, and where
    \begin{eqnarray*}\label{eq9+moabclse+}
&&\delta(x,t,\lambda,y,s,\sigma)\notag\\
&:=&\biggl(1\wedge\frac{\delta(x,t,\lambda)}{|\lambda-\sigma|+|x-y|+|t-s|^{1/2}}\biggr ) \biggl(1\wedge\frac{\delta(y,s,\sigma)}{|\lambda-\sigma|+|x-y|+|t-s|^{1/2}}\biggr ),
    \end{eqnarray*}
    $\delta(x,t,\lambda)=\lambda$, $\delta(y,s,\sigma)=\sigma$. Using this we see that
            \begin{eqnarray}\label{eq9+moabclsea}
       \frac cR\biggl(\int_{(D_2\cup D_3)\cap \Omega_{\lambda/4}}|G|^2\, \frac {dydsd\sigma}{\sigma}\biggr)^{1/2}&\leq&c_\lambda R^{-2\alpha}.
    \end{eqnarray}
    Putting these estimates together we can conclude that
            \begin{eqnarray}\label{eq9+moabclse+}
            I_2&\leq & c_\lambda R^{-2\alpha}\to 0\mbox{ as $R\to \infty$}.
    \end{eqnarray}
    Furthermore, choosing $a=\lambda/4$ in \eqref{eq9+moabcl} $(ii)$ we also see that
               \begin{eqnarray}\label{eq9+moabclse++}
            I_3&\leq & c_\lambda R^{-1}|||\sigma\nabla u|||_+\to 0\mbox{ as $R\to \infty$}.
    \end{eqnarray}
    Put together we can conclude, by letting either $\epsilon\to 0$, or using that $u(\cdot,\cdot,\sigma)\to 0$ in $L^2(\mathbb R^{n+1},\mathbb C)$ as $\sigma\to 0$, or by letting $R\to \infty$, that $I\to 0$. By similar arguments, writing $II=II_1+II_2+II_3$, $III=III_1+III_2+III_3$, again letting either $\epsilon\to 0$, or using that $u(\cdot,\cdot,\sigma)\to 0$ in $L^2(\mathbb R^{n+1},\mathbb C)$ as $\sigma\to 0$, or by letting $R\to \infty$ it also follows that  $II\to 0$,  $III\to 0$. In particular, $u\equiv 0$. We omit further details  and claim that the proof of uniqueness for $(D2)$ can be completed in this manner.

    To prove $(ii)$ we suppose that $\tilde N_\ast(\nabla u)\in L^2(\mathbb R^{n+1})$, $\tilde N_\ast(H_tD_t^{1/2}u)\in L^2(\mathbb R^{n+1})$ and that $u\to 0$ n.t in $\mathbb H(\mathbb R^{n+1},\mathbb C)$ as $\lambda\to 0$. In this case we again express $u(x,t,\lambda)=(u\theta)(x,t,\lambda)$ as above getting three terms $I$, $II$, $III$. We then split each of these terms into three terms. Choosing $a=2\epsilon$ in \eqref{eq9+moabcl} $(i)$, applying Lemma \ref{trace2} with $f\equiv 0$, using H{\"o}lder's inequality and standard energy estimate applied to $\nabla G$,  we then see that
    $$I_1+II_1\leq c\epsilon\lambda^{-(n+2)}||\tilde N_\ast(\nabla u)||_2\to 0\mbox{ as $\epsilon\to 0$}.$$
    All other pieces can be handled as well, see for instance the proof of Lemma 4.31 in \cite{AAAHK}. We here omit further details and claim that the proof of uniqueness for $(R2)$ can be completed in this manner.
     \end{proof}

             \begin{lemma}\label{trace7++} Assume that $\mathcal{H}$,  $\mathcal{H}^\ast$ satisfy \eqref{eq3}-\eqref{eq4} as well as \eqref{eq14+}-\eqref{eq14++}.  Assume the existence of solutions to $(N2)$ and assume that $\mathcal{H}$, $\mathcal{H}^\ast$ have bounded, invertible and good layer potentials in the sense of Definition \ref{blayer+}, for some constant $\Gamma$.  Then the solutions to $(N2)$ are unique in the sense that
           \begin{eqnarray}
     (iii)&&\mbox{if $u$ solves $(N2)$, and $\partial_\nu u=0$ in the sense of Lemma \ref{trace3} $(i)$ and $(ii)$,}\notag\\
     &&\mbox{then $u\equiv 0$ modulo constants.}
     \end{eqnarray}
     \end{lemma}
     \begin{proof} Assume that $\tilde N_\ast(\nabla u)\in L^2(\mathbb R^{n+1})$ and that $\partial_\nu u=0$ in the sense of Lemma \ref{trace3} $(i)$ and $(ii)$. We claim that
     \begin{eqnarray}\label{cla}
     \sup_{\lambda>0}||H_tD^t_{1/2} u(\cdot,\cdot,\lambda)||_2<\infty.
     \end{eqnarray}
     Assuming \eqref{cla} for now  we see, using  Lemma \ref{trace2} $(i)$,  that $u\to u_0$ n.t for some $u_0\in{\mathbb H}(\mathbb R^{n+1},\mathbb C)$. Using that  $\mathcal{H}$  has bounded, invertible and good layer potentials in the sense of Definition \ref{blayer+}, and in particular that
     $\mathcal{S}_\lambda^{\mathcal{H}}|_{\lambda=0}:L^2(\mathbb R^{n+1},\mathbb C)\to {\mathbb H}(\mathbb R^{n+1},\mathbb C)$ is a bijection,  and the uniqueness in $(R2)$, see Lemma \ref{trace7},  we see that
     $$u(\cdot,\cdot,\lambda)=\mathcal{S}_\lambda^{\mathcal{H}}((\mathcal{S}_0^{\mathcal{H}})^{-1}(u_0)).$$
     In particular, using Lemma \ref{trace4} we have
     $$0=\partial_\nu u=\biggl (\frac 1 2I+\tilde{\mathcal{K}}^{\mathcal{H}}\biggr )((\mathcal{S}_0^{\mathcal{H}})^{-1}(u_0)).$$
     Using the assumptions that $(\frac 1 2I+\tilde{\mathcal{K}}^{\mathcal{H}}):L^2(\mathbb R^{n+1},\mathbb C)\to L^2(\mathbb R^{n+1},\mathbb C)$ and $\mathcal{S}_0^{\mathcal{H}}:L^2(\mathbb R^{n+1},\mathbb C)\to {\mathbb H}(\mathbb R^{n+1},\mathbb C)$ are bijections, we can conclude that $u_0=0$ in the sense of
     ${\mathbb H}(\mathbb R^{n+1},\mathbb C)$. In particular, $u_0$ is constant a.e., and  by uniqueness in $(R2)$ we see that $u$ is constant. Hence it only remains to prove \eqref{cla}. To start the proof of \eqref{cla}  we  fix $\lambda_0>0$  and we let, for $R\gg \lambda_0$ given,
               \begin{eqnarray}
               D_1&=&\{(x,t,\lambda)\in\mathbb R^{n+2}_+: (x,t)\in Q_{2R},\ 0<\lambda<2R\},\notag\\
               D_2&=&\{(x,t,\lambda)\in\mathbb R^{n+2}_+: (x,t)\in Q_{2R},\ 2R\leq \lambda<4R\},\notag\\
               D_3&=&\{(x,t,\lambda)\in\mathbb R^{n+2}_+: (x,t)\in Q_{6R},\ 0<\lambda<6R\}.
\end{eqnarray}
We choose $\phi\in C_0^\infty(Q_{2R}\times (-2R,2R))$, $\phi\geq 0$, with $\phi\equiv 1$ on $Q_{R}\times (-R,R)$ and such that
$$||\partial_t\phi||_{\infty}+||\nabla^2\phi||_\infty\leq cR^{-2}.$$
We introduce
$$v(x,t,\lambda)=u(x,t,\lambda_0+\lambda),$$
and we let
$$w(x,t,\lambda)=(v(x,t,\lambda)-m_{D_1}v)\phi(x,t,\lambda),$$
where
$$m_{D_1}v=\mean{D_1}v(x,t,\lambda)\, dxdtd\lambda.$$  We note that
          \begin{eqnarray}
||H_tD^t_{1/2} u(\cdot,\cdot,\lambda_0)||_2^2&\approx&\int_{\mathbb R}\int_{\mathbb R}\int_{\mathbb R^n}\frac {|u(x,t,\lambda_0)-u(y,s,\lambda_0)|^2}{(s-t)^2}\, dxdtds\notag\\
&=&\int_{\mathbb R}\int_{\mathbb R}\int_{\mathbb R^n}\frac {|v(x,t,0)-v(y,s,0)|^2}{(s-t)^2}\, dxdtds.
\end{eqnarray}
Hence, using the definition of $w$, and that $w=v-m_{D_1}v$ on $Q_{R}\times (-R,R)$, we  see that
         \begin{eqnarray}
&&\int_{-R}^R\int_{-R}^R\int_{\mathbb R^n}\frac {|u(x,t,\lambda_0)-u(y,s,\lambda_0)|^2}{(s-t)^2}\, dxdtds\notag\\
&\leq&\int_{-R}^R\int_{-R}^R\int_{\mathbb R^n}\frac {|w(x,t,0)-w(y,s,0)|^2}{(s-t)^2}\, dxdtds\leq c||H_tD^t_{1/2} w(\cdot,\cdot,0)||_2^2.
\end{eqnarray}
Letting $R\to\infty$ we see  that \eqref{cla} follows once we can prove that
   \begin{eqnarray}\label{cla+}
     ||H_tD^t_{1/2} w(\cdot,\cdot,0)||_2\leq c||\tilde N_\ast(\nabla u)||_2,
     \end{eqnarray}
     for some $c$. To start the proof of \eqref{cla+} we note that
      \begin{eqnarray}\label{cla++}
     ||H_tD^t_{1/2} w(\cdot,\cdot,0)||_2^2&=&-\int_{0}^\infty\int_{\mathbb R^{n+1}}(D_{1/2}^tw)\overline{(D_{1/2}^t\partial_\lambda w)}\, dxdtd\lambda\notag\\
     &\leq&2\biggl (\int_{0}^\infty\int_{\mathbb R^{n+1}}|D_{1/4}^t\partial_\lambda w|^2\, dxdtd\lambda\biggr )^{1/2}\notag\\
     &&\times \biggl (\int_{0}^\infty\int_{\mathbb R^{n+1}}|D_{3/4}^tw|^2\, dxdtd\lambda\biggr )^{1/2}\notag\\
     &=:&2I_1^{1/2}I_2^{1/2}.
     \end{eqnarray}
     Integrating by parts with respect to $\lambda$ we see that
       \begin{eqnarray}\label{cla+++}
     I_1&=&-\int_{0}^\infty\int_{\mathbb R^{n+1}}(D_{1/4}^t\partial_\lambda w)\overline{(D_{1/4}^t\partial_\lambda^2 w)}\, \lambda dxdtd\lambda\notag\\
     &\leq&2\biggl (\int_{0}^\infty\int_{\mathbb R^{n+1}}|\partial_\lambda^2 w|^2\, \lambda dxdtd\lambda\biggr )^{1/2}\notag\\
     &&\times \biggl (\int_{0}^\infty\int_{\mathbb R^{n+1}}|D_{1/2}^t\partial_\lambda w|^2\, \lambda dxdtd\lambda\biggr )^{1/2},
     \end{eqnarray}
     and
         \begin{eqnarray}\label{cla++++}
     I_2&=&-\int_{0}^\infty\int_{\mathbb R^{n+1}}(D_{3/4}^t\partial_\lambda w)(D_{3/4}^t\partial_\lambda w)\, \lambda dxdtd\lambda\notag\\
     &\leq&2\biggl (\int_{0}^\infty\int_{\mathbb R^{n+1}}|D_{1/2}^t\partial_\lambda w|^2\, \lambda dxdtd\lambda\biggr )^{1/2}\notag\\
     &&\times \biggl (\int_{0}^\infty\int_{\mathbb R^{n+1}}|\partial_t w|^2\, \lambda dxdtd\lambda\biggr )^{1/2}.
     \end{eqnarray}
     We also have, by integration by parts and by using the H{\"o}lder inequality, that
             \begin{eqnarray}\label{cla++++}
&&\int_{0}^\infty\int_{\mathbb R^{n+1}}|\partial_t w|^2\, \lambda dxdtd\lambda\leq c\int_{0}^\infty\int_{\mathbb R^{n+1}}|\partial_t\partial_\lambda w|^2\, \lambda^3\,  dxdtd\lambda.
     \end{eqnarray}
     Hence, we see that the proof of \eqref{cla+} is reduced to proving that
              \begin{eqnarray}\label{cla++++a}
              (i)&&\int_{0}^\infty\int_{\mathbb R^{n+1}}|\partial_\lambda^2 w|^2\, \lambda dxdtd\lambda\leq c||\tilde N_\ast(\nabla u)||_2^2,\notag\\
              (ii)&&\int_{0}^\infty\int_{\mathbb R^{n+1}}|\partial_t \partial_\lambda w|^2\, \lambda^3 dxdtd\lambda\leq c||\tilde N_\ast(\nabla u)||_2^2,\notag\\
              (iii)&&\int_{0}^\infty\int_{\mathbb R^{n+1}}|D_{1/2}^t\partial_\lambda w|^2\, \lambda dxdtd\lambda\leq c||\tilde N_\ast(\nabla u)||_2^2.
     \end{eqnarray}
     To start the proof of \eqref{cla++++a} we note that we can apply $(D2)$ to  $\partial_\lambda v$. Indeed, by the definition of bounded, invertible and good layer potentials in the sense of Definition \ref{blayer+}, $\partial_\lambda v=\mathcal{D^{\mathcal{H}}}f$ for some $f$ such that
     $$||f||_2\leq c||\tilde N_\ast(\nabla v)||_2\leq c||\tilde N_\ast(\nabla u)||_2.$$
     Using this, and  again using the assumptions of Lemma \ref{trace7++}, see Remark \ref{resea} and Lemma \ref{th0uu} below, as well as Lemma \ref{le1a}   we see that
        \begin{eqnarray}\label{cla++++b}
              (i')&&\int_{0}^\infty\int_{\mathbb R^{n+1}}|\partial_\lambda^2 v|^2\, \lambda dxdtd\lambda\leq c||\tilde N_\ast(\nabla u)||_2^2,\notag\\
              (ii')&&\int_{0}^\infty\int_{\mathbb R^{n+1}}|\partial_t \partial_\lambda v|^2\, \lambda^3 dxdtd\lambda\leq c||\tilde N_\ast(\nabla u)||_2^2.
     \end{eqnarray}
     To continue,
              \begin{eqnarray}\label{cla++++c}
              |\partial_\lambda^2 w|^2&\leq c\bigl(|\partial_\lambda^2 v|^2+|\partial_\lambda v|^2|\partial_\lambda\phi|^2+|v-m_{D_1}v|^2|\partial_\lambda^2\phi|^2\bigr ).
     \end{eqnarray}
     Using \eqref{cla++++c} and \eqref{cla++++b}, we see that
     \begin{eqnarray}\label{cla++++ag}
              \int_{0}^\infty\int_{\mathbb R^{n+1}}|\partial_\lambda^2 w|^2\, \lambda dxdtd\lambda&\leq& c||\tilde N_\ast(\nabla u)||_2^2\notag\\
              &&+cR^{-1}\int_{D_1}|\partial_\lambda v|^2\, dxdtd\lambda\notag\\
              &&+cR^{-3}\int_{D_1}|v-m_{D_1}v|^2\, \lambda dxdtd\lambda.
              \end{eqnarray}
              Hence,
                   \begin{eqnarray}\label{cla++++agg}
              \int_{0}^\infty\int_{\mathbb R^{n+1}}|\partial_\lambda^2 w|^2\, \lambda dxdtd\lambda&\leq& c||\tilde N_\ast(\nabla u)||_2^2\notag\\
              &&+cR^{-3}\int_{D_1}|v-m_{D_1}v|^2\, \lambda dxdtd\lambda.
              \end{eqnarray}
              Also,
                   \begin{eqnarray}\label{cla++++cg}
              |\partial_t\partial_\lambda w|^2&\leq& c\bigl(|\partial_t\partial_\lambda v|^2\phi^2+|\partial_tv|^2|\partial_\lambda \phi|^2+|\partial_\lambda v|^2|\partial_t \phi|^2\bigr)\notag\\
              &&+c|v-m_{D_1}v|^2|\partial_t\partial_\lambda\phi|^2.
     \end{eqnarray}
     Hence, by similar considerations, using also Lemma \ref{le1a}, we see that
          \begin{eqnarray}\label{cla++++aggg}
              \int_{0}^\infty\int_{\mathbb R^{n+1}}|\partial_t\partial_\lambda w|^2\, \lambda^3 dxdtd\lambda&\leq& c||\tilde N_\ast(\nabla u)||_2^2\notag\\
              &&+cR^{-3}\int_{D_1}|v-m_{D_1}v|^2\, \lambda dxdtd\lambda.
              \end{eqnarray}
              Finally,
                     \begin{eqnarray}\label{cla++++af}
             &&\int_{0}^\infty\int_{\mathbb R^{n+1}}|D_{1/2}^t\partial_\lambda w|^2\, \lambda dxdtd\lambda\notag\\
             &=&-2\int_{0}^\infty\int_{\mathbb R^{n+1}}D_{1/2}^t\partial_\lambda w\overline{D_{1/2}^t\partial_\lambda^2 w}\, \lambda^2 dxdtd\lambda\notag\\
             &&\leq c\biggl (\int_{0}^\infty\int_{\mathbb R^{n+1}}|\partial_\lambda^2 w|^2\, \lambda dxdtd\lambda\biggr )^{1/2}
             \biggl (\int_{0}^\infty\int_{\mathbb R^{n+1}}|\partial_t\partial_\lambda w|^2\, \lambda^3 dxdtd\lambda\biggr )^{1/2}.
     \end{eqnarray}
     Based on this we see that to complete the proof of \eqref{cla++++a} $(i)$-$(iii)$ it suffices to prove that
     \begin{eqnarray}\label{cla++++agggs}
             R^{-3}\int_{D_1}|v-m_{D_1}v|^2\, \lambda dxdtd\lambda\leq c||\tilde N_\ast(\nabla u)||_2^2.
              \end{eqnarray}
     To prove this we first note that
                 \begin{eqnarray}\label{cla++++e}
              T&:=&R^{-3}\int_{D_1}|v-m_{D_1}v|^2\, dxdtd\lambda\notag\\
              &\leq& cR^{-n-6}\int_{D_1}\int_{D_1}|v(y,s,\sigma)-v(x,t,\lambda)|^2\, dydsd\sigma dxdtd\lambda.
     \end{eqnarray}
     Consider $(y,s,\sigma)$, $(x,t,\lambda)\in D_1$. Let
$$(x',t',\lambda')=(x,t,\lambda+2R),\    (y',s',\sigma')=(y,s,\sigma+2R).$$
Note that $(x',t',\lambda')\in D_2$, $(y',s',\sigma')\in D_2$. Furthermore,
              \begin{eqnarray}\label{cla++++ef}
             |v(y,s,\sigma)-v(x,t,\lambda)|&\leq &|v(x',t',\lambda')-v(x,t,\lambda)|+|v(y,s,\sigma)-v(y',s',\sigma')|\notag\\
              &&+|v(x',t',\lambda')-v(y',s',\sigma')|\notag\\
              &\leq&cR \bigl(N_\ast(\partial_\lambda u)(x,t)+ N_\ast(\partial_\lambda u)(y,s)\bigr)\notag\\
              &&+|v(x',t',\lambda')-v(y',s',\sigma')|.
     \end{eqnarray}
     Hence, using the fundamental theorem of calculus, standard arguments, and Lemma \ref{le1a}, we see that
               \begin{eqnarray}\label{cla++++f}
              T&\leq& c||\tilde N_\ast(\nabla u)||_2^2+cR\int_{D_2}|\partial_tv(x,t,\lambda)|^2\, dxdtd\lambda\notag\\
              &\leq&c||\tilde N_\ast(\nabla u)||_2^2+cR^{-1}\int_{D_3}|\nabla v(x,t,\lambda)|^2\, dxdtd\lambda\notag\\
              &\leq& c||\tilde N_\ast(\nabla u)||_2^2.
     \end{eqnarray}
     This completes the proof of \eqref{cla++++agggs}, \eqref{cla++++a}, and hence the proof of \eqref{cla} and the lemma. \end{proof}

     \begin{remark} We here note that as part of the proof of Lemma \ref{trace7++} we have proved that if $\mathcal{H}$, $\mathcal{H}^\ast$ have bounded, invertible and good layer potentials in the sense of Definition \ref{blayer+}, for some constant $\Gamma$, then the estimate
     $$\sup_{\lambda>0}||H_tD_{1/2}^tu(\cdot,\cdot,\lambda)||_2\leq c||\tilde N_\ast(\nabla u)||_2$$
     holds, with a  uniform constant, for all solutions $u$ to $\mathcal{H}u= 0$ in $\mathbb R^{n+2}_+$ such that $\tilde N_\ast(\nabla u)\in L^2(\mathbb R^{n+1})$.
\end{remark}

\section{Existence of non-tangential limits}\label{sec4++}

Throughout this section we will assume that
\begin{eqnarray}\label{keysea}
&&\mbox{$\mathcal{H}$,  $\mathcal{H}^\ast$ satisfy \eqref{eq3}-\eqref{eq4} as well as \eqref{eq14+}-\eqref{eq14++}, and that}\notag\\
&&\mbox{$\mathcal{H}$, $\mathcal{H}^\ast$ have bounded, invertible and good layer potentials in the}\notag\\
&&\mbox{sense of Definition \ref{blayer+}, for some constant $\Gamma$.}
 \end{eqnarray}
 Note that \eqref{keysea} implies, in particular, that \eqref{asa1} holds.

\begin{lemma}\label{trace6} Assume \eqref{keysea}. Let $\psi\in C_0^\infty(\mathbb R^{n+1},\mathbb C)$ and consider $u(\cdot,\cdot,\lambda):=\mathcal{S}_\lambda^{\mathcal{H}}\psi(\cdot,\cdot)$ in $\mathbb R^{n+2}_-$. Let $u_0(\cdot,\cdot)=u(\cdot,\cdot,0)$. Then
     $$\mathcal{D}_\lambda^{\mathcal{H}} u_0=\mathcal{S}_\lambda^{\mathcal{H}} (\partial_\nu u),$$
     in $\mathbb R^{n+2}_-$ and where $\partial_\nu u$ exists in the sense of Lemma \ref{trace3}.
     \end{lemma}
     \begin{proof} It is enough to prove that
     \begin{eqnarray}\label{ibyparts-}\int_{\mathbb R^{n+1}} (\mathcal{D}_\lambda^{\mathcal{H}}  u_0)\bar\phi\, dxdt=\int_{\mathbb R^{n+1}} (\mathcal{S}_\lambda^{\mathcal{H}} (\partial_\nu u))\bar\phi\, dxdt,
     \end{eqnarray}
     whenever $\phi\in C_0^\infty(\mathbb R^{n+1},\mathbb C)$. Recall that the hermitian adjoint of $\mathcal{D}_\lambda^{\mathcal{H}}$ equals
     $-e_{n+1}\cdot A^\ast\nabla \mathcal{S}_\sigma^{\mathcal{H}^\ast}|_{\sigma=-\lambda}$, see \eqref{eq11edsea+}, and that the hermitian adjoint of $\mathcal{S}_\lambda^{\mathcal{H}}$ equals $\mathcal{S}_{-\lambda}^{\mathcal{H}^\ast}$. Let $v(\cdot,\cdot,\lambda)=\mathcal{S}_{\lambda}^{\mathcal{H}^\ast}\phi$ so that
     ${\mathcal{H}^\ast}v=0$ in $\mathbb R^{n+2}\setminus\{\lambda=0\}$. We consider $w(\cdot,\cdot,\sigma)=v(\cdot,\cdot,\sigma-\lambda)$ in $\mathbb R^{n+2}_-$ for $\lambda\geq 0$ fixed.   We claim that
              \begin{eqnarray}\label{ibyparts+}
      u(x,t,0)\overline{\partial_{\nu^\ast}w}(x,t,0),\,  \partial_\nu u(x,t,0)\bar w(x,t,0)\in L^1(\mathbb R^{n+1},\mathbb R).
      \end{eqnarray}
      To prove \eqref{ibyparts+} we see, by \eqref{asa1} and elementary estimates for single layer potentials, that
         \begin{eqnarray}\label{ibypartsse}
      &&\int_{\mathbb R^{n+1}} \bigl (|u(x,t,0)\overline{\partial_{\nu^\ast}w}(x,t,0)|+|\partial_\nu u(x,t,0)\bar w(x,t,0)|\bigr)\, dxdt\notag\\
      &&\leq c||\mathcal{S}_0^{\mathcal{H}}\psi||_2\sup_{\lambda<0}||\nabla \mathcal{S}_{\lambda}^{\mathcal{H}^\ast}\phi||_2\notag+c\sup_{\lambda<0}||\nabla \mathcal{S}_\lambda^{\mathcal{H}}\psi||_2\sup_{\lambda<0}||\mathcal{S}_{\lambda}^{\mathcal{H}^\ast}\phi||_2\notag\\
      &&\leq c_{\psi,\phi}\bigl (||\mathcal{S}_0^{\mathcal{H}}\psi||_2+\sup_{\lambda<0}||\mathcal{S}_{\lambda}^{\mathcal{H}^\ast}\phi||_2\bigr )\leq
      \tilde c_{\psi,\phi}<\infty.
      \end{eqnarray}
      Using \eqref{ibyparts+} we see that the proof of \eqref{ibyparts-} is reduced to proving that
           \begin{eqnarray}\label{ibyparts}
      \int_{\mathbb R^{n+1}} u(x,t,0)\overline{\partial_{\nu^\ast}w}(x,t,0)\, dxdt=\int_{\mathbb R^{n+1}} \partial_\nu u(x,t,0)\bar w(x,t,0)\, dxdt.
      \end{eqnarray}
        Let $\tilde Q_\rho=Q_\rho\times(-\rho,\rho)$, $\rho>0$, where $Q_\rho\subset\mathbb R^{n+1}$  is the standard parabolic cube in $\mathbb R^{n+1}$ with center at the origin and with side length defined by $\rho$. Let $R$ be so large that the supports of $\psi$ and $\phi$ are contained in in $Q_{R/4}$. Furthermore, let $\Psi_R\in C_0^\infty(\mathbb R^{n+2},\mathbb R)$, $\Psi_R\geq 0$,  be such that the support of $\Psi_R$ is contained in $\tilde Q_{2R}$ and such that $\Psi_R\equiv 1$ on
    $\tilde Q_R$. Then, using \eqref{asa1} and \eqref{aaq} we see that
           \begin{eqnarray}\label{aaqkla-}
     &&\int_{\mathbb R^{n+2}_-}\biggl (A\nabla u\cdot\overline{\nabla (\Psi_R w)}-D_t^{1/2}u\overline{H_tD^{1/2}_t{(\Psi_R w)}}\biggr )\, dxdtd\lambda\notag\\
     &&=\int_{\mathbb R^{n+1}}\partial_\nu u\overline{(\Psi_R w)}\, dxdt.
    \end{eqnarray}
     Using \eqref{ibyparts+} and \eqref{aaqkla-} we see, by dominated convergence and by letting $R\to\infty$, that if we can prove that
             \begin{eqnarray}\label{aaqkla+}
    \int_{\mathbb R^{n+2}_-\cap(\tilde Q_{2R}\setminus \tilde Q_R)}\biggl |\biggl (A\nabla u\cdot\overline{\nabla (\Psi_R w)}-D_t^{1/2}u\overline{H_tD^{1/2}_t{(\Psi_R w)}}\biggr )\biggr |\, dxdtd\lambda
    \end{eqnarray}
    tends to 0 as $R\to \infty$, then
              \begin{eqnarray}\label{aaqkla}
     \int_{\mathbb R^{n+2}_-}\biggl (A\nabla u\cdot\overline{\nabla v}-D_t^{1/2}u\overline{H_tD^{1/2}_t{w}}\biggr )\, dxdtd\lambda
     =\int_{\mathbb R^{n+1}}\partial_\nu u\bar w\, dxdt.
    \end{eqnarray}
    By the symmetry of our hypothesis we see that this proves \eqref{ibyparts}. In particular, the proof of the lemma is complete once we have verified that the expression in \eqref{aaqkla+} tends to 0 as $R\to \infty$.
      To estimate the expression in \eqref{aaqkla+} we first note that
                  \begin{eqnarray*}\label{aaqkla+a}
     &&\int_{\mathbb R^{n+2}_-\cap(\tilde Q_{2R}\setminus \tilde Q_R)}\biggl |\biggl (A\nabla u\cdot\overline{\nabla (\Psi_R w)}-D_t^{1/2}u\overline{H_tD^{1/2}_t{(\Psi_R w)}}\biggr )\biggr |\, dxdtd\lambda\notag\\
     &\leq&c\int_{\mathbb R^{n+2}\cap(\tilde Q_{2R}\setminus \tilde Q_R)}\biggl (R^{-1}|\nabla u||w|+|\nabla u||\nabla w|+|\partial_tu||w|\biggr )\, dxdtd\lambda.
    \end{eqnarray*}
    By our choice for $R$ we see that $\mathcal{H}u=0$ and  $\mathcal{H}^\ast w=0$ in $\mathbb R^{n+2}\cap(\tilde Q_{2R}\setminus \tilde Q_R)$. Hence, using this, Lemma \ref{le1--} and Lemma \ref{le1a}, we see that
                  \begin{eqnarray*}\label{aaqkla+a}
     &&\int_{\mathbb R^{n+2}\cap(\tilde Q_{2R}\setminus \tilde Q_R)}\biggl (R^{-1}|\nabla u||w|+|\nabla u||\nabla w|+|\partial_tu||w|\biggr )\, dxdtd\lambda\notag\\
     &&\leq cR^{-2}\biggl (\int_{\mathbb R^{n+2}\cap(\tilde Q_{3R}\setminus \tilde Q_{R/2})}|u|^2\, dxdtd\lambda\biggr )^{1/2}\biggl (\int_{\mathbb R^{n+2}\cap(\tilde Q_{3R}\setminus \tilde Q_{R/2})}|w|^2\, dxdtd\lambda\biggr )^{1/2}.
    \end{eqnarray*}
    Putting these estimates together, and applying Lemma \ref{le2+}, we can conclude that
             \begin{eqnarray}\label{aaqkla+-}
    &&\int_{\mathbb R^{n+2}_-\cap(\tilde Q_{2R}\setminus \tilde Q_R)}\biggl |\biggl (A\nabla u\cdot\overline{\nabla (\Psi_R w)}-D_t^{1/2}u\overline{H_tD^{1/2}_t{(\Psi_R w)}}\biggr )\biggr |\, dxdtd\lambda\notag\\
    &&\leq c_{\psi,\phi} R^{-n-1}\to 0\mbox{ as $R\to\infty$.}
    \end{eqnarray}
    This completes the proof of the lemma.
    \end{proof}

           \begin{lemma}\label{trace5}   Assume \eqref{keysea}.  Then
     $$\mathcal{D}_{\pm\lambda}^{\mathcal{H}}f\to (\mp\frac 1 2I+ \mathcal{K}^{\mathcal{H}})f$$
     non-tangentially and in $L^2(\mathbb R^{n+1},\mathbb C)$ as $\lambda\to 0^+$ and whenever $f\in L^2(\mathbb R^{n+1},\mathbb C)$.
     \end{lemma}
     \begin{proof} Using Lemma \ref{trace4} we have that
     \begin{eqnarray}\label{ntq}\mathcal{D}_{\pm \lambda}^{\mathcal{H}}f\to (\mp\frac 1 2I+ \mathcal{K}^{\mathcal{H}})f
     \end{eqnarray}
     weakly in $L^2(\mathbb R^{n+1},\mathbb C)$ as $\lambda\to 0^+$. Hence, to prove the lemma it suffices to establish the existence of non-tangential limits and to establish establish the existence of the strong $L^2$-limits. We here give the proof only in the case of the
upper half-space, as the proof in the lower half-space is the same. Recall that $\mathcal{D}_{\lambda}^{\mathcal{H}}=-e_{n+1}A\mathcal{S}_\lambda^{\mathcal{H}}\cdot\nabla$. To establish the existence of non-tangential limits we observe that the operator adjoint to $\mathcal{S}_\lambda^{\mathcal{H}}\nabla$ is the operator
$(\nabla \mathcal{S}_\sigma^{\mathcal{H}^\ast})|_{\sigma=-\lambda}$ and that it is enough, by \eqref{asa1} and Lemma \ref{lemsl1++} $(viii)$, to prove the existence of non-tangential limits for $f$ in a dense subset of $L^2(\mathbb R^{n+1},\mathbb C)$. Recall the space
$\mathbb H^{-1}(\mathbb R^{n+1},\mathbb C)$ introduced in \eqref{fspace}. Embedded in \eqref{keysea} is the assumption that  $\mathcal{S}_0^{\mathcal{H}^\ast}:=\mathcal{S}_\lambda^{\mathcal{H}^\ast}|_{\lambda=0}$ is a bijection from $L^2(\mathbb R^{n+1},\mathbb C)$ to
${\mathbb H}(\mathbb R^{n+1},\mathbb C)$. Hence, by duality we have that $\mathcal{S}_0^{\mathcal{H}}:=\mathcal{S}_\lambda^{\mathcal{H}}|_{\lambda=0}$ is a bijection from ${\mathbb H}^{-1}(\mathbb R^{n+1},\mathbb C)$ to $L^2(\mathbb R^{n+1},\mathbb C)$. To proceed we need a better description of the elements in ${\mathbb H}^{-1}(\mathbb R^{n+1},\mathbb C)$ and to get this we consider ${\mathbb H}(\mathbb R^{n+1},\mathbb C)$ equipped with the inner product
$$(u,v):=\int_{\mathbb R^{n+1}}\bigl (\nabla_{||}u\cdot\nabla_{||}\bar v+D^t_{1/2}uD^t_{1/2}\bar v\bigr )\, dxdt.$$
Then ${\mathbb H}(\mathbb R^{n+1},\mathbb C)$ is a Hilbert space and by the Riesz representation theorem we see that
$${\mathbb H}^{-1}(\mathbb R^{n+1},\mathbb C)=\{\div_{||}{\bf g}_{||}+D^t_{1/2}{\bf g}_{n+1}:\ {\bf g}=({\bf g}_{||},{\bf g}_{n+1})\in L^2(\mathbb R^{n+1},\mathbb C^{n+1})\}.$$
Hence, as $C_0^\infty(\mathbb R^{n+1},\mathbb C^{n+1})$ is dense in $L^2(\mathbb R^{n+1},\mathbb C^{n+1})$ we can conclude that
      \begin{eqnarray}\label{eq4-}
L^2(\mathbb R^{n+1},\mathbb C)=\{\mathcal{S}_0^{\mathcal{H}}(\div_{||}{\bf g}_{||}+D^t_{1/2}{\bf g}_{n+1}):\ {\bf g}\in C_0^\infty(\mathbb R^{n+1},\mathbb C^{n+1})\}.
    \end{eqnarray}
    Using this, and given ${\bf g}\in C_0^\infty(\mathbb R^{n+1},\mathbb C^{n+1})$, we consider $$u(\cdot,\cdot,\lambda):=\mathcal{S}_\lambda^{\mathcal{H}}(\div_{||}{\bf g}_{||}+D^t_{1/2}{\bf g}_{n+1})$$ in $\mathbb R^{n+2}_-$ and we let
     $$f=u_0=u(\cdot,\cdot,0).$$ Using Lemma \ref{trace6} we obtain that
$$\mathcal{D}_\lambda^{\mathcal{H}} f=\mathcal{S}_\lambda^{\mathcal{H}}(\partial_\nu u).$$
Moreover, \eqref{asa1}, Lemma \ref{th0} and Lemma \ref{trace3} imply that $\partial_\nu u\in L^2(\mathbb R^{n+1},\mathbb C)$. Hence $\mathcal{S}_\lambda(\partial_\nu u)$ converges non-tangentially as $\lambda\to 0^-$. This prove the non-tangential version of the limit in \eqref{ntq} for $\mathcal{D}_{-\lambda}^{\mathcal{H}}f$ as $\lambda\to 0^+$. To establish the strong $L^2$-limits we first note that \eqref{asa1} implies, in particular, that uniform (in $\lambda$) $L^2$ bounds hold for
$\mathcal{D}_{\lambda}^{\mathcal{H}}$, see Remark \ref{resea}. Thus, again it is enough to establish convergence in a dense class. To this end, choose $f=u_0$ and $u$ as above. It suffices to show that $\mathcal{D}_{\lambda}^{\mathcal{H}}f$ is Cauchy convergent in  $L^2(\mathbb R^{n+1},\mathbb C)$, as $\lambda\to 0$. Suppose that $0<\lambda'<\lambda\to 0$, and observe, by Lemma \ref{trace6}, \eqref{asa1} and by the previous observation that $\partial_\nu u\in L^2(\mathbb R^{n+1},\mathbb C)$, that
 \begin{eqnarray}\label{ntq+}||\mathcal{D}_{\lambda}^{\mathcal{H}}f-\mathcal{D}_{\lambda'}^{\mathcal{H}}f||_2&=&\biggl\|\int_{\lambda'}^\lambda
 \partial_\sigma \mathcal{S}_\sigma^{\mathcal{H}}(\partial_\nu u)\, d\sigma\biggr \|_2\notag\\
 &\leq&(\lambda-\lambda')^{1/2}\bigl (\sup_{\lambda'<\sigma<\lambda}||\partial_\sigma \mathcal{S}_\sigma^{\mathcal{H}}(\partial_\nu u)||_2\bigr ) \to 0,
     \end{eqnarray}
     as $(\lambda-\lambda')\to 0$. This completes the proof of the lemma.\end{proof}

\begin{lemma} \label{tracce} Assume \eqref{keysea}.  Assume also that $\mathcal{H}u=0$ and that
     \begin{eqnarray}\label{claseat}\sup_{\lambda>0}||u(\cdot,\cdot,\lambda)||_2<\infty.
     \end{eqnarray}
     Then $u(\cdot,\cdot,\lambda)$ converges n.t and in $L^2(\mathbb R^{n+1},\mathbb C)$ as $\lambda\to 0^+$.
     \end{lemma}
     \begin{proof} By Lemma \ref{trace5} it is enough to prove that $u(\cdot,\cdot,\lambda)=\mathcal{D}_\lambda^{\mathcal{H}} h$ for some $h\in L^2(\mathbb R^{n+1},\mathbb C)$. Let $f_\epsilon(\cdot,\cdot)=u(\cdot,\cdot,\epsilon)$ and consider
     $$u_\epsilon(x,t,\lambda)=\mathcal{D}_\lambda^{\mathcal{H}} \biggl (\biggl (-\frac 1 2I+\mathcal{K}^{\mathcal{H}}\biggr) ^{-1}f_\epsilon\biggr )(x,t).$$
     Let $U_\epsilon(x,t,\lambda)=u(x,t,\lambda+\epsilon)-u_\epsilon(x,t,\lambda)$. Then $\mathcal{H}U_\epsilon=0$ in $\mathbb R^{n+2}_+$ and
     $$\sup_{\lambda>0}||U_\epsilon(\cdot,\cdot,\lambda)||_2<\infty.$$
     Furthermore, $U_\epsilon(\cdot,\cdot,0)=0$ and $U_\epsilon(\cdot,\cdot,\lambda)\to 0$ n.t in $L^2(\mathbb R^{n+1},\mathbb C)$ by Lemma \ref{trace5}. By uniqueness in the Dirichlet problem, Lemma \ref{trace7} we see that $U_\epsilon(x,t,\lambda)\equiv 0$. Furthermore, using
     \eqref{claseat} we see that $\sup_\epsilon ||f_\epsilon||_2<\infty$. Hence a subsequence of $f_\epsilon$ converges in the weakly in $L^2(\mathbb R^{n+1},\mathbb C)$ to some $f\in L^2(\mathbb R^{n+1},\mathbb C)$. Given an arbitrary $g\in L^2(\mathbb R^{n+1},\mathbb C)$ we let $h=\adj\bigl(-\frac 1 2I+\mathcal{K}^{\mathcal{H}})^{-1}(\mathcal{D}_\lambda^{\mathcal{H}})\bigr )g$ and we observe that
     \begin{eqnarray}\label{claseat+}
     &&\int_{\mathbb R^{n+1}}\biggl (\mathcal{D}_\lambda^{\mathcal{H}}\biggl (-\frac 1 2I+\mathcal{K}^{\mathcal{H}}\biggr)^{-1}f\biggr )\bar g\, dxdt\notag\\
     &=&\int_{\mathbb R^{n+1}}f\bar h\, dxdt=
     \lim_{k\to \infty}\int_{\mathbb R^{n+1}}f_{\epsilon_k}\bar h\, dxdt\notag\\
     &=&\lim_{k\to \infty}\int_{\mathbb R^{n+1}}\biggl (\mathcal{D}_\lambda^{\mathcal{H}}\biggl (-\frac 1 2I+\mathcal{K}^{\mathcal{H}}\biggr)^{-1}f_{\epsilon_k}\biggr )\bar g\, dxdt\notag\\
     &=&\lim_{k\to \infty}\int_{\mathbb R^{n+1}}u(x,t,\lambda+\epsilon_k)\bar g\, dxdt\notag\\
     &=&\int_{\mathbb R^{n+1}}u(x,t,\lambda)\bar g\, dxdt.
     \end{eqnarray}
     As $g$ is arbitrary in this argument we can conclude that $u(\cdot,\cdot,\lambda)=\mathcal{D}_\lambda^{\mathcal{H}} h$ where $$h=\biggl (-\frac 1 2I+\mathcal{K}^{\mathcal{H}}\biggr)^{-1}f\in L^2(\mathbb R^{n+1},\mathbb C).$$ This completes the proof of the lemma.
     \end{proof}

\begin{lemma}\label{trace7-} Assume \eqref{keysea}.  Then
     \begin{eqnarray*}
     (i)&&\mathcal{P}_\lambda(\nabla_{||}\mathcal{S}_{\pm\lambda}^{\mathcal{H}} f)\to \nabla_{||} \mathcal{S}_\lambda^{\mathcal{H}}|_{\lambda=0}f,\notag\\
     (ii)&&\mathcal{P}_\lambda(H_tD^t_{1/2}\mathcal{S}_{\pm\lambda}^{\mathcal{H}} f)\to H_tD^t_{1/2}\mathcal{S}_\lambda^{\mathcal{H}}|_{\lambda=0}f,\notag\\
     (iii)&&\mathcal{P}_\lambda(\partial_\lambda\mathcal{S}_{\pm\lambda}^{\mathcal{H}} f)\to\mp\frac 12\cdot\frac{f(x,t)}{A_{n+1,n+1}(x,t)}e_{n+1}+\mathcal{T}_\perp^{\mathcal{H}}f,
     \end{eqnarray*}
     non-tangentially and in $L^2(\mathbb R^{n+1},\mathbb C)$ as $\lambda\to 0^+$ and whenever $f\in L^2(\mathbb R^{n+1},\mathbb C)$.
     \end{lemma}
     \begin{proof} Again we treat only the case of the upper half space, as the proof in the other case is
the same. Since the weak limits has already been established for $\nabla \mathcal{S}_{\lambda}^{\mathcal{H}}f$ and
$H_tD^t_{1/2}\mathcal{S}_{\lambda}^{\mathcal{H}}f$, see Lemma \ref{trace4}, it is easy to verify that the strong and non-tangential limits for $\mathcal{P}_\lambda(\nabla \mathcal{S}_{\lambda}^{\mathcal{H}}f)$  and
$\mathcal{P}_\lambda(H_tD^t_{1/2}\mathcal{S}_{\lambda}^{\mathcal{H}}f)$ will take the same
value, once the existence of those limits has been established. Hence, in the following we prove the existence of these limits  as $\lambda\to 0^+$. Furthermore, using Lemma \ref{lemsl1++}, \eqref{keysea}, and the dominated convergence theorem, we see that it is
enough to establish non-tangential convergence. Using \eqref{keysea} we see that the non-tangential convergence of $\partial_\lambda\mathcal{S}_{\lambda}^{\mathcal{H}}f$  follows immediately
Lemma \ref{tracce} and a simple real variable argument yields the same conclusion for $\mathcal{P}_\lambda(\partial_\lambda\mathcal{S}_{\lambda}^{\mathcal{H}}f)$. The latter proves $(iii)$ and hence we only have to prove $(i)$ and $(ii)$.

To prove $(i)$ we fix $(x_0,t_0)\in\mathbb R^{n+1}$, we consider $(x,t,\lambda)\in \Gamma(x_0,t_0)$ and we let $k\in \{1,...,n\}$. Then
\begin{eqnarray}
\P_\lambda(\partial_{x_k}\mathcal{S}^{\mathcal{H}}_\lambda f)(x,t)&=&\partial_{x_k}\P_\lambda\bigl(\int_0^\lambda\partial_\sigma\mathcal{S}^{\mathcal{H}}_\sigma f\, d\sigma\bigr )(x,t)+\P_\lambda(\partial_{x_k}\mathcal{S}^{\mathcal{H}}_0 f)(x,t)\notag\\
&=&\mathcal{Q}_\lambda\bigl(\lambda^{-1}\int_0^\lambda\partial_\sigma\mathcal{S}^{\mathcal{H}}_\sigma f\, d\sigma\bigr )(x,t)+\P_\lambda(\partial_{x_k}\mathcal{S}^{\mathcal{H}}_0 f)(x,t),
\end{eqnarray}
 where again $\mathcal{Q}_\lambda$ is a standard approximation of the zero operator. As $\P_\lambda$ is an  approximation of the identity we see that
$$\P_\lambda(\partial_{x_k}\mathcal{S}^{\mathcal{H}}_0 f)(x,t)\to (\partial_{x_k}\mathcal{S}^{\mathcal{H}}_0 f)(x_0,t_0)$$
n.t as $\lambda\to 0$. In the following we let $Vf(x_0,t_0)$ denote the non-tangential limit $\partial_\lambda\mathcal{S}_{\lambda}^{\mathcal{H}}f(x,t)$ as $(x,t,\lambda)\to (x_0,t_0,0)$ non-tangentially. Using this notation we see that
\begin{eqnarray}
\mathcal{Q}_\lambda\bigl(\lambda^{-1}\int_0^\lambda\partial_\sigma\mathcal{S}^{\mathcal{H}}_\sigma f\, d\sigma\bigr )(x,t)&=&
\mathcal{Q}_\lambda\bigl(\lambda^{-1}\int_0^\lambda(\partial_\sigma\mathcal{S}^{\mathcal{H}}_\sigma f-Vf)\, d\sigma\bigr )(x,t)\notag\\
&&+\mathcal{Q}_\lambda(Vf-Vf(x_0,t_0))(x,t)\notag\\
&=:&I_1+I_2.
\end{eqnarray}
As $Vf\in L^2(\mathbb R^{n+1},\mathbb C)$ it follows, if $(x_0,t_0)$ is a Lebesgue point for $Vf$, that $I_2\to 0$ as $\lambda\to 0$. Furthermore, using Lemma \ref{trace2} we see that
\begin{eqnarray}\label{limq}
\quad\quad\biggl |\mathcal{Q}_\lambda\bigl(\lambda^{-1}\int_0^\lambda(\mathcal{S}^{\mathcal{H}}_\sigma f-\mathcal{S}^{\mathcal{H}}_0 f)\, d\sigma\bigr )(x,t)\biggr |\leq c\lambda M(\tilde N_\ast(\nabla \mathcal{S}^{\mathcal{H}}_\lambda f))(x_0,t_0)\to 0
\end{eqnarray}
as $\lambda\to 0$ and for a.e. $(x_0,t_0)\in\mathbb R^{n+1}$. Similarly,  if ${\bf f}\in C_0^\infty (\mathbb R^{n+1},\mathbb C^n)$ then
\begin{eqnarray}\label{limqq-}
\biggl |\mathcal{Q}_\lambda\bigl(\lambda^{-1}\int_0^\lambda((\mathcal{S}^{\mathcal{H}}_\sigma\nabla_{||})\cdot {\bf f}-(\mathcal{S}^{\mathcal{H}}_0 \nabla_{||})\cdot {\bf f})\, d\sigma\bigr )(x,t)\biggr |\to 0
\end{eqnarray}
n.t as $\lambda\to 0$. By Lemma \ref{lemsl1++} $(vii)$, the density of $C_0^\infty(\mathbb R^{n+1},\mathbb C^n)$ in $L^2(\mathbb R^{n+1},\mathbb C^n)$, and the fact that $\mathcal{Q}_\lambda$ is dominated by the Hardy-
Littlewood maximal operator, which is bounded from $L^{2,\infty}$ to $L^{2,\infty}$, the latter convergence
continues to hold for ${\bf f}\in L^2(\mathbb R^{n+1},\mathbb C^n)$.  Moreover, if $u_0$ belongs to the dense class
      \begin{eqnarray*}\label{eq4-ja}
\{\mathcal{S}_0^{\mathcal{H}}(\div_{||}{\bf g}_{||}+D^t_{1/2}{\bf g}_{n+1}):\ {\bf g}=({\bf g}_{||}, {\bf g}_{n+1})\in C_0^\infty(\mathbb R^{n+1},\mathbb C^{n+1})\},
    \end{eqnarray*}
    see \eqref{eq4-}, then using Lemma \ref{tracce} and \eqref{limq} we see that
    \begin{eqnarray}\label{limqq}
\biggl |\mathcal{Q}_\lambda\bigl(\lambda^{-1}\int_0^\lambda(\mathcal{D}_\sigma u_0-g)\, d\sigma\bigr )(x,t)\biggr |\to 0
\end{eqnarray}
n.t as $\lambda\to 0$ and where $g$ is the boundary trace according to Lemma \ref{tracce}. Again this conclusion remain true whenever $u_0\in
L^2(\mathbb R^{n+1},\mathbb C)$ by Lemma \ref{lemsl1++} $(viii)$, the density of $C_0^\infty(\mathbb R^{n+1},\mathbb C^n)$ in $L^2(\mathbb R^{n+1},\mathbb C^n)$, and the fact that $\mathcal{Q}_\lambda$ is dominated by the Hardy-
Littlewood maximal operator, which is bounded from $L^{2,\infty}$ to $L^{2,\infty}$. Combining \eqref{limqq-} and \eqref{limqq}
with the adjoint version of the identity \eqref{assumphq}, we obtain convergence to 0 for the term $I_1$ as every
$f\in L^2(\mathbb R^{n+1},\mathbb C)$ can be written in the form $f=A^\ast_{n+1,n+1}h$, $h\in L^2(\mathbb R^{n+1},\mathbb C)$. This completes the proof of $(i)$.

To prove $(ii)$ we again fix $(x_0,t_0)\in\mathbb R^{n+1}$ and we consider $(x,t,\lambda)\in \Gamma(x_0,t_0)$. Given $(x,t,\lambda)$ we let $(y,s)\in\mathbb R^{n+1}$ be such that $\P_\lambda(x-y,t-s)\neq 0$. Then $||(y-x_0,s-t_0)||<8\lambda$. To complete the proof we will perform a decomposition of $H_tD_{1/2}^t(\mathcal{S}_\lambda f)(y,s)$ similar to the one in  the proof of Lemma \ref{lemsl1++} $(vi)$ and we let $K\gg 1$ be a degree of freedom to be chosen. Then
\begin{eqnarray*}
H_tD_{1/2}^t(\mathcal{S}_\lambda f)(y,s)&=&\lim_{\epsilon\to 0}\int_{\epsilon\leq |s-\tilde t|<1/\epsilon}\frac {\mbox{sgn}(s-\tilde t)}{|s-\tilde t|^{3/2}}(\mathcal{S}_\lambda f)(y,\tilde t)\, d\tilde t\notag\\
&=&\lim_{\epsilon\to 0}\int_{\epsilon\leq |s-\tilde t|<(K\lambda)^2}\frac {\mbox{sgn}(s-\tilde t)}{|s-\tilde t|^{3/2}}(\mathcal{S}_\lambda f)(y,\tilde t)\, d\tilde t\notag\\
&&+\lim_{\epsilon\to 0}\int_{(K\lambda)^2\leq |s-\tilde t|<1/\epsilon}\frac {\mbox{sgn}(s-\tilde t)}{|s-\tilde t|^{3/2}}(\mathcal{S}_\lambda f)(y,\tilde t)\, d\tilde t\notag\\
&=:&g_1(y,s,\lambda)+g_2(y,s,\lambda).
\end{eqnarray*}
We claim that
\begin{eqnarray}\label{paris}
\P_\lambda(g_1(\cdot,\cdot,\lambda))(x,t)\to 0\mbox{ as $\lambda\to 0$}.
\end{eqnarray}
To prove this we first note that
\begin{eqnarray}\label{paris1}
\P_\lambda(|g_1(\cdot,\cdot,\lambda)|)(x,t)\leq cK\lambda g_3(x_0,t_0,\lambda)
\end{eqnarray}
where
$$g_3(x_0,t_0,\lambda):=\sup_{\{z:\ |z-x_0|\leq 8\lambda\}}\sup_{\{w:\ |w-t_0|\leq (4K\lambda)^2\}}|\partial_\tau(\mathcal{S}_\lambda f)(z,w)|.$$
Furthermore, given $(z,w)$ as in the definition of $g_3(x_0,t_0,\lambda)$ we see, using \eqref{eq14+} and Lemma \ref{le1a}, that
\begin{eqnarray*}
\lambda^2|\partial_\tau(\mathcal{S}_\lambda f)(z,w)|^2\leq c\mean{W_\lambda(z,w)}|\nabla \mathcal{S}_\sigma f(\tilde z,\tilde w)-\P_\lambda(\nabla \mathcal{S}_\lambda f)(z,w)|^2\, d\tilde zd\tilde wd\sigma.
\end{eqnarray*}
Using this, \eqref{keysea}, and arguing as in the proof of $(i)$ and $(iii)$, we can then conclude that \eqref{paris} holds. To proceed we introduce, similar to the proof of Lemma \ref{lemsl1++} $(vi)$,
\begin{eqnarray*}
g_4(\bar y,\bar s,\lambda)=\lim_{\epsilon\to 0}\int_{(K\lambda)^2\leq |\tilde s-\bar s|<1/\epsilon}\frac {\mbox{sgn}
(\bar s-\tilde s)}{|\bar s-\tilde  s|^{3/2}}({\mathcal{S}}_0 f)(\bar y,\tilde  s)\, d\tilde  s.
\end{eqnarray*}
Then, see the proof of Lemma \ref{lemsl1++} $(vi)$,
\begin{eqnarray*}
|g_2(y,s,\lambda)-g_4(y,s,\lambda)|&\leq &  cK^{-1} M^t(N_\ast(\partial_\lambda {\mathcal{S}}_\lambda f)(y,\cdot))(s),
\end{eqnarray*}
where $M^t$ is the Hardy-Littlewood maximal function in $t$ only, and where we have emphasized the presence of the degree of freedom $K$. In particular,
\begin{eqnarray*}
|\P_\lambda(g_2(\cdot,\cdot,\lambda))(x,t)-\P_\lambda(g_4(\cdot,\cdot,\lambda))(x,t)|\leq   cK^{-1} M(M^t(N_\ast(\partial_\lambda {\mathcal{S}}_\lambda f)(y,\cdot))(\cdot))(x_0,t_0).
\end{eqnarray*}
Using \eqref{keysea} and Lemma \ref{lemsl1++} $(i)$ we see that the right hand side in the above display is finite a.e. Hence
\begin{eqnarray}\label{paris2}
&&\limsup_{(x,t,\lambda)\to (x_0,t_0,0)}|\P_\lambda(g_2(\cdot,\cdot,\lambda))(x,t)-\P_\lambda(g_4(\cdot,\cdot,\lambda))(x,t)|\notag\\
&\leq &  cK^{-1} M(M^t(N_\ast(\partial_\lambda {\mathcal{S}}_\lambda f)(y,\cdot))(\cdot))(x_0,t_0)<\infty.
\end{eqnarray}
However, using Lemma 2.27 in \cite{HL} it follows that
\begin{eqnarray*}
\limsup_{(x,t,\lambda)\to (x_0,t_0,0)}\P_\lambda(g_4(\cdot,\cdot,\lambda))(x,t)=H_tD_{1/2}^t{\mathcal{S}}_0 f(x_0,t_0).
\end{eqnarray*}
Hence, letting $K\to\infty$ in \eqref{paris2} we can conclude the validity of $(ii)$.\end{proof}

\section{Square function estimates for composed operators}\label{sec8}

As in the statement of Theorem \ref{thper2} we here consider two operators  $\mathcal{H}_0=\partial_t-\div A^0\nabla$, $\mathcal{H}_1=\partial_t-\div A^1\nabla$. Throughout the section we will assume that
\begin{eqnarray}\label{keyseauu}
&&\mbox{$\mathcal{H}_0$,  $\mathcal{H}_0^\ast$, $\mathcal{H}_1$,  $\mathcal{H}_1^\ast$, satisfy \eqref{eq3}-\eqref{eq4} as well as \eqref{eq14+}-\eqref{eq14++}, and that}\notag\\
&&\mbox{$\mathcal{H}_0$,  $\mathcal{H}_0^\ast$ have bounded, invertible and good layer potentials in the}\notag\\
&&\mbox{sense of Definition \ref{blayer+}, for some constant $\Gamma_0$.}
 \end{eqnarray}
 Note that \eqref{keyseauu} implies, in particular, that \eqref{asa1} holds for $\mathcal{H}_0$,  $\mathcal{H}_0^\ast$. In the following we let
     \begin{eqnarray}\label{keyestaahmo}{\bf \varepsilon}(x):=A^1(x)-A^0(x).
      \end{eqnarray}
     Then ${\bf \varepsilon}$ is a (complex) matrix valued function and we assume that
        \begin{eqnarray}\label{keyestaahmeddo}||\varepsilon||_\infty\leq\epsilon\leq \varepsilon_0.
      \end{eqnarray}
      Furthermore, we write
     ${\bf \varepsilon}=({\bf \varepsilon}_1,...,{\bf \varepsilon}_{n+1})$ where ${\bf \varepsilon}_i$, for $i\in \{1,...,n+1\}$, is a $(n+1)$-dimensional column vector. In the following we let $\tilde{\bf \varepsilon}$ be the $(n+1)\times n$ matrix defined to equal the first $n$ columns of ${\bf \varepsilon}$, i.e.,
         \begin{eqnarray}\label{keyestaahmouu}\tilde{\bf \varepsilon}=({\bf \varepsilon}_1,...,{\bf \varepsilon}_{n}).
      \end{eqnarray}
     \begin{lemma}\label{ilem1} Assume \eqref{keyseauu}. Let
            \begin{eqnarray}\label{keyestaah}
          \theta_\lambda{\bf f}:=\lambda^2\partial_\lambda^2 (\mathcal{S}_{\lambda}^{\mathcal{H}_0}\nabla)\cdot{\bf f},
     \end{eqnarray}
     whenever ${\bf f}\in L^2(\mathbb R^{n+1},\mathbb C^{n+1})$. Then Lemma \ref{le8} is applicable to the operator $\theta_\lambda$. In particular, ${\theta}_\lambda$ is a linear operator satisfying  \eqref{stand-} and \eqref{stand}, for some $d\geq 0$, and
\begin{eqnarray}\label{square}\int_0^\infty\int_{\mathbb R^{n+1}} |\theta_\lambda {\bf f}(x,t)|^2\, \frac {dxdtd\lambda}\lambda\leq \hat\Gamma||{\bf f}||_2^2,
\end{eqnarray}
whenever ${\bf f}\in L^2(\mathbb R^{n+1},\mathbb C^{n+1})$, for some constant $\hat\Gamma\geq 1$, depending  at most
     on $n$, $\Lambda$,  the De Giorgi-Moser-Nash constants and $\Gamma_0$.
     \end{lemma}
     \begin{proof} Recall that the estimate
               \begin{eqnarray}\label{keyestinted}
\sup_{\lambda> 0}||\partial_\lambda \mathcal{S}_\lambda^{\mathcal{H}_0} f||_2+|||\lambda \partial_\lambda^2\mathcal{S}_{\lambda}^{\mathcal{H}_0}f|||\leq \Gamma_0||f||_2,
     \end{eqnarray}
    for $f\in L^2(\mathbb R^{n+1},\mathbb C)$, is embedded in \eqref{keyseauu}. In the following we write, for simplicity,  $\mathcal{S}_{\lambda}^0:= \mathcal{S}_{\lambda}^{\mathcal{H}_0}$. Note that
                 \begin{eqnarray*}
          \theta_\lambda{\bf f}=\lambda^2\partial_\lambda^2 (\mathcal{S}_{\lambda}^0\nabla_{||})\cdot{\bf f}_{||}-\lambda^2\partial_\lambda^3 (\mathcal{S}_{\lambda}^0{\bf f}_{n+1}),
     \end{eqnarray*}
     where ${\bf f}=({\bf f}_{||},{\bf f}_{n+1})$. That ${\theta}_\lambda$ satisfies  \eqref{stand-} follows from Lemma \ref{le5} $(i)$ and $(iii)$. That ${\theta}_\lambda$ satisfies  \eqref{stand} follow from  Lemma \ref{le4}.  Hence, we only have to prove \eqref{square}. To start the proof of
      \eqref{square} we have
         \begin{eqnarray}\label{keyestaa+}
    \int_{0}^\infty\int_{\mathbb R^{n+1}}|\theta_\lambda{\bf f}|^2\,  \frac{dxdtd\lambda}\lambda&\leq&\int_{0}^\infty\int_{\mathbb R^{n+1}}|\partial_\lambda^3 \mathcal{S}_{\lambda}^0{\bf f}_{n+1}|^2\,  \lambda^3{dxdtd\lambda}\notag\\
    &&+\int_{0}^\infty\int_{\mathbb R^{n+1}}|\partial_\lambda^2 (\mathcal{S}_{\lambda}^0\nabla_{||})\cdot{\bf f_{||}}|^2\,  \lambda^3{dxdtd\lambda}\notag\\
    &&=:I+II,
     \end{eqnarray}
     and we note that  $I\leq c||{\bf f}||_2^2$ by Lemma \ref{le1--} and \eqref{keyestinted}. To estimate $II$ we first note, using Lemma \ref{le1--} and the ellipticity of $ A_{||}$,  that to bound $II$ it suffices to bound
           \begin{eqnarray}\label{keyestaa+a}
    \widetilde {II}:=\int_{0}^\infty\int_{\mathbb R^{n+1}}|\partial_\lambda (\mathcal{S}_{\lambda}^0\nabla_{||})\cdot A_{||}{\bf f_{||}}|^2\,  \lambda{dxdtd\lambda}.
     \end{eqnarray}
     Using Lemma \ref{parahodge} we see  that there exists $u\in\mathbb H(\mathbb R^{n+1},\mathbb C)$ such that $-\div_{||}(A_{||}{\bf f_{||}})=\mathcal{H}_{||}u$ and such that
     \begin{eqnarray}\label{ua}||u||_{\mathbb H(\mathbb R^{n+1},\mathbb C)}\leq c||{\bf f}||_2.
     \end{eqnarray}
     Using this we see that
                \begin{eqnarray}\label{keyestaa+aa}
  \partial_\lambda (\mathcal{S}_{\lambda}^0\nabla_{||})\cdot A_{||}{\bf f_{||}}=\partial_\lambda (\mathcal{S}_{\lambda}^0(\mathcal{H}_{||}u)).
     \end{eqnarray}
     Using this, \eqref{eq4} and that, for $(x,t,\lambda)$ fixed,
     \begin{eqnarray}\label{keyestaa+aaa}
                 \mathcal{H}_{||}^\ast\Gamma(x,t,\lambda,y,s,\sigma)
                 &=&\sum_{i=1}^n\partial_{y_i}\bigl(\overline{A_{i,n+1}^\ast(y)}\partial_\sigma\Gamma(x,t,\lambda,y,s,\sigma)\bigr)\notag\\
                 &&+\sum_{j=1}^{n+1}\overline{A_{n+1,j}^\ast(y)}\partial_{y_i}\partial_\sigma\Gamma(x,t,\lambda,y,s,\sigma),
     \end{eqnarray}
     we see that
                \begin{eqnarray}\label{keyestaa+aab}
  \partial_\lambda (\mathcal{S}_{\lambda}^0\nabla_{||})\cdot A_{||}{\bf f_{||}}=\sum_{i=1}^n\partial_\lambda^2 \mathcal{S}_{\lambda}^0(A_{n+1,i}D_iu)+
  \partial_\lambda^2 (\mathcal{S}_{\lambda}^0\overline{\partial_{\nu^\ast}u}),\end{eqnarray}
  where $\overline{\partial_{\nu^\ast}}=-\sum_{i=1}^{n+1}\overline{A_{n+1,i}^\ast}D_i=-\sum_{i=1}^{n+1}{A_{n+1,i}}D_i$. Hence,
          \begin{eqnarray}\label{keyestaa+ak}
    \widetilde {II}&\leq &\sum_{i=1}^n\int_{0}^\infty\int_{\mathbb R^{n+1}}|\partial_\lambda^2 \mathcal{S}_{\lambda}^0(A_{n+1,i}D_iu)|^2\,  \lambda{dxdtd\lambda}\notag\\
    &&+\int_{0}^\infty\int_{\mathbb R^{n+1}}|\partial_\lambda^2 \mathcal{S}_{\lambda}^0(\overline{\partial_{\nu^\ast}u})|^2\,  \lambda{dxdtd\lambda}\notag\\
    &:=&\widetilde {II}_1+\widetilde {II}_2.
     \end{eqnarray}
          Again using \eqref{keyestinted}, and \eqref{ua},  we see that $\widetilde {II}_1\leq c||{\bf f}||_2^2$. Furthermore, as $u\in\mathbb H(\mathbb R^{n+1},\mathbb C)$ and as, by assumption,
           $\mathcal{S}_0:=\mathcal{S}_\lambda|_{\lambda=0}:L^2(\mathbb R^{n+1},\mathbb C)\to {\mathbb H}(\mathbb R^{n+1},\mathbb C)$ is invertible, we can conclude that there exists $v\in L^2(\mathbb R^{n+1},\mathbb C)$ such that
     $u=\mathcal{S}_0v$. We now let $v(\cdot,\cdot, \sigma)=\mathcal{S}_\sigma v(\cdot,\cdot)$ for $\sigma<0$ so that $v(\cdot,\cdot,0)=u(\cdot,\cdot)$. Then
     $$||v||_{\mathbb H(\mathbb R^{n+1},\mathbb C)}\leq c||u||_{\mathbb H(\mathbb R^{n+1},\mathbb C)}\leq c||{\bf f}||_2,$$
     by Theorem \ref{th0} and \eqref{ua}. Furthermore, as $(\mathcal{S}_{\lambda}^0\overline{\partial_{\nu^\ast}})=\mathcal{D}_\lambda$, Lemma \ref{trace6} implies that
     $$\partial_\lambda^2 (\mathcal{S}_{\lambda}^0\overline{\partial_{\nu^\ast}u})=\partial_\lambda^2 \mathcal{S}_{\lambda}^0(\partial_{\nu}v(\cdot,\cdot,0)).$$ Hence, using \eqref{keyestinted} once more we see that
\begin{eqnarray}\label{keyestaa+ak+}
\widetilde {II}_2&\leq& \int_{0}^\infty\int_{\mathbb R^{n+1}}|\partial_\lambda^2 \mathcal{S}_{\lambda}^0(\partial_{\nu}v(\cdot,\cdot,0))|^2\,  \lambda{dxdtd\lambda}\notag\\
&\leq&||\nabla v(\cdot,\cdot,0)||_2^2\leq c||u||_{\mathbb H(\mathbb R^{n+1},\mathbb C)}^2\leq c||{\bf f}||_2^2.
     \end{eqnarray}
     This completes the proof of \eqref{square} and the lemma.
     \end{proof}

     \begin{lemma} \label{ilem2--} Assume \eqref{keyseauu}. Let $\theta_\lambda$ be as in the Lemma \ref{ilem1}, let ${\bf \varepsilon}$, $\tilde{\bf \varepsilon}$ be as in \eqref{keyestaahmo}, \eqref{keyestaahmouu}. Let $\mathcal{E}_\lambda^1:=(I+\lambda^2(\mathcal{H}_1)_{||})^{-1}=(I+\lambda^2(\partial_t+(\mathcal{L}_1)_{||}))^{-1}$. Let $A_{||}^1=(A_{1,||}^1,...,A_{n,||}^1)$ where $A_{i,||}^1\in \mathbb C^n$ for all $i\in \{1,...,n\}$. Let
          \begin{eqnarray}\label{Udef}
\mathcal{U}_\lambda =\theta_\lambda \tilde{\bf \varepsilon}\nabla_{||}\mathcal{E}_\lambda^1\lambda^2\div_{||}.
\end{eqnarray}
and consider $\mathcal{U}_\lambda A_{||}^1:=(\mathcal{U}_\lambda A_{1,||}^1,...,\mathcal{U}_\lambda A_{n,||}^1)$.
Then
\begin{eqnarray}\label{crucacar+}\int_0^{l(Q)}\int_Q|\mathcal{U}_\lambda A_{||}^1|^2\frac {dxdtd\lambda}\lambda\leq c\varepsilon_0|Q|,
\end{eqnarray}
for all cubes $Q\subset\mathbb R^{n+1}$ and for some constant $c$ depending  at most
     on $n$, $\Lambda$,  the De Giorgi-Moser-Nash constants and $\Gamma_0$.
     \end{lemma}
     \begin{proof} Using Lemma \ref{ilem2--+} applied to $\gamma_\lambda=\mathcal{U}_\lambda A_{||}^1$ we see that to prove Lemma \ref{ilem2--} it suffices to prove that
\begin{eqnarray}\label{ff1}
\int _{0}^{l(Q)}\int_{Q}|(\mathcal{U}_\lambda A_{||}^1)\cdot\mathcal{A}_\lambda^Q\nabla_{||} f^{\epsilon}_{Q,w}|^2
\frac {dxdtd\lambda}\lambda\leq c|Q|
\end{eqnarray}
for all $Q\subset\mathbb R^{n+1}$ and for a constant $c$ depending only on $n$, $\Lambda$. In the following we will simply, with a slight abuse of notation but consistently, drop the $\cdot$ in \eqref{ff1}. We write
\begin{eqnarray*}
(\mathcal{U}_\lambda A_{||}^1)\mathcal{A}_\lambda^Q\nabla_{||} f^{\epsilon}_{Q,w}=\mathcal{R}_\lambda^{(1)}\nabla_{||} f^{\epsilon}_{Q,w}+\mathcal{R}_\lambda^{(2)}\nabla_{||} f^{\epsilon}_{Q,w}+\mathcal{U}_\lambda A_{||}^1\nabla f^{\epsilon}_{Q,w},
\end{eqnarray*}
where
\begin{eqnarray*}
\mathcal{R}_\lambda^{(1)}\nabla_{||} f^{\epsilon}_{Q,w}&=&(\mathcal{U}_\lambda A_{||}^1)(\mathcal{A}_\lambda^Q-\mathcal{A}_\lambda^Q \P_\lambda)\nabla_{||} f^{\epsilon}_{Q,w},\notag\\
\mathcal{R}_\lambda^{(2)}\nabla_{||} f^{\epsilon}_{Q,w}&=&((\mathcal{U}_\lambda A_{||}^1)\mathcal{A}_\lambda^Q \P_\lambda-\mathcal{U}_\lambda A_{||}^1)\nabla_{||} f^{\epsilon}_{Q,w},
\end{eqnarray*}
and where $\P_\lambda$ is a standard parabolic approximation of the identity.  We first note
that
\begin{eqnarray}
\mathcal{U}_\lambda A_{||}^1\nabla_{||}f^{\epsilon}_{Q,w}=\theta_\lambda \tilde{\bf \varepsilon} \lambda^2 \nabla_{||}\mathcal{E}_\lambda^1
(\mathcal{L}_1)_{||}f^{\epsilon}_{Q,w}.
\end{eqnarray}
Hence,  using $L^2$ boundness of $\theta_\lambda$, see Lemma \ref{ilem1}, and the $L^2$-boundedness of $\lambda\nabla_{||}\mathcal{E}_\lambda^1$, see Lemma \ref{le8-}, we see that
\begin{eqnarray*}
\int _{0}^{l(Q)}\int_{Q}|\mathcal{U}_\lambda A_{||}^1\nabla_{||} f^{\epsilon}_{Q,w}(x,t)|^2\frac {dxdtd\lambda}\lambda &\leq& c\varepsilon_0 l(Q)^2\int_{\mathbb R^{n+1}}|(\mathcal{L}_1)_{||}f^{\epsilon}_{Q,w}|\, dxdt\notag\\
&\leq& c\varepsilon_0 |Q|,
\end{eqnarray*}
where we in the last step have used Lemma \ref{ilem2--+} $(ii)$. Note that
\begin{eqnarray*}
\mathcal{R}_\lambda^{(1)}&=&(\mathcal{U}_\lambda A_{||}^1)(\mathcal{A}_\lambda^Q-\mathcal{A}_\lambda^Q \P_\lambda)=(\mathcal{U}_\lambda A_{||}^1)\mathcal{A}_\lambda^Q(\mathcal{A}_\lambda^Q-\P_\lambda).
\end{eqnarray*}
We want to apply Lemma \ref{le11+} with $\Theta_\lambda$ replaced by $ \mathcal{U}_\lambda$. $\theta_\lambda$ satisfies  \eqref{stand-} and \eqref{stand}, see Lemma \ref{ilem1}, and $\lambda^2\nabla_{||}\mathcal{E}_\lambda^1\div_{||}$ satisfies \eqref{stand-}, see Lemma \ref{le8-}. Furthermore, using Lemma \ref{le8-+} we see that the latter operator also satisfies assumption \eqref{stand+} in Lemma \ref{le9}. Hence, applying
 Lemma \ref{le9} we can first conclude that \eqref{stand-} and \eqref{stand} hold with $\tilde\Theta_\lambda$ replaced by $ \mathcal{U}_\lambda$, and hence that
 Lemma \ref{le11+} is applicable to $ \mathcal{U}_\lambda$. Using Lemma \ref{le11+}  we see that
$$||(\mathcal{U}_\lambda A_{||}^1)\mathcal{A}_\lambda^Q||_{2\to 2}\leq c\varepsilon_0.$$
Thus
\begin{eqnarray*}
\int _{0}^{l(Q)}\int_{Q}| \mathcal{R}_\lambda^{(1)}\nabla_{||} f^{\epsilon}_{Q,w}(x,t)|^2\, \frac {dxdtd\lambda}\lambda
 &\leq& c\varepsilon_0 \int_{\mathbb R^{n+2}_+}|(\mathcal{A}_\lambda^Q-\P_\lambda)\nabla_{||} f^{\epsilon}_{Q,w}(x,t)|^2\, \frac {dxdtd\lambda}\lambda\notag\\
&\leq &c\varepsilon_0 \int_{\mathbb R^{n+1}}|\nabla_{||} f^{\epsilon}_{Q,w}(x,t)|^2\, {dxdt}\leq c\varepsilon_0|Q|,
\end{eqnarray*}
where we have used the $L^2$-boundedness of the operator
$$g\to \biggl (\int_0^\infty|(\mathcal{A}_\lambda^Q-\P_\lambda)g|^2\frac {d\lambda}\lambda\biggr )^{1/2},$$
see Lemma \ref{little3}, and  Lemma \ref{ilem2--+} $(i)$.  Left to estimate is
\begin{eqnarray}
\int _{0}^{l(Q)}\int_{Q}| R_\lambda^{(2)}\nabla_{||}f^{\epsilon}_{Q,w}(x,t)|^2\, \frac {dxdtd\lambda}\lambda.
\end{eqnarray}
 Arguing as above we see that Lemma \ref{le11-} applies to the operator $\mathcal{R}_\lambda^{(2)}$. To explore this we write
$$\mathcal{R}_\lambda^{(2)}\nabla f^{\epsilon}_{Q,w}=\mathcal{R}_\lambda^{(2)}(I-\P_\lambda)\nabla f^{\epsilon}_{Q,w}+
\mathcal{R}_\lambda^{(2)}\P_\lambda\nabla f^{\epsilon}_{Q,w}.$$
Now, using Lemma \ref{le11-} we see that
\begin{eqnarray*}
&&\int _{0}^{l(Q)}\int_{Q}| \mathcal{R}_\lambda^{(2)}\P_\lambda\nabla_{||} f^{\epsilon}_{Q,w}(x,t)|^2\, \frac {dxdtd\lambda}\lambda\notag\\
&\leq&c \varepsilon_0\int _{0}^{l(Q)}\int_{Q}| \nabla \P_\lambda\nabla_{||} f^{\epsilon}_{Q,w}(x,t)|^2\, \lambda{dxdtd\lambda}\notag\\
&+&c\varepsilon_0\int _{0}^{l(Q)}\int_{Q}| \partial_t\P_\lambda\nabla_{||} f^{\epsilon}_{Q,w}(x,t)|^2\, \lambda^3{dxdtd\lambda}.
\end{eqnarray*}
In particular, by Littlewood Paley theory, see Lemma \ref{little2}, we can conclude that
\begin{eqnarray*}
\int _{0}^{l(Q)}\int_{Q}| \mathcal{R}_\lambda^{(2)}\P_\lambda\nabla_{||} f^{\epsilon}_{Q,w}(x,t)|^2\, \frac {dxdtd\lambda}\lambda&\leq&
c\varepsilon_0\int_{\mathbb R^{n+1}}|\nabla_{||} f^{\epsilon}_{Q,w}(x,t)|^2\, dxdt\notag\\
&\leq& c\varepsilon_0|Q|,
\end{eqnarray*}
where we again also have used Lemma \ref{ilem2--+} $(i)$. To continue we decompose
\begin{eqnarray*}
\mathcal{R}_\lambda^{(2)}(I-\P_\lambda)\nabla_{||} f^{\epsilon}_{Q,w}=(\mathcal{U}_\lambda A)\mathcal{A}_\lambda^Q \mathcal{Q}_\lambda \nabla_{||} f^{\epsilon}_{Q,w}-\mathcal{U}_\lambda A_{||}^1\nabla_{||} (I-\P_\lambda) f^{\epsilon}_{Q,w},
\end{eqnarray*}
where $ \mathcal{Q}_\lambda=\P_\lambda(I-\P_\lambda)$. Then, again using Lemma \ref{le11+}, standard  Littlewood Paley theory, see Lemma \ref{little2}, and
Lemma \ref{ilem2--+} $(i)$ we see that
\begin{eqnarray*}
\int _{0}^{l(Q)}\int_{Q}| (\mathcal{U}_\lambda A_{||}^1)\mathcal{A}_\lambda^Q \mathcal{Q}_\lambda \nabla_{||} f^{\epsilon}_{Q,w}(x,t)|^2\, \frac {dxdtd\lambda}\lambda&\leq&
c\varepsilon_0\int_{\mathbb R^{n+1}}|\nabla_{||} f^{\epsilon}_{Q,w}(x,t)|^2\, dxdt\notag\\
&\leq& c\varepsilon_0|Q|.
\end{eqnarray*}
Furthermore,
\begin{eqnarray}
\mathcal{U}_\lambda A_{||}^1\nabla_{||}(I-\P_\lambda) f^{\epsilon}_{Q,w}&=&
\theta_\lambda \tilde{\bf \varepsilon} \lambda^2 \nabla_{||} \mathcal{E}_\lambda^1\div_{||}(A_{||}^1\nabla_{||}(I-\P_\lambda) f^{\epsilon}_{Q,w})\notag\\
&=&\theta_\lambda \tilde{\bf \varepsilon} \lambda^2 \nabla_{||} \mathcal{E}_\lambda^1(\partial_t+(\mathcal{L}_1)_{||})(I-\P_\lambda) f^{\epsilon}_{Q,w}\notag\\
&&-\theta_\lambda \tilde{\bf \varepsilon} \lambda^2 \nabla_{||} \mathcal{E}_\lambda^1\partial_t(I-\P_\lambda) f^{\epsilon}_{Q,w}.
\end{eqnarray}
In particular,
\begin{eqnarray}
\mathcal{U}_\lambda A_{||}^1\nabla_{||}(I-\P_\lambda) f^{\epsilon}_{Q,w}=I+II+III+IV,
\end{eqnarray}
where
\begin{eqnarray}
I&=&-\theta_\lambda \tilde{\bf \varepsilon} \nabla_{||} \mathcal{E}_\lambda^1(I-\P_\lambda) f^{\epsilon}_{Q,w},\notag\\
II&=&+\theta_\lambda \tilde{\bf \varepsilon} \nabla_{||}f^{\epsilon}_{Q,w},\notag\\
III&=&-\theta_\lambda \tilde{\bf \varepsilon} \nabla_{||}\P_\lambda f^{\epsilon}_{Q,w},\notag\\
IV&=&-\theta_\lambda \tilde{\bf \varepsilon} \lambda^2 \nabla_{||} \mathcal{E}_\lambda^1\partial_t(I-\P_\lambda) f^{\epsilon}_{Q,w}.
\end{eqnarray}
Using the $L^2$-boundedness of $\theta_\lambda$ and $\nabla_{||} \mathcal{E}_\lambda^1$ we see that
\begin{eqnarray*}
\int _{0}^{l(Q)}\int_{Q}|I|^2\, \frac {dxdtd\lambda}\lambda&\leq& c\varepsilon_0\int _{0}^{l(Q)}\int_{Q}
|\lambda^{-1}(I-\P_\lambda)  f^{\epsilon}_{Q,w}|^2\, \frac {dxdtd\lambda}\lambda\notag\\
&\leq& c\varepsilon_0||\mathbb D f^{\epsilon}_{Q,w}||_2^2\leq c\varepsilon_0|Q|,
\end{eqnarray*}
by Lemma \ref{little2} and Lemma \ref{ilem2--+} $(i)$. Furthermore,
\begin{eqnarray*}
\int _{0}^{l(Q)}\int_{Q}|II|^2\, \frac {dxdtd\lambda}\lambda\leq c\varepsilon_0||\nabla _{||}f^{\epsilon}_{Q,w}||_2^2\leq c\varepsilon_0|Q|,
\end{eqnarray*}
by Lemma \ref{ilem1} and  Lemma \ref{ilem2--+} $(i)$.  To estimate $III$ we choose $\P_\lambda=\tilde \P_\lambda^2$, where $\tilde \P_\lambda$ is of the same type, and write
\begin{eqnarray*}
-III&=&\theta_\lambda \tilde{\bf \varepsilon} \P_\lambda \nabla_{||}f^{\epsilon}_{Q,w}\notag\\
&=&(\theta_\lambda \tilde{\bf \varepsilon} \P_\lambda-(\theta_\lambda \tilde{\bf \varepsilon}) \P_\lambda)\nabla_{||}f^{\epsilon}_{Q,w}+(\theta_\lambda \tilde{\bf \varepsilon}) \P_\lambda\nabla_{||}f^{\epsilon}_{Q,w}\notag\\
&=&(\theta_\lambda \tilde\varepsilon \tilde \P_\lambda-(\theta_\lambda \tilde{\bf \varepsilon}) \tilde \P_\lambda)\tilde \P_\lambda\nabla_{||}f^{\epsilon}_{Q,w}+(\theta_\lambda \tilde{\bf \varepsilon}) \P_\lambda\nabla_{||}f^{\epsilon}_{Q,w}\notag\\
&=:&\mathcal{R}_\lambda^{(3)} \tilde \P_\lambda\nabla_{||}f^{\epsilon}_{Q,w}+(\theta_\lambda \tilde{\bf \varepsilon}) \P_\lambda\nabla_{||}f^{\epsilon}_{Q,w}.
\end{eqnarray*}
Then
\begin{eqnarray*}
\int _{0}^{l(Q)}\int_{Q}|III|^2\, \frac {dxdtd\lambda}\lambda&\leq& c\int _{0}^{l(Q)}\int_{Q}|\mathcal{R}_\lambda^{(3)} \tilde \P_\lambda\nabla_{||}f^{\epsilon}_{Q,w}|^2\, \frac {dxdtd\lambda}\lambda\notag\\
&&+c\int _{0}^{l(Q)}\int_{Q}|(\theta_\lambda \tilde{\bf \varepsilon}) \P_\lambda\nabla_{||}f^{\epsilon}_{Q,w}|^2\, \frac {dxdtd\lambda}\lambda.
\end{eqnarray*}
Now Lemma \ref{le11-} applies to $\mathcal{R}_\lambda^{(3)}$ and, by Lemma \ref{ilem1}, Lemma \ref{le8} applies to $\theta_\lambda$. Hence using these results we deduce that
\begin{eqnarray*}
\int _{0}^{l(Q)}\int_{Q}|III|^2\, \frac {dxdtd\lambda}\lambda&\leq& c\varepsilon_0 \int_{\mathbb R^{n+2}_+}
|\lambda\nabla_{||}(\P_\lambda\nabla_{||} f^{\epsilon}_{Q,w})(x,t)|^2\, \frac {dxdtd\lambda}\lambda\notag\\
&&+c\varepsilon_0 \int_{\mathbb R^{n+2}_+}
|\lambda^2\partial_t(\P_\lambda\nabla_{||} f^{\epsilon}_{Q,w})(x,t)|^2\, \frac {dxdtd\lambda}\lambda\notag\\
&&+c\varepsilon_0\int_{\mathbb R^{n+1}}|\nabla_{||}f^{\epsilon}_{Q,w}|^2\,  {dxdt}\notag\\
&\leq&c\varepsilon_0\int_{\mathbb R^{n+1}}|\nabla_{||}f^{\epsilon}_{Q,w}|^2\, {dxdt}\leq c\varepsilon_0|Q|,
\end{eqnarray*}
by Lemma \ref{little2} $(i)$ and Lemma \ref{ilem2--+} $(i)$. To handle $IV$ we first note that
\begin{eqnarray}
IV=(\lambda^2\partial_t  \theta_\lambda) \tilde{\bf \varepsilon} \biggl(\lambda\nabla_{||} \mathcal{E}_\lambda^1\frac 1 \lambda(I-\P_\lambda) f^{\epsilon}_{Q,w}\biggr )
\end{eqnarray}
by the facts that $\tilde{\bf \varepsilon}$ is independent of $t$, \eqref{eq4}, and that $\partial_t$ and $\mathcal{E}_\lambda^1$ commute.
By definition $$\lambda^2\partial_t  \theta_\lambda=\lambda^4\partial_t \partial_\lambda^2 (\mathcal{S}_{\lambda}^{\mathcal{H}_0}\nabla)\cdot.$$
Hence, using Lemma \ref{le5} $(i)$ and $(ii)$  we see that $\lambda^2\partial_t  \theta_\lambda$ is uniformly (in $\lambda$) bounded on $L^2(\mathbb R^{n+1},\mathbb C)$. The same applies to $\lambda\nabla_{||} \mathcal{E}_\lambda^1$ by Lemma \ref{le8-}. Hence,
\begin{eqnarray*}
\int _{0}^{l(Q)}\int_{Q}|IV|^2\, \frac {dxdtd\lambda}\lambda&\leq& c\varepsilon_0 \int_0^\infty\int_{\mathbb R^{n+1}}|
\frac 1 \lambda \mathbb I_1(I-\P_\lambda) \mathbb D f^{\epsilon}_{Q,w}(x,t)|^2\, \frac {dxdtd\lambda}\lambda\notag\\
&\leq &c\varepsilon_0||\mathbb D f^{\epsilon}_{Q,w}||^2\leq c\varepsilon_0|Q|,
\end{eqnarray*}
by Lemma \ref{little2} and
Lemma \ref{ilem2--+} $(i)$. This completes the proof of the lemma.
\end{proof}

     \begin{lemma} \label{ilem2-} Assume \eqref{keyseauu}. Let $\theta_\lambda$ be as in the Lemma \ref{ilem1}, let ${\bf \varepsilon}$, $\tilde{\bf \varepsilon}$ be as in \eqref{keyestaahmo}, \eqref{keyestaahmouu}. Let $\mathcal{E}_\lambda^1:=(I+\lambda^2(\mathcal{H}_1)_{||})^{-1}=(I+\lambda^2(\partial_t+(\mathcal{L}_1)_{||}))^{-1}$.  Let
     \begin{eqnarray}
\mathcal{R}_\lambda=\lambda\theta_\lambda \tilde{\bf \varepsilon} \nabla_{||} \mathcal{E}_\lambda^1.
\end{eqnarray}
Then $\mathcal{R}_\lambda$ is an operator satisfying  \eqref{stand-} and \eqref{stand} for some $d\geq 0$ and $\mathcal{R}_\lambda1=0$. Furthermore,
   \begin{eqnarray}\label{cruca}
    \int_{0}^\infty\int_{\mathbb R^{n+1}}|\mathcal{R}_\lambda u|^2\,  \frac{dxdtd\lambda}{\lambda^3}\leq c\varepsilon_0||\mathbb Du||_2^2,
     \end{eqnarray}
     whenever $u\in \mathbb H(\mathbb R^{n+1},\mathbb C)$ and for some constant $c$ depending  at most
     on $n$, $\Lambda$,  the De Giorgi-Moser-Nash constants and $\Gamma_0$.
     \end{lemma}
     \begin{proof} $\theta_\lambda$ satisfies  \eqref{stand-} and \eqref{stand}, see Lemma \ref{ilem1}, and $\lambda\nabla_{||}\mathcal{E}_\lambda^1$ satisfies \eqref{stand-}, see Lemma \ref{le8-}. Furthermore, using Lemma \ref{le8-+} we see that the latter operator also satisfies assumption \eqref{stand+} in Lemma \ref{le9}. Hence, applying
 Lemma \ref{le9} we can first conclude that \eqref{stand-} and \eqref{stand} hold with $\Theta_\lambda$ replaced by $ \mathcal{R}_\lambda$, and hence that
 Lemma \ref{le11} is applicable to $ \mathcal{R}_\lambda$. Hence, based on Lemma \ref{le11} we see that to prove \eqref{cruca} it suffices to prove that
\begin{eqnarray}\label{crucacar}\biggl|\frac 1 \lambda R_\lambda\Psi(x,t)\biggr|^2\frac {dxdtd\lambda}\lambda,
\end{eqnarray}
where $\Psi(x,t)=x$, defines a Carleson measure on $\mathbb R^{n+2}_+$ with constant bounded by $c\varepsilon_0$. We write
\begin{eqnarray}
\frac 1 \lambda R_\lambda\Psi(x,t)=\theta_\lambda \tilde{\bf \varepsilon} \nabla_{||} (\mathcal{E}_\lambda^1-I)\Psi+
\theta_\lambda \tilde{\bf \varepsilon} \nabla_{||} \Psi.
\end{eqnarray}
However, $\nabla_{||} \Psi$ is the identity matrix and hence, using Lemma \ref{ilem1} and Lemma \ref{le8} we see that
$$\int_0^{l(Q)}\int_Q|\theta_\lambda \tilde{\bf \varepsilon} \nabla_{||} \Psi(x,t)|^2\frac {dxdtd\lambda}\lambda\leq c\varepsilon_0|Q|.$$
To continue we note that
\begin{eqnarray}
\theta_\lambda \tilde{\bf \varepsilon} \nabla_{||} (\mathcal{E}_\lambda^1-I)\Psi&=&\theta_\lambda \tilde{\bf \varepsilon} \lambda^2 \nabla_{||}\mathcal{E}_\lambda^1(\partial_t+(\mathcal{L}_1)_{||})\Psi\notag\\
&=&\theta_\lambda \tilde{\bf \varepsilon} \lambda^2 \nabla_{||} \mathcal{E}_\lambda^1\div_{||}A_{||}^1=\mathcal{U}_\lambda A_{||}^1
\end{eqnarray}
as $\Psi$ is independent of $t$ and where $\mathcal{U}_\lambda$ was introduced in \eqref{Udef}.  Hence it suffices to prove the estimate
\begin{eqnarray}\label{crucacar+}\int_0^{l(Q)}\int_Q|\mathcal{U}_\lambda A_{||}^1|^2\frac {dxdtd\lambda}\lambda\leq c\varepsilon_0|Q|.
\end{eqnarray}
However, this is Lemma \ref{ilem2--}.
\end{proof}

     \begin{lemma} \label{ilem2} Assume \eqref{keyseauu}. Let $\theta_\lambda$ be as in the Lemma \ref{ilem1}, let ${\bf \varepsilon}$, $\tilde{\bf \varepsilon}$ be as in \eqref{keyestaahmo}, \eqref{keyestaahmouu}. Let $\mathcal{E}_\lambda^1:=(I+\lambda^2(\mathcal{H}_1)_{||})^{-1}=(I+\lambda^2(\partial_t+(\mathcal{L}_1)_{||}))^{-1}$. Let $\mathcal{U}_\lambda$ be as in the Lemma \ref{ilem2--}. Then
   \begin{eqnarray}\label{cruc}
    \int_{0}^\infty\int_{\mathbb R^{n+1}}|\mathcal{U}_\lambda{\bf f}|^2\,  \frac{dxdtd\lambda}\lambda\leq c||{\bf f}||_2^2
     \end{eqnarray}
     whenever ${\bf f}\in  L^2(\mathbb R^{n+1},\mathbb C^n)$ and for some constant $c$ depending  at most
     on $n$, $\Lambda$,  the De Giorgi-Moser-Nash constants and $\Gamma_0$.
     \end{lemma}
     \begin{proof}
Let ${\bf f}\in  L^2(\mathbb R^{n+1},\mathbb C^n)$ and let, see Lemma \ref{parahodge},  $u\in \mathbb H(\mathbb R^{n+1},\mathbb C)$ be a weak solution to the equation
$$-\div_{||}(A_{||}^1{\bf f})=(\mathcal{H}_1)_{||}u=\partial_tu+(\mathcal{L}_1)_{||}u,$$
such that
 \begin{eqnarray}\label{cruc-}
 ||u||_{\mathbb H(\mathbb R^{n+1},\mathbb C)}\leq c||{\bf f}||_2.
 \end{eqnarray}
 Using the ellipticity of $A_{||}^1$ we see that  to prove \eqref{cruc} it suffices to prove that
  \begin{eqnarray}\label{cruc+}
    \int_{0}^\infty\int_{\mathbb R^{n+1}}|\mathcal{U}_\lambda A_{||}^1{\bf f}|^2\,  \frac{dxdtd\lambda}\lambda\leq c||{\bf f}||_2^2,
     \end{eqnarray}
     whenever ${\bf f}\in  L^2(\mathbb R^{n+1},\mathbb C^n)$. Now
\begin{eqnarray}
\mathcal{U}_\lambda A_{||}^1{\bf f}(x,t)&=&\theta_\lambda \tilde{\bf \varepsilon}\nabla_{||}\mathcal{E}_\lambda^1\lambda^2 (\partial_t+(\mathcal{L}_1)_{||})u\notag\\
&=&\theta_\lambda \tilde{\bf \varepsilon} \nabla_{||} ((I+\lambda^2(\partial_t+(\mathcal{L}_1)_{||}))^{-1}-I)u\notag\\
&=&\theta_\lambda \tilde{\bf \varepsilon} \nabla_{||} \mathcal{E}_\lambda^1u-\theta_\lambda \tilde{\bf \varepsilon} \nabla_{||} u.
\end{eqnarray}
Using Lemma \ref{ilem1} and \eqref{cruc-} we see that
  \begin{eqnarray}\label{cruc++}
    \int_{0}^\infty\int_{\mathbb R^{n+1}}|\theta_\lambda \tilde{\bf \varepsilon} \nabla_{||} u|^2\,
    \frac{dxdtd\lambda}\lambda\leq c\varepsilon_0||{\bf f}||_2^2.
     \end{eqnarray}
Hence, the new estimate we need to prove is that
  \begin{eqnarray}\label{cruc++}
    \int_{0}^\infty\int_{\mathbb R^{n+1}}|\theta_\lambda \tilde{\bf \varepsilon} \nabla_{||} \mathcal{E}_\lambda^1u|^2\,
    \frac{dxdtd\lambda}\lambda\leq c||{\bf f}||_2^2.
     \end{eqnarray}
Define $\mathcal{R}_\lambda$ through the relation
\begin{eqnarray}\label{aaa}
\theta_\lambda \tilde{\bf \varepsilon} \nabla_{||} \mathcal{E}_\lambda^1u=\frac 1 {\lambda}{\mathcal{R}_\lambda u}.
\end{eqnarray}
The estimate in \eqref{cruc++} now follows from Lemma \ref{ilem2-}.
\end{proof}

\begin{lemma}\label{th0uu} Assume \eqref{keysea}. Then
\begin{eqnarray}\label{keyestint+adase}
|||\lambda \nabla \mathcal{D}_\lambda^{\mathcal{H}} f|||_\pm+|||\lambda \nabla \mathcal{D}_\lambda^{\mathcal{H}^\ast} f|||_\pm\leq c  ||f||_2.
\end{eqnarray}
whenever $f\in L^2(\mathbb R^{n+1},\mathbb C)$ and for some constant $c$ depending  at most
     on $n$, $\Lambda$,  the De Giorgi-Moser-Nash constants and $\Gamma$.
\end{lemma}
               \begin{proof} We will only prove the estimate for $|||\lambda \nabla \mathcal{D}_\lambda^{\mathcal{H}} f|||_+$. To start the proof  we first note that
     \begin{eqnarray}\label{eq11edsea+fr}
    I^2:=|||\lambda\nabla\mathcal{D}_\lambda f|||_+^2&=&-\int_0^\infty\int_{\mathbb R^{n+1}}\nabla\mathcal{D}_\lambda f\cdot \overline{\partial_\lambda\nabla\mathcal{D}_\lambda f}\, \lambda^2dxdtd\lambda\notag\\
    &&+\lim_{\epsilon\to 0} J_{1/\epsilon}-\lim_{\epsilon\to 0} J_{\epsilon},
    \end{eqnarray}
    where
    $$J_\lambda=\int_{\mathbb R^{n+1}}|\nabla\mathcal{D}_\lambda f|^2\, \lambda^2dxdt.$$
    However, by energy estimates, see Lemma \ref{le1--} and Lemma \ref{le1}, \eqref{eq11edsea} and duality we see that
     \begin{eqnarray}\label{eq11edsea+fri}
     J_\lambda\leq c\int_{\mathbb R^{n+1}}|\mathcal{D}_\lambda f|^2\, dxdt\leq c||f||_2^2.
      \end{eqnarray}
      Hence it suffices to estimate
          \begin{eqnarray}\label{eq11edsea+frii}
    |||\lambda^2\nabla\partial_\lambda\mathcal{D}_\lambda f|||_+&\leq& c|||\lambda\partial_\lambda\mathcal{D}_\lambda f|||_+\notag\\
    &=&c|||\lambda\partial_\lambda(\mathcal{S}_\lambda\nabla_{||})\cdot {\bf f}|||_++c|||\lambda\partial_\lambda^2\mathcal{S}_\lambda f|||_+
    \end{eqnarray}
    where we again have used energy estimates, see Lemma \ref{le1--},  \eqref{eq11edsea}, and where we have introduced ${\bf f}$. To complete the proof we only have
    to estimate $|||\lambda\partial_\lambda(\mathcal{S}_\lambda\nabla_{||})\cdot {\bf f}|||_+$. However this is the term $ \widetilde {II}$ introduced in
    \eqref{keyestaa+a} in the proof of Lemma \ref{ilem1}. Hence, reusing that estimate we can conclude, using \eqref{keysea}, that
           \begin{eqnarray}\label{eq11edsea+frii}
    |||\lambda^2\nabla\partial_\lambda\mathcal{D}_\lambda f|||_+\leq c||f||_2.
    \end{eqnarray}
    Hence the proof of the lemma is complete.
    \end{proof}

\section{Proof of Theorem \ref{thper2}: preliminary technical estimates}\label{sec9}
\noindent

In this section we prove a number of technical estimates to be used in the proof of Theorem \ref{thper2}. As in the statement of Theorem \ref{thper2}, and as in Section \ref{sec8}, we throughout this section consider two operators  $\mathcal{H}_0=\partial_t-\div A^0\nabla$, $\mathcal{H}_1=\partial_t-\div A^1\nabla$. We will assume  \eqref{keyseauu}. By definition, \eqref{keyseauu} implies that
          \begin{eqnarray}\label{keyestinteda}
\sup_{\lambda\neq 0}||\partial_\lambda \mathcal{S}_\lambda^{\mathcal{H}_0} f||_2+\sup_{\lambda\neq 0}||\partial_\lambda \mathcal{S}_\lambda^{\mathcal{H}_0^\ast} f||_2&\leq& \Gamma_0||f||_2\notag\\
|||\lambda \partial_\lambda^2\mathcal{S}_{\lambda}^{\mathcal{H}_0}f|||_\pm+|||\lambda \partial_\lambda^2\mathcal{S}_{\lambda}^{\mathcal{H}_0^\ast}f|||_\pm&\leq& \Gamma_0||f||_2,
     \end{eqnarray}
whenever  $f\in L^2(\mathbb R^{n+1},\mathbb C)$. We let ${\bf \varepsilon}$ be as in \eqref{keyestaahmo}, we assume
     \eqref{keyestaahmeddo} and we let $\tilde{\bf \varepsilon}$ be as introduced in \eqref{keyestaahmouu}. We also introduce
     \begin{eqnarray}\label{iest1ha}
A_\pm^{\mathcal{H}_1,\eta}(f)&:=& |||\lambda \nabla \partial_\lambda \mathcal{S}_\lambda^{\mathcal{H}_1,\eta}f|||_{\pm}+||N_\ast^\pm(\P_\lambda \partial_\lambda \mathcal{S}_\lambda^{\mathcal{H}_1,\eta}f)||_2\notag\\
&&+\sup_{\pm\lambda>0}||\mathbb D\mathcal{S}_\lambda^{\mathcal{H}_1,\eta}f||_2+\sup_{\pm\lambda> 0}||\partial_\lambda\mathcal{S}_\lambda^{\mathcal{H}_1,\eta}f||_2+||f||_2.
\end{eqnarray}
In this section we prove the following technical lemmas.
\begin{lemma}\label{telem1} Assume \eqref{keyseauu}. Let $a\in \mathbb R\setminus\{0\}$. Then there exists a constant $c$, depending at most
     on $n$, $\Lambda$,  the De Giorgi-Moser-Nash constants, $\Gamma_0$, and $a$, such that
$$|||\lambda^2(\partial_\lambda^2\mathcal{S}_{a\lambda}^{\mathcal{H}_0}\nabla) \cdot\varepsilon \nabla \mathcal{S}_\lambda^{\mathcal{H}_1,\eta}f|||_{\pm}\leq c\varepsilon_0A_\pm^{\mathcal{H}_1,\eta}(f).$$
\end{lemma}
\begin{lemma}\label{telem2} Assume \eqref{keyseauu}. Let $a\in \mathbb R\setminus\{0\}$. Then there exists a constant $c$, depending at most
     on $n$, $\Lambda$,  the De Giorgi-Moser-Nash constants, $\Gamma_0$, and $a$, such that
$$|||\lambda(\partial_\lambda \mathcal{S}_{a\lambda}^{\mathcal{H}_0}\nabla) \cdot\varepsilon \nabla \mathcal{S}_\lambda^{\mathcal{H}_1,\eta}f|||_\pm\leq c\varepsilon_0A_\pm^{\mathcal{H}_1,\eta}(f).$$
\end{lemma}

\begin{lemma} \label{telem3} Assume \eqref{keyseauu}. Let $a,b\in \mathbb R\setminus\{0\}$. Then there exists a constant $c$, depending at most
     on $n$, $\Lambda$,  the De Giorgi-Moser-Nash constants, $\Gamma_0$, and $a$, $b$, such that
\begin{eqnarray}
\sup_{0\leq\lambda_1<\lambda_2<\infty}||\int_{\lambda_1}^{\lambda_2}(D_{n+1}\mathcal{S}_{a\lambda}^{\mathcal{H}_0}\nabla)\cdot\varepsilon\nabla \mathcal{S}_{b\lambda}^{\mathcal{H}_1,\eta}f\, d\lambda||_2&\leq& c\varepsilon_0 A_+^{\mathcal{H}_1,\eta}(f),\notag\\
\sup_{0\leq\lambda_1<\lambda_2<\infty}||\int_{-\lambda_2}^{-\lambda_1}(D_{n+1}\mathcal{S}_{a\lambda}^{\mathcal{H}_0}\nabla)\cdot\varepsilon\nabla \mathcal{S}_{b\lambda}^{\mathcal{H}_1,\eta}f\, d\lambda||_2&\leq& c\varepsilon_0 A_-^{\mathcal{H}_1,\eta}(f).
\end{eqnarray}
\end{lemma}

      Below we prove Lemma \ref{telem1}-Lemma \ref{telem3}. We will consequently only establish the estimates involving $|||\cdot|||_+$, $A^+$, as the corresponding estimates involving $|||\cdot|||_-$, $A^-$, can be proved analogously. Furthermore, we will in the case of Lemma \ref{telem1}, Lemma \ref{telem2}, only give the details assuming that $a=1$, and in the case of Lemma \ref{telem3}, we will give the details assuming that $a=2$ and $b=1$.

\subsection{Proof of Lemma \ref{telem1}} We are going to prove that
$$|||\lambda^2(\partial_\lambda^2\mathcal{S}_{\lambda}^{\mathcal{H}_0}\nabla) \cdot\varepsilon \nabla \mathcal{S}_\lambda^{\mathcal{H}_1,\eta}f|||_{+}\leq c\varepsilon_0A_+^{\mathcal{H}_1,\eta}(f).$$
Let
\begin{eqnarray*}
          \theta_\lambda{\bf f}:=\lambda^2\partial_\lambda^2 (\mathcal{S}_{\lambda}^{\mathcal{H}_0}\nabla)\cdot{\bf f},
     \end{eqnarray*}
     whenever ${\bf f}\in L^2(\mathbb R^{n+1},\mathbb C^{n+1})$. Then $\theta_\lambda$ is the operator explored in Section \ref{sec8}. We write
          \begin{eqnarray}
\lambda^2(\partial_\lambda^2\mathcal{S}_{\lambda}^{\mathcal{H}_0}\nabla) \cdot\varepsilon \nabla \mathcal{S}_\lambda^{\mathcal{H}_1,\eta}f &=&\theta_\lambda\varepsilon \nabla \mathcal{S}_\lambda^{\mathcal{H}_1,\eta}f\notag\\
&=&\theta_\lambda\tilde\varepsilon\nabla_{||}\mathcal{S}_\lambda^{\mathcal{H}_1,\eta}f+\theta_\lambda \varepsilon_{n+1}\partial_\lambda \mathcal{S}_\lambda^{\mathcal{H}_1,\eta}f,
\end{eqnarray}
and
\begin{eqnarray}
\theta_\lambda\varepsilon_{n+1}\partial_\lambda \mathcal{S}_\lambda^{\mathcal{H}_1,\eta}f&=&\mathcal{R}_\lambda\partial_\lambda \mathcal{S}_\lambda^{\mathcal{H}_1,\eta}f+(\theta_\lambda \varepsilon_{n+1})\mathcal{P}_\lambda\partial_\lambda \mathcal{S}_\lambda^{\mathcal{H}_1,\eta}f,
\end{eqnarray}
where
\begin{eqnarray}
\mathcal{R}_\lambda=\theta_\lambda\varepsilon_{n+1}-(\theta_\lambda \varepsilon_{n+1})\P_\lambda,
\end{eqnarray}
and where $\P_\lambda$ is a standard parabolic approximation of the identity. Using Lemma \ref {ilem1} we see that
 Lemma \ref{le8} is applicable to  $\theta_\lambda$ and that  Lemma \ref{le11-} is applicable to $\mathcal{R}_\lambda$.  Hence,
  \begin{eqnarray}\label{ew1}
 |||\theta_\lambda\varepsilon_{n+1}\partial_\lambda \mathcal{S}_\lambda^{\mathcal{H}_1,\eta}f|||_+\leq c\varepsilon_0 ||N_\ast(P_\lambda\partial_\lambda \mathcal{S}_\lambda^{\mathcal{H}_1,\eta}f)||_2,
\end{eqnarray}
and
  \begin{eqnarray}\label{aa-}
|||\mathcal{R}_\lambda^{(1)}\partial_\lambda \mathcal{S}_\lambda^{\mathcal{H}_1,\eta}f|||_+&\leq& c\varepsilon_0(|||\lambda\nabla\partial_\lambda \mathcal{S}_\lambda^{\mathcal{H}_1,\eta}f|||_++|||\lambda^{2}\partial_\lambda \partial_t \mathcal{S}_\lambda^{\mathcal{H}_1,\eta}f|||_+).
\end{eqnarray}
Using Lemma \ref{lemsl1} $(iv)$ we see that
  \begin{eqnarray}\label{aa-pp}|||\lambda^{2}\partial_\lambda \partial_t \mathcal{S}_\lambda^{\mathcal{H}_1,\eta}f|||_+\leq cA_+^{\mathcal{H}_1,\eta}(f),
\end{eqnarray}
and we can conclude that \begin{eqnarray}\label{aa}
|||\theta_\lambda\varepsilon_{n+1}\partial_\lambda \mathcal{S}_\lambda^{\mathcal{H}_1,\eta}f|||_++|||\mathcal{R}_\lambda^{(1)}\partial_\lambda \mathcal{S}_\lambda^{\mathcal{H}_1,\eta}f|||_+&\leq& c\varepsilon_0 A_+^{\mathcal{H}_1,\eta}(f).
\end{eqnarray}
To start the estimate of $|||\theta_\lambda\tilde\varepsilon\nabla_{||}\mathcal{S}_\lambda^{\mathcal{H}_1,\eta}f|||_+$, we let $$\mathcal{E}_\lambda^1:=(I+\lambda^2(\partial_t+(\mathcal{L}_1)_{||}))^{-1}$$ and write
 \begin{eqnarray}
 \theta_\lambda\tilde\varepsilon\nabla_{||}\mathcal{S}_\lambda^{\mathcal{H}_1,\eta}f=\theta_\lambda\tilde\varepsilon\nabla_{||}(I-\mathcal{E}_\lambda^1)\mathcal{S}_\lambda^{\mathcal{H}_1,\eta}f+ \theta_\lambda\tilde\varepsilon\nabla_{||}\mathcal{E}_\lambda^1\mathcal{S}_\lambda^{\mathcal{H}_1,\eta}f.
\end{eqnarray}
Hence,
 \begin{eqnarray}
 \theta_\lambda\tilde\varepsilon\nabla_{||}\mathcal{S}_\lambda^{\mathcal{H}_1,\eta}f=:Y_\lambda f+Z_\lambda f,
\end{eqnarray}
where
 \begin{eqnarray}
Y_\lambda &=&\theta_\lambda \tilde\varepsilon\nabla_{||}\mathcal{E}_\lambda^1\lambda^2(\partial_t+(\mathcal{L}_1)_{||})\mathcal{S}_\lambda ^{\mathcal{H}_1,\eta},\notag\\
Z_\lambda&=& \theta_\lambda\tilde\varepsilon\nabla_{||}\mathcal{E}_\lambda^1\mathcal{S}_\lambda ^{\mathcal{H}_1,\eta}.
\end{eqnarray}
Recall that  $f_\eta(x,t,\lambda)=f(x,t)\varphi_\eta(\lambda)$, see \eqref{fsmooth}, and note that
 \begin{eqnarray}
 (\partial_t+(\mathcal{L}_1)_{||})\mathcal{S}_\lambda ^{\mathcal{H}_1,\eta}f&=& \sum_{i=1}^{n}D_i(A_{i,n+1}^1D_{n+1}\mathcal{S}_\lambda ^{\mathcal{H}_1,\eta}f) \notag\\
 &&+ \sum_{j=1}^{n+1}A_{n+1,j}^1D_{n+1}D_j\mathcal{S}_\lambda ^{\mathcal{H}_1,\eta}f\notag\\
 &&+f_\eta,
 \end{eqnarray}
 in a weak sense.  As a result we  get a natural decomposition
  \begin{eqnarray}
Y_\lambda f =Y_\lambda^1 f +Y_\lambda^2 f +Y_\lambda^3 f_\eta.
\end{eqnarray}
Using the $L^2(\mathbb R^{n+1},\mathbb C)$ boundedness of $\theta_\lambda$ and $\lambda\nabla_{||}\mathcal{E}_\lambda^1$, see Lemma \ref {ilem1} and Lemma \ref{le4}, and elementary estimates for $f_\eta$, we see that
  \begin{eqnarray}
|||Y_\lambda^2 f|||_+&\leq& c\varepsilon_0|||\lambda\nabla\partial_\lambda \mathcal{S}_\lambda ^{\mathcal{H}_1,\eta}f|||_+,\notag\\
|||Y_\lambda^3 f|||_+&\leq& c\varepsilon_0|||\lambda f_\eta|||_+\leq c||f||_2.
\end{eqnarray}
To estimate $|||Y_\lambda^1 f|||_+$ we let $\tilde A_{n+1}^1=(A_{1,n+1}^1,...,A_{n,n+1}^1)$ and we let $\mathcal{U}_\lambda$ be as in the statement of Lemma \ref{ilem2--}. Using this notation we see that
\begin{eqnarray*}
Y_\lambda^1=\mathcal{U}_\lambda \tilde A_{n+1}^1\ \partial_\lambda \mathcal{S}_\lambda^{\mathcal{H}_1,\eta}.
\end{eqnarray*}
To proceed we write
\begin{eqnarray*}
Y_\lambda^1=\mathcal{R}_\lambda^{(2)} \partial_\lambda \mathcal{S}_\lambda ^{\mathcal{H}_1,\eta}+
(\mathcal{U}_\lambda \tilde A_{n+1}^1)\P_\lambda \partial_\lambda \mathcal{S}_\lambda ^{\mathcal{H}_1,\eta},
\end{eqnarray*}
where
\begin{eqnarray*}
\mathcal{R}_\lambda^{(2)}=(\mathcal{U}_\lambda \tilde A_{n+1}^1-(\mathcal{U}_\lambda\tilde A_{n+1}^1)\P_\lambda),
\end{eqnarray*}
and where again $\P_\lambda$ is a standard approximation of the identity. Again applying Lemma \ref{ilem2} we see that
 Lemma \ref{le8} is applicable to  $\mathcal{U}_\lambda$, and that Lemma \ref{le11-} is applicable to $\mathcal{R}_\lambda^{(2)}$. Hence,
 \begin{eqnarray*}
 \quad|||Y_\lambda^1f|||_+\leq c\varepsilon_0\bigl (
 ||N_\ast(P_\lambda\partial_\lambda \mathcal{S}_\lambda^{\mathcal{H}_1,\eta}f)||_2+|||\lambda \nabla\partial_\lambda \mathcal{S}_\lambda^{\mathcal{H}_1,\eta}f|||_++
 |||\lambda^2 \partial_t\partial_\lambda \mathcal{S}_\lambda^{\mathcal{H}_1,\eta}f|||_+\bigr).
\end{eqnarray*}
Putting all estimates together we can conclude, using \eqref{aa-pp}, that\begin{eqnarray}
\quad\quad|||Y_\lambda^1 f|||_++|||Y_\lambda^2 f|||+|||Y_\lambda^3 f|||_+\leq c\varepsilon_0A_+^{\mathcal{H}_1,\eta}(f).
\end{eqnarray}
 This completes the proof of $|||Y_\lambda f|||_+$. To estimate $|||Z_\lambda f|||_+$ we write
\begin{eqnarray}
Z_\lambda &=&\theta_\lambda \tilde\varepsilon\nabla_{||}\mathcal{E}_\lambda^1(\mathcal{S}_\lambda ^{\mathcal{H}_1,\eta}-\mathcal{S}_\delta^{\mathcal{H}_1,\eta})+\theta_\lambda \tilde\varepsilon\nabla_{||}\mathcal{E}_\lambda^1\mathcal{S}_\delta^{\mathcal{H}_1,\eta}\notag\\
&=:&Z_\lambda ^1+Z_\lambda ^2,
\end{eqnarray}
for some $\delta>0$ small. Furthermore,
\begin{eqnarray}
Z_\lambda^1=\theta_\lambda \tilde\varepsilon\nabla_{||}\mathcal{E}_\lambda^1\int_\delta^\lambda\partial_\sigma
\mathcal{S}_\sigma^{\mathcal{H}_1,\eta}\, d\sigma=\Omega_\lambda^1+\Omega_\lambda^2,
\end{eqnarray}
by partial integration, and where
\begin{eqnarray}
\Omega_\lambda^1&=&\theta_\lambda \tilde\varepsilon\nabla_{||}\mathcal{E}_\lambda^1\lambda\partial_\lambda \mathcal{S}_\lambda^{\mathcal{H}_1,\eta},\notag\\
\Omega_\lambda^2&=&-\theta_\lambda \tilde\varepsilon\nabla_{||}\mathcal{E}_\lambda^1\int_\delta^\lambda\sigma\partial_\sigma^2
\mathcal{S}_\sigma^{\mathcal{H}_1,\eta}\, d\sigma.
\end{eqnarray}
Now Lemma \ref{ilem2-} applies to the operator $\mathcal{R}_\lambda=\lambda\theta_\lambda \tilde\varepsilon\nabla_{||}\mathcal{E}_\lambda^1$ and hence
\begin{eqnarray}
\quad|||\Omega_\lambda^1f|||_+\leq c\varepsilon_0(|||\lambda\nabla\partial_\lambda \mathcal{S}_\lambda^{\mathcal{H}_1,\eta}|||_++
|||\lambda^2\partial_t\partial_\lambda \mathcal{S}_\lambda^{\mathcal{H}_1,\eta}|||_+)\leq c\varepsilon_0A_+^{\mathcal{H}_1,\eta}(f).
\end{eqnarray}
Furthermore,
\begin{eqnarray}
\Omega^2_\lambda=-\lambda\int_\delta^\lambda \frac \sigma\lambda
\theta_\lambda \tilde\varepsilon\nabla_{||}\mathcal{E}_\lambda^1\sigma\partial_\sigma^2
\mathcal{S}_\sigma^{\mathcal{H}_1,\eta}\, \frac {d\sigma}\sigma.
\end{eqnarray}
Hence, using Lemma \ref{le11++} we can conclude that
\begin{eqnarray}
|||\Omega_\lambda^1f|||_+&\leq&c\varepsilon_0|||\lambda^2\nabla_{||}\mathcal{E}_\lambda^1\partial_\lambda^2 \mathcal{S}_\lambda^{\mathcal{H}_1,\eta}f|||_+\notag\\
&\leq&c\varepsilon_0 |||\lambda\partial_\lambda^2 \mathcal{S}_\lambda^{\mathcal{H}_1,\eta}f|||_+\leq c\varepsilon_0A_+^{\mathcal{H}_1,\eta}(f),
\end{eqnarray}
by the $L^2(\mathbb R^{n+1},\mathbb C)$ boundedness of $\lambda \nabla_{||}\mathcal{E}_\lambda^1$.  Finally, using Lemma \ref{ilem2-} we see that
\begin{eqnarray}
|||Z_\lambda^2f|||_+\leq c\varepsilon_0(\sup_{\lambda>0}||\mathbb D\mathcal{S}_\lambda^{\mathcal{H}_1,\eta}f||_2).
\end{eqnarray}
Put together we can conclude that
\begin{eqnarray}
|||Z_\lambda f|||_+\leq |||Z_\lambda^1f|||_++|||Z_\lambda^2f|||_+\leq c\varepsilon_0A_+^{\mathcal{H}_1,\eta}(f).
\end{eqnarray}
This completes the proof of Lemma \ref{telem1}.

\subsection{Proof of Lemma \ref{telem2}} Consider $\delta>0$ and let
$$I_\delta=\int_\delta^{1/\delta}\int_{\mathbb R^{n+1}}|(\partial_\lambda \mathcal{S}_{\lambda}^{\mathcal{H}_0}\nabla)
 \cdot\varepsilon \nabla \mathcal{S}_\lambda^{\mathcal{H}_1,\eta}f|^2\, \lambda {dxdtd\lambda}.$$
Integrating by parts with respect to $\lambda$  we see that
 \begin{eqnarray}
I_\delta&=&-\int_\delta^{1/\delta}\int_{\mathbb R^{n+1}}\partial_\lambda\bigl ((\partial_\lambda \mathcal{S}_{\lambda}^{\mathcal{H}_0}\nabla)
 \cdot\varepsilon \nabla \mathcal{S}_\lambda^{\mathcal{H}_1,\eta}f\bigr )\overline{(\partial_\lambda \mathcal{S}_{\lambda}^{\mathcal{H}_0}\nabla)
 \cdot\varepsilon \nabla \mathcal{S}_\lambda^{\mathcal{H}_1,\eta}f)} \, \lambda^2 {dxdtd\lambda}\notag\\
 &&+\int_{\mathbb R^{n+1}}|(\partial_\lambda \mathcal{S}_{\lambda}^{\mathcal{H}_0}\nabla)
 \cdot\varepsilon \nabla \mathcal{S}_\lambda^{\mathcal{H}_1,\eta}f|^2 \, \lambda^2 {dxdt}\biggl |_{\lambda=1/\delta}\notag\\
  &&-\int_{\mathbb R^{n+1}}|(\partial_\lambda \mathcal{S}_{\lambda}^{\mathcal{H}_0}\nabla)
 \cdot\varepsilon \nabla \mathcal{S}_\lambda^{\mathcal{H}_1,\eta}f|^2 \, \lambda^2 {dxdt}\biggl |_{\lambda=\delta}.
\end{eqnarray}
Hence,
 \begin{eqnarray}
I_\delta&\leq& \frac 1 2I_\delta+|||\lambda^2(\partial_\lambda^2 \mathcal{S}_{\lambda}^{\mathcal{H}_0}\nabla)
 \cdot\varepsilon \nabla \mathcal{S}_\lambda^{\mathcal{H}_1,\eta}f|||_+^2+|||\lambda^2(\partial_\lambda \mathcal{S}_{\lambda}^{\mathcal{H}_0}\nabla)
 \cdot\varepsilon \nabla \partial_\lambda \mathcal{S}_\lambda^{\mathcal{H}_1,\eta}f|||_+^2\notag\\
 &&+c\sup_{\lambda>0}\int_{\mathbb R^{n+1}}|(\partial_\lambda \mathcal{S}_{\lambda}^{\mathcal{H}_0}\nabla)
 \cdot\varepsilon \nabla \mathcal{S}_\lambda^{\mathcal{H}_1,\eta}f|^2 \, \lambda^2 {dxdt}.
\end{eqnarray}
Using this and Lemma \ref{le5} we see that
 \begin{eqnarray}
I_\delta&\leq& c|||\lambda^2(\partial_\lambda^2 \mathcal{S}_{\lambda}^{\mathcal{H}_0}\nabla)
 \cdot\varepsilon \nabla \mathcal{S}_\lambda^{\mathcal{H}_1,\eta}f|||_+^2+c\varepsilon_0^2 |||\lambda \nabla \partial_\lambda \mathcal{S}_\lambda^{\mathcal{H}_1,\eta}f|||_+^2\notag\\
 &&+c\varepsilon_0^2 \sup_{\lambda>0}||\nabla \mathcal{S}_\lambda^{\mathcal{H}_1,\eta}f||^2_2.
\end{eqnarray}
Based on this we see that Lemma \ref{telem2} now follows from Lemma \ref{telem1}. This completes the proof of Lemma \ref{telem2}.

\subsection{Proof of Lemma \ref{telem3}} Fix $0\leq\lambda_1<\lambda_2<\infty$. To estimate \begin{eqnarray}
\int_{\mathbb R^{n+1}}\biggl |\int_{\lambda_1}^{\lambda_2}(D_{n+1}\mathcal{S}_{2\lambda}^{\mathcal{H}_0}\nabla)\cdot\varepsilon\nabla \mathcal{S}_{\lambda}^{\mathcal{H}_1,\eta}f\, d\lambda\biggr |^2\, dxdt
\end{eqnarray}
we will bound $|I|$ where
$$I=:\int_{\lambda_1}^{\lambda_2}\int_{\mathbb R^{n+1}}\nabla \partial_\lambda \mathcal{S}_{-2\lambda}^{\mathcal{H}_0^\ast}\bar h\cdot \varepsilon \nabla \mathcal{S}_\lambda^{\mathcal{H}_1,\eta}f\, dxdtd\lambda,$$
and where $h\in L^2(\mathbb R^{n+1},\mathbb C)$, $||h||_2=1$. To start the estimate we first integrate by parts in $I$ with respect to $\lambda$ and we see that
 \begin{eqnarray}
I&=&-\int_{\lambda_1}^{\lambda_2}\int_{\mathbb R^{n+1}}\nabla \partial_\lambda^2 \mathcal{S}_{-2\lambda}^{\mathcal{H}_0^\ast}\bar h\cdot \varepsilon \nabla \mathcal{S}_\lambda^{\mathcal{H}_1,\eta}f\, \lambda dxdtd\lambda\notag\\
&&-\int_{\lambda_1}^{\lambda_2}\int_{\mathbb R^{n+1}}\nabla \partial_\lambda \mathcal{S}_{-2\lambda}^{\mathcal{H}_0^\ast}\bar h\cdot \varepsilon \nabla\partial_\lambda \mathcal{S}_\lambda^{\mathcal{H}_1,\eta}f\, \lambda dxdtd\lambda\notag\\
&&+\int_{\mathbb R^{n+1}}\nabla \partial_\lambda \mathcal{S}_{-2\lambda}^{\mathcal{H}_0^\ast}\bar h\cdot \varepsilon \nabla \mathcal{S}_\lambda^{\mathcal{H}_1,\eta}f\, \lambda dxdt\biggl |_{\lambda=\lambda_2}\notag\\
&&-\int_{\mathbb R^{n+1}}\nabla \partial_\lambda \mathcal{S}_{-2\lambda}^{\mathcal{H}_0^\ast}\bar h\cdot \varepsilon \nabla \mathcal{S}_\lambda^{\mathcal{H}_1,\eta}f\, \lambda dxdt\biggl |_{\lambda=\lambda_1}\notag\\
&=:&I_1+I_2+I_3+I_4.
\end{eqnarray}
Again, using Lemma \ref{le5} applied $\mathcal{S}_{-2\lambda}^{\mathcal{H}_0^\ast}$ we see
 \begin{eqnarray}
|I_3+I_4|\leq c\varepsilon_0\sup_{\lambda>0}||\nabla \mathcal{S}_\lambda^{\mathcal{H}_1,\eta}f||_2.
\end{eqnarray}
Furthermore,
 \begin{eqnarray}
\quad|I_2|\leq c\varepsilon_0|||\lambda\nabla \partial_\lambda \mathcal{S}_{-2\lambda}^{\mathcal{H}_0^\ast}h|||_+ |||\lambda\nabla \partial_\lambda \mathcal{S}_\lambda^{\mathcal{H}_1,\eta}f|||_+\leq
c\varepsilon_0|||\lambda\nabla \partial_\lambda \mathcal{S}_\lambda^{\mathcal{H}_1,\eta}f|||_+,
\end{eqnarray}
where we have used \eqref{keyestinteda} and Lemma \ref{lemsl1} applied to  $\mathcal{S}_{-2\lambda}^{\mathcal{H}_0^\ast}$. To handle $I_1$ we again integrate by parts with respect to $\lambda$,
 \begin{eqnarray}
2I_1&=&\int_{\lambda_1}^{\lambda_2}\int_{\mathbb R^{n+1}}\nabla \partial_\lambda^3 \mathcal{S}_{-2\lambda}^{\mathcal{H}_0^\ast}\bar h\cdot \varepsilon \nabla \mathcal{S}_\lambda^{\mathcal{H}_1,\eta}f\, \lambda^2 dxdtd\lambda\notag\\
&&+\int_{\lambda_1}^{\lambda_2}\int_{\mathbb R^{n+1}}\nabla \partial_\lambda^2 \mathcal{S}_{-2\lambda}^{\mathcal{H}_0^\ast}\bar h\cdot \varepsilon \nabla \partial_\lambda\mathcal{S}_\lambda^{\mathcal{H}_1,\eta}f\, \lambda^2 dxdtd\lambda\notag\\
&&+\int_{\mathbb R^{n+1}}\nabla \partial_\lambda^2 \mathcal{S}_{-2\lambda}^{\mathcal{H}_0^\ast}\bar h\cdot \varepsilon \nabla \mathcal{S}_\lambda^{\mathcal{H}_1,\eta}f\, \lambda^2 dxdt\biggl |_{\lambda=\lambda_2}\notag\\
&&-\int_{\mathbb R^{n+1}}\nabla \partial_\lambda^2 \mathcal{S}_{-2\lambda}^{\mathcal{H}_0^\ast}\bar h\cdot \varepsilon \nabla \mathcal{S}_\lambda^{\mathcal{H}_1,\eta}f\, \lambda^2 dxdt\biggl |_{\lambda=\lambda_1}\notag\\
&=:&I_{11}+I_{12}+I_{13}+I_{14}.
\end{eqnarray}
Arguing as above we see that
 \begin{eqnarray}
|I_{12}+I_{13}+I_{14}| \leq
c\varepsilon_0\bigl (\sup_{\lambda>0}||\nabla \mathcal{S}_\lambda^{\mathcal{H}_1,\eta}f||_2+|||\lambda\nabla \partial_\lambda \mathcal{S}_\lambda^{\mathcal{H}_1,\eta}f|||_+\bigr),
\end{eqnarray}
and we can conclude that
$$|I-I_{11}|\leq c\varepsilon_0A_+^{\mathcal{H}_1,\eta}(f).$$
To estimate $I_{11}$ we note that
$$\partial_\lambda^3 \mathcal{S}_{-2\lambda}^{\mathcal{H}_0^\ast}\bar h=\partial_{\tilde\lambda}\partial_\lambda^2
\mathcal{S}_{-\lambda-\tilde\lambda}^{\mathcal{H}_0^\ast}\bar h\bigl |_{\tilde \lambda=\lambda}.$$
We now let, considering $\lambda\in (\lambda_1,\lambda_2)$ as fixed,   $g(x,t)=\partial_\lambda^2
\mathcal{S}_{-\lambda}^{\mathcal{H}_0^\ast}\bar h(x,t)$ and we let $u$ solve $\mathcal{H}_0^{\ast}u=0$ in $\mathbb R^{n+2}_-$ with $u(\cdot,\cdot,0)=g(\cdot,\cdot)$ on
$\mathbb R^{n+1}$. Then $u(\cdot,\cdot,-\tilde\lambda)=\partial_\lambda^2
\mathcal{S}_{-\lambda-\tilde\lambda}^{\mathcal{H}_0^\ast}\bar h(\cdot,\cdot)$  by the uniqueness in $(D2)$ for $\mathcal{H}_0^{\ast}$, see Lemma \ref{trace7}. Furthermore, by invertibility of layer potentials for $\mathcal{H}_0^{\ast}$, and uniqueness in $(D2)$ for $\mathcal{H}_0^{\ast}$,   we also have
$$u(\cdot,\cdot,-\tilde\lambda)=\mathcal{D}_{-\tilde\lambda}^{\mathcal{H}_0^{\ast}}\biggl (\frac 1 2I+\mathcal{K}^{\mathcal{H}_0^{\ast}}\biggr )^{-1}g.$$
Consequently,
$$\partial_{\tilde\lambda}\nabla u(\cdot,-\tilde\lambda)=\partial_{\tilde\lambda}\nabla
\mathcal{D}_{-\tilde\lambda}^{\mathcal{H}_0^{\ast}}\biggl (\frac 1 2I+\mathcal{K}^{\mathcal{H}_0^{\ast}}\biggr )^{-1}g=\partial_{\tilde\lambda}\nabla
\partial_\lambda^2
\mathcal{S}_{-\lambda-\tilde\lambda}^{\mathcal{H}_0^\ast}\bar h.$$
Setting $\tilde\lambda=\lambda$ we see that
 \begin{eqnarray}
\nabla \partial_\lambda^3 \mathcal{S}_{-2\lambda}^{\mathcal{H}_0^\ast}\bar h&=&-\partial_\lambda\nabla
\mathcal{D}_{-\lambda}^{\mathcal{H}_0^{\ast}}\biggl (\frac 1 2I+\mathcal{K}^{\mathcal{H}_0^{\ast}}\biggr )^{-1}g\notag\\
&=&\partial_\lambda\nabla
\mathcal{D}_{-\lambda}^{\mathcal{H}_0^{\ast}}\biggl (\frac 1 2I+\mathcal{K}^{\mathcal{H}_0^{\ast}}\biggr )^{-1}\partial_\lambda^2
\mathcal{S}_{-\lambda}^{\mathcal{H}_0^\ast}\bar h.
\end{eqnarray}
But $\mathcal{D}_{-\lambda}^{\mathcal{H}_0^{\ast}}=(\mathcal{S}_{-\lambda}^{\mathcal{H}_0^\ast}\overline{\partial_{\nu_0}})$ where $\overline{\partial_{\nu_0}}$ denotes
the conjugate exterior co-normal differentiation associated to $\mathcal{H}_0$. Thus
$$\mbox{adj}(\nabla\partial_\lambda \mathcal{D}_{-\lambda}^{\mathcal{H}_0^{\ast}})=(\partial_{\nu_0}\partial_\lambda \mathcal{S}_\lambda^{\mathcal{H}_0}\nabla).$$
In particular, using this we see that $|I_{11}|$ equals
 \begin{eqnarray*}
 \quad\biggl |\int_{\lambda_1}^{\lambda_2}\int_{\mathbb R^{n+1}}\biggl(\frac 1 2I+\mathcal{K}^{\mathcal{H}_0^{\ast}}\biggr )^{-1}\partial_\lambda^2 \mathcal{S}_{-\lambda}^{\mathcal{H}_0^\ast}\bar h\biggl ( (\partial_{\nu_0}\partial_\lambda \mathcal{S}_\lambda^{\mathcal{H}_0}\nabla)\cdot\varepsilon \nabla \mathcal{S}_\lambda^{\mathcal{H}_1,\eta}f\biggr )\, \lambda^2dxdtd\lambda\biggr |.
\end{eqnarray*}
Hence
 \begin{eqnarray}
 |I_{11}|&\leq&c|||\lambda\partial_\lambda^2 \mathcal{S}_{-\lambda}^{\mathcal{H}_0^\ast}h|||_+|||\lambda^2(\nabla\partial_\lambda \mathcal{S}_\lambda^{\mathcal{H}_0}\nabla)\cdot\varepsilon \nabla \mathcal{S}_\lambda^{\mathcal{H}_1,\eta}f|||_+\notag\\
 &\leq&c|||\lambda^2(\nabla\partial_\lambda \mathcal{S}_\lambda^{\mathcal{H}_0}\nabla)\cdot\varepsilon \nabla \mathcal{S}_\lambda^{\mathcal{H}_1,\eta}f|||_+
\end{eqnarray}
by the estimate stated in Definition \ref{blayer+} $(ix)$ applied to $\frac 1 2I+\mathcal{K}^{\mathcal{H}_0^\ast}$,  and \eqref{keyestinteda} applied to $\lambda\partial_\lambda^2 \mathcal{S}_{-\lambda}^{\mathcal{H}_0^\ast}h$. Hence it remains to estimate
$|||\lambda^2(\nabla\partial_\lambda \mathcal{S}_\lambda^{\mathcal{H}_0}\nabla)\cdot\varepsilon \nabla \mathcal{S}_\lambda^{\mathcal{H}_1,\eta}f|||_+$. However, arguing analogous to the argument below
Lemma 7.11 in \cite{AAAHK} it is easily seen that
 \begin{eqnarray}
&&|||\lambda^2(\nabla\partial_\lambda \mathcal{S}_\lambda^{\mathcal{H}_0}\nabla)\cdot\varepsilon \nabla \mathcal{S}_\lambda^{\mathcal{H}_1,\eta}f|||_+^2\notag\\
&\leq &c\varepsilon_0
|||\lambda\nabla\partial_\lambda \mathcal{S}_\lambda^{\mathcal{H}_1,\eta}f|||_+^2\notag\\
&&+c\sum_{k=-\infty}^\infty\int_{2^k}^{2^{k+1}}\int_{\mathbb R^{n+1}}
|\lambda^2(\nabla\partial_\lambda \mathcal{S}_\lambda^{\mathcal{H}_0}\nabla)\cdot\varepsilon \nabla \mathcal{S}_{2^{k-1}}^{\mathcal{H}_1,\eta}f(x,t)|^2\, \frac {dxdtd\lambda}{\lambda}.
\end{eqnarray}
Next, using that $(\partial_\lambda \mathcal{S}_\lambda^{\mathcal{H}_0}\nabla)\cdot\varepsilon \nabla \mathcal{S}_{2^{k-1}}^{\mathcal{H}_1,\eta}f$ is, for fixed $k$, a
solution to the operator $\mathcal{H}_0$ we see,  by now standard applications of energy estimates, see Lemma \ref{le1--}, that
 \begin{eqnarray}
&&\sum_{k=-\infty}^\infty\int_{2^k}^{2^{k+1}}\int_{\mathbb R^{n+1}}
|\lambda^2(\nabla\partial_\lambda \mathcal{S}_\lambda^{\mathcal{H}_0}\nabla)\cdot\varepsilon \nabla \mathcal{S}_{2^{k-1}}^{\mathcal{H}_1,\eta}f(x,t)|^2\, \frac {dxdtd\lambda}{\lambda}\notag\\
&\leq&c \sum_{k=-\infty}^\infty\int_{2^k}^{2^{k+1}}\int_{\mathbb R^{n+1}}
|\lambda(\partial_\lambda \mathcal{S}_\lambda^{\mathcal{H}_0}\nabla)\cdot\varepsilon \nabla \mathcal{S}_{2^{k-1}}^{\mathcal{H}_1,\eta}f(x,t)|^2\, \frac {dxdtd\lambda}{\lambda}.
\end{eqnarray}
Putting these estimates together, and again using a parabolic version of Lemma 7.11 in \cite{AAAHK} we can conclude that
 \begin{eqnarray}
|||\lambda^2(\nabla\partial_\lambda \mathcal{S}_\lambda^{\mathcal{H}_0}\nabla)\cdot\varepsilon \nabla \mathcal{S}_\lambda^{\mathcal{H}_1,\eta}f|||_+^2&\leq &c\varepsilon_0
|||\lambda\nabla\partial_\lambda \mathcal{S}_\lambda^{\mathcal{H}_1,\eta}f|||_+^2\notag\\
&&+c|||\lambda(\partial_\lambda \mathcal{S}_{\lambda}^{\mathcal{H}_0}\nabla) \cdot\varepsilon \nabla \mathcal{S}_\lambda^{\mathcal{H}_1,\eta}f|||_+^2.
\end{eqnarray}
Hence, summarizing our estimates we see that
$$|I|\leq c\varepsilon_0A_+^{\mathcal{H}_1,\eta}(f)+|I_{11}|\leq c\varepsilon_0A_+^{\mathcal{H}_1,\eta}(f)+c|||\lambda(\partial_\lambda \mathcal{S}_{\lambda}^{\mathcal{H}_0}\nabla) \cdot\varepsilon \nabla \mathcal{S}_\lambda^{\mathcal{H}_1,\eta}f|||_+. $$
Hence Lemma \ref{telem3} now follows by an application of Lemma \ref{telem2}. This completes the proof of Lemma \ref{telem3}.

    \section{Proof of Theorem \ref{thper2} and Corollary \ref{thper3}}\label{sec10}

    In this section we prove Theorem \ref{thper2} and Corollary \ref{thper3}. As in the statement of Theorem \ref{thper2}, and as in Section \ref{sec8}
    and Section \ref{sec9}, we throughout this section consider two operators  $\mathcal{H}_0=\partial_t-\div A^0\nabla$, $\mathcal{H}_1=\partial_t-\div A^1\nabla$. We will assume  \eqref{keyseauu} and recall that the constant $\Gamma_0$ appears in \eqref{keyseauu}.  We let ${\bf \varepsilon}$ be as in \eqref{keyestaahmo}, we assume
     \eqref{keyestaahmeddo} and we let $\tilde{\bf \varepsilon}$ be as introduced in \eqref{keyestaahmouu}. In the following we will use the notation
          \begin{eqnarray}\label{keyestint-ex+edap+}
\Phi^{\mathcal{H}_1,\eta}(f):=|||\lambda \partial_\lambda^2  \mathcal{S}_\lambda^{\mathcal{H}_1,\eta}f|||+\sup_{\lambda \neq 0} ||\partial_\lambda  \mathcal{S}_\lambda^{\mathcal{H}_1,\eta}f||_2
     \end{eqnarray}
     and
      \begin{eqnarray}\label{iest1haf}
A^{\mathcal{H}_1,\eta}(f)&:=& A_+^{\mathcal{H}_1,\eta}(f)+A_-^{\mathcal{H}_1,\eta}(f)\notag\\
&=&|||\lambda \nabla \partial_\lambda \mathcal{S}_\lambda^{\mathcal{H}_1,\eta}f|||+\sup_{\lambda\neq 0}||\nabla \mathcal{S}_\lambda^{\mathcal{H}_1,\eta}f||_2+\sup_{\lambda\neq 0}||H_tD^t_{1/2}\mathcal{S}_\lambda^{\mathcal{H}_1,\eta}f||_2\notag\\
&&+||N_\ast^{+}(\P_\lambda \partial_\lambda \mathcal{S}_\lambda^{\mathcal{H}_1,\eta}f)||_2+||N_\ast^{-}(\P_\lambda \partial_\lambda \mathcal{S}_\lambda^{\mathcal{H}_1,\eta}f)||_2+||f||_2.
\end{eqnarray}
Note that by the results of Section \ref{sec3} we always have, a priori, that $\Phi^{\mathcal{H}_1,\eta}(f)<\infty$ and $A^{\mathcal{H}_1,\eta}(f)<\infty$ whenever $f\in C_0^\infty(\mathbb R^{n+1},\mathbb C)$. Our proof of Theorem \ref{thper2} is based on the following lemma the proof of which is given below.
    \begin{lemma}\label{finall} Assume \eqref{keyseauu}. Then there exists a constant $c$, depending at most
     on $n$, $\Lambda$,  the De Giorgi-Moser-Nash constants and $\Gamma_0$, such that
    \begin{eqnarray}\label{iest1}
\Phi^{\mathcal{H}_1,\eta}(f) \leq c\varepsilon_0A^{\mathcal{H}_1,\eta}(f)+c||f||_2,
     \end{eqnarray}
   whenever $\eta\in (0,1/10)$ and $f\in C_0^\infty(\mathbb R^{n+1},\mathbb C)$.\end{lemma}

   \subsection{Proof of Theorem \ref{thper2}}

      The proof of Lemma \ref{finall} is given below. We here use Lemma \ref{finall} to complete the proof of Theorem \ref{thper2}. Using Lemma \ref{lemsl1} and Lemma \ref{lemsl1c} we first see that
    \begin{eqnarray}\label{keyest+aaaca}
    |||\lambda \nabla \partial_\lambda \mathcal{S}_\lambda^{\mathcal{H}_1,\eta}f|||\leq c\bigl (\Phi^{\mathcal{H}_1,\eta}(f)+||f||_2\bigr),
\end{eqnarray}
     and
         \begin{eqnarray}\label{keyest+aaacafa}
    \sup_{\lambda\neq 0}||\nabla \mathcal{S}_\lambda^{\mathcal{H}_1,\eta}f||_2&\leq&c\bigl (\Phi^{\mathcal{H}_1,\eta}(f)+||N_\ast^{\pm}(\P_\lambda \partial_\lambda \mathcal{S}_\lambda^{\mathcal{H}_1,\eta}f)||_2\bigr)\notag\\
    &&+c\bigl(\sup_{\lambda\neq 0}||\partial_\lambda \mathcal{S}_\lambda^{\mathcal{H}_1,\eta}f||_2+||f||_2\bigr),\notag\\
    \sup_{\lambda\neq 0}||H_tD^t_{1/2}\mathcal{S}_\lambda^{\mathcal{H}_1,\eta}f||_2&\leq&c\bigl (\Phi^{\mathcal{H}_1,\eta}(f)+||N_\ast^{\pm}(\P_\lambda \partial_\lambda \mathcal{S}_\lambda^{\mathcal{H}_1,\eta}f)||_2\bigr)\notag\\
    &&+c\bigl(\sup_{\lambda\neq 0}||\partial_\lambda \mathcal{S}_\lambda^{\mathcal{H}_1,\eta}f||_2+||f||_2\bigr).
    \end{eqnarray}
Hence, using Lemma \ref{finall} and hiding terms, we first see that,
    \begin{eqnarray}\label{keyest+aaacaff}
    |||\lambda \nabla \partial_\lambda \mathcal{S}_\lambda^{\mathcal{H}_1,\eta}f|||\leq c\varepsilon_0\bigl (A^{\mathcal{H}_1,\eta}(f)-|||\lambda \nabla \partial_\lambda \mathcal{S}_\lambda^{\mathcal{H}_1,\eta}f|||\bigr )+c||f||_2.
     \end{eqnarray}
     Using  Lemma \ref{finall} again, as well as \eqref{keyest+aaacaff}, we can again hide terms and conclude that
              \begin{eqnarray}\label{keyest+aaacafb}
    \sup_{\lambda\neq 0}||\nabla \mathcal{S}_\lambda^{\mathcal{H}_1,\eta}f||_2&\leq&c\bigl (||N_\ast^{\pm}(\P_\lambda \partial_\lambda \mathcal{S}_\lambda^{\mathcal{H}_1,\eta}f)||_2+\sup_{\lambda\neq 0}||\partial_\lambda \mathcal{S}_\lambda^{\mathcal{H}_1,\eta}f||_2+||f||_2\bigr),\\
    \sup_{\lambda\neq 0}||H_tD^t_{1/2}\mathcal{S}_\lambda^{\mathcal{H}_1,\eta}f||_2&\leq&c\bigl (|N_\ast^{\pm}(\P_\lambda \partial_\lambda \mathcal{S}_\lambda^{\mathcal{H}_1,\eta}f)||_2+\sup_{\lambda\neq 0}||\partial_\lambda \mathcal{S}_\lambda^{\mathcal{H}_1,\eta}f||_2+||f||_2\bigr).\notag
    \end{eqnarray}
    In particular, putting the estimates in \eqref{keyest+aaacafb} in \eqref{keyest+aaacaff} we see that
        \begin{eqnarray}\label{keyest+aaacaffh}
    |||\lambda \nabla \partial_\lambda \mathcal{S}_\lambda^{\mathcal{H}_1,\eta}f|||&\leq& c\varepsilon_0\bigl (||N_\ast^{\pm}(\P_\lambda \partial_\lambda \mathcal{S}_\lambda^{\mathcal{H}_1,\eta}f)||_2+\sup_{\lambda\neq 0}||\partial_\lambda \mathcal{S}_\lambda^{\mathcal{H}_1,\eta}f||_2\bigr )\notag\\
    &&+c||f||_2.
     \end{eqnarray}
     Using Lemma \ref{finall} once more, and the above deductions, we have
             \begin{eqnarray}\label{keyest+aaacaffhha}
    \sup_{\lambda\neq 0}||\partial_\lambda \mathcal{S}_\lambda^{\mathcal{H}_1,\eta}f||_2&\leq& c\varepsilon_0\bigl (||N_\ast^{\pm}(\P_\lambda \partial_\lambda \mathcal{S}_\lambda^{\mathcal{H}_1,\eta}f)||_2+\sup_{\lambda\neq 0}||\partial_\lambda \mathcal{S}_\lambda^{\mathcal{H}_1,\eta}f||_2\bigr )\notag\\
    &&+c||f||_2.
     \end{eqnarray}
     As $f\in C_0^\infty(\mathbb R^{n+1},\mathbb C)$ the support of $f$ is contained in some cube $Q\subset\mathbb R^{n+1}$. Hence, using
     \eqref{keyest+aaacaffhha}, Lemma \ref{lemsl1++} $(iv)$ and taking the supremum over all $f\in C_0^\infty(Q,\mathbb C)$ with $||f1_Q||_2=1$, we see that
                \begin{eqnarray*}\label{keyest+aaacaffhhal}
    \sup_{\lambda\neq 0}||\partial_\lambda \mathcal{S}_\lambda^{\mathcal{H}_1,\eta}||_{L^2(Q,\mathbb C)\to L^2(\mathbb R^{n+1},\mathbb C)}\leq c\bigl (1+\varepsilon_0
    \sup_{\lambda\neq 0}||\partial_\lambda \mathcal{S}_\lambda^{\mathcal{H}_1,\eta}||_{L^2(Q,\mathbb C)\to L^2(\mathbb R^{n+1},\mathbb C)}\bigr ).
     \end{eqnarray*}
     Hence, using Lemma \ref{smooth1} $(viii)$ we can conclude that
                     \begin{eqnarray}\label{keyest+aaacaffhhal+}
    \sup_{\lambda\neq 0}||\partial_\lambda \mathcal{S}_\lambda^{\mathcal{H}_1,\eta}||_{L^2(Q,\mathbb C)\to L^2(\mathbb R^{n+1},\mathbb C)}\leq c,
     \end{eqnarray}
     uniformly with respect to $Q$. Thus, using Lemma \ref{smooth1} $(v)$, and first letting $l(Q)\to\infty$, then $\eta\to 0$, we can conclude that
                     \begin{eqnarray}\label{keyest+aaacaffhhala+}
    \sup_{\lambda\neq 0}||\partial_\lambda \mathcal{S}_\lambda^{\mathcal{H}_1}||_{2\to 2}\leq c.
     \end{eqnarray}
     In addition, using \eqref{keyest+aaacaffhhal+}, Lemma \ref{lemsl1++} and a limiting argument as $l(Q)\to\infty$, we have that
                  \begin{eqnarray}\label{keyest+aaacaffhhagg}
\sup_{\lambda_0\geq 0}||N_\ast^{\pm}(\P_{\lambda} \partial_\lambda \mathcal{S}_{\lambda\pm\lambda_0}^{\mathcal{H}_1,\eta}f)||_2\leq c||f||_2,
     \end{eqnarray}
     whenever $f\in C_0^\infty(\mathbb R^{n+1},\mathbb C)$. Putting all these conclusions together, and using Lemma \ref{smooth1}, we can conclude that there exists $\varepsilon_0$, depending at most
     on $n$, $\Lambda$,  the De Giorgi-Moser-Nash constants and $\Gamma_0$,  such that if
    $$||A^1-A^0||_\infty\leq\varepsilon_0,$$
    then
             \begin{eqnarray}\label{keyest+aaa}
 \quad|||\lambda \nabla\partial_\lambda  \mathcal{S}_\lambda^{\mathcal{H}_1}f|||+\sup_{\lambda \neq 0}||\nabla \mathcal{S}_\lambda^{\mathcal{H}_1}f||_{2}+
 \sup_{\lambda \neq 0}||H_tD^t_{1/2} \mathcal{S}_\lambda^{\mathcal{H}_1}f||_{2}\leq c||f||_2,
     \end{eqnarray}
     whenever $f\in C_0^\infty(\mathbb R^{n+1},\mathbb C)$ and for some constant $c$ having the dependence stated in Lemma \ref{finall}. Using \eqref{keyest+aaa} it follows that the statements in  Definition \ref{blayer+} $(i)-(vi)$ hold for $\mathcal{H}_1$ and for some constant $\Gamma_1$, the statements for $\mathcal{H}_1^\ast$  follow by duality. $\Gamma_1$ depends at most
     on $n$, $\Lambda$,  the De Giorgi-Moser-Nash constants and $\Gamma_0$. Furthermore, using this result and using
     \eqref{keyest+aaa}, Lemma \ref{lemsl1++},  Lemma \ref{trace4}, Lemma \ref{trace5}, and Lemma \ref{trace7-}, we can conclude there exist  operators $\mathcal{K}^{\mathcal{H}_1}$, $\tilde {\mathcal{K}}^{\mathcal{H}_1}$,
     $\nabla_{||} \mathcal{S}_\lambda^{\mathcal{H}_1}|_{\lambda=0}$, $H_tD^t_{1/2}\mathcal{S}_\lambda^{\mathcal{H}_1}|_{\lambda=0}$, in the sense of
      Lemma \ref{trace4}, Lemma \ref{trace5}, and Lemma \ref{trace7-}. Furthermore, all these operators are bounded operators on $L^2(\mathbb R^{n+1},\mathbb C)$. Hence to complete the proof of Theorem \ref{thper2} the statements in  Definition \ref{blayer+} $(viii)-(xiii)$ for $\mathcal{H}_1$ remain to be verified. To do this we need the following lemma.
     \begin{lemma}\label{analpert}  Assume \eqref{keyseauu}. There exists a constant $c$, depending at most
     on $n$, $\Lambda$,  such that if
    $$||A^1-A^0||_\infty\leq\varepsilon_0,$$
    then
     \begin{eqnarray}\label{keyest+aaaed}
     ||\mathcal{K}^{\mathcal{H}_0}-\mathcal{K}^{\mathcal{H}_1}||_{2\to 2}+||\tilde {\mathcal{K}}^{\mathcal{H}_0}-\tilde {\mathcal{K}}^{\mathcal{H}_1}||_{2\to 2}&\leq& c\varepsilon_0,\notag\\
     ||\nabla_{||} \mathcal{S}_\lambda^{\mathcal{H}_0}|_{\lambda=0}-\nabla_{||} \mathcal{S}_\lambda^{\mathcal{H}_1}|_{\lambda=0}||_{2\to 2}
     &\leq& c\varepsilon_0,\notag\\
     ||H_tD^t_{1/2}\mathcal{S}_\lambda^{\mathcal{H}_0}|_{\lambda=0}-H_tD^t_{1/2}\mathcal{S}_\lambda^{\mathcal{H}_1}|_{\lambda=0}||_{2\to 2}&\leq& c\varepsilon_0.
    \end{eqnarray}
     \end{lemma}

     The short proof of Lemma \ref{analpert} is for completion included below. We here use Lemma \ref{analpert} to complete the proof of Theorem \ref{thper2} by verifying the statements in  Definition \ref{blayer+} $(viii)-(xiii)$ for $\mathcal{H}_1$. Let, for $\tau\in [0,1]$,
     $\mathcal{H}_\tau$ be the operator which has coefficients $(1-\tau)A^0+\tau A^1$ and let $\mathcal{K}^{\mathcal{H}_\tau}$, $\tilde {\mathcal{K}}^{\mathcal{H}_\tau}$,
     $\nabla_{||} \mathcal{S}_\lambda^{\mathcal{H}_\tau}|_{\lambda=0}$, $H_tD^t_{1/2}\mathcal{S}_\lambda^{\mathcal{H}_\tau}|_{\lambda=0}$, be the boundary operators associated to $\mathcal{H}_\tau$ and in the sense of Lemma \ref{trace4}. Let
      $\mathcal{O}_\tau$ denote any of these operators. Using  Lemma \ref{trace4} we see that any such operator $\mathcal{O}_\tau$ is a (uniformly in $\tau$) bounded operator on $L^2(\mathbb R^{n+1},\mathbb C)$.  By
      Lemma \ref{analpert} $\tau\to \mathcal{O}_\tau$ is continuous in the $2\to 2$-norm. Furthermore, by assumption
      \begin{eqnarray}\label{keyest+aaaed}\mp\frac 1 2I+\mathcal{K}^{\mathcal{H}_0}:L^2(\mathbb R^{n+1},\mathbb C)&\to& L^2(\mathbb R^{n+1},\mathbb C),\notag\\
       \pm\frac 1 2I+\tilde{\mathcal{K}}^{\mathcal{H}_0}:L^2(\mathbb R^{n+1},\mathbb C)&\to& L^2(\mathbb R^{n+1},\mathbb C),\notag\\
       \mathcal{S}^{\mathcal{H}_0}_0:=\mathcal{S}^{\mathcal{H}_0}_\lambda|_{\lambda=0}: L^2(\mathbb R^{n+1},\mathbb C)&\to&  \mathbb H(\mathbb R^{n+1},\mathbb C),
    \end{eqnarray}
    are all bounded, invertible and they satisfy, by \eqref{keyseauu}, the quantitative estimates stated in Definition \ref{blayer+}. Hence, using this, the above facts, and the method of continuity we can conclude the invertibility of
        \begin{eqnarray}\label{keyest+aaaed}\pm\frac 1 2I+\mathcal{K}^{\mathcal{H}_1}:L^2(\mathbb R^{n+1},\mathbb C)&\to& L^2(\mathbb R^{n+1},\mathbb C),\notag\\
       \pm\frac 1 2I+\tilde{\mathcal{K}}^{\mathcal{H}_1}:L^2(\mathbb R^{n+1},\mathbb C)&\to& L^2(\mathbb R^{n+1},\mathbb C),\notag\\
       \mathcal{S}^{\mathcal{H}_1}_0:=\mathcal{S}^{\mathcal{H}_1}_\lambda|_{\lambda=0}: L^2(\mathbb R^{n+1},\mathbb C)&\to&  \mathbb H(\mathbb R^{n+1},\mathbb C),
    \end{eqnarray}
    In particular, we can conclude the validity of the statements in Definition \ref{blayer+} $(viii)-(xiii)$ also for $\mathcal{H}_1$. This completes the proof of Theorem \ref{thper2} modulo  Lemma \ref{finall} and Lemma \ref{analpert}. The proof of these lemmas are given below.

        \subsection{Proof of Corollary \ref{thper3}} By Theorem \ref{thper2} we have that  if
    $$||A^1-A^0||_\infty\leq\varepsilon_0,$$
    then there exists a constant $\Gamma_1$,  depending at most
     on $n$, $\Lambda$,  the De Giorgi-Moser-Nash constants and  $\Gamma_0$, such that
     \begin{eqnarray}\label{afa+lu}
  &&\mbox{$\mathcal{H}_1$,  $\mathcal{H}_1^\ast$, have bounded, invertible and good layer potentials}\notag\\
      &&\mbox{in the sense of Definition \ref{blayer+}, with constant $\Gamma_1$}.
     \end{eqnarray}
     This implies, as discussed in Remark \ref{resea}, that we also have the for $(D2)$ relevant quantitative estimates of the double layer potential
     $\mathcal{D}^{\mathcal{H}_1}$. This, \eqref{afa+lu}, Lemma \ref{trace5} and the uniqueness result for $(D2)$ in  Lemma \ref{trace7}  prove that
      Corollary \ref{thper3} follows  from Theorem \ref{thper2} in the case of $(D2)$. Lemma \ref{trace7-}, and the uniqueness result for $(N2)$ in Lemma \ref{trace7++} prove that
      Corollary \ref{thper3} follows  from Theorem \ref{thper2} in the case of $(N2)$. Finally, Lemma \ref{trace7-}, and the uniqueness result for $(R2)$ in Lemma \ref{trace7} prove that
      Corollary \ref{thper3} follows  from Theorem \ref{thper2} in the case of $(R2)$.

  \subsection{Proof of Lemma \ref{finall}} Having developed many of the key estimates in the previous sections, at this stage the remaining arguments become quite similar to the corresponding arguments in \cite{AAAHK}. Because of this we will, at instances, be a bit brief. The proof of Lemma \ref{finall} is based on a perturbation argument using a representation formula for the difference
    \begin{eqnarray}\label{iest1mod++pre}
\quad\partial_\lambda \mathcal{S}_\lambda^{\mathcal{H}_1,\eta}f(x,t)- \partial_\lambda\mathcal{S}_\lambda^{\mathcal{H}_0,\eta}f(x,t)=\mathcal{H}_0^{-1}\div \varepsilon\nabla D_{n+1}
\mathcal{S}_\cdot^{\mathcal{H}_1,\eta}f(x,t).
     \end{eqnarray}
     We will only supply the proofs of Lemma \ref{finall} in the case of
     $$|||\lambda \partial_\lambda^2  \mathcal{S}_\lambda^{\mathcal{H}_1,\eta}f|||_{+},\ \sup_{\lambda > 0} ||\partial_\lambda  \mathcal{S}_\lambda^{\mathcal{H}_1,\eta}f||_2,$$ as the estimates of the remaining terms/cases in the definition of $\Phi^{\mathcal{H}_1,\eta}(f)$ are similar. To start the estimate of $|||\lambda \partial_\lambda^2  \mathcal{S}_\lambda^{\mathcal{H}_1,\eta}f|||_{+}$ we let
     $$\mbox{$\Psi\in C_0^\infty(\mathbb R^{n+2}_+,\mathbb C)$, $|||\Psi|||_+\leq 1$,
$\Psi_\delta(x,t,\lambda )=\varphi_\delta\ast \Psi(x,t,\cdot )(\lambda)$,}$$
for $\delta>0$ sufficiently small. To estimate $|||\lambda \partial_\lambda^2  \mathcal{S}_\lambda^{\mathcal{H}_1,\eta}f|||_{+}$
we intend to bound
   \begin{eqnarray}\label{iest1mod++pre1}
\int_0^\infty\int_{\mathbb R^{n+1}}\lambda\partial_\lambda^2\mathcal{S}_\lambda^{\mathcal{H}_1,\eta}f(x,t)\overline{\Psi_\delta(x,t,\lambda )}\, \frac {dxdtd\lambda}\lambda.
     \end{eqnarray}
Using \eqref{keyseauu} and Lemma \ref{smooth1} $(vii)$ we see that to estimate the expression in \eqref{iest1mod++pre1} we only have to bound
   \begin{eqnarray*}\label{iest1mod++pre1+}
\int_0^\infty\int_{\mathbb R^{n+1}}\lambda\partial_\lambda\bigl (\partial_\lambda\mathcal{S}_\lambda^{\mathcal{H}_1,\eta}f(x,t)-\partial_\lambda\mathcal{S}_\lambda^{\mathcal{H}_0,\eta}f(x,t)\bigr)\overline{\Psi_\delta(x,t,\lambda )}\, \frac {dxdtd\lambda}\lambda.
     \end{eqnarray*}
     Furthermore, using \eqref{iest1mod++pre} we see that it suffices to bound
     \begin{eqnarray*}\label{iest1mod++pre1+a}
\mathcal{E}:=\int_{\mathbb R^{n+2}}\varepsilon(y,s)\nabla \partial_\lambda \mathcal{S}_\lambda^{\mathcal{H}_1,\eta}f(y,s)\cdot\overline{\nabla (\mathcal{H}_0^\ast)^{-1}D_{n+1}\Psi_\delta(y,s,\lambda)}
dyds d\lambda.
     \end{eqnarray*}
We intend to prove that
        \begin{eqnarray}\label{iest1mod++pre2+aa-}
\mathcal{E}\leq c\varepsilon_0A^{\mathcal{H}_1,\eta}(f)+c||f||_2.
     \end{eqnarray}
To start the proof of \eqref{iest1mod++pre2+aa-} we note that
\begin{eqnarray}\label{splitrep}
\nabla(\mathcal{H}_0^\ast)^{-1}D_{n+1}\Psi_\delta(y,s,\lambda)&=&\int\nabla\partial_\lambda \mathcal{S}_{\lambda-\lambda'}^{\mathcal{H}_0^\ast,\delta}(\Psi(\cdot,\cdot,\lambda'))(y,s)\, d\lambda'.
\end{eqnarray}
Furthermore, using this and following \cite{FJK} and \cite{AAAHK} we first write
\begin{eqnarray}
\nabla(\mathcal{H}_0^\ast)^{-1}D_{n+1}\Psi_\delta(y,s,\lambda)&=&\int_{\lambda'>2|\lambda|}\nabla\partial_\lambda \mathcal{S}_{\lambda-\lambda'}^{\mathcal{H}_0^\ast,\delta}(\Psi(\cdot,\cdot,\lambda'))(y,s)\, d\lambda'\notag\\
&&+\int_{\lambda'\leq2|\lambda|}\nabla\partial_\lambda \mathcal{S}_{\lambda-\lambda'}^{\mathcal{H}_0^\ast,\delta}(\Psi(\cdot,\cdot,\lambda'))(y,s)\, d\lambda',
\end{eqnarray}
and then
\begin{eqnarray}
\nabla(\mathcal{H}_0^\ast)^{-1}D_{n+1}\Psi_\delta(y,s,\lambda)&=&{e}_1(y,s,\lambda)+e_2(y,s,\lambda)\notag\\
&&+e_3(y,s,\lambda)+e_4(y,s,\lambda)\notag\\
&&+e_5(y,s,\lambda),
\end{eqnarray}
where
\begin{eqnarray*}
e_1(y,s,\lambda)&=&\int_{\lambda'>2|\lambda|}\biggl (\nabla\partial_\lambda \mathcal{S}_{\lambda-\lambda'}^{\mathcal{H}_0^\ast,\delta}(\Psi(\cdot,\cdot,\lambda'))(y,s)\notag\\
&&\quad\quad\quad\quad-
\nabla\partial_\lambda \mathcal{S}_{\lambda-\lambda'}^{\mathcal{H}_0^\ast,\delta}|_{\lambda=0}(\Psi(\cdot,\cdot,\lambda'))(y,s)\biggr )\, d\lambda',\notag\\
e_2(y,s,\lambda)&=&\int_{\lambda'>2|\lambda|}\nabla\partial_\lambda \mathcal{S}_{\lambda-\lambda'}^{\mathcal{H}_0^\ast,\delta}|_{\lambda=0}(\Psi(\cdot,\cdot,\lambda'))(y,s)\, d\lambda',\notag\\
e_3(y,s,\lambda)&=&\int_{\lambda'\leq2|\lambda|}\biggl (1-\biggl (\frac {|\lambda|}{|\lambda'|}\biggr )^{1/2}\biggr )\nabla\partial_\lambda \mathcal{S}_{\lambda-\lambda'}^{\mathcal{H}_0^\ast,\delta}(\Psi(\cdot,\cdot,\lambda'))(y,s)\, d\lambda',\notag\\
e_4(y,s,\lambda)&=&\int\biggl (\frac {|\lambda|}{|\lambda'|}\biggr )^{1/2}\nabla\partial_\lambda \mathcal{S}_{\lambda-\lambda'}^{\mathcal{H}_0^\ast,\delta}(\Psi(\cdot,\cdot,\lambda'))(y,s)\, d\lambda',\notag\\
e_5(y,s,\lambda)&=&\int_{\lambda'>2|\lambda|}\biggl (\frac {|\lambda|}{|\lambda'|}\biggr )^{1/2}\nabla\partial_\lambda \mathcal{S}_{\lambda-\lambda'}^{\mathcal{H}_0^\ast,\delta}(\Psi(\cdot,\cdot,\lambda'))(y,s)\, d\lambda'.
\end{eqnarray*}
Then, using this decomposition we see that
\begin{eqnarray}
\mathcal{E}=\mathcal{E}_1+\mathcal{E}_2+\mathcal{E}_3+\mathcal{E}_4+\mathcal{E}_5,
\end{eqnarray}
where
\begin{eqnarray}
\mathcal{E}_1&=&\int_{\mathbb R^{n+2}}\varepsilon(y,s)\nabla \partial_\lambda \mathcal{S}_\lambda^{\mathcal{H}_1,\eta}f(y,s)\cdot \overline{e_1(y,s,\lambda)}\, dyds d\lambda,\notag\\
\mathcal{E}_2&=&\int_{\mathbb R^{n+2}}\varepsilon(y,s)\nabla \partial_\lambda \mathcal{S}_\lambda^{\mathcal{H}_1,\eta}f(y,s)\cdot
\overline{e_2(y,s,\lambda)}\,  dyds d\lambda,\notag\\
\mathcal{E}_3&=&\int_{\mathbb R^{n+2}}\varepsilon(y,s)\nabla \partial_\lambda \mathcal{S}_\lambda^{\mathcal{H}_1,\eta}f(y,s)\cdot
\overline{e_3(y,s,\lambda)}\, dyds d\lambda,\notag\\
\mathcal{E}_4&=&\int_{\mathbb R^{n+2}}\varepsilon(y,s)\nabla \partial_\lambda \mathcal{S}_\lambda^{\mathcal{H}_1,\eta}f(y,s)\cdot
\overline{e_4(y,s,\lambda)}\, dyds d\lambda,\notag\\
\mathcal{E}_5&=&\int_{\mathbb R^{n+2}}\varepsilon(y,s)\nabla \partial_\lambda \mathcal{S}_\lambda^{\mathcal{H}_1,\eta}f(y,s)\cdot
\overline{e_5(y,s,\lambda)}\, dyds d\lambda.
\end{eqnarray}
Using \eqref{splitrep} we see that $\mathcal{E}_4$ equals
\begin{eqnarray*}
\int_{\mathbb R^{n+2}}|\lambda|^{1/2}\varepsilon(y,s)\nabla \partial_\lambda \mathcal{S}_\lambda^{\mathcal{H}_1,\eta}f(y,s)\cdot
\overline{\nabla(\mathcal{H}_0^\ast)^{-1}(D_{n+1}(\varphi_\delta\ast(\Psi/\sqrt{\lambda'})))(y,s,\lambda)}\, dyds d\lambda.
\end{eqnarray*}
In particular,
\begin{eqnarray}
|\mathcal{E}_4|&\leq& c\varepsilon_0|||\lambda \nabla \partial_\lambda \mathcal{S}_\lambda^{\mathcal{H}_1,\eta}f|||_+\notag\\
&&\quad\quad\times\biggl (\int_{\mathbb R^{n+2}}|\nabla(\mathcal{H}_0^\ast)^{-1}(D_{n+1}(\varphi_\delta\ast(\Psi/\sqrt{\lambda'})))(y,s,\lambda)|^2\, dyds d\lambda\biggr )^{1/2}\notag\\
&\leq& c\varepsilon_0|||\lambda \nabla \partial_\lambda \mathcal{S}_\lambda^{\mathcal{H}_1,\eta}f|||_+\times\biggl (\int_{\mathbb R^{n+2}}|(\varphi_\delta\ast(\Psi/\sqrt{\lambda'})))(y,s,\lambda)|^2\, dyds d\lambda\biggr )^{1/2}\notag\\
&\leq& \varepsilon_0|||\lambda \nabla \partial_\lambda \mathcal{S}_\lambda^{\mathcal{H}_1,\eta}f|||_+,
\end{eqnarray}
as $\nabla \mathcal{H}_0^{-1}\div:L^2(\mathbb R^{n+2},\mathbb C)\to L^2(\mathbb R^{n+2},\mathbb C)$, see  Lemma \ref{gara}, and by the properties of
$\Psi$. $\mathcal{E}_1$, $\mathcal{E}_2$, $\mathcal{E}_3$, and $\mathcal{E}_5$ remain to be estimated and to estimate  $\mathcal{E}_2$ is the heart of the matter. Indeed, we claim that
\begin{eqnarray}\label{dda}
|\mathcal{E}_1|+|\mathcal{E}_3|+|\mathcal{E}_5|\leq c\varepsilon_0|||\lambda \nabla \partial_\lambda \mathcal{S}_\lambda^{\mathcal{H}_1,\eta}f|||_+,
\end{eqnarray}
and we leave it to the reader to verify, by arguing as in the proof of Lemma 6.5 in \cite{AAAHK} and by using Hardy's inequality, that \eqref{dda} holds. We will here show how to control  $\mathcal{E}_2$ using Lemma \ref{telem2}. To estimate $\mathcal{E}_2$ we first note that $\mathcal{E}_2$ equals
\begin{eqnarray*}
\int_{\mathbb R^{n+2}}\varepsilon(y,s)\nabla \partial_\lambda \mathcal{S}_\lambda^{\mathcal{H}_1,\eta}f(y,s)\biggl (\int_{\lambda'>2|\lambda|}\overline{\nabla D_{n+1} \mathcal{S}_{-\lambda'}^{\mathcal{H}_0^\ast,\delta}(\Psi(\cdot,\cdot,\lambda'))(y,s)}\, d\lambda'\biggr )\,  dyds d\lambda
\end{eqnarray*}
which in turns equals
\begin{eqnarray*}
-\int_0^\infty\int_{\mathbb R^{n+1}}(\partial_\lambda \mathcal{S}_\lambda^{\mathcal{H}_0}\nabla)\cdot\varepsilon(y,s)\nabla
(\mathcal{S}_{\lambda/2}^{\mathcal{H}_1,\eta}f-\mathcal{S}_{-\lambda/2}^{\mathcal{H}_1,\eta}f)(x,t)\overline{\Psi_\delta(x,t,\lambda)}\, dxdtd\lambda.
\end{eqnarray*}
In the latter deduction we have used that $\partial_\lambda\mathcal{S}_\lambda^\eta$ does not, for $\eta>0$,  jump across the boundary. Using that $\Psi$ is compactly supported in
$\mathbb R^{n+2}_+$ we see, for $\delta$ small enough, that
$$\lambda^{-1/2}|\Psi_\delta(x,t,\lambda)|\leq c\int\varphi(\lambda-\lambda')|\Psi(x,t,\lambda')|\lambda'^{-1/2}\, d\lambda'.$$
Using this we see that
\begin{eqnarray*}
|\mathcal{E}_2|&\leq &|||\lambda(\partial_\lambda \mathcal{S}_{\lambda}^{\mathcal{H}_0}\nabla) \cdot\varepsilon \nabla \mathcal{S}_{-\lambda/2}^{\mathcal{H}_1,\eta}f|||_++|||\lambda(\partial_\lambda \mathcal{S}_{\lambda}^{\mathcal{H}_0}\nabla) \cdot\varepsilon \nabla \mathcal{S}_{\lambda/2}^{\mathcal{H}_1,\eta}f|||_+.
\end{eqnarray*}
Applying Lemma \ref{telem2} we can therefore conclude that
\begin{eqnarray*}
|\mathcal{E}_2|\leq c\varepsilon_0A^{\mathcal{H}_1,\eta}(f)+c||f||_2,
\end{eqnarray*}
and hence that \eqref{iest1mod++pre2+aa-} holds. This completes the estimate of $\mathcal{E}$ and hence the estimate of $|||\lambda \partial_\lambda^2  \mathcal{S}_\lambda^{\mathcal{H}_1,\eta}f|||_+$.

To start the estimate of  $\sup_{\lambda > 0} ||\partial_\lambda  \mathcal{S}_\lambda^{\mathcal{H}_1,\eta}f||_2$ we intend to prove that
    \begin{eqnarray}\label{iest1mod}
\sup_{0<\eta<10^{-10}}\sup_{\lambda > 0} ||\partial_\lambda  \mathcal{S}_\lambda^{\mathcal{H}_1,\eta}f||_2&\leq& c\varepsilon_0A^{\mathcal{H}_1,\eta}(f)+c||f||_2.
     \end{eqnarray}
     By elementary estimates it is easy to see that if $0\leq\lambda<4\eta$, then
         \begin{eqnarray}\label{iest1mod+}
|\partial_\lambda  \mathcal{S}_\lambda^{\mathcal{H}_1,\eta}f(x,t)-D_{n+1}  \mathcal{S}_{4\eta}^{\mathcal{H}_1,\eta}f(x,t)|\leq M(f)(x,t),
     \end{eqnarray}
     where $M$ is the parabolic Hardy-Littlewood maximal function. Hence, from now on we consider $\lambda_0\geq 4\eta$ fixed. Using \eqref{eq14+} and \eqref{keyseauu}  we see that
              \begin{eqnarray}\label{iest1mod++}
||D_{n+1} \mathcal{S}_{\lambda_0}^{\mathcal{H}_1,\eta}||_2^2&\leq& \frac c\lambda_0\int_{\lambda_0/2}^{3\lambda_0/2}\int_{\mathbb R^{n+1}}|\partial_\lambda \mathcal{S}_\lambda^{\mathcal{H}_1,\eta}f- \partial_\lambda\mathcal{S}_\lambda^{\mathcal{H}_0,\eta}f|^2\, dxdxtd\lambda\notag\\
&&+c||f||_2^2.
     \end{eqnarray}
With  $\lambda_0\geq 4\eta$ fixed,  we let
$$\mbox{$\tilde \Psi\in C_0^\infty(\mathbb R^{n+1}\times (\lambda_0/2,3\lambda_0/2))$, $\lambda_0^{-1/2}||\tilde\Psi||_2=1$, $\tilde\Psi_\delta=\varphi_\delta\ast\tilde\Psi$.}$$ Let $f\in C_0^\infty(\mathbb R^{n+1},\mathbb C)$. Based on the above we can conclude, that to prove \eqref{iest1mod} it suffices to bound
   \begin{eqnarray}\label{iest1mod++pre2+}
\quad\quad\lambda_0^{-1}\biggl |\int_0^\infty\int_{\mathbb R^{n+1}}\bigl (\partial_\lambda\mathcal{S}_\lambda^{\mathcal{H}_1,\eta}f(x,t)-\partial_\lambda\mathcal{S}_\lambda^{\mathcal{H}_0,\eta}f(x,t)\bigr)\overline{\tilde\Psi_\delta(x,t,\lambda )}\, {dxdtd\lambda}\biggr |
     \end{eqnarray}
    by $c\varepsilon_0A^{\mathcal{H}_1,\eta}(f)+c||f||_2$. Furthermore, using this and \eqref{iest1mod++pre} we see that it suffices to bound
   \begin{eqnarray*}\label{iest1mod++pre2+a}
\tilde{\mathcal{E}}:=\lambda_0^{-1}\int_{\mathbb R^{n+2}}\varepsilon(y,s)\nabla\mathcal{S}_\lambda^{\mathcal{H}_1,\eta}f(y,s)\cdot\overline{\nabla(\mathcal{H}_0^\ast)^{-1}D_{n+1}\tilde\Psi_\delta(y,s,\lambda)}
dyds d\lambda,
     \end{eqnarray*}
     and we intend to prove that
        \begin{eqnarray}\label{iest1mod++pre2+aa}
\tilde{\mathcal{E}}\leq c\varepsilon_0A^{\mathcal{H}_1,\eta}(f)+c||f||_2.
     \end{eqnarray}
To start the estimate of $\tilde{\mathcal{E}}$ we write
     \begin{eqnarray}
\tilde{\mathcal{E}}=\tilde{\mathcal{E}}_1+\tilde{\mathcal{E}}_2+\tilde{\mathcal{E}}_3+\tilde{\mathcal{E}}_4,
\end{eqnarray}
where
\begin{eqnarray*}
\tilde{\mathcal{E}}_1&=&\lambda_0^{-1}\int_{-\lambda_0/4}^{\lambda_0/4}\int_{\mathbb R^{n+1}}\varepsilon(y,s)\nabla\mathcal{S}_\lambda^{\mathcal{H}_1,\eta}f(y,s)\cdot\overline{\nabla(\mathcal{H}_0^\ast)^{-1}D_{n+1}\tilde\Psi_\delta(y,s,\lambda)}\,
dyds d\lambda,\notag\\
\tilde{\mathcal{E}}_2&=&\lambda_0^{-1}\int_{\lambda_0/4}^{4\lambda_0}\int_{\mathbb R^{n+1}}\varepsilon(y,s)\nabla\mathcal{S}_\lambda^{\mathcal{H}_1,\eta}f(y,s)\cdot\overline{\nabla(\mathcal{H}_0^\ast)^{-1}D_{n+1}\tilde\Psi_\delta(y,s,\lambda)}\,
dyds d\lambda,\notag\\
\tilde{\mathcal{E}}_3&=&\lambda_0^{-1}\int_{4\lambda_0}^{\infty}\int_{\mathbb R^{n+1}}\varepsilon(y,s)\nabla\mathcal{S}_\lambda^{\mathcal{H}_1,\eta}f(y,s)\cdot\overline{\nabla(\mathcal{H}_0^\ast)^{-1}D_{n+1}\tilde\Psi_\delta(y,s,\lambda)}\,
dyds d\lambda,\notag\\
\tilde{\mathcal{E}}_4&=&\lambda_0^{-1}\int_{-\infty}^{-\lambda_0/4}\int_{\mathbb R^{n+1}}\varepsilon(y,s)\nabla\mathcal{S}_\lambda^{\mathcal{H}_1,\eta}f(y,s)\cdot\overline{\nabla(\mathcal{H}_0^\ast)^{-1}D_{n+1}\tilde\Psi_\delta(y,s,\lambda)}\,
dyds d\lambda.
\end{eqnarray*}
 Using Lemma \ref{gara} we see that $\nabla (\mathcal{H}_0^\ast)^{-1}\div:L^2(\mathbb R^{n+2},\mathbb C)\to L^2(\mathbb R^{n+2},\mathbb C)$, and hence
\begin{eqnarray*}
|\tilde{\mathcal{E}}_2|&=&c\varepsilon_0\bigl (\lambda_0^{-1}\int_{\lambda_0/4}^{4\lambda_0}\int_{\mathbb R^{n+1}}|\nabla\mathcal{S}_\lambda^{\mathcal{H}_1,\eta}f(y,s)|^2\, dyds\bigr )^{1/2}\leq c\varepsilon_0\sup_{\lambda>0}||\nabla \mathcal{S}_\lambda^{\mathcal{H}_1,\eta}f||_2.
\end{eqnarray*}
We next consider $\tilde{\mathcal{E}}_3$ and $\tilde{\mathcal{E}}_4$ and as these terms can be treated similarly we here only treat $\tilde{\mathcal{E}}_3$.  Using
\eqref{splitrep} we see that $\tilde{\mathcal{E}}_3$ equals
\begin{eqnarray*}
&&\lambda_0^{-1}\int\int_{4\lambda_0}^{\infty}\int_{\mathbb R^{n+1}}\varepsilon(y,s)\nabla\mathcal{S}_\lambda^{\mathcal{H}_1,\eta}f(y,s)\cdot \overline{\nabla\partial_\lambda \mathcal{S}_{\lambda-\lambda'}^{\mathcal{H}_0^\ast,\delta}(\tilde \Psi_\delta(\cdot,\cdot,\lambda'))(y,s)}\,
dyds d\lambda d\lambda'\notag\\
&=&\lambda_0^{-1}\int\int_{4\lambda_0}^{\infty}\int_{\mathbb R^{n+1}}(\partial_\lambda \mathcal{S}_{\lambda-\lambda'}^{\mathcal{H}_0}\nabla)\cdot \varepsilon(y,s)\nabla\mathcal{S}_\lambda^{\mathcal{H}_1,\eta}f(y,s)\tilde \Psi_\delta(y,s,\lambda')\,
dyds d\lambda d\lambda'\notag\\
&=:&\tilde{\mathcal{E}}_{31}+\tilde{\mathcal{E}}_{32},
\end{eqnarray*}
where
\begin{eqnarray*}
\tilde{\mathcal{E}}_{31}&=&\lambda_0^{-1}\int\int_{2\lambda'}^{\infty}\int_{\mathbb R^{n+1}}(\partial_\lambda \mathcal{S}^{\mathcal{H}_0}_{\lambda-\lambda'}\nabla)\cdot \varepsilon(y,s)\nabla\mathcal{S}_\lambda^{\mathcal{H}_1,\eta}f(y,s)\tilde \Psi_\delta(y,s,\lambda')\,
dyds d\lambda d\lambda',\notag\\
\tilde{\mathcal{E}}_{32}&=&-\lambda_0^{-1}\int\int_{2\lambda'}^{4\lambda_0}\int_{\mathbb R^{n+1}}(\partial_\lambda \mathcal{S}^{\mathcal{H}_0}_{\lambda-\lambda'}\nabla)\cdot \varepsilon (y,s)\nabla\mathcal{S}_\lambda^{\mathcal{H}_1,\eta}f(y,s)\tilde \Psi_\delta(y,s,\lambda')\,
dyds d\lambda d\lambda'.
\end{eqnarray*}
In $\tilde{\mathcal{E}}_{32}$ we see that $\lambda-\lambda'\approx\lambda\approx\lambda'\approx\lambda_0$ if $\delta$ is sufficiently small. Hence, using Lemma \ref{le5} we see that
\begin{eqnarray*}
|\tilde{\mathcal{E}}_{32}|\leq c\varepsilon_0\sup_{\lambda>0}||\nabla \mathcal{S}_\lambda^{\mathcal{H}_1,\eta}f||_2.
\end{eqnarray*}
To estimate $\tilde{\mathcal{E}}_{31}$ we let, for $R\gg 1$ large,
\begin{eqnarray*}
\Theta_R(y,s,\lambda')&=&\int_{2\lambda'}^{2R}(\partial_\lambda \mathcal{S}^{\mathcal{H}_0}_{\lambda'-\lambda}\nabla)\cdot (\varepsilon(y,s)\nabla\mathcal{S}_\lambda^{\mathcal{H}_1,\eta}f(y,s))\, d\lambda,
\end{eqnarray*}
and we note that
\begin{eqnarray*}
\tilde{\mathcal{E}}_{31}&=&\lambda_0^{-1}\int\lim_{R\to\infty}\int_{\mathbb R^{n+1}}\Theta_R(y,s,\lambda')\tilde \Psi_\delta(y,s,\lambda')\,
dydsd\lambda'.
\end{eqnarray*}
However, $\Theta_R(y,s,\lambda')$ equals
\begin{eqnarray*}
&&-\int_{\lambda'}^R\partial_{\sigma}\biggl (\int_{2\sigma}^{2R}(\partial_\sigma \mathcal{S}^{\mathcal{H}_0}_{\sigma-\lambda}\nabla)\cdot (\varepsilon(y,s)\nabla\mathcal{S}_\lambda^{\mathcal{H}_1,\eta}f(y,s))\, d\lambda\biggr )\, d\sigma\notag\\
&=&-\int_{\lambda'}^R\partial_{\sigma}\biggl (\int_{\sigma}^{2R-\sigma}(D_{n+1}\mathcal{S}^{\mathcal{H}_0}_{-\lambda}\nabla)\cdot (\varepsilon(y,s)\nabla\mathcal{S}_{\sigma+\lambda}^{\mathcal{H}_1,\eta}f(y,s))\, d\lambda\biggr )\, d\sigma.
\end{eqnarray*}
Hence,
\begin{eqnarray*}
\Theta_R(y,s,\lambda')=\Theta_R^{'}(y,s,\lambda')+\Theta_R^{''}(y,s,\lambda')+\Theta_R^{'''}(y,s,\lambda'),
\end{eqnarray*}
where
\begin{eqnarray*}
\Theta_R^{'}(y,s,\lambda')&=& \int_{\lambda'}^R(D_{n+1}\mathcal{S}^{\mathcal{H}_0}_{-\lambda}\nabla)\cdot (\varepsilon(y,s)\nabla\mathcal{S}_{2\lambda}^{\mathcal{H}_1,\eta}f(y,s))\, d\lambda, \notag\\
\Theta_R^{''}(y,s,\lambda')&=&-\int_{\lambda'}^R(D_{n+1}\mathcal{S}^{\mathcal{H}_0}_{\lambda-2R}\nabla)\cdot (\varepsilon(y,s)\nabla\mathcal{S}_{2R}^{\mathcal{H}_1,\eta}f(y,s))\, d\lambda,\notag\\
\Theta_R^{'''}(y,s,\lambda')&=&-\int_{\lambda'}^R\biggl (\int_{\sigma}^{2R-\sigma}(D_{n+1}\mathcal{S}^{\mathcal{H}_0}_{-\lambda}\nabla)\cdot (\varepsilon(y,s)\nabla\partial_\lambda\mathcal{S}_{\sigma+\lambda}^{\mathcal{H}_1,\eta}f(y,s))\, d\lambda\biggr )\, d\sigma.
\end{eqnarray*}
Using this decomposition for $\Theta_R$ we get a decomposition for $\tilde{\mathcal{E}}_{31}$:
\begin{eqnarray*}
\tilde{\mathcal{E}}_{31}=\tilde{\mathcal{E}}_{311}+\tilde{\mathcal{E}}_{312}+\tilde{\mathcal{E}}_{313}.
\end{eqnarray*}
Using that $|\sigma-2R|\approx R$ we see that it follows from Lemma \ref{le5} that
\begin{eqnarray*}
\sup_{\lambda', R:\  0<\lambda'<R}||\Theta_R^{''}(\cdot,\cdot,\lambda')||_2\leq c\varepsilon_0\sup_{\lambda>0}||\nabla \mathcal{S}_\lambda^{\mathcal{H}_1,\eta}f||_2,
\end{eqnarray*}
and hence
\begin{eqnarray*}
|\tilde{\mathcal{E}}_{312}|\leq c\varepsilon_0\sup_{\lambda>0}||\nabla \mathcal{S}_\lambda^{\mathcal{H}_1,\eta}f||_2.
\end{eqnarray*}
Furthermore, using Lemma \ref{telem3} we see that
\begin{eqnarray*}
|\tilde{\mathcal{E}}_{311}|\leq c\varepsilon_0A^{\mathcal{H}_1,\eta}(f)+c||f||_2.
\end{eqnarray*}
Hence only $\tilde{\mathcal{E}}_{313}$ remains to be estimated. Note that
\begin{eqnarray*}
\Theta_R^{'''}(y,s,\lambda')&=&-\int_{\lambda'}^R\int_{2\sigma}^{2R}(\partial_\lambda\mathcal{S}^{\mathcal{H}_0}_{\sigma-\lambda}\nabla)\cdot (\varepsilon(y,s)\partial_\lambda\nabla\mathcal{S}_{\lambda}^{\mathcal{H}_1,\eta}f(y,s))\, d\lambda d\sigma.
\end{eqnarray*}
To estimate $||\Theta_R^{'''}(\cdot,\cdot,\lambda')||_2$, consider $h\in L^2(\mathbb R^{n+1},\mathbb C)$, $||h||_2=1$. Then
\begin{eqnarray*}
&&\biggl|\int_{\mathbb R^{n+1}}\Theta_R^{'''}(y,s,\lambda')\overline{h(y,s)}\, dyds\biggr |\notag\\
&=& \biggl|\int_{\lambda'}^R\int_{2\sigma}^{2R}\int_{\mathbb R^{n+1}}
(\nabla D_{n+1}\mathcal{S}^{\mathcal{H}_0^\ast}_{\sigma-\lambda}h(y,s))\cdot \overline{(\varepsilon(y,s)\partial_\lambda\nabla\mathcal{S}_{\lambda}^{\mathcal{H}_1,\eta}f(y,s))}\, dydsd\lambda d\sigma\biggr |,
\end{eqnarray*}
where we have used that $\mbox{adj}(\mathcal{S}^{\mathcal{H}_0}_{\sigma-\lambda})=\mathcal{S}^{\mathcal{H}_0^\ast}_{\sigma-\lambda}$. Using this we see that
\begin{eqnarray*}
&&\biggl|\int_{\mathbb R^{n+1}}\Theta_R^{'''}(y,s,\lambda')\overline{h(y,s)}\, dyds\biggr |\notag\\
&\leq& c\varepsilon_0 \biggl(\int_{0}^\infty\int_{2\sigma}^{\infty}
||\nabla\partial_\lambda\mathcal{S}^{\mathcal{H}_0^\ast}_{\sigma-\lambda}h||_2^2\, d\lambda d\sigma\biggr )^{1/2}\biggl (\int_{0}^\infty
||\partial_\lambda\nabla\mathcal{S}_{\lambda}^{\mathcal{H}_1,\eta}f||_2^2\lambda\, d\lambda\biggr )^{1/2}\notag\\
&\leq& c\varepsilon_0 \biggl(\int_{0}^\infty\int_{2\sigma}^{\infty}
||\nabla\partial_\lambda\mathcal{S}^{\mathcal{H}_0^\ast}_{\sigma-\lambda}h||_2^2\, d\lambda d\sigma\biggr )^{1/2}|||\lambda\partial_\lambda\nabla\mathcal{S}_{\lambda}^{\mathcal{H}_1,\eta}f|||_+\notag\\
&\leq& c\varepsilon_0 |||\lambda\nabla\partial_\lambda\mathcal{S}^{\mathcal{H}_0^\ast}_{\lambda}h|||_+|||\lambda\partial_\lambda\nabla\mathcal{S}_{\lambda}^{\mathcal{H}_1,\eta}f|||_+\leq  c\varepsilon_0|||\lambda\nabla\partial_\lambda\mathcal{S}_{\lambda}^{\mathcal{H}_1,\eta}f|||_+,
\end{eqnarray*}
by \eqref{keyseauu} applied to $\mathcal{S}^{\mathcal{H}_0^\ast}_{\lambda}$. Hence,
\begin{eqnarray*}
|\tilde{\mathcal{E}}_{313}|\leq c\varepsilon_0A^{\mathcal{H}_1,\eta}(f),
\end{eqnarray*}
and we can conclude that
\begin{eqnarray*}
|\tilde{\mathcal{E}}-\tilde{\mathcal{E}}_1|\leq c\varepsilon_0A^{\mathcal{H}_1,\eta}(f)+c||f||_2.
\end{eqnarray*}
To estimate $\tilde{\mathcal{E}}_1$ we first see, using \eqref{splitrep} and that the support of $\tilde\Psi$, for $\delta$ small, is contained in $\{\lambda_0/2<\lambda<3\lambda_0/2\}$, that
\begin{eqnarray*}
\tilde{\mathcal{E}}_1&=&\lambda_0^{-1}\int_{-\lambda_0/4}^{\lambda_0/4}\int_{\mathbb R^{n+1}}\varepsilon(y,s)\nabla\mathcal{S}_\lambda^{\mathcal{H}_1,\eta}f(y,s)\cdot\overline{\nabla(\mathcal{H}_0^\ast)^{-1}D_{n+1}\tilde\Psi_\delta(y,s,\lambda)}
dyds d\lambda\notag\\
&=&\lambda_0^{-1}\int\int_{-\lambda_0/4}^{\lambda_0/4}\int_{\mathbb R^{n+1}}(\partial_\lambda \mathcal{S}_{\lambda-\lambda'}^{\mathcal{H}_0}\nabla)\cdot \varepsilon(y,s)\nabla\mathcal{S}_\lambda^{\mathcal{H}_1,\eta}f(y,s)\tilde \Psi_\delta(y,s,\lambda')\,
dyds d\lambda d\lambda'\notag\\
&=&\tilde{\mathcal{E}}_{11}+\tilde{\mathcal{E}}_{12},
\end{eqnarray*}
where
\begin{eqnarray*}
\tilde{\mathcal{E}}_{11}&=&\lambda_0^{-1}\int\int_{-\lambda'/2}^{\lambda'/2}\int_{\mathbb R^{n+1}}(\partial_\lambda \mathcal{S}^{\mathcal{H}_0}_{\lambda-\lambda'}\nabla)\cdot \varepsilon(y,s)\nabla\mathcal{S}_\lambda^{\mathcal{H}_1,\eta}f(y,s)\tilde \Psi_\delta(y,s,\lambda')\,
dyds d\lambda d\lambda',\notag\\
\tilde{\mathcal{E}}_{12}&=&-\lambda_0^{-1}\int\int_{\lambda_0/4<|\lambda|<\lambda'/2}\int_{\mathbb R^{n+1}}(\partial_\lambda \mathcal{S}^{\mathcal{H}_0}_{\lambda-\lambda'}\nabla)\cdot\varepsilon(y,s)\nabla\mathcal{S}_\lambda^{\mathcal{H}_1,\eta}f(y,s)\tilde \Psi_\delta(y,s,\lambda')\,
dyds d\lambda d\lambda'.
\end{eqnarray*}
Again by Cauchy-Schwarz and Lemma \ref{le5} we see, as $\lambda-\lambda'\approx\lambda_0$, that
\begin{eqnarray*}
|\tilde{\mathcal{E}}_{12}|\leq c\varepsilon_0\sup_{\lambda>0}||\nabla \mathcal{S}_\lambda^{\mathcal{H}_1,\eta}f||_2.
\end{eqnarray*}
Furthermore,
\begin{eqnarray*}
\tilde{\mathcal{E}}_{11}&=&\tilde{\mathcal{E}}_{111}+\tilde{\mathcal{E}}_{112},
\end{eqnarray*}
where
\begin{eqnarray*}
\tilde{\mathcal{E}}_{111}&=&\lambda_0^{-1}\int\int_{0}^{\lambda'/2}\int_{\mathbb R^{n+1}}(\partial_\lambda \mathcal{S}^{\mathcal{H}_0}_{\lambda-\lambda'}\nabla)\cdot \varepsilon(y,s)\nabla\mathcal{S}_\lambda^{\mathcal{H}_1,\eta}f(y,s)\tilde \Psi_\delta(y,s,\lambda')\,
dyds d\lambda d\lambda',\notag\\
\tilde{\mathcal{E}}_{112}&=&\lambda_0^{-1}\int\int_{-\lambda'/2}^{0}\int_{\mathbb R^{n+1}}(\partial_\lambda \mathcal{S}^{\mathcal{H}_0}_{\lambda-\lambda'}\nabla)\cdot \varepsilon(y,s)\nabla\mathcal{S}_\lambda^{\mathcal{H}_1,\eta}f(y,s)\tilde \Psi_\delta(y,s,\lambda')\,
dyds d\lambda d\lambda'.
\end{eqnarray*}
We only estimate $\tilde{\mathcal{E}}_{111}$, the term $\tilde{\mathcal{E}}_{112}$ being treated similar. We write
\begin{eqnarray*}
\tilde{\mathcal{E}}_{111}&=&\lambda_0^{-1}\int \int_{\mathbb R^{n+1}} F(y,s,\lambda')\tilde \Psi_\delta(y,s,\lambda')\, dyds  d\lambda',
\end{eqnarray*}
where
\begin{eqnarray*}
F(y,s,\lambda')=\int_{0}^{\lambda'/2}(\partial_\lambda \mathcal{S}^{\mathcal{H}_0}_{\lambda'-\lambda}\nabla)\cdot \varepsilon(y,s)\nabla\mathcal{S}_\lambda^{\mathcal{H}_1,\eta}f(y,s)\, d\lambda.
\end{eqnarray*}
Now
\begin{eqnarray*}
F(y,s,\lambda')=\int_{0}^{\lambda'}\partial_\sigma\biggl(\int_{0}^{\sigma/2} (\partial_\sigma \mathcal{S}^{\mathcal{H}_0}_{\sigma-\lambda}\nabla)\cdot \varepsilon(y,s)\nabla\mathcal{S}_\lambda^{\mathcal{H}_1,\eta}f(y,s)\, d\lambda\biggr )\, d\sigma.
\end{eqnarray*}
However, now using \eqref{keyseauu} and Lemma \ref{telem3}, and proceeding as in the estimates of $\Theta_R$ above, one can prove the appropriate bound for
$\tilde{\mathcal{E}}_{111}$ and $\tilde{\mathcal{E}}_{1}$. We omit further details and claim that this completes the proof of \eqref{iest1mod++pre2+aa} and hence the proof of Lemma \ref{finall}.

\subsection{Proof of Lemma \ref{analpert}} Recall that $\mathcal{H}_0=\partial_t+\mathcal{L}_0=\partial_t-\div A^0\nabla$. By assumption we have that  $A^0$ satisfies \eqref{eq3}-\eqref{eq4} as well as the De Giorgi-Moser-Nash estimates stated in \eqref{eq14+}-\eqref{eq14++}. We let
$$A^z=A^0+zM,\ z\in \mathbb C,$$
where $M$ is a $(n+1)\times (n+1)$-dimensional matrix which is measurable, bounded, complex and satisfies \eqref{eq4} and $||M||_\infty\leq 1$. We let
$$\mathcal{H}_z:=\partial_t+\mathcal{L}_z:=\partial_t-\div A^z\nabla.$$
Following \cite{A}, there exists $\varepsilon_0=\varepsilon_0(n,\Lambda)$, $0<\varepsilon_0<1$, such that if $|z|<\varepsilon_0$, then $L_z$ defines an
$L^2$-contraction semigroup $e^{-tL_z}$, for $t>0$, generated by $L_z$. $e^{-t\mathcal{L}_z}$ is defined using functional calculus, see \cite{A}, \cite{AT}, \cite{K} for instance, and the map $z\to e^{-tL_z}$ is analytic for $|z|<\varepsilon_0$. We let $K_t^z(X,Y)$ denote the distribution kernel of $e^{-t\mathcal{L}_z}$ and by definition
$$\Gamma^{\mathcal{H}_z}(X,t,Y,s)=\Gamma^{\mathcal{H}_z}(x,t,\lambda,y,s,\sigma)=K_{t-s}^z(x,\lambda,y,\sigma)=K_{t-s}^z(X,Y)$$
whenever $t-s>0$. In particular, the fundamental solution associated to $\mathcal{H}_z$, $\Gamma^{\mathcal{H}_z}$, coincides with the kernel
$K_t^z$. Furthermore, by construction the map $z\to \Gamma^{\mathcal{H}_z}(x,t,\lambda,y,s,\sigma)$ is analytic for $|z|<\varepsilon_0$. Assuming
\eqref{keyseauu} we have proved that there exists a constant $c$, depending at most
     on $n$, $\Lambda$,  such that if
    $$|z|<\varepsilon_0,$$
    then
     \begin{eqnarray}\label{keyest+aaaedop}
     ||\mathcal{K}^{\mathcal{H}_z}||_{2\to 2}+||\tilde {\mathcal{K}}^{\mathcal{H}_z}||_{2\to 2}&\leq& c,\notag\\
    \sup_{\lambda\neq 0} ||\nabla_{||} \mathcal{S}_\lambda^{\mathcal{H}_z}||_{2\to 2}+\sup_{\lambda\neq 0} ||H_tD^t_{1/2}\mathcal{S}_\lambda^{\mathcal{H}_z}||_{2\to 2}&\leq& c.
    \end{eqnarray}
    To complete the proof of Lemma \ref{analpert} it suffices to prove that
         \begin{eqnarray}\label{keyest+aaaedop1}
    (i)&& z\to\mathcal{K}^{\mathcal{H}_z},\ z\to \tilde {\mathcal{K}}^{\mathcal{H}_z},\notag\\
    (ii)&&z\to\nabla_{||} \mathcal{S}_\lambda^{\mathcal{H}_z}|_{\lambda=0},\ z\to H_tD^t_{1/2}\mathcal{S}_\lambda^{\mathcal{H}_z}|_{\lambda=0},
    \end{eqnarray}
    are analytic for $|z|<\varepsilon_0$. Indeed, if this is true, then it follows from the operator valued form of the Cauchy formula that
        \begin{eqnarray}\label{keyest+aaaedop2}
     ||\frac {d}{dz}\mathcal{K}^{\mathcal{H}_z}||_{2\to 2}+||\frac {d}{dz}\tilde {\mathcal{K}}^{\mathcal{H}_z}||_{2\to 2}&\leq& c,\notag\\
    \sup_{\lambda\neq 0} ||\frac {d}{dz}\nabla_{||} \mathcal{S}_\lambda^{\mathcal{H}_z}||_{2\to 2}+\sup_{\lambda\neq 0} ||\frac {d}{dz}H_tD^t_{1/2}\mathcal{S}_\lambda^{\mathcal{H}_z}||_{2\to 2}&\leq& c,
    \end{eqnarray}
    and it is clear that Lemma \ref{analpert} follows. To prove \eqref{keyest+aaaedop1} we first note, using that  $C_0^\infty(\mathbb R^{n+1},\mathbb C^k)$ is dense in
    $L^2(\mathbb R^{n+1},\mathbb C^k)$, and as we have proved \eqref{keyest+aaaedop}, that to prove  \eqref{keyest+aaaedop1} it suffices to verify the criterium  for analyticity stated on p. 365 in \cite{K}. Indeed, we only have to verify that
          \begin{eqnarray}\label{keyest+aaaedop3}
    (i')&& z\to( \mathcal{K}^{\mathcal{H}_z}f,g),\ z\to (\tilde {\mathcal{K}}^{\mathcal{H}_z}f,g),\notag\\
    (ii')&&z\to(\nabla_{||} \mathcal{S}_\lambda^{\mathcal{H}_z}|_{\lambda=0}f,{\bf g}),\  z\to ( H_tD^t_{1/2}\mathcal{S}_\lambda^{\mathcal{H}_z}|_{\lambda=0}f,g),
    \end{eqnarray}
    are analytic for $|z|<\varepsilon_0$ whenever $f,g\in C_0^\infty(\mathbb R^{n+1},\mathbb C)$, ${\bf g}\in C_0^\infty(\mathbb R^{n+1},\mathbb C^n)$. Here $(\cdot,\cdot)$ is the standard inner product on $L^2(\mathbb R^{n+1},\mathbb C^k)$. To prove $(i')$ it suffices, by duality, to prove that
             \begin{eqnarray}\label{keyest+aaaedop3g}
 z\to ((\frac I2+\tilde {\mathcal{K}}^{\mathcal{H}_z})f,g)\mbox{ is analytic for $|z|<\varepsilon_0$},
    \end{eqnarray}
    whenever $f,g\in C_0^\infty(\mathbb R^{n+1},\mathbb C)$. Fix $f,g\in C_0^\infty(\mathbb R^{n+1},\mathbb C)$ and let
    $$g_j(z):=(-e_{n+1}\cdot A\nabla \mathcal{S}_{1/j}^{\mathcal{H}_z}f,g),\ j\in\mathbb Z_+.$$
    Using the bounds established we have that $\{g_j\}$ is a uniformly bounded family of analytic functions in $|z|<\varepsilon_0$ and by Lemma \ref{trace4} $(i)$ we have that
    $$g_j(z)\to ((\frac I2+\tilde {\mathcal{K}}^{\mathcal{H}_z})f,g)\mbox{ for all $|z|<\varepsilon_0$ as $j\to\infty$}.$$
    Using these facts we can use Montel's theorem to conclude \eqref{keyest+aaaedop3g}. To prove $(ii')$ we can essentially argue as above using instead
    Lemma \ref{trace4} $(iii)$-$(iv)$.

\section{Proof of Theorem \ref{th2-}- Theorem \ref{th2+}}\label{sec5}

In this section we prove Theorem \ref{th2-}-Theorem \ref{th2+} using Theorem \ref{thper2} and  Corollary \ref{thper3}.

 \subsection{Proof of Theorem \ref{th2-} } Consider $\mathcal{H}^0=\partial_t+\mathcal{L}=\partial_t-\div(A^0\nabla)$ where $A^0$ now is a constant complex matrix. Let
   \begin{eqnarray}Q(\xi,\zeta)&=&A^0_{n+1,n+1}\zeta^2+\zeta\bigl(\sum_{k=1}^n\xi_k(A^0_{k,n+1}+A^0_{n+1,k})\bigr)+A^0_{||}\xi\cdot\xi
  \end{eqnarray}
  where  $(\xi,\zeta)\in \mathbb R^n\times\mathbb R$ and where again $A^0_{||}$ is the $n\times n$-dimensional sub matrix of $A^0$ defined by $\{A^0_{i,j}\}_{i,j=1}^n$. Using \eqref{eq3} we see that $\mbox{Re }A^0_{n+1,n+1}\geq\Lambda^{-1}$ and that
  $$\mbox{Re }Q(\xi,\zeta)\geq \Lambda^{-1}(|\xi|^2+|\zeta|^2).$$ The Fourier transform, with respect to the spatial variables, of the fundamental solution associated to $\mathcal{H}^0$ equals
  $\exp(-t Q(\xi,\zeta))$, and taking also the Fourier transform in the $t$-variable we see that the Fourier transform of $\Gamma$ with respect to all variables, $\hat\Gamma(\xi,\tau,\zeta)$, equals $(Q(\xi,\zeta)-i\tau)^{-1}$ which of course is the symbol associated to $\mathcal{H}^0$. We let
  $$F(\xi,\tau,\lambda)=\int_{-\infty}^\infty (Q(\xi,\zeta)-i\tau)^{-1}\exp(-i\lambda\zeta)\, d\zeta,$$
  $(\xi,\tau,\lambda)\in \mathbb R^{n-1}\times\mathbb R\times\mathbb R_+$. Then $F$ equals $\hat\Gamma$ inverted in the $\zeta$-variable only and when $\lambda\geq 0$.   In the following we write
\begin{eqnarray}Q(\xi,\zeta)-i\tau&=&A^0_{n+1,n+1}\zeta^2+\zeta\bigl(\sum_{k=1}^n\xi_k(A^0_{k,n+1}+A^0_{n+1,k})\bigr)+A^0_{||}\xi\cdot\xi\notag\\
  &=&A^0_{n+1,n+1}\biggl (\biggl (\zeta+\frac {(\xi\cdot w)}{2A^0_{n+1,n+1}}\biggr )^{2}-B(\xi,\tau)\biggr )\notag\\
  \end{eqnarray}
  where
  \begin{eqnarray}
  w_k&=&(A^0_{k,n+1}+A^0_{n+1,k})\mbox{ for $k\in \{1,...,n\}$, and}\notag\\
  B(\xi,\tau)&=&\biggl (\frac {(\xi\cdot w)}{2A^0_{n+1,n+1}}\biggr )^{2}-\frac {A^0_{||}\xi\cdot\xi}{A^0_{n+1,n+1}}+\frac {i\tau}{A^0_{n+1,n+1}},\notag\\
  \end{eqnarray}
  Then, using  the above notation we see that
    \begin{eqnarray}\label{pparis}
&&2 A^0_{n+1,n+1} \sqrt{B(\xi,\tau)}F(\xi,\tau,\lambda)\notag\\
&=&-\int_{-\infty}^\infty \frac{1}{\zeta+\frac {(\xi\cdot w)}{2A^0_{n+1,n+1}}+\sqrt{B(\xi,\tau)}}\exp(-i\lambda\zeta)\, d\zeta\notag\\
&&+\int_{-\infty}^\infty \frac{1}{\zeta+\frac {(\xi\cdot w)}{2A^0_{n+1,n+1}}-\sqrt{B(\xi,\tau)}}\exp(-i\lambda\zeta)\, d\zeta.
  \end{eqnarray}
  Hence, using the residue theorem,
  \begin{eqnarray}\label{pparis+}
&&2 A^0_{n+1,n+1} \sqrt{B(\xi,\tau)}F(\xi,\tau,\lambda)\notag\\
&=&\exp\biggl (i\lambda \frac {(\xi\cdot w)}{2A^0_{n+1,n+1}}\biggr )\biggl (\exp(-i\lambda\sqrt{B(\xi,\tau)})-\exp(i\lambda\sqrt{B(\xi,\tau)})\biggr )
  \end{eqnarray}
  Furthermore, using that
    \begin{eqnarray}\label{pparis++}
\sqrt{B(\xi,\tau)}&=&\frac 1{\sqrt{2}}\sqrt{|B(\xi,\tau)|+\mbox{Re }B(\xi,\tau)}\notag\\
&&+i\frac {\mbox{sgn}(\mbox{Im }B(\xi,\tau))}{\sqrt{2}}
\sqrt{|B(\xi,\tau)|-\mbox{Re }B(\xi,\tau)},
  \end{eqnarray}
   \eqref{pparis}, \eqref{pparis+}, and \eqref{eq3} it is not hard to see that Definition \ref{blayer+} $(i)$-$(ii)$ hold for
  some $\Gamma=\Gamma(n,\Lambda)$. Using this,  Lemma \ref{lemsl1++}, Lemma \ref{trace4}, Lemma \ref{trace5}, and Lemma \ref{trace7-}, we see that
  also Definition \ref{blayer+} $(i)$-$(vii)$ hold. Finally, evaluating \eqref{pparis} at $\lambda=0$ it also follows, similar to the corresponding argument in \cite{AAAHK}, that the conditions in Definition \ref{blayer+} $(viii)$-$(xiii)$ hold for $\mathcal{H}^0$. An application of  Theorem \ref{thper2} completes the proof of Theorem \ref{th2-}.

    \subsection{Proof of Theorem \ref{th2} } The proof of Theorem \ref{th2} is based on the following lemma proved at the end of the section.

    \begin{lemma}\label{th0++}   Assume that  $\mathcal{H}=\partial_t-\div A\nabla$ satisfies \eqref{eq3}-\eqref{eq4}. Assume that \begin{eqnarray}\label{resq} \mbox{$A$ is a real and symmetric matrix}.
       \end{eqnarray}
Then there exists a constant $\Gamma$, depending at most
     on $n$, $\Lambda$ , such that Definition \ref{blayer+} $(i)$-$(x)$ hold with this $\Gamma$.\end{lemma}

    We here use Lemma \ref{th0++} to complete the proof of Theorem \ref{th2}. Given $\sigma\in [0,1]$ we let
    $$A_\sigma=(1-\sigma) \mathbb I_{n+1}+\sigma A$$
    where $\mathbb I_{n+1}$ is the $(n+1)\times (n+1)$ identity matrix. Based on $A_\sigma$ we introduce $\mathcal{H}_\sigma=\partial_t-\div(A_\sigma\nabla)$. Then Lemma \ref{th0++} applies to $\mathcal{H}_\sigma$ with a constant $\Gamma$ which can be chosen independent of $\sigma$. Hence, by arguing as in the proof of  Corollary \ref{thper3} we see that to prove Theorem \ref{th2+} we only have to verify Definition \ref{blayer+} $(xi)-(xiii)$ for
    $\mathcal{H}_1$. However, by repeating the constant coefficient arguments in \cite{B} we see that Definition \ref{blayer+} $(xi)-(xiii)$ holds for $\mathcal{H}_0$. Hence, invoking Theorem \ref{thper2} we see that Definition \ref{blayer+} $(xi)-(xiii)$ holds $\mathcal{H}_\sigma$ whenever $|\sigma|\leq \tilde\varepsilon$ for some
    $\tilde\varepsilon=\tilde\varepsilon(n,\Lambda)$. Iterating this procedure step by step we see that Definition \ref{blayer+} $(xi)-(xiii)$ also hold for $\mathcal{H}_1$. This completes the proof of Proof of Theorem \ref{th2+}.
    \subsection{Proof of Theorem \ref{th2+} } Theorem \ref{th2+} follows directly from  Theorem \ref{th2}, Theorem \ref{thper2} and Corollary \ref{thper3}. Indeed, by Theorem \ref{th2} we have that $\mathcal{H}_0$ satisfies all statements of Definition \ref{blayer+}. An application of Theorem \ref{thper2} and Corollary \ref{thper3}  then completes the proof of Theorem \ref{th2+}.

    \subsection{Proof of Lemma \ref{th0++}} To start the proof we first record the following lemma proved in  \cite{CNS}.

    \begin{theorem}\label{th2ag} Assume that $\mathcal{H}$ satisfies \eqref{eq3}-\eqref{eq4}. Assume in addition that
 $A$ is real and symmetric. Let $\Phi_+(f)$ be defined as in \eqref{keyestint-ex+}. Then there exists a constant $\Gamma$, depending at most
     on $n$, $\Lambda$, such that
     $$\Phi_+(f)\leq \Gamma ||f||_2.$$ In particular,  there exists a constant $c$ depending only
     on $n$, $\Lambda$, such that
     \begin{eqnarray}\label{keyestint+ars}
||N_\ast(\partial_\lambda \mathcal{S}_\lambda f)||_2+||\tilde N_\ast(\nabla_{||}\mathcal{S}_\lambda f)||_2+||\tilde N_\ast(H_tD^t_{1/2}\mathcal{S}_\lambda f)||_2&\leq &c||f||_2,\notag\\
\sup_{\lambda> 0}||\mathbb D\mathcal{S}_{\lambda}f||_{2}&\leq &c||f||_2,
\end{eqnarray}
whenever $f\in L^2(\mathbb R^{n+1},\mathbb R)$.
     \end{theorem}
     \begin{proof} This is Theorem 1.5 and Theorem 1.8 in \cite{CNS}. In \cite{CNS} Theorem 1.8  is proved by first establishing a local parabolic Tb-theorem for square functions, see Theorem 8.4 in \cite{CNS} , and then by establishing a version of the main result in \cite{FS} for equation of the form \eqref{eq1}, assuming in addition that $A$ is real and symmetric, see Theorem 8.7 in \cite{CNS}. Both Theorem 8.4 and Theorem 8.7 in \cite{CNS} are of independent interest.
\end{proof}

    Using Lemma \ref{th2ag} we see that Definition \ref{blayer+} $(i)-(vi)$ hold.  Definition \ref{blayer+} $(vi)$ is consequence of these estimates,  Lemma \ref{trace4}, Lemma \ref{trace5}, and Lemma \ref{trace7-}. Hence, to complete the proof of the lemma it suffices to prove Definition \ref{blayer+} $(viii)-(x)$ and to do this it suffices to prove that
          \begin{eqnarray}
(i)&& ||f||_2\leq c\min\biggl\{||\frac 1 2I+\tilde{\mathcal{K}}^{\mathcal{H}}f||_2,||-\frac 1 2I+\tilde{\mathcal{K}}^{\mathcal{H}}f||_2\biggr\},\notag\\
 (ii)&&||f||_2\leq c||\mathbb D\mathcal{S}_\lambda^{\mathcal{H}}|_{\lambda=0}f||_{2},
\end{eqnarray}
whenever $f\in L^2(\mathbb R^{n+1},\mathbb R)$. To start the proof of these two inequalities, let  $\Phi_+(f)$ be defined as in \eqref{keyestint-ex+} and let
     \begin{eqnarray}\label{keyestint-ex+sea}
\Phi_-(f):=\sup_{\lambda< 0}||\partial_\lambda \mathcal{S}_\lambda^{\mathcal{H}} f||_2+|||\lambda\partial_\lambda^2 \mathcal{S}_{\lambda}^{\mathcal{H}}f|||_-.
     \end{eqnarray}
     By  Lemma \ref{th2ag} we have
     \begin{eqnarray}\label{korees}\Phi_\pm(f)\leq \Gamma ||f||_2,
     \end{eqnarray}
     whenever $f\in L^2(\mathbb R^{n+1},\mathbb R)$. Let  $\delta>0$ be fixed. Let $u_\delta^+(x,t,\lambda)=\mathcal{S}_{\lambda+\delta}^{\mathcal{H}}f(x,t)$ whenever $(x,t,\lambda)\in \mathbb R_+^{n+2}$ and let
 $u_\delta^-(x,t,\lambda)=\mathcal{S}_{\lambda-\delta}^{\mathcal{H}}f(x,t)$ whenever $(x,t,\lambda)\in \mathbb R_-^{n+2}$. Then, simply using the equation and \eqref{eq4} we see that
\begin{eqnarray*}
\div(e_{n+1}A\nabla u_\delta^\pm \cdot{\nabla u_\delta^\pm})=2\div(\partial_\lambda u_\delta^\pm A\nabla u_\delta^\pm)+2\partial_tu_\delta^\pm\partial_\lambda u_\delta^\pm,
     \end{eqnarray*}
     in $\mathbb R_\pm^{n+2}$. Hence
       \begin{eqnarray}\label{rec1}
       -\int_{\mathbb R^{n+1}}A\nabla u_\delta^\pm\cdot{\nabla u_\delta^\pm}\, dxdt&=&2\int_{\mathbb R^{n+2}_\pm}\partial_tu_\delta^\pm\partial_\lambda u_\delta^\pm\, dxdtd\lambda\notag\\
       &&+
       2\int_{\mathbb R^{n+1}}\partial_\lambda u_\delta^\pm(-e_{n+1}\cdot A\nabla u_\delta^\pm)\, dxdt.
          \end{eqnarray}
          Let
          $$I_\delta^\pm=\int_{\mathbb R^{n+2}_\pm}\partial_tu_\delta^\pm\partial_\lambda u_\delta^\pm\, dxdtd\lambda.$$
         Then, using \eqref{rec1} we can conclude that
              \begin{eqnarray}\label{rec2-}
     ||\nabla_{||} u_\delta^\pm||_2^2&\leq& c||\partial_\nu u_\delta^\pm||_2^2+c|I_\delta^\pm|,\notag\\
     ||\partial_\nu u_\delta^\pm||_2^2&\leq& c||\nabla_{||} u_\delta^\pm||_2^2+c|I_\delta^\pm|.
          \end{eqnarray}
          We claim that
                   \begin{eqnarray}\label{invkey}
      |I_\delta^\pm|+||D_{1/2}^tu_\delta^\pm||_2^2\leq c||f||_2||D_{1/2}^t u_\delta^\pm||_2^{1/2}||\partial_\nu u_\delta^\pm||_2^{1/2}.
      \end{eqnarray}
      We postpone the proof of \eqref{invkey} for now to complete the proof of Lemma \ref{th0++}. Indeed, given a degree of freedom $\tilde\delta\in (0,1)$ we see that
\eqref{rec2-} and \eqref{invkey} imply that
        \begin{eqnarray}\label{rec2+-}
     ||\mathbb Du_\delta^\pm||_2^2&\leq& c(\tilde\delta)||\partial_\nu u_\delta^\pm||_2^2+\tilde\delta||f||_2^2,\notag\\
     ||\partial_\nu u_\delta^\pm||_2^2&\leq& c(\tilde\delta)||\mathbb Du_\delta^\pm||_2^2+\tilde\delta||f||_2^2.
          \end{eqnarray}
          Using this, letting $\delta\to 0$ and applying Lemma \ref{trace4} and Lemma \ref{trace7-}, we see that
               \begin{eqnarray*}\label{rec2+-a}
                ||\mathbb D\mathcal{S}_\lambda^{\mathcal{H}}|_{\lambda=0}f||_{2}^2\leq c(\tilde\delta)\min\biggl\{||\frac 1 2I+\tilde{\mathcal{K}}^{\mathcal{H}}f||_2^2,||-\frac 1 2I+\tilde{\mathcal{K}}^{\mathcal{H}}f||_2^2\biggr\}+\tilde\delta||f||_2^2,
          \end{eqnarray*}
          and that
                       \begin{eqnarray*}\label{rec2+-b}
     \max\biggl\{||\frac 1 2I+\tilde{\mathcal{K}}^{\mathcal{H}}f||_2^2,||-\frac 1 2I+\tilde{\mathcal{K}}^{\mathcal{H}}f||_2^2\biggr\}\leq c(\tilde\delta)||\mathbb D\mathcal{S}_\lambda^{\mathcal{H}}|_{\lambda=0}f||_{2}^2+\tilde\delta||f||_2^2.
          \end{eqnarray*}
          Using the inequalities in the last two displays and the fact that
          $$f=\frac 1 2I+\tilde{\mathcal{K}}^{\mathcal{H}}f-(-\frac 1 2I+\tilde{\mathcal{K}}^{\mathcal{H}}f),$$
          we see that Lemma \ref{th0++} $(i)$, $(ii)$ hold.

     We next prove the claim  in \eqref{invkey} and we will here only prove that
               \begin{eqnarray}\label{invkey+}
      |I_\delta^+|+||D_{1/2}^tu_\delta^+||_2^2\leq c||f||_2||D_{1/2}^t u_\delta^+||_2^{1/2}||\partial_\nu u_\delta^+||_2^{1/2},
      \end{eqnarray}
      as the corresponding estimate involving $I_\delta^-$ and $u_\delta^-$ follows similarly. Based on this we in the following let, for simplicity, $u_\delta=u_\delta^+$, and we introduce
\begin{eqnarray*}
    I_\delta&=&\int_0^\infty\int_{\mathbb R^{n+1}}|D_{1/4}^t\partial_\lambda u_\delta|^2\, dxdtd\lambda,\notag\\
    II_\delta&=&\int_0^\infty\int_{\mathbb R^{n+1}}|D_{3/4}^tu_\delta|^2\, dxdtd\lambda.
          \end{eqnarray*}
Then
           \begin{eqnarray*}
    |I_\delta^+|+||D_{1/2}^tu_\delta||_2^2\leq cI_\delta^{1/2}II_\delta^{1/2}.
          \end{eqnarray*}
          We first estimate $I_\delta$. Integrating by parts with respect to $\lambda$ twice, and using Cauchy-Schwarz,  see that
                   \begin{eqnarray*}
  I_\delta\leq c\int_0^\infty\int_{\mathbb R^{n+1}}|\partial_\lambda^2 u_\delta|\, \lambda dxdtd\lambda+c\int_0^\infty\int_{\mathbb R^{n+1}}|\partial_t\partial_\lambda u_\delta|^2\, \lambda^3 dxdtd\lambda+\tilde I_\delta,
  \end{eqnarray*}
  where     \begin{eqnarray*}
  \tilde I_\delta&=&\sup_{\lambda>0}\int_{\mathbb R^{n+1}}|D_{1/4}^t\partial_\lambda u_\delta(x,t,\lambda)|^2\, \lambda dxdt\notag\\
  &&+\sup_{\lambda>0}\int_{\mathbb R^{n+1}}|D_{1/2}^t\partial_\lambda u_\delta(x,t,\lambda)|^2\, \lambda^2 dxdt.
    \end{eqnarray*}
    Hence, using Lemma \ref{lemsl1} and \eqref{korees}  we see that
    \begin{eqnarray*}
  I_\delta\leq  c\Phi_+(f)+\tilde I_\delta\leq c||f||_2^2+\tilde I_\delta.
    \end{eqnarray*}
    However,
    \begin{eqnarray*}
&&\int_{\mathbb R^{n+1}}|D_{1/4}^t\partial_\lambda u_\delta|^2\, \lambda dxdt\notag\\
&\leq&\biggl (\int_{\mathbb R^{n+1}}|\partial_\lambda u_\delta|^2\, dxdt\biggr )^{1/2}
\biggl (\int_{\mathbb R^{n+1}}|D_{1/2}^t\partial_\lambda u_\delta|^2\, \lambda^2 dxdt\biggr )^{1/2}\notag\\
&\leq&c||f||_2\biggl (\int_{\mathbb R^{n+1}}|D_{1/2}^t\partial_\lambda u_\delta|^2\, \lambda^2 dxdt\biggr )^{1/2},
    \end{eqnarray*}
   by \eqref{korees}. Similarly,
        \begin{eqnarray*}
&&\int_{\mathbb R^{n+1}}|D_{1/2}^t\partial_\lambda u_\delta|^2\, \lambda^2 dxdt\notag\\
&\leq&\biggl (\int_{\mathbb R^{n+1}}|\partial_\lambda u_\delta|^2\, dxdt\biggr )^{1/2}
\biggl (\int_{\mathbb R^{n+1}}|\partial_t\partial_\lambda u_\delta|^2\, \lambda^4 dxdt\biggr )^{1/2}\notag\\
&\leq&c||f||_2^2
    \end{eqnarray*}
    by \eqref{korees} and Lemma \ref{le5}. Put together we can conclude that
       \begin{eqnarray}\label{estqa}
  I_\delta\leq c||f||_2^2.
    \end{eqnarray}
  To estimate  $II_\delta$ we see that
    \begin{eqnarray*}
    II_\delta&=&\int_0^\infty\int_{\mathbb R^{n+1}}(H_tD^t_{1/2}u_\delta )\partial_t{u_\delta}\, dxdtd\lambda.
    \end{eqnarray*}
    Using the equation,
      \begin{eqnarray*}
    II_\delta&=&\sum_{k,m=1}^{n+1}\int_0^\infty\int_{\mathbb R^{n+1}}(H_tD^t_{1/2}u_\delta)\partial_{x_k}( A_{k,m}\partial_{x_m}{u_\delta})\, dxdtd\lambda\notag\\
    &=&\sum_{m=1}^{n+1}\int_0^\infty\int_{\mathbb R^{n+1}}(H_tD^t_{1/2}u_\delta)\partial_{x_{n+1}}
    ( A_{n+1,m}\partial_{x_m}{u_\delta})\, dxdtd\lambda\notag\\
    &&+\sum_{k=1}^{n}\int_0^\infty\int_{\mathbb R^{n+1}}(H_tD^t_{1/2}{u_\delta})\partial_{x_k}
    ( A_{k,n+1}\partial_{x_{n+1}}{u_\delta})\, dxdtd\lambda\notag\\
    &&+\int_0^\infty\int_{\mathbb R^{n+1}}(H_tD^t_{1/2}{u_\delta})\nabla_{||}\cdot (A_{||}\nabla_{||} {u_\delta})\, dxdtd\lambda\notag\\
    &=&II_{\delta,1}+II_{\delta,2}+II_{\delta,3}.
    \end{eqnarray*}
    Using that $A$ is real and symmetric, and the anti-symmetry of $H_tD^t_{1/2}$, we see that $II_{\delta,3}=0$.
    By partial integration,
    \begin{eqnarray*}
    II_{\delta,1}&=&\sum_{m=1}^{n+1}\int_0^\infty\int_{\mathbb R^{n+1}}(H_tD^t_{1/2}(u_\delta)\partial_{x_{n+1}}
    ( A_{n+1,m}\partial_{x_m}{u_\delta})\, dxdtd\lambda\notag\\
    &=&-\sum_{m=1}^{n+1}\int_0^\infty\int_{\mathbb R^{n+1}}(H_tD^t_{1/2}\partial_{x_{n+1}}{u_\delta})  A_{n+1,m}\partial_{x_m}{u_\delta}\, dxdtd\lambda\notag\\
    &&+\lim_{R\to\infty}\sum_{m=1}^{n+1}\int_{\mathbb R^{n+1}}(H_tD^t_{1/2}{u_\delta}) A_{n+1,m}\partial_{x_m}{u_\delta}\, dxdt\biggl |_{\lambda=R}\notag\\
&&-\sum_{m=1}^{n+1}\int_{\mathbb R^{n+1}}(H_tD^t_{1/2}{u_\delta}) A_{n+1,m}\partial_{x_m}{u_\delta}\, dxdt\biggl |_{\lambda=0}\notag\\
&=&II_{\delta,11}+II_{\delta,12}+II_{\delta,13}.
    \end{eqnarray*}
    Using Lemma \ref{th2ag} we see that $II_{\delta,12}=0$. Furthermore,
    \begin{eqnarray*}
|II_{\delta,13}|\leq c||H_tD^t_{1/2}u_\delta||_2||\partial_\nu u_\delta||_2.
 \end{eqnarray*}
 Next, by definition
   \begin{eqnarray*}
 II_{\delta,2}+II_{\delta,11}&=&\sum_{k=1}^{n}\int_0^\infty\int_{\mathbb R^{n+1}}(H_tD^t_{1/2}u_\delta)\partial_{x_k}
    ( A_{k,n+1}\partial_{x_{n+1}}{u_\delta})\, dxdtd\lambda\notag\\
    &&-\sum_{m=1}^{n+1}\int_0^\infty\int_{\mathbb R^{n+1}}(H_tD^t_{1/2}\partial_{x_{n+1}}u_\delta)  A_{n+1,m}\partial_{x_m}{u_\delta}\,  dxdtd\lambda.
 \end{eqnarray*}
 Hence, integrating by parts with respect  to $x_k$ in the first term, again using the anti-symmetry of $H_tD^t_{1/2}$, \eqref{eq4} and that $A$ is symmetric, we see that
    \begin{eqnarray*}
 II_{\delta,2}+II_{\delta,11}=0.
 \end{eqnarray*}
 Put together we can conclude that
        \begin{eqnarray}\label{estqb}
|II_{\delta}|\leq c||H_tD^t_{1/2}u_\delta||_2||\partial_\nu u_\delta||_2.
    \end{eqnarray}
This completes the proof of \eqref{invkey+} and hence the proof of the claim in \eqref{invkey}.

\end{document}